\documentclass[11pt,letterpaper]{article}
\usepackage{amsmath,amssymb,mathrsfs,amsthm}
\usepackage{microtype}
\usepackage{booktabs}
\usepackage{graphicx}
\usepackage{bbm}
\usepackage{blkarray}
\usepackage{enumerate}

\usepackage{bm}
\usepackage{listings}
\usepackage{caption}
\usepackage{subcaption}
\usepackage{econometrics}

\usepackage{setspace}
\usepackage{multirow}

\usepackage{tabularx,array,booktabs} 
\newcolumntype{C}{>{\centering\arraybackslash}X}

\usepackage[authoryear]{natbib}

\usepackage[colorlinks=true,linkcolor=magenta,citecolor=blue,backref=page]{hyperref}
\let \backreforig \backref
\renewcommand*{\backref}[1]{[\backreforig{#1}]}

\usepackage[toc,title,titletoc,header]{appendix}

\usepackage{hyperref}
\hypersetup{
colorlinks=true,
linkcolor=blue,
filecolor=magenta,      
urlcolor=magenta,
pdftitle={Overleaf Example},
pdfpagemode=FullScreen,
}

\usepackage[margin=1in]{geometry}
\usepackage{setspace}

\theoremstyle{definition}

\newtheorem{theorem}{Theorem}[section]

\newtheorem{lemma}{Lemma}[section]
\newtheorem{proposition}{Proposition}[section]
\newtheorem{remark}{Remark}[section]

\newtheorem{assumption}{Assumption}

\newcommand{\indep}{\perp \!\!\! \perp}

\newcommand{\V}{\mathbb{V}}
\newcommand{\Exp}{\mathbb{E}}

\allowdisplaybreaks

\usepackage{xr}
\usepackage{array}
\usepackage{arydshln}

\usepackage{lineno}
\setpagewiselinenumbers

\title{Misspecified regressions with mixed regressors:\\ 
robust inference and causal interpretation}
\author{Mengsi Gao \\
Department of Economics\\
University of Southern California\\
\url{mengsi.gao@usc.edu}
\and 
Peng Ding \\
Department of Statistics\\
UC Berkeley \\
\url{pengdingpku@berkeley.edu} \\
}
\date{\today}

\begin{document}
\maketitle

\onehalfspacing
\begin{abstract}
For analytic convenience, existing statistical frameworks either assume random or fixed regressors. However, it is a little awkward that they do not cover the practical case of estimating the average treatment effect in experiments with randomized treatments and non-randomized, fixed pretreatment covariates. We unify the literature by providing the theory for regressions with mixed regressors that contain both random and fixed components. Importantly, our theory allows for misspecification of the regression functions. We first establish general results for estimating equations with both random and fixed components and then use it to analyze misspecified linear regression, with applications to completely randomized experiments. We focus on the causal interpretation of the regression coefficients and standard errors even when the models are wrong. We start with the theory for independent data and then extend the discussion to clustered data.
\end{abstract}


\section{Introduction to statistical inference with misspecified models} \label{sec:intro}
Empirical researchers use models to extract information from data. However, models are only approximations to real world phenomena and are often misspeficied. The theory of misspecified models has been of continuing interest in statistics and econometrics \citep[e.g.,][]{Cox1961, Huber1967, White1980a, White1982, Imbens1997, AbadieImbensZheng2014, BujaBrownBerk2019, BujaBrownKuchibhotla2019}. 
A well-known result from \citet{Huber1967} and \citet{White1982} is that we must use the Huber--White (HW) (also known as robust or sandwich) standard errors when the models are misspecified.

\citet{AbadieImbensZheng2014} pointed out an interesting distinction between random regressors and fixed regressors in misspecified regressions: while the  robust standard errors are consistent with random regressors, they are only conservative with fixed regressors (see also \citet{White1983} and \citet{Chow1984}). In the theoretical literature, the choice between random regressors and fixed regressors is often driven by analytic convenience. They do not cover the practical setting of randomized experiments with randomized treatments and fixed pretreatment covariates. Motivated by this setting due to its relevance for regression-based analysis for causal inference, we unify the literature by developing the theory for misspecified regressions with mixed regressors. We show that similar to the setting of fixed regressors, the robust standard errors are conservative with mixed regressors. This constitutes our first contribution.

As \citet{Freedman2006} critically pointed out, although robust standard errors can be useful for approximating the large-sample uncertainty of the estimators based on misspecified models, they do not solve the first-order problem that the targeted parameters themselves may not be meaningful in general. 
\citet{Geer2019} made a similar point. A more optimistic quote from   
\citet{BoxDraper1987} is that ``Essentially, all models are wrong, but some are useful.'' 
To complement \citet{BoxDraper1987}, we argue that when we use models for inference, we must first verify that the parameters from misspecified models have meaningful interpretations. In linear regression, a celebrated result is that least squares gives the best linear approximation to the conditional mean function of the outcome given the regressors \citep{White1980a, AngristPischke2009, BujaBrownBerk2019}. While this is a correct mathematical statement, it does answer the question whether linear approximation is a good approximation in the first place (\citealp{Sims2010}; \citealp[Chapter~12]{Ding2023}). We report some positive results about linear regression for causal inference in randomized experiments. In particular, we analyze various regressions for causal inference in randomized experiments with covariate adjustment, and extend them to the local average treatment effect framework \citep{ImbensAngrist1994, AngristImbensRubin1996}.    
This constitutes our second contribution.

We further extend the theory to deal with clustered data. Moreover, we detail the theory for regression analysis of cluster randomized experiments \citep{BarriosDiamondImbens2012, SuDing2021, AbadieAtheyImbens2023, BugniCanayShaikh2025}. 
In particular, we propose a novel correction to the usual Liang--Zeger (LZ) cluster robust standard error \citep{LiangZeger1986, Arellano1987, AbadieAtheyImbens2023} when the covariates are treated as random in the regression with treatment-covariates interaction. This constitutes our third contribution.

Table \ref{tab:overview} summarizes the main results of our paper. 
The remainder of the paper is organized as follows. 
Section \ref{sec:ME} displays the general results of $Z$-estimation with independent and identically distributed (i.i.d.) data. 
Section \ref{sec:i.i.d.Theory} displays the results from linear regressions. 
Section \ref{sec:application} displays the results under complete randomization. 
Section \ref{sec:cluster} displays the general results of $Z$-estimation under clustered data, and apply it to cluster randomization. 
Section \ref{sec:Simulation} studies the finite-sample performance of our point and variance estimators based on simulation.
Section \ref{sec:Discussion} provides the concluding remarks. 
The Supplementary Material contains the results for the local average treatment effect framework, all proofs, and several intermediate results.

\begin{table}[!ht]
\centering
\caption{Overview of results}
\label{tab:overview}
\begin{subtable}[t]{0.95\textwidth}
\centering
\caption{i.i.d. data}
\label{tab:overview-i.i.d.}

\begin{tabularx}{\linewidth}{
>{\hsize=1.10\hsize\centering\arraybackslash}X|
>{\hsize=0.75\hsize\centering\arraybackslash}X|
>{\hsize=0.85\hsize\centering\arraybackslash}X|
>{\hsize=1.45\hsize\centering\arraybackslash}X|
>{\hsize=0.85\hsize\centering\arraybackslash}X
}
\hline
\multicolumn{2}{c|}{Setting} & Design & HW SE & Theorem \\
\hline
\multicolumn{2}{c|}{} & Random & consistent & Theorem \ref{thm:ME_random}  \\
\cline{3-5}
\multicolumn{2}{c|}{$Z$-estimation} & Fixed & conservative & Theorem \ref{thm:ME_fixed} \\
\cline{3-5}
\multicolumn{2}{c|}{} & Mixed & conservative & Theorem \ref{thm:ME_mixed} \\
\hline
\multicolumn{2}{c|}{} & Random & consistent & Theorem \ref{thm:random}  \\
\cline{3-5}
\multicolumn{2}{c|}{OLS} & Fixed & conservative & Theorem \ref{thm:fixed} \\
\cline{3-5}
\multicolumn{2}{c|}{} & Mixed & conservative & Theorem \ref{thm:AsyDist_mixed} \\
\hline
& 
no $X$ & Random & consistent & Theorem \ref{thm:y_Z_r} \\
\cline{2-5}
Complete & \multirow{2}{*}{Fisher} & Random & consistent & Theorem \ref{thm:PRA_r} \\
\cline{3-5}
Randomization & & Mixed & conservative & Theorem \ref{thm:PRA_m} \\
\cline{2-5}
& \multirow{2}{*}{Lin} & Random & anti-conservative & Theorem \ref{thm:FRA_random} \\
\cline{3-5}
& & Mixed & conservative & Theorem \ref{thm:FRA_m} \\
\hline
\end{tabularx}
\end{subtable}

\vspace{0.8em}

\begin{subtable}[t]{0.95\textwidth}
\centering
\caption{Clustered data}
\label{tab:overview-cluster}

\begin{tabularx}{\linewidth}{
>{\hsize=1.10\hsize\centering\arraybackslash}X|
>{\hsize=0.75\hsize\centering\arraybackslash}X|
>{\hsize=0.85\hsize\centering\arraybackslash}X|
>{\hsize=1.45\hsize\centering\arraybackslash}X|
>{\hsize=0.85\hsize\centering\arraybackslash}X
}
\hline
\multicolumn{2}{c|}{Setting} & Design & LZ SE & Theorem \\
\hline
\multicolumn{2}{c|}{} & Random & consistent & Theorem \ref{thm:ME_random_cluster}  \\
\cline{3-5}
\multicolumn{2}{c|}{$Z$-estimation} & Fixed & conservative & Theorem \ref{thm:ME_fixed_cluster} \\
\cline{3-5}
\multicolumn{2}{c|}{} & Mixed & conservative & Theorem \ref{thm:ME_mixed_cluster} \\
\hline
& no $X$ & Random & consistent & Theorem \ref{thm:cluster_random} \\
\cline{2-5}
Cluster & \multirow{2}{*}{Fisher} & Random & consistent & Theorem \ref{thm:cluster_add_r} \\
\cline{3-5}
Randomization & & Mixed & conservative & Theorem \ref{thm:cluster_add_m} \\
\cline{2-5}
& \multirow{2}{*}{Lin} & Random & {anti-conservative} &  Theorem \ref{thm:cluster_adj_r} \\
\cline{3-5}
& & Mixed & conservative & Theorem \ref{thm:cluster_adj_m} \\
\hline
\end{tabularx}
\end{subtable}

\vspace{0.3em}
\footnotesize
\emph{Notes:} 
``OLS'' denotes ordinary least squares; ``HW SE'' denotes the Huber--White robust standard error, and ``LZ SE'' denotes the Liang--Zeger cluster-robust standard error. 
``No $X$'' refers to the regression without covariates, 
``Fisher'' refers to additive regression adjustment \citep{Fisher1935} and ``Lin'' refers to fully interacted regression adjustment \citep{Lin2013}.
\end{table}

\paragraph{Notation}
We use $\mathbb{N}$ to denote the set of all non-negative integers.
Let $1(\cdot)$ denote the indicator function.
Let ${I}_{m}$ be a $m \times m$ identity matrix. We suppress the dimension $m$ when it is clear from the context. 
Let $\|\cdot\|$ denote the Euclidean norm, i.e., $\|w\| = \sqrt{w^{\mathrm{T}} w}$ for $w\in\mathbb{R}^v$. 
Unless stated otherwise, all vectors are assumed to be column vectors.
Let $X$, $Y$, and $\varepsilon$ be the $n \times K$ matrix with $i$th row equal to $x_i^{\mathrm{T}}$, the $n$-vector with $i$th element equal to $y_i$, and the $n$-vector with $i$th element equal to $\varepsilon_i$, respectively.
Define $\dot{x}_i = x_i - \Exp(x_i)$ and $\ddot{x}_i = x_i - \bar{x}$ where $\bar{x} = n^{-1}\sum_{i=1}^n x_i$, following the notation in \cite{NegiWooldridge2021}. 
Let $[\cdot]_{(a,a)}$ denote the $(a,a)$th element of the matrix inside $[\cdot]$. 
We use $\text{lm}(y_i \sim x_i)$ to denote the least-squares
regression of $y_i$ on $x_i$ and focus on the associated Eicker--Huber--White (EHW) variance estimator and $\text{lm}(y_{ij} \sim x_{ij})$ to denote the least-squares
regression of $y_{ij}$ on $x_{ij}$ and focus on the associated LZ variance estimator clustered at the level of cluster $i$.
The terms ``regression'' and ``EHW variance'' refer to the numerical outputs of the least-squares fit without any modeling assumptions.
Let \(\mathbb{E}_\circ[\cdot]\) denote the relevant expectation operator (unconditional or conditional depending on the design): 
\(\mathbb{E}_\circ = \mathbb{E}_{x}\) refers to the expectation under the joint distribution of \(x\), while
\(\mathbb{E}_\circ = \mathbb{E}_{x_1|x_2}\) refers to the expectation under the conditional distribution of \(x_1\) given \(x_2\).
The same convention applies to the probability measure $\mathbb{P}_\circ$.
We use $o(1;\mathbb{P}_{\circ})$ as a less cluttered notation for
$o_{\mathbb{P}_{\circ}}(1)$, denoting a sequence of random variables
that converges to zero in $\mathbb{P}_{\circ}$-probability.
We use \(\V\) to denote variance.
To present our asymptotic results, we introduce the notion of conditional convergence in distributions.
Let $\mathcal{L}(t;W \mid X)=\mathbb{P}(W \le t \mid X)$ denote the cumulative distribution function of the random variable $W$.
We say that $W_n \mid X_n \overset{\textup{d}}{\to} W $ a.s.
if $\limsup_{n\to\infty}\;\sup_{t\in\mathbb{R}}
\bigl|\mathcal{L}\bigl(t;W_n \mid X_n\bigr)
-\mathcal{L}\bigl(t;W\bigr)\bigr|
=0$, $\mathbb{P}$-a.s.
We omit the measure-theoretic language ``a.s." below.
Let $v^{\otimes 2} = vv^{\mathrm{T}}$ denote the outer product of a vector $v$.

\section{Robust inference based on $Z$-estimation} \label{sec:ME}
We begin with the familiar random-design setting with i.i.d. observations and establishes a benchmark in Section \ref{sec:Z_random}. Because this is the standard framework for \(Z\)-estimation, we use it to introduce the basic notation and variance estimator before turning to fixed and mixed designs.
Sections \ref{sec:Z_fixed} and \ref{sec:Z_mixed} then develop the corresponding fixed- and mixed-design results. The key distinction across these designs is the source of randomness: random-design inference is unconditional, fixed-design inference conditions on all covariates, and mixed-design inference conditions on only a subset of the covariates.

\subsection{Random design}\label{sec:Z_random}

Under random design, we observe independent observations \( w_i \), and the parameter of interest \( \beta^\textup{r} \) is identified through a population moment condition
\[
\mathbb{E}[\psi(w_i; \beta^\textup{r})] = 0,
\]
where we use $\beta^\textup{r}$ to denote the estimand under random design.
We focus on the class of estimators \( \hat{\beta} \) that can be written as the unique solution to the corresponding sample estimating equations:
\begin{equation}
\bar{\psi}(W;b) =
\frac{1}{n}\sum_{i=1}^n \psi(w_i;b) = 0,
\label{eq:ME}
\end{equation}
where $W = (w_1,\cdots, w_n)$ collects all observations.
A first-order Taylor
expansion of $\bar{\psi}(W;\hat{\beta})$ around $\beta^\textup{r}$ gives
\[
0=\bar{\psi}(W;\hat{\beta}) \approx \bar{\psi}(W;{\beta}^\textup{r}) +\Gamma(\beta^\textup{r})\,(\hat\beta-\beta^\textup{r}),
\]
where ${\Gamma}(\beta^\textup{r})
= \left.\frac{\partial}{\partial b^{\mathrm{T}}} \bar{\psi}(W;{b}) \right|_{b=\beta^\textup{r}}$.
This yields the linearization of $Z$-estimator in \eqref{eq:ME} as $\hat{\beta}
\approx \beta^\textup{r} - \Gamma(\beta^\textup{r})^{-1}\bar{\psi}(W;{\beta}^\textup{r})$, 
and the HW variance estimator for the asymptotic variance of $\hat{\beta}$:
\begin{equation}
\hat{V}_{\textsc{hw}}
= n^{-1} 
\hat{\Gamma} ^{-1}
\hat{\Delta}
\hat{\Gamma} ^{-\mathrm{T}} 
\label{eq:EHW_ME}
\end{equation}
where 
\begin{align*}
\hat{\Gamma}
= \left.\frac{\partial}{\partial b^{\mathrm{T}}} \bar{\psi}(W;{b}) \right|_{b=\hat{\beta}}
\quad \text{and} \quad
\hat{\Delta}
= \frac{1}{n} \sum_{i=1}^n \psi(w_i;\hat{\beta})\psi(w_i;\hat{\beta})^{\mathrm{T}}.
\end{align*}
We use the label HW for the general $Z$-estimation setting.
Maximum likelihood estimation with possibly misspecified models  is a special case of this framework with $\psi$ as the score function \citep{Cox1961,Huber1967,White1980a, White1982}; we therefore focus on the general $Z$-estimation.

We apply the general $Z$-estimation results to linear regression in Section~\ref{sec:i.i.d.Theory} and to randomized experiments in Section~\ref{sec:application}. 
The framework accommodates not only Lin-type fully interacted regression adjustment, but also the IV regression, with the local average treatment effect as a leading application of the latter \citep{ImbensAngrist1994}. Throughout this section, we write $w_i=(y_i,x_i)$, where $y_i$ is the response and $x_i$ denotes the covariates.
Specifically, we define the estimand $\beta^\textup{r}$ as the unique solution to the estimating equation:
\[
\mathbb{E}\left[\psi(y_i, x_i;\beta^\textup{r})\right] = 0.  
\]
Define the asymptotic variance of $\hat{\beta}$ under random design as
$V^\textup{r} 
= (\Gamma^\textup{r})^{-1}\Delta^\textup{r}(\Gamma^{\textup{r} })^{-\mathrm{T}}$, with 
\begin{align*}
\Gamma^\textup{r}
&= \mathbb{E}\left[ \frac{\partial \psi(y_i, x_i;\beta^\textup{r})}{\partial b^{\mathrm{T}}}  \right] 
\text{ and }
\Delta^\textup{r} 
= \mathbb{E}\left[\psi(y_i, x_i;\beta^\textup{r})\psi(y_i, x_i;\beta^\textup{r})^{\mathrm{T}}\right]. 
\end{align*}
More generally, to present the three cases in a unified notation, we use \(\beta^{\diamond}\) to denote the estimand under different designs, and \(\mathbb{E}_{\circ}\) and \(\mathbb{P}_{\circ}\) to denote the corresponding expectation operator and probability measure, respectively.
Specifically, let \(\diamond = \textup{r}\) under a random design with \(\circ = (y, x)\); 
\(\diamond = \textup{f}\) under a fixed design with \(\circ = y| x \); 
and \(\diamond = \textup{m}\) under a mixed design with \(\circ = (y, x_1) | x_2\). 
This notation allows us to present the following assumption uniformly across the three design regimes.

\begin{assumption} \label{asu:ME}
Let \(\beta^{\diamond} \in \Theta \subset \mathbb{R}^p\) be the target parameter. 
Suppose that $w_i$, $i=1,\ldots,n$, are i.i.d.. 
\begin{enumerate}[(a)] 
\item The parameter space $\Theta$ is compact, \(\beta^{\diamond}\) lies in the interior of the  \(\Theta\), and $\beta^{\diamond}$ is the unique solution to \(\mathbb{E}_\circ[\psi(w, b)] = 0\).

\item The function $\psi(w,b)$ is continuous in $b\in\Theta$ almost
surely, and $\mathbb{E}\left[\sup_{b\in\Theta}\|\psi(w,b)\|\right]<\infty$.

\item There exists a neighborhood $\mathscr{N}$ of $\beta^\diamond$
such that $\psi(w,b)$ is continuously differentiable in $b\in \mathscr{N}$ almost surely.

\item $\mathbb{E}_\circ\left[
    \sup_{b\in\mathscr{N}}\|\psi(w,b)\|^{2+\delta}
\right]<\infty$ for some $\delta>0$ and $\mathbb{E}_\circ\left[
    \sup_{b\in\mathscr{N}}\|\nabla_b\psi(w,b)\|
\right]<\infty$.

\item \(\Gamma^\diamond := \mathbb{E}_\circ[\nabla_b \psi(w, \beta^{\diamond})]\) is nonsingular.
\end{enumerate}
\end{assumption}

Theorem \ref{thm:ME_random} below reviews the asymptotic normality of the $Z$-estimator $\hat{\beta}$ and the consistency of the HW variance estimator $\hat{V}_{\textsc{hw}}$ in \eqref{eq:EHW_ME} under random design.

\begin{theorem}\label{thm:ME_random}
Assume that $W = \{y_i, x_i\}_{i=1}^n$ are i.i.d..
Under random design and Assumption \ref{asu:ME} with \(\beta^{\diamond} = \beta^{\textup{r}}\), and \(\mathbb{E}_\circ = \mathbb{E}_{(y,x)}\) and \(\mathbb{P}_\circ = \mathbb{P}_{(y,x)}\), we have 
\[
(V^\textup{r})^{-1/2}
\sqrt{n}(\hat{\beta} - \beta^\textup{r}) 
\overset{\textup{d}}{\to} \mathcal{N}\left(0, I \right)
\text{ and } 
n \hat{V}_{\textsc{hw}} = V^\textup{r} + o(1;\mathbb{P}_{(y,x)}).
\] 
\end{theorem}
Theorem \ref{thm:ME_random} is the standard random-design result for $Z$-estimation equations; see, for example, \cite{NeweyMcFadden1994}. We include it to establish notation and to serve as the benchmark for the fixed- and mixed-design results below.

\subsection{Fixed design} \label{sec:Z_fixed}
Under fixed design, the $x_i$'s are fixed, or, equivalently, we condition on them.
Let $X = (x_1, \ldots, x_n)^{\mathrm{T}}.$
Define the estimand $\beta^\textup{f}$ as the solution to the estimating equations:
\[
\mathbb{E}\left[ \frac{1}{n} \sum_{i=1}^n \psi(y_i,x_i;\beta^\textup{f} ) \mid {X} \right] = 0.   
\]
Define the conditional variance $n\cdot \V(\hat{\beta} \mid X)$ as $V^\textup{f} = (\Gamma^\textup{f})^{-1}\Delta^\textup{f}(\Gamma^{\textup{f} })^{-\mathrm{T}}$ with 
\begin{align*}
\Gamma^\textup{f} 
&= \frac{1}{n} \sum_{i=1}^n \mathbb{E}\left[  \frac{\partial \psi(y_i,x_i;\beta^\textup{f} )}{\partial b^{\mathrm{T}}}  \mid x_i \right] 
\text{ and }
\Delta^\textup{f} 
= \frac{1}{n} \sum_{i=1}^n \V\left[ \psi(y_i,x_i;\beta^\textup{f}) \mid x_i \right],
\end{align*}
and the asymptotic bias of $\hat{V}_{\textsc{ehw}}$ in \eqref{eq:EHW_ME} for $V^\textup{f}$ as
\begin{equation}
B^{\textup{f}}
= (\Gamma^\textup{f})^{-1} \left( \frac{1}{n} \sum_{i=1}^n \left( \mathbb{E}\left[ \psi(y_i,x_i;\beta^\textup{f}) \mid x_i \right] \right)^{\otimes 2}  \right) (\Gamma^{\textup{f} })^{-\mathrm{T}}. 
\label{eq:bias_me_fixed}
\end{equation}

\begin{theorem}\label{thm:ME_fixed}
Assume that $\{y_i, x_i\}_{i=1}^n$ are i.i.d..
Under fixed design and Assumption \ref{asu:ME} with \(\beta^{\diamond} = \beta^{\textup{f}}\), \(\mathbb{E}_\circ = \mathbb{E}_{y|x}\) and \(\mathbb{P}_\circ = \mathbb{P}_{y|x}\), we have 
\[
(V^\textup{f})^{-1/2}
\sqrt{n}(\hat{\beta} - \beta^\textup{f}) \mid X
\overset{\textup{d}}{\to} \mathcal{N}\left(0, I \right)
\text{ and }
n \hat{V}_{\textsc{hw}} 
= V^\textup{f} + B^{\textup{f}}
+ o(1;\mathbb{P}_{y|x}).
\] 
\end{theorem}

Theorem \ref{thm:ME_fixed} is the fixed-design counterpart of the standard
$Z$-estimation result in Theorem \ref{thm:ME_random}. 
Theorem \ref{thm:ME_fixed} indicates that under fixed design, $\hat{V}_{\textsc{hw}}$ in \eqref{eq:EHW_ME} is conservative with asymptotic bias $B^{\textup{f}}$ in \eqref{eq:bias_me_fixed}.
\citet{AbadieImbensZheng2014} also study general $Z$-estimators under misspecification and distinguish between population and covariate-conditional estimands. Their asymptotic results, however, are formulated under i.i.d.\ sampling and unconditional inference, whereas
our fixed-design analysis conditions on the realized regressors.

\subsection{Mixed design} \label{sec:Z_mixed}
We partition all the regressors $X= (X_1, X_2)$. Under mixed design, our analysis conditions on part of the regressors, i.e., ${X}_2$, and take the remaining regressors ${X}_1$ as random.
Define the estimand $\beta^\textup{m}$ as the solution to the estimating equations:
\[
\mathbb{E}\left[\frac{1}{n} \sum_{i=1}^n \psi(y_i,x_i;\beta^\textup{m})\mid {X}_2 \right] = 0. 
\]
Define the conditional variance $n\cdot \V(\hat{\beta} \mid X_2)$ as
$V^\textup{m} = (\Gamma^\textup{m})^{-1}\Delta^\textup{m}(\Gamma^{\textup{m} })^{-\mathrm{T}}$ with 
\begin{align*}
\Gamma^\textup{m} 
&= \frac{1}{n} \sum_{i=1}^n \mathbb{E}\left[  \frac{\partial \psi(y_i,x_i;\beta^\textup{m} )}{\partial b^{\mathrm{T}}}  \mid x_{i2} \right] 
\text{ and }
\Delta^\textup{m} 
= \frac{1}{n} \sum_{i=1}^n \V\left[ \psi(y_i,x_i;\beta^\textup{m}) \mid x_{i2} \right]
\end{align*}
and the asymptotic bias of $\hat{V}_{\textsc{ehw}}$ in \eqref{eq:EHW_ME} for $V^\textup{m}$ as 
\begin{equation}
B^\textup{m}
= (\Gamma^\textup{m})^{-1} \left( \frac{1}{n} \sum_{i=1}^n \left( \mathbb{E}\left[ \psi(y_i,x_i;\beta^\textup{m}) \mid x_{i2} \right] \right)^{\otimes 2}  \right) 
(\Gamma^{\textup{m} })^{-{\mathrm{T}}}.
\label{eq:bias_me_mixed}
\end{equation}
\begin{theorem} \label{thm:ME_mixed}
Assume that $\{y_i, x_i\}_{i=1}^n$ are i.i.d..
Under mixed design and Assumption \ref{asu:ME} with \(\beta^{\diamond} = \beta^{\textup{m}}\), \(\mathbb{E}_\circ = \mathbb{E}_{(y,x_1)|x_2}\) and \(\mathbb{P}_\circ = \mathbb{P}_{(y,x_1)|x_2}\), we have 
\begin{equation*}
(V^\textup{m})^{-1/2} \sqrt{n}(\hat{\beta} - \beta^\textup{m}) \mid X_2
\overset{\textup{d}}{\to} \mathcal{N}\left(0, I \right)
\text{ and }
n \hat{V}_{\textsc{hw}}
=  V^\textup{m}
+ B^\textup{m}
+ o(1;\mathbb{P}_{(y,x_1)\mid x_2}). 
\end{equation*}
\end{theorem}
Theorem \ref{thm:ME_mixed} generalizes the random- and fixed-design benchmarks to mixed designs, where only part of the regressors are conditioned on. This formulation is useful for settings such as randomized experiments with fixed or pre-determined covariates. Under mixed design, $\hat{V}_{\textsc{hw}}$ remains conservative for estimating $V^\textup{m}$ under misspecification, with the additional term $B^{\textup{m}}$ in \eqref{eq:bias_me_mixed} capturing the approximation-error component induced by conditioning on \(X_2\).
In the special case in which either $X_2$ is empty or $\mathbb{E}\!\left[
\psi(y_i,x_i;\beta^{\textup{m}})
\mid x_{i2}
\right]
=0$, the HW variance estimator $\hat{V}_{\textsc{hw}}$ is consistent rather than strictly conservative.

\section{OLS estimation as a special case}
\label{sec:i.i.d.Theory}
We now specialize the framework in Section \ref{sec:ME} to OLS. The structure of this section parallels that of Section \ref{sec:ME}.
Assume:
\begin{equation}
y_i = x_i^{\mathrm{T}}\beta + \varepsilon_i 
\label{eq:y_x}
\end{equation}
with $y_i$ being the outcome of interest, $x_i$ a $K$-vector of observed
covariates, possibly including an intercept, and $\varepsilon_i$ an unobserved
error. 
If $\sum_{i=1}^n x_ix_i^{\mathrm{T}}>0$, the OLS estimator $\hat{\beta}$ solves 
\begin{equation}
\frac{1}{n} \sum_{i=1}^n \psi(y_i,x_i; b)
= \frac{1}{n} \sum_{i=1}^n x_i(y_i-x_i^{\mathrm{T}} b)
= 0. \label{eq:beta_ols}
\end{equation}
Define the residual from the OLS fit as $\hat{\varepsilon}_i = y_i-x_i^{\mathrm{T}}\hat{\beta}$. 
Specializing the general sandwich estimator $\hat{V}_{\textsc{hw}}$ in \eqref{eq:EHW_ME} to
the OLS estimating equations in \eqref{eq:beta_ols} gives the Eicker--Huber--White (EHW)
variance estimator \citep{Eicker1967,Huber1967,White1980a}:
\begin{equation}
\hat{V}_{\textsc{ehw}}
= \frac{1}{n} \left( \frac{1}{n}\sum_{i=1}^n x_ix_i^{\mathrm{T}} \right)^{-1}
\left(\frac{1}{n}\sum_{i=1}^n \hat{\varepsilon}_i^2 x_ix_i^{\mathrm{T}} \right)
\left(\frac{1}{n}\sum_{i=1}^n x_ix_i^{\mathrm{T}} \right)^{-1}.  
\label{eq:EHW}
\end{equation}
We use the label EHW to emphasize that, for OLS, the general
HW sandwich estimator coincides with \cite{Eicker1967}'s heteroskedasticity-robust covariance estimator.

In this section, we impose the following assumption. 
\begin{assumption}\label{asu:i.i.d.}
\begin{enumerate}[{(a)}]
\item The variables $(y_i, x_i)$, $i=1, \ldots, n$, are i.i.d.;
\item $\mathbb{E}_\circ(y_i^4)<\infty$;
\item $\mathbb{E}_\circ\|x_i\|^4<\infty$;
\item $\mathbb{E}_\circ(x_ix_i^{\mathrm{T}})$ is positive definite.
\end{enumerate}
\end{assumption}
Below, we study the performance of the OLS estimator $\hat{\beta}$ in \eqref{eq:beta_ols} and $\hat{V}_{\textsc{ehw}}$ in \eqref{eq:EHW} under different sources of randomness.

\subsection{Random Design}\label{random}
Under random design, the estimand of interest $\beta^{\textup{r}}$ is the solution to the estimating equations:
\begin{align}
0 = \frac{1}{n} \sum_{i=1}^n \mathbb{E} \left[  x_i (y_i-x_i^{\mathrm{T}} b) \right].
\end{align}
Equivalently, $\beta^{\textup{r}}$ is the population OLS projection coefficient because our theory does not assume the linear model \eqref{eq:y_x} to be correctly specified. 
Define the projection error $\varepsilon_i^{\textup{r}} = y_i-x_i^{\mathrm{T}}\beta^{\textup{r}}$. 
The OLS estimator $\hat{\beta}$ in \eqref{eq:beta_ols} is biased for $\beta^{\textup{r}}$ since $\mathbb{E}(\hat{\beta}-\beta^{\textup{r}})\ne 0$, but is consistent for $\beta^{\textup{r}}$ by the law of large numbers under Assumption \ref{asu:i.i.d.}. 
Define 
\begin{align}
V^{\textup{r}} = \mathbb{E}\left[x_ix_i^{\mathrm{T}}\right]^{-1}\mathbb{E}\left[(\varepsilon_i^{\textup{r}})^2 x_ix_i^{\mathrm{T}}\right]\mathbb{E}\left[x_ix_i^{\mathrm{T}}\right]^{-1}.
\label{eq:Vr}
\end{align}
\begin{theorem}\label{thm:random}
Under random design and Assumption \ref{asu:i.i.d.} with \(\mathbb{E}_\circ = \mathbb{E}_{(y,x)}\), we have
\begin{equation}
(V^{\textup{r}})^{-1/2} \sqrt{n}(\hat{\beta} - \beta^{\textup{r}}) \overset{\textup{d}}{\to} \mathcal{N}(0, I) 
\text{ and } 
n \hat{V}_{\textsc{ehw}} = V^{\textup{r}} + o(1;\mathbb{P}_{(y,x)}).
\end{equation} 
\end{theorem}
Theorem \ref{thm:random} states that the EHW variance estimator $\hat{V}_{\textsc{ehw}}$ in \eqref{eq:EHW} is consistent for estimating $V^{\textup{r}}$ in \eqref{eq:Vr}.
Theorem \ref{thm:random} is the standard random-design OLS result under possible misspecification; see \cite{White1980, White1982} and \cite{NeweyMcFadden1994}.

\subsection{Fixed Design}\label{fixed}
Under fixed design, we condition on all regressors $X$.
Define the estimand $\beta^{\textup{f}}$ as the solution to the estimating equations:
\begin{align}
0 = \frac{1}{n} \sum_{i=1}^n \mathbb{E} \left(  x_i (y_i-x_i^{\mathrm{T}} b) \mid x_i \right). \label{eq:beta_f}
\end{align}
When conditional on ${X}$, $\hat{\beta}$ is an unbiased estimator for $\beta^{\textup{f}}$ since $\mathbb{E}(\hat{\beta}-\beta^{\textup{f}} \mid {X})=0$, and also consistent for $\beta^{\textup{f}}$ by the law of large numbers under Assumption \ref{asu:i.i.d.}. 
Define the conditional variance $n\cdot \V(\hat{\beta} \mid X)$ as
\begin{equation}
V^{\textup{f}} 
= \left( \frac{1}{n}\sum_{i=1}^n x_ix_i^{\mathrm{T}} \right)^{-1}
\left( \frac{1}{n}\sum_{i=1}^n \sigma^2(x_i)x_ix_i^{\mathrm{T}} \right)
\left( \frac{1}{n}\sum_{i=1}^n x_ix_i^{\mathrm{T}} \right)^{-1},
\label{eq:Vf}
\end{equation}
and the asymptotic bias of $\hat{V}_{\textsc{ehw}}$ in \eqref{eq:EHW} for $V^{\textup{f}}$ as
\begin{equation}
B^{\mathrm{f}}
= \left( \frac{1}{n}\sum_{i=1}^n x_ix_i^{\mathrm{T}} \right)^{-1}
\left( \frac{1}{n}\sum_{i=1}^n \mathbb{E}[\varepsilon_i^{\textup{f}}\mid x_{i}]^2 x_i x_i^{\mathrm{T}} \right)
\left( \frac{1}{n}\sum_{i=1}^n x_ix_i^{\mathrm{T}} \right)^{-1}
\label{bias:OLS_f}
\end{equation}
where $\varepsilon_i^{\textup{f}} = y_i-x_i^{\mathrm{T}}\beta^{\textup{f}}$.
\begin{theorem}\label{thm:fixed}
Under fixed design and Assumption \ref{asu:i.i.d.} with \(\mathbb{E}_\circ = \mathbb{E}_{y|x}\), we have
\begin{align*}
(V^{\textup{f}})^{-1/2} \sqrt{n}(\hat{\beta} - \beta^{\textup{f}}) \mid {X}
\overset{\textup{d}}{\to} \mathcal{N} (0, I)
\text{ and }
n \hat{V}_{\textsc{ehw}} 
= V^{\textup{f}} + B^{\mathrm{f}} + o(1;\mathbb{P}_{y|x}).
\end{align*}
\end{theorem}
Theorem \ref{thm:fixed} states that under fixed design, the EHW variance estimator $\hat{V}_{\textsc{ehw}}$ in \eqref{eq:EHW} is a conservative estimator for $V^{\textup{f}}$ in \eqref{eq:Vf} with asymptotic bias $B^{\mathrm{f}}$ in \eqref{bias:OLS_f}.
Theorem \ref{thm:fixed} is the OLS specialization of the mixed-design $Z$-estimation result in Theorem \ref{thm:ME_fixed}.

\citet{AbadieImbensZheng2014} also study the covariate-conditional OLS
estimand and derive the corresponding asymptotic distribution. 
Their
analysis is unconditional, so the realized design $X$, and hence the
conditional estimand $\beta^{\mathrm f}(X)$, varies across repeated
samples. By contrast, Theorem~\ref{thm:fixed} establishes asymptotic
normality conditional on the realized design $X$.
\cite{AbadieImbensZheng2014} show that the conditional variance of the OLS estimator $\hat{\beta}$ converges in probability to
\[
V_{\text {cond }}
= \operatorname{plim}(n \cdot \V(\hat{\beta} \mid {X}))
= \left(\mathbb{E}\left[x_i x_i^{\mathrm{T}}\right]\right)^{-1} 
\left(\mathbb{E}\left[\sigma^2\left(x_i\right) x_i x_i^{\mathrm{T}} \right]\right)\left(\mathbb{E}\left[x_i x_i^{\mathrm{T}}\right]\right)^{-1}.
\]
The middle part of $V^{\textup{r}}$ can be decomposed into  conditional variance and approximation error:
\[
\mathbb{E}\left[(y_i-x_i^{\mathrm{T}}\beta^{\textup{r}})^2 \mid x_i \right]
= \mathbb{E}\left[(y_i-\mu(x_i)+\mu(x_i)-x_i^{\mathrm{T}}\beta^{\textup{r}})^2 \mid x_i \right]
= \sigma^2(x_i) + (\mu(x_i)-x_i^{\mathrm{T}}\beta^{\textup{r}})^2,
\]
where $\sigma^2(x_i)= \V(y_i \mid x_i)$ and $\mu(x_i)=\mathbb{E}(y_i \mid x_i)$.
Comparing $V_{\text {cond }}$ with $V^{\textup{r}}$ shows that the latter is generally larger: $V^{\textup{r}} = V_{\text {cond }} + \V(\beta^{\textup{f}})$
where $\V(\beta^{\textup{f}}) = \operatorname{plim}n \cdot \mathbb{E}[(\beta^{\textup{f}} - \beta^{\textup{r}})(\beta^{\textup{f}} - \beta^{\textup{r}})^{\mathrm{T}}]$. The variance estimator $\hat{V}_{\textsc{ehw}}$ is conservative for $V^{\textup{f}}$ since $\beta^{\textup{f}}$ cannot be consistently estimated.

\subsection{Mixed Design}\label{mixed}
We partition the regressors as \(X = (X_1, X_2)\). Under the mixed-design setting, our analysis conditions on part of the regressors, namely \(X_2\), while treating the remaining regressors \(X_1\) as random. A leading application of this framework arises in randomized controlled trials, where \(X_1\) represents the randomized treatment and \(X_2\) consists of fixed or pretreatment covariates.
The estimand $\beta^{\textup{m}}$ is therefore defined as the solution to the estimating equations:
\begin{align}
0 = \frac{1}{n} \sum_{i=1}^n \mathbb{E} \left(  x_i (y_i-x_i^{\mathrm{T}} b) \mid x_{i2} \right). 
\label{eq:beta_m}
\end{align}
With $\varepsilon_i^{\textup{m}} = y_i-x_i^{\mathrm{T}}\beta^{\textup{m}}$, define the conditional variance $n\cdot \V(\hat{\beta} \mid X_2)$ as
\begin{equation}
V^{\mathrm{m}} 
= \left( \frac{1}{n} \sum_{i=1}^n \mathbb{E}(x_ix_i^{\mathrm{T}} \mid x_{i2}) \right)^{-1}
\frac{1}{n} \sum_{i=1}^n \V(x_i \varepsilon_i^{\textup{m}} \mid x_{i2})
\left( \frac{1}{n} \sum_{i=1}^n \mathbb{E}(x_ix_i^{\mathrm{T}} \mid x_{i2}) \right)^{-1}, 
\label{eq:Vm}
\end{equation}
and the asymptotic bias of $\hat{V}_{\textsc{ehw}}$ in \eqref{eq:EHW} for $V^{\mathrm{m}} $ as
\begin{equation}
B^{\textup{m}}
= \left( \frac{1}{n}\sum_{i=1}^n \mathbb{E}(x_ix_i^{\mathrm{T}}\mid x_{i2}) \right)^{-1}
\left( \frac{1}{n}\sum_{i=1}^n \left( \mathbb{E}[x_i \varepsilon_i^{\textup{m}}\mid x_{i2}] \right)^{\otimes 2} \right)
\left( \frac{1}{n}\sum_{i=1}^n \mathbb{E}(x_ix_i^{\mathrm{T}}\mid x_{i2}) \right)^{-1}. \label{bias_OLS_m}
\end{equation}
\begin{theorem}\label{thm:AsyDist_mixed}
Under mixed design and Assumption \ref{asu:i.i.d.} with \(\mathbb{E}_\circ = \mathbb{E}_{(y,x_1)|x_2}\), we have
\[
(V^{\textup{m}})^{-1/2} \sqrt{n}(\hat{\beta} - \beta^{\textup{m}}) \mid {X}_2
\overset{\textup{d}}{\to}
\mathcal{N}\left(0, I \right) 
\text{ and } 
n \hat{V}_{\textsc{ehw}} 
= V^{\mathrm{m}} 
+ B^{\mathrm{m}} + o(1;\mathbb{P}_{(y,x_1)\mid x_2}).
\]
\end{theorem}
Theorem \ref{thm:AsyDist_mixed} indicates that under mixed design, the EHW variance estimator $\hat{V}_{\textsc{ehw}}$ in \eqref{eq:EHW} is conservative for $V^{\mathrm{m}}$ in \eqref{eq:Vm} with asymptotic bias $B^{\mathrm{m}}$ in \eqref{bias_OLS_m}.
Theorem \ref{thm:AsyDist_mixed} is the OLS specialization of the mixed-design $Z$-estimation result in Theorem \ref{thm:ME_mixed}, and the result is new in the literature.
Related work by \cite{ChetverikovHahnLiao2023} studies OLS when a regressor of interest is randomly assigned, but does not consider our conditional mixed-design framework or the resulting conservativeness of the EHW variance estimator under misspecification.

\section{Interpretation of misspecified regressions for causal inference} \label{sec:application}

\citet{Freedman2008} and \citet{Geer2019} question the relevance of
inference based on $\hat{V}_{\textsc{hw}}$ when the target parameters lack a
meaningful interpretation under model misspecification. To address this
concern, we provide a detailed discussion of the causal interpretations of
coefficients from regressions commonly used in causal inference when the
working linear models may be misspecified. Our main objective is not merely
to derive asymptotic variance formulas, which follow from the general results 
in Sections~\ref{sec:ME} and \ref{sec:i.i.d.Theory}, but to clarify which
causal parameters the regression coefficients identify.
We study completely randomized experiments and focus on the comparison of random and mixed designs with covariates in Section \ref{sec:FishRCT}, Section \ref{sec:LinRCT} because treatment $Z$ is randomized. 
We do not consider inference conditional on $Z$ because the coefficient does not have causal interpretation if the linear model is incorrect. For this reason, we omit the fixed-design analysis, since it does not quantify the advantages of randomization with a misspecified linear model. 
We focus here on OLS and relegate the IV results for the local average treatment effect framework to Section \ref{sec:IV} in the Supplementary Material .

Consider an experiment with a binary treatment $Z_i\in\{0,1\}$.
Let $y_i(z)$
denote the potential outcome of unit $i$ under treatment status
$z\in\{0,1\}$. This notation allows us to characterize the causal
interpretation of the OLS coefficient even when the working linear model is
misspecified.
The observed outcome $y_i$ is connected with potential outcomes in the usual manner: $y_i = Z_i Y_i(1) + (1-Z_i)Y_i(0)$.
Researchers also observe the covariate vector $x_i = (x_{i1}, \ldots, x_{iJ})$ for $i = 1, \ldots, n$.

In this section, we impose the following assumption. 
\begin{assumption} \label{asu:completeRandom}
\begin{enumerate}[(a)]
\item The variables $(Y_i(1), Y_i(0), Z_i, x_i)$, $i=1, \ldots, n$, are i.i.d.;
\item $(Y_i(0), Y_i(1), x_i) \indep Z_i$;
\item $\mathbb{P}(Z_i=1)= e \in(0,1)$;
\item $\mathbb{E}_\circ[ |Y_{i}(z)|^{4} ] < \infty $ for all $z \in \{0,1\}$ and $\mathbb{E}_\circ[ \left\|x_{i}\right\|^{4} ] < \infty $ for some $\delta > 0$;
\item $\mathbb{E}_\circ(x_ix_i^{\mathrm{T}})$ is positive definite.
\end{enumerate}

\end{assumption}

\subsection{Simple difference in means and its regression implementation}
To estimate the average treatment effect (ATE), we consider the linear regression: 
\begin{equation}
\text{lm}(y_i \sim 1 + Z_i). 
\label{eq:YonZ}
\end{equation}
Let $\hat{\beta}$ denote the coefficient of $Z_i$ from the above OLS fit
and $\hat{\varepsilon}_{i}$ denote the residual from the same OLS fit. 
Define $\hat{V}_{\textsc{ehw}}$ as the $(2,2)$ entry of the EHW variance estimator in the form of \eqref{eq:EHW} with regressors $x_{i} = (1, Z_i)^{\mathrm{T}}$; this is the variance of the coefficient on $Z_i$.

The estimands are $(\alpha^{\textup{r}}, (\beta^{\textup{r}})^{\mathrm{T}} )^{\mathrm{T}} = \mathbb{E}(x_{i} x_{i}^{\mathrm{T}} )^{-1}\mathbb{E}(x_{i} y_i )$.
By the general theory in Section \ref{random}, we derive the asymptotic variance of $\hat{\beta}$ as
\begin{equation}
V^{\textup{r}}
= \frac{\V(\varepsilon^{\textup{r}}_i(1))}{e} + \frac{\V(\varepsilon^{\textup{r}}_i(0))}{1-e}
\label{eq:V_OLS_RCT}
\end{equation}
with $\varepsilon^{\textup{r}}_i(z) = y_i(z) - \Exp(y_i(z))$ for $z\in \{0,1\}$.

\begin{theorem} \label{thm:y_Z_r}
Under random design and Assumption \ref{asu:completeRandom} with \(\mathbb{E}_\circ = \mathbb{E}_{(y,Z)}\), we have $\beta^{\textup{r}} = \mathbb{E}(Y_i(1) - Y_i(0) )$ and
\begin{equation*}
(V^{\textup{r}})^{-1/2}\sqrt{n}(\hat{\beta} - \beta^{\textup{r}} ) 
\overset{\textup{d}}{\to} 
\mathcal{N}(0, 1) 
\text{ and } 
n \hat{V}_{\textsc{ehw}} = V^{\textup{r}} + o(1;\mathbb{P}_{(y,Z)}).
\end{equation*}
\end{theorem}
The coefficient of $Z_i$ from the projection of $y_i$ on $x_{i}$ identifies the ATE.
Theorem \ref{thm:y_Z_r} also indicates that under random design, $\hat{V}_{\textsc{ehw}}$ is consistent for estimating $V^{\textup{r}}$ in \eqref{eq:V_OLS_RCT}.
Theorem \ref{thm:y_Z_r} is the standard difference-in-means result for randomized experiments; see \cite{Wooldridge2020}.

\subsection{Fisher's analysis of covariance} 
\label{sec:FishRCT}
We consider the additive regression: 
\begin{equation}
\text{lm}(y_i \sim 1 + Z_i + x_i). 
\label{eq:yonZx}
\end{equation}
Let $\hat{\beta}_{\textsc{f}}$ denote the coefficient of $Z_i$ from the OLS in \eqref{eq:yonZx}. We use the subscript ``F'' to signify \cite{Fisher1935}. Denote $\hat{\varepsilon}_{i,\textsc{f}}$ as the residual from the same OLS fit.
Define $\hat{V}_{\textsc{ehw},\textsc{f}}$ as the $(2,2)$ entry of the EHW variance estimator in the form of \eqref{eq:EHW} with regressors $x_{i,\textsc{f}} = (1, Z_i, x_i^{\mathrm{T}})^{\mathrm{T}}$.

\paragraph*{Random Design}
By the general theory in Section \ref{random}, we derive the asymptotic variance of $\hat{\beta}_{\textsc{f}}$ as
\begin{equation*}
V^{\textup{r}}_{\textsc{f}} 
= \frac{\V(\varepsilon^{\textup{r}}_{i,\textsc{f}}(1))}{e} 
+ \frac{\V(\varepsilon^{\textup{r}}_{i,\textsc{f}}(0))}{1-e},
\end{equation*}
where $\varepsilon_{i,\textsc{f}}^{\textup{r}}(z)
= y_i(z) - \mathbb{E}(y_i(z)) - \dot{x}_i^{\mathrm{T}}\gamma_\textsc{f}^{\textup{r}}$ for $z\in\{0, 1\}$ with $\gamma^{\textup{r}}_{\textsc{f}}$ being the probability limit of the coefficient of $x_i$ from the OLS fit in \eqref{eq:yonZx}. 
\begin{theorem} \label{thm:PRA_r}
Under random design and Assumption \ref{asu:completeRandom} with \(\mathbb{E}_\circ = \mathbb{E}_{(y,Z,x)}\), we have $\beta^{\textup{r}} = \mathbb{E}(Y_i(1) - Y_i(0) )$ and
\begin{equation*}
(V_\textsc{f}^{\textup{r}}  )^{-1/2}\sqrt{n}(\hat{\beta}_{\textsc{f}} -\beta^{\textup{r}} ) \overset{\textup{d}}{\to} \mathcal{N}(0,  1)
\text{ and }
n \hat{V}_{\textsc{ehw,f}} = V_\textsc{f}^{\textup{r}} + o(1;\mathbb{P}_{(y,Z,x)}).
\end{equation*}
\end{theorem}
The coefficient of $Z_i$ from the projection of $y_i$ on $x_{i,\textsc{f}}$ identifies the ATE.
Theorem \ref{thm:PRA_r} indicates that under random design, $\hat{V}_{\textsc{ehw},\textsc{f}}$ is consistent.
This specification corresponds
to Fisher's additive regression adjustment, whose design-based properties are
studied by \citet{Freedman2008} and \citet{Lin2013}. In their finite-population
framework, the usual EHW variance estimator is generally conservative, whereas our random-design formulation,
closer to the perspective of \citet{NegiWooldridge2021}, yields consistency.
We include the result to
connect the random-design OLS theory in Section \ref{random} to Fisher's analysis of covariance adjustment and to provide a benchmark for the mixed-design result below.

\paragraph*{Mixed Design}
Define $\varepsilon_{i,\textsc{f}}^{\textup{m}}(z)
= (y_i(z) - n^{-1} \sum_{i=1}^n \mathbb{E}(y_i(z) \mid x_i)) - \ddot{x}_i^{\mathrm{T}}\gamma_{\textsc{f}}^{\textup{m}}$ for $z=0, 1$ with $\gamma^{\textup{m}}_{\textsc{f}}$ being the probability limit of the coefficient of $x_i$ from the OLS fit in \eqref{eq:yonZx} under mixed design. 
By the general theory in Section \ref{mixed}, we derive the conditional variance of $n\cdot \V(\hat{\beta}_{\textsc{f}} \mid X)$ as
\begin{equation*}
V_{\textsc{f}}^{\textup{m}}
= \frac{1}{n}\sum_{i=1}^n \left( \frac{\mathbb{E}\left(\varepsilon^{\textup{m}}_{i,\textsc{f}}(1)^2 \mid x_i\right)}{e}  
+ \frac{\mathbb{E}\left(\varepsilon^{\textup{m}}_{i,\textsc{f}}(0)^2 \mid x_i\right)}{1-e}  
- \mathbb{E}\left(\varepsilon_{i,\textsc{f}}^{\textup{m}}(1) - \varepsilon_{i,\textsc{f}}^{\textup{m}}(0) \mid x_i \right)^2  \right)
\end{equation*}
and the asymptotic bias of $\hat{V}_{\textsc{ehw},\textsc{f}}$ for $V_{\textsc{f}}^{\textup{m}}$ as
\begin{equation}
B_{\textsc{f}}^{\textup{m}}
=  \frac{1}{n} \sum_{i=1}^n \mathbb{E}(\varepsilon_{i,\textsc{f}}^{\textup{m}}(1) - \varepsilon_{i,\textsc{f}}^{\textup{m}}(0) \mid x_i)^2
\geq 0.
\label{eq:bias_com_F_m}
\end{equation}
\begin{theorem} \label{thm:PRA_m}
Under mixed design and Assumption \ref{asu:completeRandom} with \(\mathbb{E}_\circ = \mathbb{E}_{(y,Z)|x}\), we have $\beta^{\textup{m}} = n^{-1} \sum_{i=1}^n \mathbb{E}(Y_i(1) - Y_i(0) \mid x_i )$ and 
\begin{equation*}
(V_{\textsc{f}}^{\textup{m}})^{-1/2}\sqrt{n}(\hat{\beta}_{\textsc{f}}-\beta^{\textup{m}} ) \mid X \overset{\textup{d}}{\to} \mathcal{N}(0,  1)
\text{ and }
n \hat{V}_{\textsc{ehw},{\textsc{f}}}  
= V_{\textsc{f}}^{\textup{m}} + B_{\textsc{f}}^{\textup{m}} + o(1;\mathbb{P}_{(y,Z) \mid x}).
\end{equation*}
\end{theorem}
The coefficient of $Z_i$ from the projection of $y_i$ on $x_{i,\textsc{f}}$ identifies the average of the empirical conditonal average treatment effect, which is also known as the mixed average treatment effect by \citet{LiDingMealli2023}.
Theorem \ref{thm:PRA_m} also indicates that under mixed design, $\hat{V}_{\textsc{ehw},\textsc{f}}$ is conservative in general with asymptotic bias $B^{\textup{m}}_{\textsc{f}}$ in \eqref{eq:bias_com_F_m}.

\subsection{Lin's fully interacted adjustment} \label{sec:LinRCT}
Now consider the regression with fully-interacted covariates, i.e., 
\begin{equation}
\text{lm}(y_i \sim 1 + Z_i + \ddot{x}_i + Z_i \cdot \ddot{x}_i),
\label{eq:Lin}
\end{equation}
where $\ddot{x}_i = x_i - \bar{x}$.
Let \(\hat{\beta}_{\textsc{l}}\) denote the coefficient on \(Z_i\) from the OLS regression in \eqref{eq:Lin}. We use the subscript ``L'' to reference \citet{Lin2013}. 
The coefficient on \(Z_i\) from the OLS fit in \eqref{eq:Lin} provides an estimate of the ATE when covariates are centered around their sample mean. 
Without centering, we need to estimate the ATE using \(\hat{\beta}_{\textsc{l}} + \hat{\xi}_{\textsc{l}} \bar{x}\) where $\hat{\xi}_{\textsc{l}}$ denotes the coefficient on the interaction term \(Z_i x_i\).
Denote $\hat{\varepsilon}_{i,\textsc{l}}$ as the residual from the OLS fit in \eqref{eq:Lin}.
Define $\hat{V}_{\textsc{ehw},\textsc{l}}$ as the $(2,2)$ entry of the EHW variance estimator in the form of \eqref{eq:EHW} with regressors $x_{i,\textsc{l}} = (1, Z_i,\ddot{x}_i, Z_i\ddot{x}_i)^{\mathrm{T}}$.

\paragraph*{Random Design}
We cannot naively apply the OLS theory without modifications because it ignores the uncertainty in $\bar x$. We can either modify the OLS theory or apply the general $Z$-estimation theory. We adopt the second strategy here.
We apply the $Z$-estimation framework in Section~\ref{sec:Z_random} to derive the asymptotic variance of the estimator $\hat{\beta}_{\textsc{l}}$. 
The variance estimator $\hat{V}_{\textsc{ehw},\textsc{l}}$, obtained by directly applying the OLS results in Section~\ref{random}, is incorrect because it ignores the additional uncertainty introduced by centering the covariates.
In particular, obtaining a consistent EHW covariance estimator requires using augmented estimating equations \citep{Newey1984}.
Specifically, $\hat{\beta}_{\textsc{l}}$ can be expressed as $\hat{Y}_{\textsc{l}}(1) - \hat{Y}_{\textsc{l}}(0)$ where
$\left(\mu_1, \gamma_1, \mu_0, \gamma_0, \mu_x\right)= (\hat{Y}_{\textsc{l}}(1), \hat{\gamma}_1, \hat{Y}_{\textsc{l}}(0), \hat{\gamma}_0, \bar{x})$ jointly solves the estimating equations:
\begin{equation}
0
= \frac{1}{n} \sum_{i=1}^n \left(\begin{array}{c}
Z_i \cdot\left\{y_i-\left(x_i-\mu_x\right)^{\mathrm{T}} \gamma_1-\mu_1\right\} \left(\begin{array}{c}
1 \\
x_i-\mu_x
\end{array}\right) \\
(1-Z_i)  \cdot\left\{y_i-\left(x_i-\mu_x\right)^{\mathrm{T}} \gamma_0-\mu_0\right\} \left(\begin{array}{c}
1 \\
x_i-\mu_x
\end{array}\right) \\
x_i-\mu_x
\end{array}\right),
\label{eq:augmented estimating equations}
\end{equation}
where the last line accounts for the estimation of \(\mu_x\).
With these equations, we can construct the variance estimator using \eqref{eq:EHW_ME} and apply Theorem \ref{thm:random} to demonstrate the consistency of the EHW variance estimator.

Let $\gamma_z^{\textup{r}}$ be the coefficient of $x_{i}$ in the OLS fit of $y_{i}$ on $1$ and $\ddot{x}_{i}$ over $\{i: Z_i = z\}$ under random design.
Define $\varepsilon_{i,\textsc{l}}^{\textup{r}}(z) = y_i(z) - \mathbb{E}(y_i(z)) - \dot{x}_i^{\mathrm{T}}\gamma_z^{\textup{r}}$
for $z\in \{0, 1\}$.
By the general theory in Section \ref{random}, we derive the asymptotic variance of $\hat{\beta}_{\textsc{l}}$ under random design as 
\begin{equation}
V^{\textup{r}}_{\textsc{l}}
= \frac{\V(\varepsilon^{\textup{r}}_{i,\textsc{l}}(1))}{e} 
+ \frac{\V(\varepsilon^{\textup{r}}_{i,\textsc{l}}(0))}{1-e}
+ (\gamma^{\textup{r}}_1 - \gamma_0^{\textup{r}})^{\mathrm{T}} 
\V(x_i) 
(\gamma^{\textup{r}}_1-\gamma^{\textup{r}}_0).
\label{eq:V_lin_random}
\end{equation}
Let $\Sigma_x
= n^{-1} \sum_{i=1}^n \ddot{x}_i\ddot{x}_i^{\mathrm{T}}$.
We estimate $V^{\textup{r}}_{\textsc{l}}$ using the following estimator:
\begin{align}
\hat{V}_{\textsc{ehw},\textsc{l},\text{adj}}
= \hat{V}_{\textsc{ehw},\textsc{l}}
+ \frac{1}{n} (\hat{\gamma}_1 - \hat{\gamma}_0)^{\mathrm{T}} \Sigma_x 
(\hat{\gamma}_1 - \hat{\gamma}_0).
\label{eq:Lin_correction}
\end{align}
Theorem \ref{thm:FRA_random} below establishes the asymptotic normality of $\hat{\beta}_{\textsc{l}}$ and the consistency of $\hat{V}_{\textsc{ehw},\textsc{l},\text{adj}}$ for estimating $V_{\textsc{l}}^{\textup{r}}$ in \eqref{eq:V_lin_random} under random design.

\begin{theorem} \label{thm:FRA_random}
Under random design and Assumption \ref{asu:completeRandom} with \(\mathbb{E}_\circ = \mathbb{E}_{(y,Z,x)}\), we have 
$\beta^{\textup{r}} = \mathbb{E}(Y_i(1) - Y_i(0) )$
and
\[
(V^{\textup{r}}_{\textsc{l}})^{-1/2}\sqrt{n}(\hat{\beta}_{\textsc{l}} - \beta^{\textup{r}})  \overset{\textup{d}}{\to} \mathcal{N}(0, 1)
\text{ and }
n \hat{V}_{\textsc{ehw},\textsc{l},\text{adj}}
= V_{\textsc{l}}^{\textup{r}} + o(1;\mathbb{P}_{(y,Z,x)}).
\]
Moreover, 
\[
n \hat{V}_{\textsc{ehw},\textsc{l}}
= V_{\textsc{l}}^{\textup{r}}
- B_{\textsc{l}}^{\textup{r}}
+ o(1;\mathbb{P}_{(y,Z,x)})
\]
with \begin{equation}
B_{\textsc{l}}^{\textup{r}} 
= (\gamma^{\textup{r}}_1-\gamma^{\textup{r}}_0)^{\mathrm{T}}
\V(x_i) 
(\gamma^{\textup{r}}_1-\gamma^{\textup{r}}_0). 
\label{eq:bias_com_L_r}
\end{equation}
\end{theorem}
Under random design, the coefficient of $Z_i$ from the projection of $y_i$ on $x_{i,\textsc{l}}$ identifies the ATE.
The result of asymptotic normality in Theorem \ref{thm:FRA_random} is also proved in \citet[Theorem 5.1]{NegiWooldridge2021}, whereas our proof relies on the general $Z$-estimation framework developed in Section~\ref{sec:Z_random}.
\cite{NegiWooldridge2021} and \cite{ZhaoDing2021} similarly propose correcting the EHW variance estimator by adding the second term on the right-hand side of \eqref{eq:Lin_correction}; see also \citet[Page 131]{Ding2023}.

As noted above, under a random design, centering around $\bar{x}$ introduces additional uncertainty, and the theory of OLS with i.i.d. data does not apply here. 
Consequently, using the OLS-based variance estimator $\hat{V}_{\textsc{ehw},\textsc{l}}$ from \eqref{eq:EHW} results in an anti-conservative variance estimator.
To address this issue, a correction term in \eqref{eq:Lin_correction} is necessary.
This contrasts with the design-based framework of \citet{Lin2013}, where the
covariates are treated as fixed. In that setting, centering by the
sample mean $\bar x$ introduces no additional uncertainty, and the usual EHW
variance estimator for Lin's fully interacted regression is asymptotically
conservative for the randomization variance.

\begin{remark}\label{re:correction_Lin}
We can construct the variance estimator in two ways: by applying the HW variance formula to the augmented estimating equations in \eqref{eq:augmented estimating equations} and then using a plug-in approach as in \eqref{eq:EHW}, or by adding the correction term to the EHW variance estimator $\hat{V}_{\textsc{ehw},\textsc{l}}$ as in \eqref{eq:Lin_correction}. Although these two methods are asymptotically equivalent, they are not numerically identical due to small finite-sample differences.
\end{remark}

\paragraph*{Mixed Design}
We can apply the results of OLS in Section \ref{mixed} because the covariates $\{x_i\}_{i=1}^n$ are fixed.
Let $\gamma_z^{\textup{m}}$ be the coefficient of $x_{i}$ in the OLS fit of $y_{i}$ on $1$ and $\ddot{x}_{i}$ over $\{i: Z_i = z\}$ under mixed design.
Define $\varepsilon_{i,\textsc{l}}^{\textup{m}}(z) = (y_i(z) - n^{-1}\sum_{i=1}^n \mathbb{E}(y_i(z) \mid x_i)) - \ddot{x}_i^{\mathrm{T}}\gamma_z^{\textup{m}}$ for $z\in\{0, 1\}$.
By the general theory in Section \ref{mixed}, we derive the asymptotic variance of $\hat{\beta}_{\textsc{l}}$ under mixed design as
\begin{equation*}
V^{\textup{m}}_{\textsc{l}}
= \frac{1}{n} \sum_{i=1}^n 
\left( \frac{\mathbb{E}(\varepsilon^{\textup{m}}_{i,\textsc{l}}(1)^2 \mid x_i)}{e}
+ \frac{\mathbb{E}(\varepsilon^{\textup{m}}_{i,\textsc{l}}(0)^2 \mid x_i)}{1-e} 
- \mathbb{E}( \varepsilon^{\textup{m}}_{i,\textsc{l}}(1) - \varepsilon^{\textup{m}}_{i,\textsc{l}}(0) \mid x_i)^2
\right).  
\end{equation*}
Define the asymptotic bias of $\hat{V}_{\textsc{ehw},{\textsc{l}}}$ for $V^{\textup{m}}_{\textsc{l}}$ as
\begin{equation}
B_{\textsc{l}}^{\textup{m}} 
= \frac{1}{n}\sum_{i=1}^n \mathbb{E}(\varepsilon_{i,\textsc{l}}^{\textup{m}}(1) - \varepsilon_{i,\textsc{l}}^{\textup{m}}(0) \mid x_i)^2. 
\label{eq:bias_com_L_m} 
\end{equation}
\begin{theorem} \label{thm:FRA_m}
Under mixed design and Assumption \ref{asu:completeRandom} with \(\mathbb{E}_\circ = \mathbb{E}_{(y,Z)|x}\), we have
$\beta^{\textup{m}} = n^{-1} \sum_{i=1}^n \mathbb{E}(Y_i(1) - Y_i(0) \mid x_i )$ and 
\begin{equation*}
(V^{\textup{m}}_{\textsc{l}})^{-1/2} \sqrt{n}(\hat{\beta}_{\textsc{l}}-\beta^{\textup{m}} ) \mid  X \overset{\textup{d}}{\to } \mathcal{N}(0, 1 )
\text{ and }
n \hat{V}_{\textsc{ehw},{\textsc{l}}} 
= V_{\textsc{l}}^{\textup{m}} + B_{\textsc{l}}^{\textup{m}} + o(1;\mathbb{P}_{(y, Z) \mid x}).
\end{equation*}
\end{theorem}
The coefficient of $Z_i$ from the projection of $y_i$ on $x_{i,\textsc{l}}$ identifies empirical average of the conditional average treatment effects under mixed design.
Theorem \ref{thm:FRA_m} indicates that under mixed design, $\hat{V}_{\textsc{ehw},{\textsc{l}}}$ is conservative in general with asymptotic bias $B^{\textup{m}}_{\textsc{l}}$ in \eqref{eq:bias_com_L_m}.
The mixed design formulation here has the advantage of simplifying the estimation of the asymptotic variance of $\hat{\beta}_{\textsc{l}}$. 

Both the Fisher specification and the Lin specification target the same causal estimand under the random and mixed designs considered here. Their distinction therefore concerns efficiency rather than interpretation. Proposition \ref{prop:fisher-lin-mixed} shows that, under mixed design, the Lin specification is asymptotically no less efficient than the Fisher specification.

\begin{proposition}\label{prop:fisher-lin-mixed}
Under mixed design, $V_{\textsc{l}}^{\mathrm{m}}
\leq
V_{\textsc{f}}^{\mathrm{m}}$.
If $\Sigma_x$ is positive definite,
equality holds if and only if either $e=1/2$ or $\gamma_1^{\mathrm{m}}=\gamma_0^{\mathrm{m}}$.
\end{proposition}

This result complements existing efficiency comparisons under alternative sources of randomness. In the finite-population, design-based framework, \cite{Freedman2008} shows that additive covariate adjustment need not improve precision relative to the difference-in-means estimator, whereas \cite{Lin2013} shows that fully interacted adjustment is asymptotically no less efficient, even under misspecification. Under random sampling, \cite{NegiWooldridge2021} establishes an analogous ranking: full regression adjustment is asymptotically no less efficient than either the difference-in-means estimator or pooled regression adjustment, without requiring linear conditional mean functions. Proposition~\ref{prop:fisher-lin-mixed} extends the comparison between additive and fully interacted adjustment to the mixed-design setting.

\section{Clustered Data}\label{sec:cluster}
Clustered data are common in empirical work, especially when observations are grouped by schools, classrooms, villages, firms, or geographic units.
In such settings, researchers routinely use cluster-robust standard errors to account for within-cluster dependence. 
\cite{AbadieAtheyImbens2023} emphasize that the interpretation of LZ cluster robust standard error depends on the source of randomness, and \cite{SuDing2021} provide design-based theory for regression estimators in cluster-randomized experiments.
We extend the $Z$-estimation framework in Section \ref{sec:ME} to clustered data and apply it to cluster-randomized experiments, focusing on parameter interpretation and robust inference under misspecified models and different sources of randomness.

\subsection{M-estimation}\label{sec:cluster_ME}

Parallel to Section \ref{sec:ME}, we develop a general \(Z\)-estimation framework for clustered data. We begin with the familiar random-design setting as a benchmark in Section \ref{sec:Z_random} and then extend the analysis to fixed- and mixed-design settings in Sections \ref{sec:Z_fixed} and \ref{sec:Z_mixed}, respectively.

\subsubsection{Random design}
\label{sec:Z_random_cluster}
Under random design, we observe a random sample of clusters.
Let $w_{ij}$ denote the observation for unit $j$ in cluster $i$, for $j=1,\ldots, n_i$, $i=1, \ldots M$, and let the total sample size be $N=\sum_{i=1}^M n_i$. 
Let $\sum_{i j}=\sum_{i=1}^M \sum_{j=1}^{n_i}$ denote the summation over all units.
We assume that $n_i$ is fixed for each $i$.
The parameter of interest \( \beta^\textup{r} \) is identified through a population moment condition
\[
\mathbb{E}\left[ \sum_{ij} \psi(w_{ij};\beta^\textup{r}) \right] = 0.  
\]
where we use $\beta^\textup{r}$ to denote the estimand under random design.
Let $\hat{\beta}$ be the solution to the following equation:
\begin{equation}
\bar{\psi}(W;\hat{\beta})
= \frac{1}{N} \sum_{i=1}^M \sum_{j=1}^{n_i} \psi(w_{ij};\hat{\beta}) = 0, 
\label{eq:ME_cluster}
\end{equation}
where $W=\{w_{ij}: j=1,\ldots,n_i; i=1,\ldots,M\}$ collects all observations.
Define $\psi_{i}(\hat{\beta})$ as an $n_{i} \times p$ matrix with row $j$ equaling $\psi(w_{ij}; \hat{\beta})$, for $j=1, \ldots, n_{i}$.  
A first-order Taylor
expansion of $\bar{\psi}(W;\hat{\beta})$ around $\beta$ gives
\[
0=\bar{\psi}(W;\hat{\beta}) \approx \bar{\psi}(W;{\beta^\textup{r}}) +\Gamma(\beta^\textup{r})\,(\hat\beta-\beta^\textup{r}),
\]
where ${\Gamma}(\beta^\textup{r})
= \left.\frac{\partial}{\partial b^{\mathrm{T}}} \bar{\psi}(W;{b}) \right|_{b=\beta^\textup{r}}$. 
This yields the linearization of $Z$-estimator as $\hat{\beta}
\approx \beta^\textup{r} - \Gamma(\beta^\textup{r})^{-1}\bar{\psi}(W;{\beta}^\textup{r})$, and the cluster-robust variance estimator:
\begin{equation}
\hat{V}_{\textsc{lz}} 
= \frac{1}{M} \left( \frac{1}{M} \sum_{ij} \frac{\partial \psi(w_{ij};\hat{\beta})}{\partial b^{\mathrm{T}}}  \right)^{-1}
\left( \frac{1}{M} \sum_{i=1}^{M} \psi_{i}(\hat{\beta}) \psi_{i}(\hat{\beta})^{\mathrm{T}} \right)
\left( \frac{1}{M} \sum_{ij} \frac{\partial \psi(w_{ij};\hat{\beta})}{\partial b^{\mathrm{T}}}  \right)^{-\mathrm{T}}.
\label{eq:V_LZ_mestimation}
\end{equation}
Throughout this section, we write $w_{ij}=(y_{ij},x_{ij})$ to emphasize the response $y_{ij}$ and covariate $x_{ij}$.
We impose the following assumptions on the clustered data.
\begin{assumption} \label{asu:cluster_ME}
\begin{enumerate}[(a)]
\item The variables $(y_{ij}, x_{ij})$ have the same marginal distribution across $(i = 1, \ldots, M; j = 1, \ldots, n_i)$;
\item The variables $\{(y_{ij}, x_{ij})\}_{j=1}^{n_i}$ are independent across cluster $i$, but can be arbitrarily correlated within the same cluster $i$.
\end{enumerate}
\end{assumption}

Specifically, we define the estimand $\beta^\textup{r}$ as the unique solution to the estimating equation under random design:
\[
\mathbb{E}\left[ \sum_{ij} \psi(y_{ij}, x_{ij};\beta^\textup{r}) \right] = 0.  
\]
Define the asymptotic variance of $\hat{\beta}$ under random design as
$V^\textup{r} = (\Gamma^\textup{r})^{-1}\Delta^{\textup{r}}(\Gamma^{\textup{r}})^{-\mathrm{T}}$ with 
\begin{align*}
\Gamma^\textup{r}
=& \frac{1}{M} \sum_{i=1}^M \mathbb{E}\left[ \sum_{j=1}^{n_i} \frac{\partial \psi(y_{ij}, x_{ij};\beta^\textup{r})}{\partial \beta^{\mathrm{T}}}  \right], \\
\Delta^\textup{r} 
=& \frac{1}{M} \sum_{i=1}^M \mathbb{E}\left[ \left( \sum_{j=1}^{n_i}\psi(y_{ij}, x_{ij}; \beta^\textup{r}) \right)^{\otimes 2} \right].
\end{align*}

Theorem \ref{thm:ME_random_cluster} below indicates that under random design, $\hat{V}_{\textsc{lz}}$ in \eqref{eq:V_LZ_mestimation} is consistent for $V^\textup{r}$. 
\begin{theorem}\label{thm:ME_random_cluster}
Under random design and Assumptions \ref{asu:ME} and \ref{asu:cluster_ME} with \(\mathbb{E}_\circ = \mathbb{E}_{(y,x)}\) and \(\mathbb{P}_\circ = \mathbb{P}_{(y,x)}\), we have 
\[
(V^\textup{r})^{-1/2}
\sqrt{M}(\hat{\beta} - \beta^\textup{r}) 
\overset{\textup{d}}{\to} \mathcal{N}\left(0, I \right)
\text{ and }
M \hat{V}_{\textsc{lz}}
= V^\textup{r} + o(1;\mathbb{P}_{(y,x)}).
\] 
\end{theorem}
Theorem \ref{thm:ME_random_cluster} extends the usual $Z$-estimation asymptotic normality result with i.i.d. data to the cluster-level setting; see, for example, \cite{LiangZeger1986} for estimating equations $Z$-estimation and \cite{NeweyMcFadden1994} for general large-sample theory for $Z$-estimation. We include it here to establish notation and to serve as the random-design benchmark for the fixed- and mixed-design results below.

\subsubsection{Fixed design}
\label{sec:Z_fixed_cluster}
Under fixed design, the $x_{ij}$'s are fixed, or, equivalently, we condition on them.
Let $X = (x_{ij})_{1\le i\le M, 1\le j \le n_i}$ denote the stacked covariate vector under fixed design.
Define the estimand $\beta^\textup{f}$ as the unique solution to
\[
\mathbb{E}\left[ \frac{1}{N} \sum_{ij} \psi(y_{ij},x_{ij};\beta^\textup{f} ) \mid {X} \right] = 0.   
\]
Define $V^\textup{f} = (\Gamma^\textup{f})^{-1}\Delta^\textup{f}(\Gamma^\textup{f})^{-\mathrm{T}}$ with 
\begin{align*}
\Gamma^\textup{f} 
&= \frac{1}{M} \sum_{i=1}^M \mathbb{E}\left[ \sum_{j=1}^{n_i}  \frac{\partial \psi(y_{ij},x_{ij};\beta^\textup{f} )}{\partial b^{\mathrm{T}}}  \mid x_i \right] 
\text{ and }
\Delta^\textup{f} 
= \frac{1}{M} \sum_{i=1}^M \V\left[ \sum_{j=1}^{n_i} \psi(y_{ij},x_{ij};\beta^\textup{f}) \mid x_i \right],
\end{align*}
and the asymptotic bias of $\hat{V}_{\textsc{lz}}$ in \eqref{eq:V_LZ_mestimation} for $V^\textup{f}$ as
\begin{equation}
B^{\textup{f}}
= (\Gamma^\textup{f})^{-1} \left( \frac{1}{M} \sum_{i=1}^M 
\left( \mathbb{E}\left[ \sum_{j=1}^{n_i} \psi(y_{ij},x_{ij};\beta^\textup{f}) \mid x_i \right] \right)^{\otimes 2}  \right) (\Gamma^{\textup{f} })^{-\mathrm{T}}. 
\label{eq:bias_me_fixed_cluster}
\end{equation}
\begin{theorem}\label{thm:ME_fixed_cluster}
Under fixed design and Assumptions \ref{asu:ME} and \ref{asu:cluster_ME} with \(\mathbb{E}_\circ = \mathbb{E}_{y|x}\) and \(\mathbb{P}_\circ = \mathbb{P}_{y|x}\), we have 
\[
(V^\textup{f})^{-1/2}
\sqrt{M}(\hat{\beta} - \beta^\textup{f}) \mid X
\overset{\textup{d}}{\to} \mathcal{N}\left(0, 1 \right)
\text{ and }
M \hat{V}_{\textsc{lz}} 
= V^\textup{f} + B^{\textup{f}}
+ o(1;\mathbb{P}_{y|x}).
\] 
\end{theorem}
Theorem \ref{thm:ME_fixed_cluster} is the clustered analogue of the fixed-regressor
conservativeness result in \cite{AbadieImbensZheng2014} and Theorem \ref{sec:Z_fixed}. 
Theorem \ref{thm:ME_fixed_cluster} indicates that under fixed design, $\hat{V}_{\textsc{lz}}$ in \eqref{eq:V_LZ_mestimation} is conservative with asymptotic bias $B^{\textup{f}}$ in \eqref{eq:bias_me_fixed_cluster}.
Conditional on the
fixed regressors, the cluster-level estimating equations can have
nonzero conditional means under misspecification, so the LZ
middle matrix estimates the sum of the conditional variance and a
positive semidefinite approximation-error component.

\subsubsection{Mixed design}
\label{sec:Z_mixed_cluster}
We partition the regressors as $X = (X_1, X_2)$. Under mixed design, our analysis is based on conditioning on part of the regressors, i.e., ${X}_2$, and take the remaining regressors ${X}_1$ as random.
Define the estimand $\beta^\textup{m}$ as the solution to 
\[
\mathbb{E}\left[ \frac{1}{N} \sum_{ij} \psi(y_{ij},x_{ij};\beta^\textup{m})\mid {X}_2 \right] = 0. 
\]
Define 
$V^\textup{m} = (\Gamma^\textup{m})^{-1}\Delta^\textup{m}(\Gamma^\textup{m})^{-\mathrm{T}}$, with 
\begin{align*}
\Gamma^\textup{m} 
&= \frac{1}{M} \sum_{i=1}^M \mathbb{E}\left[ \sum_{j=1}^{n_i}  \frac{\partial \psi(y_{ij},x_{ij};\beta^\textup{m} )}{\partial b^{\mathrm{T}} } \mid x_{i2} \right] 
\text{ and }
\Delta^\textup{m} 
= \frac{1}{M} \sum_{i=1}^M \V\left[ \sum_{j=1}^{n_i} \psi(y_{ij},x_{ij};\beta^\textup{m}) \mid x_{i2} \right]
\end{align*}
and 
\begin{equation}
B^\textup{m}
= (\Gamma^\textup{m})^{-1} \left( \frac{1}{M} \sum_{i=1}^M \left( \mathbb{E}\left[ \sum_{j=1}^{n_i} \psi(y_{ij},x_{ij};\beta^\textup{m}) \mid x_{i2} \right] \right)^{\otimes 2}  \right) (\Gamma^\textup{m})^{- \mathrm{T}}.
\label{eq:bias_me_mixed_cluster}
\end{equation}

\begin{theorem} \label{thm:ME_mixed_cluster}
Under mixed design and Assumptions \ref{asu:ME} and \ref{asu:cluster_ME} with \(\mathbb{E}_\circ = \mathbb{E}_{(y,x_1)|x_2}\) and \(\mathbb{P}_\circ = \mathbb{P}_{(y,x_1)\mid x_2}\), we have 
\begin{equation*}
(V^\textup{m})^{-1/2} \sqrt{M}(\hat{\beta} - \beta^\textup{m}) \mid X_2
\overset{\textup{d}}{\to} \mathcal{N}\left(0, I \right)
\text{ and }
M \hat{V}_{\textsc{lz}}
=  V^\textup{m}
+ B^\textup{m}
+ o(1;\mathbb{P}_{(y,x_1)\mid x_2}). 
\end{equation*}
\end{theorem}
Theorem \ref{thm:ME_mixed_cluster} is the clustered analogue of the mixed-design result in Theorem \ref{thm:ME_mixed}. 
This setting is particularly relevant for cluster-randomized experiments with fixed or pre-determined covariates. 
The theorem shows that under mixed design, the usual LZ variance estimator $\hat{V}_{\textsc{lz}}$ in \eqref{eq:V_LZ_mestimation} remains conservative under misspecification, with the additional term $B^\textup{m}$ capturing the approximation-error component induced by conditioning on the fixed part of the regressors.

\begin{remark}
Analogous to Section \ref{sec:i.i.d.Theory}, the $Z$-estimation framework in Section \ref{sec:cluster_ME} applies directly to the OLS regression as a special case: $y_{ij} = x_{ij}^\mathrm{T}\beta + \varepsilon_{ij}$, where $\beta$ is the coefficient in the population linear
projection of \(y_{ij}\) on \(x_{ij}\).  
From that theory one can show that the cluster-robust variance estimator is consistent under a random design and remains conservative under fixed or mixed designs.  We omit the details to avoid repetitiveness and turn directly to its role in cluster randomized trials in Section \ref{sec:cluster_RCT}.
\end{remark}

\subsection{Cluster randomization and regression analysis} 
\label{sec:cluster_RCT}
Consider a study with $N$ units, clustered, for example, by classrooms or villages. Cluster $i$ has $n_i$ units $(i=1, \ldots, M)$, and the total number of units is $N=\sum_{i=1}^M n_i$. Let $(i, j)$ index the $j$th unit within cluster $i$ $(i=1, \ldots, M ; j=1, \ldots, n_i)$. Unit $(i, j)$ has covariates $x_{i j}$. 
Let $Z_i$ be the treatment indicator for cluster $i$ and $Z_{i j}$ be the treatment indicator for unit $(i, j)$. In a cluster-randomized experiment, units within a cluster receive identical treatment levels. So if cluster $i$ receives treatment, then $Z_{i j}=Z_i=1$; if cluster $i$ receives control, then $Z_{i j}=Z_i=0$. Let $\mathcal{T}= \{(i, j): Z_{i j}=1 \}$ be the indices of units under treatment and $\mathcal{C}=\{(i, j): Z_{i j}=0\}$ be the indices of units under control. Their cardinalities $n_{\mathcal{T}}=\sum_{i j} Z_{i j}$ and $n_{\mathcal{C}}=\sum_{i j}(1-Z_{i j})$ represent the total numbers of units under treatment and control, respectively.
Similar to Section \ref{sec:application}, we omit the fixed-design analysis, since it does not quantify the advantages of randomization with a misspecified linear model.

For unit $(i, j)$, let $y_{ij}(1)$ and $y_{ij}(0)$ be the potential outcomes under treatment and control, respectively.
The observed outcome is related to the potential outcomes through
$y_{ij}=Z_{ij}y_{ij}(1)+(1-Z_{ij})y_{ij}(0)$, for
$i=1,\ldots,M$ and $j=1,\ldots,n_i$.
A central goal in analysing a cluster-randomized
experiment is to make inference using the observed data $\{(Z_{ij}, y_{ij}, x_{ij}): i = 1, \ldots, M; j = 1, \ldots, n_i\}$.
\cite{SuDing2021} introduced the following notation to measure the heterogeneity of the cluster sizes:
\[
\omega_i= \frac{n_i}{N}, \quad \Omega=\max _{1 \leq i \leq M} \omega_i, \quad \tilde{\omega}_i=\omega_i M= \frac{n_i}{N / M}.  
\]
When all clusters have equal sizes, $\omega_i = \Omega = 1/M$.

In this subsection, we impose the following assumption. 
\begin{assumption} \label{asu:cluster}
\begin{enumerate}[(a)]
\item The variables $(y_{ij}(0),y_{ij}(1),x_{ij})$ have identical marginal distribution across $(i = 1, \ldots, M; j = 1, \ldots, n_i)$;
\item The variables $((y_{ij}(0),y_{ij}(1),x_{ij}):j = 1, \ldots, n_i)$ are independent across $i$ but allow for arbitrary dependence within cluster;
\item $(y_{ij}(0), y_{ij}(1), x_{ij}:j = 1, \ldots, n_i) \indep Z_i$ for $i = 1, \ldots, M$;
\item $\mathbb{P}_\circ(Z_i=1)= e$;
\item $\mathbb{E}_\circ(y_{ij}^4)<\infty$;
\item $\mathbb{E}_\circ\left( \|x_{ij}\|^4 \right) <\infty$;
\item $\Omega = o(M^{-2/3})$. 
\end{enumerate}
\end{assumption}

\subsubsection{Without covariates} 
We consider the OLS fit with individual-level data:
\begin{equation}
\text{lm}(y_{ij} \sim 1 + Z_{ij}). 
\label{eq:cluster_fit}
\end{equation}
Denote by $\hat{\beta}_{\textsc{i}}$ the coefficient of $Z_{ij}$ from the above OLS fit
and $\hat{\varepsilon}_{i j}$ the residual from the same OLS fit. 
Define $X_{i}$ as an $n_{i} \times 2$ matrix with row $j$ equaling $(1, Z_{i j}), j=1, \ldots, n_{i}$. 
Stack $X_{i}$ together to obtain an $n \times 2$ matrix $X$. 
Define an $n_{i} \times n_{i}$ matrix $\hat{U}_{i}=\left(\hat{\varepsilon}_{i j} \hat{\varepsilon}_{i k}\right)_{1 \leq j, k \leq n_{i}}$. 
The cluster-robust variance estimator of $\hat{\beta}_{\textsc{i}}$ can be derived as the $(2,2)$ entry of $\hat{V}_{\textsc{lz}}$ in \eqref{eq:V_LZ_mestimation}:
\begin{align}
\hat{V}_{\textsc{lz,i}} 
=& \frac{1}{M} \left[\left(\frac{1}{M}X^{\mathrm{T}} X\right)^{-1}\left(\frac{1}{M}\sum_{i=1}^M X_i^{\mathrm{T}} \hat{U}_i X_i\right)\left(\frac{1}{M}X^{\mathrm{T}} X\right)^{-1}\right]_{(2,2)}.
\label{eq:LZ_cluster_I}
\end{align}
Define $\varepsilon_{ij}(z) = y_{ij}(z) - \mathbb{E}(y_{ij}(z))$ and 
${\varepsilon}_{i\cdot,\textsc{i}}(z) = \sum_{j=1}^{n_i}\varepsilon_{ij}(z)M/N$ for $z=0,1$.
By the theory in Section \ref{sec:Z_random_cluster}, we derive the asymptotic variance of $\hat{\beta}_{\textsc{i}}$ as 
\begin{equation*}
V_{\textsc{i}}^{\textup{r}}
= \frac{1}{M} \sum_{i=1}^M \left( \frac{\V( {\varepsilon}_{i\cdot,\textsc{i}}(1) )}{e}  + \frac{\V ( {\varepsilon}_{i\cdot,\textsc{i}}(0))}{1-e} \right). 
\end{equation*}

\begin{theorem}  \label{thm:cluster_random}
Under random design and Assumption \ref{asu:cluster} with \(\mathbb{E}_\circ = \mathbb{E}_{(y,Z)}\) and \(\mathbb{P}_\circ = \mathbb{P}_{(y,Z)}\), we have
$\beta^{\textup{r}} = \mathbb{E}(y_{ij}(1) - y_{ij}(0))$
and
\begin{equation*}
\left( V_{\textsc{i}}^{\textup{r}} \right)^{-1/2} M^{1/2}(\hat{\beta}_{\textsc{i}} - \beta^{\textup{r}}) 
\overset{\textup{d}}{\to} \mathcal{N} \left( 0, 1  \right)
\text{ and }
M \hat{V}_{\textsc{lz,i}} =
V_{\textsc{i}}^{\textup{r}} + o(1;\mathbb{P}_{(y,Z)}).
\end{equation*}

\end{theorem}
The coefficient of $Z_i$ from individual-level OLS fit in \eqref{eq:cluster_fit} identifies the ATE.
Theorem \ref{thm:cluster_random} provides a random-design counterpart to
existing design-based results for cluster-randomized experiments. The
identification of $\beta^{\textup{r}}$ follows from cluster-level random assignment,
while its asymptotic normality and the consistency of the LZ
variance estimator follow from clustered-regression asymptotics; see
\citet{LiangZeger1986} and \citet{HansenLee2019a}. 
Related design-based results for individual-level regressions in cluster-randomized experiments appear in \cite{Schochet2013} and \cite{SuDing2021}.

\subsubsection{With additive covariates}
We consider the OLS fit with additive regressors:
\begin{equation}
\text{lm}(Y_{ij} \sim 1 + Z_{ij} + {x}_{ij}). 
\label{eq:cluster_add}
\end{equation}
Let $\hat{\beta}_{\textsc{f}}$ denote the coefficient of $Z_{ij}$ from the above OLS fit
and $\hat{\varepsilon}_{ij,\textsc{f}}$ denote the residual from the same OLS fit.
Define $X_{i,\textsc{f}}$ as an $n_{i} \times(2+ p_{x})$ matrix with row $j$ equaling $(1, Z_{i j}, {x}_{i j}^{\mathrm{T}}), j=1, \ldots, n_{i}$. Stack $X_{i,\textsc{f}}$ together to obtain an $n \times(2+p_{x})$ matrix $X_{\textsc{f}}$. Define $\hat{\varepsilon}_{ij,\textsc{f}}$ as the residual from the OLS fit in \eqref{eq:cluster_add}. 
Define an $n_{i} \times n_{i}$ matrix $\hat{U}_{i,\textsc{f}}=(\hat{\varepsilon}_{ij,\textsc{f}} \hat{\varepsilon}_{ik,\textsc{f}})_{1 \leq j, k \leq n_{i}}$. 
The cluster-robust variance estimator of $\hat{\beta}_{\textsc{f}}$ equals
\begin{equation}
\hat{V}_{\textsc{lz,f}} 
= \frac{1}{M} \left[\left( \frac{1}{M} X_{\textsc{f}}^{\mathrm{T}} X_{\textsc{f}}\right)^{-1}\left( \frac{1}{M} \sum_{i=1}^{M} X_{i,\textsc{f}}^{\mathrm{T}} \hat{U}_{i,\textsc{f}} X_{i,\textsc{f}}\right)\left( \frac{1}{M} X_{\textsc{f}}^{\mathrm{T}} X_{\textsc{f}}\right)^{-1}\right]_{(2,2)}.
\label{eq:V_LZ_F}
\end{equation}

\paragraph*{Random Design}
Define ${\varepsilon}^{\textup{r}}_{i\cdot,\textsc{f}}(z) = \sum_{j=1}^{n_i} (\varepsilon^{\textup{r}}_{i j}(z) -\dot{x}_{ij}^{\mathrm{T}} \gamma^{\textup{r}}_{\textsc{f}}) M/N $ for $z\in\{0,1\}$ with $\gamma^{\textup{r}}_{\textsc{f}}$ being the probability limit of the coefficient of ${x}_{ij}$ from the OLS fit in \eqref{eq:cluster_add}. 
By the theory in Section \ref{sec:Z_random_cluster}, we derive the asymptotic variance of $\hat{\beta}_{\textsc{f}}$ under random design as 
\begin{equation*}
V_{\textsc{f}}^{\textup{r}}
= \frac{1}{M} \sum_{i=1}^M
\left( \frac{\V( {\varepsilon}^{\textup{r}}_{i\cdot,\textsc{f}}(1))}{e} + \frac{\V( {\varepsilon}^{\textup{r}}_{i\cdot,\textsc{f}}(0))}{1-e} \right).
\end{equation*}
\begin{theorem} \label{thm:cluster_add_r}
Under random design and Assumption \ref{asu:cluster} with \(\mathbb{E}_\circ = \mathbb{E}_{(y,Z,x)}\) and \(\mathbb{P}_\circ = \mathbb{P}_{(y,Z,x)}\), we have
$\beta^{\textup{r}} = \mathbb{E}(y_{ij}(1) - y_{ij}(0))$
and
\begin{equation*}
\left( V_{{\textsc{f}}}^{\textup{r}} \right)^{-1/2} M^{1/2}(\hat{\beta}_{\textsc{f}} - \beta^{\textup{r}})
\overset{\textup{d}}{\to} \mathcal{N} \left( 0, 1 \right)
\text{ and }
M \hat{V}_{\textsc{lz,f}} 
= V_{{\textsc{f}}}^{\textup{r}} 
+ o(1;\mathbb{P}_{(y,Z,x)}),
\end{equation*}
\end{theorem}
The coefficient of $Z_{ij}$ from individual-level OLS fit in \eqref{eq:cluster_add} identifies the ATE.
Under random design, Theorem \ref{thm:cluster_add_r} further establishes the asymptotic normality of $\hat\beta_{\textsc{f}}$ and the consistency of $\hat{V}_{\textsc{lz,f}}$ for $V_{\textsc{f}}^{\textup{r}}$. 
As in Theorem \ref{thm:cluster_random}, these conclusions follow from the general clustered-sample asymptotic theory of \cite{LiangZeger1986} and \citet{HansenLee2019a}.

\paragraph*{Mixed Design}
Stack the covariates in cluster $i$ to obtain $x_i = (x_{ij}: j = 1, \ldots, n_i)$. 
Define ${\varepsilon}^{\textup{m}}_{i\cdot,\textsc{f}}(z) = \sum_{j=1}^{n_i} (\varepsilon^{\textup{m}}_{i j}(z) - \ddot{x}_{ij}^{\mathrm{T}} \gamma^{\textup{m}}_{\textsc{f}})  M/N$ for $z\in\{0,1\}$ with $\gamma^{\textup{m}}_{\textsc{f}}$ being the probability limit of the coefficient of $\ddot{x}_{ij}$ from the OLS fit in \eqref{eq:cluster_add}. 
By the theory in Section 
\ref{sec:Z_mixed_cluster}, we derive the asymptotic variance of $\hat{\beta}_{\textsc{f}}$ under mixed design as 
\begin{align*}
V^{\textup{m}}_{\textsc{f}}
= \frac{1}{M} \sum_{i=1}^M \left( 
\frac{ \mathbb{E}\left( {\varepsilon}^{\textup{m}}_{i\cdot,\textsc{f}}(1)^2 \mid x_i \right)}{e}
+ \frac{\mathbb{E}\left( {\varepsilon}^{\textup{m}}_{i\cdot,\textsc{f}}(0)^2 \mid x_i \right)}{1-e}  - \mathbb{E}\left( {\varepsilon}^{\textup{m}}_{i\cdot,\textsc{f}}(1) - {\varepsilon}^{\textup{m}}_{i\cdot,\textsc{f}}(0) \mid x_i \right)^2 \right) 
\end{align*}
and the asymptotic bias of $\hat{V}_{\textsc{lz,f}} $ for $V^{\textup{m}}_{\textsc{f}}$ as
\begin{equation}
B_{{\textsc{f}}}^{\textup{m}}
= \frac{1}{M} \sum_{i=1}^M \mathbb{E}
\left(
{\varepsilon}^{\textup{m}}_{i\cdot,\textsc{f}}(1) - {\varepsilon}^{\textup{m}}_{i\cdot,\textsc{f}}(0) \mid x_i \right)^2.\label{eq:bias_cluster_F_m}
\end{equation}
\begin{theorem} \label{thm:cluster_add_m}
Under mixed design and Assumption \ref{asu:cluster} with \(\mathbb{E}_\circ = \mathbb{E}_{(y,Z)|x}\) and \(\mathbb{P}_\circ = \mathbb{P}_{(y,Z)|x}\), we have 
$\beta^{\textup{m}} = N^{-1}\sum_{ij}\mathbb{E}\left(y_{ij}(1) - y_{ij}(0) \mid x_{ij} \right)$
and
\begin{equation*}
\left( V_{{\textsc{f}}}^{\textup{m}} \right)^{-1/2} M^{1/2}(\hat{\beta}_{{\textsc{f}}} - \beta^{\textup{m}}) \mid X
\overset{\textup{d}}{\to} \mathcal{N} \left( 0, 1 \right)
\text{ and }
M \hat{V}_{\textsc{lz,f}} 
= V_{{\textsc{f}}}^{\textup{m}} + B_{{\textsc{f}}}^{\textup{m}} + o(1;\mathbb{P}_{(y,Z)|x}).
\end{equation*}
\end{theorem}
The coefficient of $Z_{ij}$ from individual-level OLS fit in \eqref{eq:cluster_add} identifies the ATE.
In parallel to Theorem \ref{thm:PRA_m},
Theorem \ref{thm:cluster_add_m} indicates that under mixed design, $\hat{V}_{\textsc{lz,f}}$ in \eqref{eq:V_LZ_F} is conservative with asymptotic bias $B^{\textup{m}}_{\textsc{f}}$ in \eqref{eq:bias_cluster_F_m}.

\subsubsection{With fully-interacted covariates}
Define $\bar{x} = N^{-1} \sum_{ij} x_{ij}$ and $\ddot{x}_{ij} = x_{ij} - \bar{x}$. Consider the OLS fit with fully-interacted covariates:
\begin{equation}
\text{lm}(Y_{ij} \sim 1 + Z_{ij} + \ddot{x}_{ij} + Z_{ij} \ddot{x}_{ij}). 
\label{eq:cluster_adj}
\end{equation}
Let $\hat{\beta}_{\textsc{l}}$ denote the coefficient of $Z_{ij}$ from the above OLS fit
and $\hat{\varepsilon}_{ij,\textsc{l}}$ denote the residual from the same fit.
Define $X_{i,\textsc{l}}$ as an $n_{i} \times (2+2 p_{x})$ matrix with row $j$ equaling $(1, Z_{i j}, \ddot{x}_{i j}^{\mathrm{T}}, Z_{ij} \ddot{x}_{i j}^{\mathrm{T}}), j=1, \ldots, n_{i}$. Stack $X_{i,\textsc{l}}$ together to obtain an $n \times(2+2 p_{x})$ matrix $X_{\textsc{l}}$. 
Define $\hat{\varepsilon}_{ij,\textsc{l}}$ as the residual from the above OLS fit. Define an $n_{i} \times n_{i}$ matrix $\hat{U}_{i,\textsc{l}}=(\hat{\varepsilon}_{ij,\textsc{l}} \hat{\varepsilon}_{i k,\textsc{l}})_{1 \leq j, k \leq n_{i}}$.

\paragraph*{Random Design}
We apply the $Z$-estimation theory under random design in Section \ref{sec:Z_random_cluster} to obtain the asymptotic variance of the estimator $\hat{\beta}_{\textsc{l}}$. 
Define $\varepsilon_{ij}^{\textup{r}}(z) = y_{ij}(z) - \mathbb{E}(y_{ij}(z) )$
and let $\gamma^{\textup{r}}_z$ be the coefficient of $\ddot{x}_{ij}$ in the OLS fit of $\varepsilon^{\textup{r}}_{ij}(z)$ on $\ddot{x}_{ij}$ over $\{i: Z_i = z\}$.
Define ${\varepsilon}^{\textup{r}}_{i\cdot,\textsc{l}}(z) = \sum_{j=1}^{n_i} (\varepsilon^{\textup{r}}_{ij}(z) - \ddot{x}_{i j}^{\mathrm{T}} \gamma^{\textup{r}}_z )M/N$ and ${x}_{i\cdot} = \sum_{j=1}^{n_i} x_{i j}(z) M/N$. 
By the theory in Section \ref{sec:Z_random_cluster}, we derive the asymptotic variance of $\hat{\beta}_{\textsc{l}}$ under random design as 
\begin{equation}
V^{\textup{r}}_{\textsc{l}} 
= \frac{1}{M} \sum_{i=1}^M
\left( \frac{\V( {\varepsilon}^{\textup{r}}_{i\cdot,\textsc{l}}(1) )}{e} 
+ \frac{\V( {\varepsilon}^{\textup{r}}_{i\cdot,\textsc{l}}(0) )}{1-e}
+ (\gamma^{\textup{r}}_1 - \gamma^{\textup{r}}_0)^{\mathrm{T}}
\V\left( {x}_{i\cdot} \right)
(\gamma^{\textup{r}}_1 - \gamma^{\textup{r}}_0) \right),
\end{equation}
which is the the asymptotic variance of $\hat{\beta}_{\textsc{l}}$ under random design, as established in Theorem \ref{thm:cluster_adj_r}.
We consider the following variance estimator for $V^{\textup{r}}_{\textsc{l}} $: 
\begin{align}
\hat{V}_{\textsc{lz},\textsc{l},\text{adj}}
= \hat{V}_{\textsc{lz},\textsc{l}}
+ (\hat{\gamma}_1 - \hat{\gamma}_0)^{\mathrm{T}}
\frac{1}{M^2} \sum_{i=1}^M ({x}_{i\cdot}-\bar{x})({x}_{i\cdot}-\bar{x})^{\mathrm{T}} 
(\hat{\gamma}_1 - \hat{\gamma}_0),
\label{eq:Lin_correction_cluster}
\end{align}
where
\begin{equation}
\hat{V}_{\textsc{lz,l}} 
= \frac{1}{M} \left[\left( \frac{1}{M} X_{\textsc{l}}^{\mathrm{T}} X_{\textsc{l}}\right)^{-1}
\left( \frac{1}{M} \sum_{i=1}^{M} X_{i,\textsc{l}}^{\mathrm{T}} \hat{U}_{i,\textsc{l}} X_{i,\textsc{l}}\right)
\left( \frac{1}{M} X_{\textsc{l}}^{\mathrm{T}} X_{\textsc{l}}\right)^{-1}\right]_{(2,2)}
\label{eq:V_LZ_L}
\end{equation}
is the LZ variance estimator from regression in \eqref{eq:cluster_adj}.
Similar to Theorem~\ref{thm:FRA_random}, we require a correction term in the variance estimator to account for the additional uncertainty introduced by centering the covariates. This issue cannot be resolved by directly applying Theorem~\ref{thm:ME_random_cluster} to the OLS fit in \eqref{eq:cluster_adj}. Accordingly, we employ augmented estimating equations to derive the consistent variance estimator in \eqref{eq:Lin_correction_cluster}, in parallel with the development in Section \ref{sec:LinRCT}.

\begin{theorem} \label{thm:cluster_adj_r}
Under random design and Assumption \ref{asu:cluster} with \(\mathbb{E}_\circ = \mathbb{E}_{(y,Z,x)}\) and \(\mathbb{P}_\circ = \mathbb{P}_{(y,Z,x)}\), we have
$\beta^{\textup{r}} = \mathbb{E}(y_{ij}(1) - y_{ij}(0))$
and
\[
\left( V_{\textsc{l}}^{\textup{r}} \right)^{-1/2} M^{1 / 2}\left(\hat{\beta}_{\textsc{l}}-\beta^{\textup{r}} \right) 
\overset{\textup{d}}{\to} \mathcal{N}(0,1)
\text{ and }
M \hat{V}_{\textsc{lz},\textsc{l},\text{adj}}
= V^{\textup{r}}_{\textsc{l}} + o(1;\mathbb{P}_{(y,Z,x)}).
\]
Moreover,  
\[
M \hat{V}_{\textsc{lz,l}} 
= V_{\textsc{l}}^{\textup{r}}
- B_{\textsc{l}}^{\textup{r}}
+ o(1;\mathbb{P}_{(y,Z,x)})
\] 
with \begin{equation}
B_{\textsc{l}}^{\textup{r}} = 
(\gamma^{\textup{r}}_1 - \gamma^{\textup{r}}_0)^{\mathrm{T}}
\frac{1}{M} \sum_{i=1}^M \V\left( {x}_{i\cdot} \right)
(\gamma^{\textup{r}}_1 - \gamma^{\textup{r}}_0). 
\label{eq:bias_cluster_L_r}
\end{equation}
\end{theorem}
The coefficient of $Z_i$ from individual-level OLS fit in \eqref{eq:cluster_fit} identifies the ATE.
Theorem \ref{thm:cluster_adj_r} indicates that under random design, $\hat{V}_{\textsc{lz,l},\text{adj}}$ in \eqref{eq:Lin_correction_cluster} is consistent.
Importantly, the adjusted estimator $\hat V_{\textsc{lz},\textsc{l},\text{adj}}$ in \eqref{eq:Lin_correction_cluster} is a novel contribution.
Moreover, under random design, the cluster-robust variance estimator $\hat{V}_{\textsc{lz},\textsc{l}}$ in \eqref{eq:V_LZ_L} is anti-conservative with asymptotic bias $B_{\textsc{l}}^{\textup{r}}$ in \eqref{eq:bias_cluster_L_r}.
\cite{SuDing2021}
establish the validity of $\hat V_{\textsc{lz,l}}$ under design-based framework.
where $x$ is fixed.

\begin{remark}
Similar to Remark \ref{re:correction_Lin}, we can also obtain consistent variance estimator based on the cluster-robust covariance estimator with the augmented estimating equations, though with finite sample difference to $\hat{V}_{\textsc{lz},\textsc{l},\text{adj}}$. 
\end{remark}

\paragraph*{Mixed Design}
Define $\varepsilon_{ij}^{\textup{m}}(z) = y_{ij}(z) - N^{-1} \sum_{ij} \mathbb{E}(y_{ij}(z) \mid  x_{i} )$
and let $\gamma^{\textup{m}}_z$ be the coefficient of $\ddot{x}_{ij}$ in the OLS fit of $\varepsilon^{\textup{m}}_{ij}(z)$ on $\ddot{x}_{ij}$ over $\{i: Z_i = z\}$.
Define ${\varepsilon}^{\textup{m}}_{i\cdot,\textsc{l}}(z) = \sum_{j=1}^{n_i} (\varepsilon^{\textup{m}}_{ij}(z) - \ddot{x}_{i j}^{\mathrm{T}} \gamma^{\textup{m}}_z )M/N$.
By the theory in Section \ref{sec:Z_mixed_cluster}, we derive the asymptotic variance of $\hat{\beta}_{\textsc{l}}$ under mixed design as 
\begin{align*}
V^{\textup{m}}_{\textsc{l}}
= \frac{1}{M} \sum_{i=1}^M 
\left( 
\frac{ \mathbb{E}\left( {\varepsilon}^{\textup{m}}_{i\cdot,\textsc{l}}(1)^2 \mid x_i \right)}{e}
+ \frac{\mathbb{E}\left( {\varepsilon}^{\textup{m}}_{i\cdot,\textsc{l}}(0)^2 \mid x_i \right)}{1-e}  - \mathbb{E}\left( {\varepsilon}^{\textup{m}}_{i\cdot,\textsc{l}}(1) - {\varepsilon}^{\textup{m}}_{i\cdot,\textsc{l}}(0) \mid x_i \right)^2 
\right) 
\end{align*}
and the asymptotic bias of $\hat{V}_{\textsc{lz,l}}$ in \eqref{eq:V_LZ_L} for $V^{\textup{m}}_{\textsc{l}}$ as
\begin{equation}
B_{\textsc{l}}^{\textup{m}}
= \frac{1}{M} \sum_{i=1}^M \mathbb{E}\left( {\varepsilon}^{\textup{m}}_{i\cdot,\textsc{l}}(1) - {\varepsilon}^{\textup{m}}_{i\cdot,\textsc{l}}(0) \mid x_i \right)^2. \label{eq:bias_cluster_L_m}
\end{equation}
\begin{theorem} \label{thm:cluster_adj_m}
Under mixed design and Assumption \ref{asu:cluster} with \(\mathbb{E}_\circ = \mathbb{E}_{(y,Z)|x}\) and \(\mathbb{P}_\circ = \mathbb{P}_{(y,Z)|x}\), we have
$\beta^{\textup{m}} = N^{-1}\sum_{ij}\mathbb{E}\left(y_{ij}(1) - y_{ij}(0) \mid x_{i} \right)$ 
and
\begin{equation*}
\left( V_{\textsc{l}}^{\textup{m}} \right)^{-1/2} M^{1 / 2}\left(\hat{\beta}_{\textsc{l}}-\beta^{\textup{m}} \right)  \mid X
\overset{\textup{d}}{\to}
\mathcal{N}(0,1)
\text{ and }
M \hat{V}_{\textsc{lz,l}} 
= V_{\textsc{l}}^{\textup{m}} 
+ B_{\textsc{l}}^{\textup{m}}
+ o(1;\mathbb{P}_{(y,Z)|x}).
\end{equation*}
\end{theorem}
The coefficient of $Z_{ij}$ from individual-level OLS fit in \eqref{eq:cluster_adj} identifies the conditional average treatment effect.
Theorem \ref{thm:cluster_adj_m} indicates that under mixed design, $\hat{V}_{\textsc{lz,l}}$ in \eqref{eq:V_LZ_L} is conservative with asymptotic bias $B^{\textup{m}}_{\textsc{l}}$ in \eqref{eq:bias_cluster_L_m}.
Unlike in the random-design case, centering around covariates introduces additional uncertainty, whereas under a mixed design, no such issue arises, and the variance estimator remains conservative.

\section{Simulation} \label{sec:Simulation}
We report simulation results for complete randomization and cluster randomization.
For each setting, we consider: random and mixed designs, because the fixed design formulation assumes away the benefits of randomization. 
In the random design, all random variables are independently redrawn in each simulation iteration. 
In the mixed design, we hold $\{x_i\}_{i=1}^n$ fixed while redrawing all other random variables. To examine robustness to model misspecification, the outcome equations include higher order of covariate $x_i$.
For each case, we report the estimand (``estimand''), the simulation estimate (``estimate''), the asymptotic standard error (``asym SE''), the estimated standard error (EHW for i.i.d. data under ``EHW SE'' and LZ for clustered data under ``LZ SE''), and the empirical coverage of the 95\% confidence interval based on the corresponding variance estimator (``coverage'').
Rows labeled ``Fisher'' correspond to regressions with covariates, and rows labeled ``Lin'' correspond to regressions with fully interacted covariates.
For the fully interacted specification, we also report the corrected EHW or LZ variance estimator and the coverage of the associated 95\% CI, marked with an asterisk (*).

\subsection{Complete randomization}
We conduct a Monte Carlo simulation with $n = 1000$ observations per sample and $B = 5000$ simulation replications. 
We generate the covariate $x_i$ independently from a standard normal distribution: $x_i \overset{\text{i.i.d.}}{\sim} \mathcal{N}(0,1)$.
Treatment assignment is randomized such that: $\mathbb{P}(Z_i = 1) = 0.5$.
Potential outcomes are generated as follows:
\begin{align*}
Y_i(1) &= \mu_1 + \gamma_1 x_i^2 + \varepsilon_i(1), \quad \varepsilon_i(1) \sim \mathcal{N}(0,1), \\
Y_i(0) &= \mu_0 + \gamma_0 x_i^3 + \varepsilon_i(0), \quad \varepsilon_i(0) \sim \mathcal{N}(0,1),
\end{align*}
where the parameters are set as $(\mu_1, \mu_0, \gamma_1, \gamma_0) = (3,2,2,1)$.
The observed outcome is defined as: 
\[
y_i = Z_i Y_i(1) + (1 - Z_i) Y_i(0).
\]
We present the results for three OLS regression specifications: without covariates, with covariates, and with fully interacted covariates.
As established in Theorem \ref{thm:y_Z_r}, the difference-in-means estimator is consistent under the random design, and the EHW variance estimator is also consistent in this setting.
Similarly, Theorems \ref{thm:PRA_r} and \ref{thm:PRA_m} verifies that the estimator from the regression with covariates is consistent under both the random and mixed designs. The EHW variance estimator remains consistent under the random design but is conservative under the mixed design.
For the regression with fully interacted covariates, Theorems \ref{thm:FRA_random} and \ref{thm:FRA_m} confirms that the estimator is consistent under both random and mixed designs. However, the EHW variance estimator is anti-conservative under the random design and conservative under the mixed design. Nonetheless, the corrected EHW variance estimator is consistent under the random design, as verified by Theorem \ref{thm:FRA_random}.

\subsection{Cluster randomization}
We conduct a Monte Carlo simulation with $M = 160$ clusters per sample and $B = 5000$ simulation replications. 
The cluster sizes are drawn from a uniform distribution.
Each cluster $i$ contains $n_i$ units where 
$n_i=\operatorname{round}\!\left(\frac{1000}{M}U_i\right)$ with $U_i \sim \operatorname{Unif}(0.6, 1.4)$.
where: 
The cluster sizes $n_i$ are fixed across simulation replications.
Each unit $(i, j)$ has a covariate drawn independently from a standard normal distribution: $x_{i j} \overset{\text{i.i.d.}}{\sim} \mathcal{N}(0,1)$.
The covariance matrix $\Sigma_i$ for each cluster $i$ is generated as:
\[
\Sigma_{i} = \frac{(\tilde{\Sigma}_{ij}: j=1,\ldots,n_i)(\tilde{\Sigma}_{ij}: j=1,\ldots,n_i)^{\mathrm{T}}}{\sqrt{(\tilde{\Sigma}_{ij}: j=1,\ldots,n_i)(\tilde{\Sigma}_{ij}: j=1,\ldots,n_i)^{\mathrm{T}}}_{ii}},
\]
where the elements $\tilde{\Sigma}_{ij}$ are drawn independently from a uniform distribution: $\tilde{\Sigma}_{ij} \overset{\text{i.i.d.}}{\sim} \text{Unif}[0,1]$.
The potential outcomes for each unit in cluster $i$ are generated as:
\begin{align*}
(Y_{i j}(1))_{j=1}^{n_i} 
&~=~ 
\mu_1 + \gamma_1 \cdot (x_{ij}^2)_{j=1}^{n_i} + \mathcal{N}(0, 4\Sigma_{i}), \\
(Y_{i j}(0))_{j=1}^{n_i} 
&~=~ 
\mu_0 + \gamma_0 \cdot (x_{ij}^3)_{j=1}^{n_i} + \mathcal{N}(0, \Sigma_{i}).
\end{align*}
The error terms are drawn from multivariate normal distributions with covariance matrices that vary by treatment status, with the treated potential outcome having a larger variance structure than the control.

We present the results for three OLS regression specifications: without covariates, with covariates, and with fully interacted covariates.
As established in Theorem \ref{thm:cluster_random}, the difference-in-means estimator is consistent under the random design, and the LZ variance estimator is also consistent in this setting.
Similarly, Theorems \ref{thm:cluster_add_r} and \ref{thm:cluster_add_m} verify that the estimator from the regression with covariates is consistent under both the random and mixed designs. The LZ variance estimator remains consistent under the random design but is conservative under the mixed design.
For the regression with fully interacted covariates, Theorems \ref{thm:cluster_adj_r} and \ref{thm:cluster_adj_m} confirm that the estimator is consistent under both random and mixed designs. However, the LZ variance estimator is anti-conservative under the random design and conservative under the mixed design. Nonetheless, the corrected LZ variance estimator is consistent under the random design, as verified by Theorem \ref{thm:cluster_adj_r}.

\newcolumntype{P}[1]{>{\centering\arraybackslash}p{#1}}
\newcolumntype{Y}{>{\centering\arraybackslash}X}

\begin{table}[ht]
\caption{Simulation results} 
\label{table:simulation}
\centering

\small
\setlength{\tabcolsep}{3pt}
\renewcommand{\arraystretch}{1.05}

\begin{subtable}[t]{0.99\textwidth}
\centering
\caption{Complete randomization}

\begin{tabularx}{0.99\textwidth}{
@{}
P{0.18\textwidth}|
P{0.09\textwidth}|
P{0.10\textwidth}|
Y|Y|Y|Y
@{}
} 
\hline
Specification 	&	 Design 	&	estimand	&	estimate	&	asym SE	&	EHW SE	&	coverage	\\
\hline
no $X$ &	Random &	3.000	&	3.001	&	0.223	&	0.223	&	0.953		\\
\cline{1-7}	
\multirow{2}{*}{Fisher}	&	Random &	3.000	&	2.998	&	0.202	&	0.201	&	0.948	\\
\cline{2-7} 
&	Mixed	&	3.032	&	3.028	&	0.108	&	0.144	&	0.988		\\
\cline{1-7}		
\multirow{3}{*}{Lin}	&	\multirow{2}{*}{Random}	&	\multirow{2}{*}{3.000}	&	\multirow{2}{*}{2.998}	&	\multirow{2}{*}{0.202}	&	0.177	&	0.915		\\
\cline{6-7}
&		&		&		&		&	0.201*	&	0.948*	\\
\cline{2-7}	
&	Mixed	&	3.032	&	3.027	&	0.108	&	0.141	&	0.987	\\
\bottomrule
\end{tabularx}
\end{subtable}

\vspace{0.8em}
\begin{subtable}[t]{0.99\textwidth}
\centering
\caption{Cluster randomization}
\label{tab:overview-cluster}

\begin{tabularx}{0.99\textwidth}{
@{}
P{0.18\textwidth}|
P{0.09\textwidth}|
P{0.10\textwidth}|
Y|Y|Y|Y
@{}
} 
\hline
Specification 	&	 Design 	&	estimand	&	estimate	&	asym SE	&	LZ SE	&	coverage	\\
\hline
no $X$	&	Random	&	4.000	&	3.995	&	0.261	&	0.258	&	0.944		\\
\cline{1-7}			
\multirow{2}{*}{Fisher}	&	Random	&	4.000	&	3.990	&	0.242	&	0.240	&	0.945		\\
\cline{2-7}	
&	Mixed	&	3.863	&	3.861	&	0.192	&	0.231	&	0.981	\\
\cline{1-7}	
\multirow{3}{*}{Lin} &	\multirow{2}{*}{Random} &	\multirow{2}{*}{4.000}	&	\multirow{2}{*}{3.990}	&	\multirow{2}{*}{0.242}	&	0.220	&	0.924	\\
\cline{6-7}		
&		&		&		&		&	0.239*	&	0.947*	\\
\cline{2-7}		
&	Mixed	&	3.863	&	3.861	&	0.191	&	0.213	&	0.970 \\																		
\bottomrule
\end{tabularx}
\end{subtable}
\caption*{ \footnotesize Note: We use * to denote the corrected EHW or LZ SE and the coverage of the associated 95\% CI. }
\end{table}

\section{Discussion}
\label{sec:Discussion}
We unify the literature by developing a general theory for $Z$-estimation with random, fixed, and mixed regressors, covering both independent and clustered data. We clarify how regression coefficients should be interpreted under misspecification and different sources of randomness, and derive the corresponding inference theory. The usual robust variance estimator is consistent under random design but generally conservative under fixed and mixed designs. We apply these results to OLS and regression adjustment in completely randomized and cluster-randomized experiments, characterize the associated causal estimands, and provide corrections for conventional variance estimators in fully interacted specifications with random covariates.

Our framework complements the design-based inference literature \cite{AbadieAtheyImbens2020} by clarifying how different sources of sampling randomness affects both parameter interpretation and robust inference. 
Although robust variance estimators may be conservative under both mixed and design-based analyses, the sources of conservativeness differ. In our mixed-design framework, the additional term arises from nonzero conditional means of the estimating equations under misspecification, whereas in design-based inference conservativeness typically reflects unidentified treatment-effect heterogeneity.

Although our applications focus on linear and instrumental-variable regressions, the $Z$-estimation inference result applies more broadly to nonlinear models. In such settings, however, the interpretation of the resulting parameters is more challenging and requires further study.
Under misspecification, such models generally target pseudo-true parameters defined by their population objectives or estimating equations. Their interpretation depends on the estimation criterion: likelihood-based
estimators minimize the expected Kullback--Leibler divergence from the true conditional distribution, whereas nonlinear least squares provides an $L^2$ approximation to the conditional mean. Under mixed design, these objectives are evaluated conditional on the fixed regressors. The resulting pseudo-true parameter may therefore depend on which regressors are treated as random and which are conditioned on.
This distinction also matters for inference. 
Under likelihood misspecification
\citep{White1982}, the sandwich variance accounts separately for the curvature of
the objective and the variance of the score. More generally, for $Z$-estimation, it accounts separately for the sensitivity and variability of the estimating equations.
Under mixed design, the additional term $B^{\mathrm m}$ arises when the estimating equations have nonzero conditional means given the fixed regressors.
Valid inference for a pseudo-true parameter, however, does not by itself provide that parameter with a causal or otherwise substantive interpretation.

\citet{AngristChernozhukovFernandez-Val2006} study the inference and interpretation of quantile regression under random design, whereas \citet{AbadieImbensZheng2014} consider a covariate-conditional estimand but derive unconditional asymptotic results. Extending these ideas to fixed and mixed designs is a natural direction. Under misspecification, quantile-regression coefficients can be viewed as weighted linear approximations to conditional quantile functions. In a mixed-design framework, however, the approximation target is evaluated conditional on the realized fixed regressors. Hence both the interpretation and the asymptotic variance may depend on the source of randomness. Because quantile regression is nonsmooth, this extension requires separate asymptotic arguments beyond our differentiable \(Z\)-estimation framework.

\section*{Acknowledgement}
Peng Ding is partially supported by the U.S. National Science Foundation \# 2514234.

\section*{Supplementary material}
The supplementary material includes the results for the local average treatment effect framework and proofs of all theorems.

\bibliographystyle{ecta}
\bibliography{MRMR}

\appendix

\renewcommand{\theproposition}{S\arabic{proposition}}
\renewcommand{\theexample}{S\arabic{example}}
\renewcommand{\thetable}{S\arabic{table}}
\renewcommand{\theequation}{S\arabic{equation}}
\renewcommand{\thesection}{S\arabic{section}}
\renewcommand{\thecorollary}{S\arabic{corollary}}
\renewcommand{\theremark}{S\arabic{remark}}

\setcounter{equation}{0}
\renewcommand {\theequation} {S\arabic{equation}}
\setcounter{definition}{0}
\renewcommand {\thedefinition} {S\arabic{definition}}
\setcounter{example}{0}
\renewcommand {\theexample} {S\arabic{example}}
\setcounter{proposition}{0}
\renewcommand {\theproposition} {S\arabic{proposition}}
\setcounter{corollary}{0}
\renewcommand {\thecorollary} {S\arabic{corollary}}

\newpage

\begin{center}
\Large\bfseries{Supplementary Material}
\end{center}

Section~\ref{sec:add} presents the results for the local average treatment effect framework.

Section~\ref{sec:Proof of sec ME} gives the proof of the results in Section \ref{sec:ME}.

Section~\ref{sec:proof of i.i.d.Theory} gives the proof of the results in Section \ref{sec:i.i.d.Theory}. 

Section~\ref{sec:proof of application} gives the proof of the results in Sections \ref{sec:application} and \ref{sec:add}. 

Section \ref{sec:Proof of cluster} gives the proof of the results in Section \ref{sec:cluster}.

\paragraph{Notation:}
Let $[\cdot]_{(1,1)+(2,2)-2(1,2)}$ denote the sum of the $(1,1)$th and the $(2,2)$th elements minus twice the $(1,2)$th element of the matrix $[\cdot]$.
Let $[\cdot]_{(1:2,1:2)}$ denote the submatrix of the first two rows and the first two columns.
We use * to denote the terms that are irrelevant in the matrix. 
We use \(\operatorname{tsls}(y_i \sim x_i \mid z_i)\) to denote the two-stage least-squares regression of \(y_i\) on \(x_i\), using \(z_i\) as the instrument for \(x_i\).
We use \(\V\) to denote variance and \(\mathbb{C}\) to denote covariance.
We use ``LLN'' and ``CLT'' to denote ``law of large numbers'' and ``central limit theorem,'' respectively.
We will use ``$\leq_{\mathrm{HI}}$'' for the steps invoking H{\"o}lder's\ inequality.

\section{IV regression } \label{sec:add}
In this section, we consider the local average treatment effect (LATE) framework.

\subsection{Main results} \label{sec:IV}
Consider the setting where the treatment choice $D_i$ is endogenous and we use $Z_i$ as the instrumental variable for $D_i$. 
We assume both $Z_i$ and $D_i$ are binary.
We consider potential outcome models for both outcomes and treatment decisions. For each participant $i=1,\dots,n$, we use $Y_i(d)$ to denote the potential outcome of participant $i$ if he/she makes treatment decision $d$, and we use $D_i(z)$ to denote the potential treatment decision of participant $i$ if he/she has assigned treatment $z$. Their observed counterparts are:
\begin{align*}
D_i &~=~ D_i(1) Z_i + D_i(0) (1-Z_i), \\
y_i &~=~ Y_i(1) D_i + Y_i(0) (1-D_i). 
\end{align*}
Following the usual classification in the local average treatment effect framework in \cite{ImbensAngrist1994} and \cite{AngristImbensRubin1996}, each participant can be one of four types: complier ($\textup{c}$), always taker ($\textup{a}$), never taker ($\textup{n}$), or a defier ($\textup{d}$). An individual $i$ is said to be a complier if $\{D_i(0)=0,D_i(1)=1\}$, an always taker if $\{D_i(0)=D_i(1)=1\}$, a never taker if $\{D_i(0)=D_i(1)=0\}$, and a defier if $\{D_i(0)=1,D_i(1)=0\}$.

We consider the two-stage least-squares (2SLS) fit for estimation, which is a special case of $Z$-estimation.
Similar to OLS analysis in Section \ref{sec:application}, we omit the discussion of the fixed design theory, because it does not quantify the advantages of randomization with a misspecified linear instrumental variable model. 

In this section, we impose the following assumption. 
\begin{assumption}\label{asu:IV}
$\{W_i\}_{i=1}^n = \left\{(Z_i, Y_i(0), Y_i(1),D_i(0),D_i(1),x_i)\right\}_{i=1}^n$ is an i.i.d.\ sample that satisfies
\begin{enumerate}[(a)]
\item $Z_i \indep \left\{(Y_i(0),Y_i(1),D_i(0),D_i(1),x_i)\right\}$ and  
$\mathbb{P}(Z_i=1)=e\in(0,1)$;
\item $\mathbb{E}_\circ[ Y_{i}( d)^{4}]<\infty $ for all $d\in \{0,1\}$;
\item $\mathbb{P}_\circ( D_{i}( 0) =1,D_{i}(1) =0) =0$ or, equivalently, $\mathbb{P}_\circ( D_{i}(1)\geq D_{i}(0)) =1$;
\item $\mathbb{E}_\circ[D_i(1) - D_i(0)] > 0$;
\item $\mathbb{E}_\circ(x_ix_i^{\mathrm{T}})$ is positive definite.
\end{enumerate}
\end{assumption}
As usual in the literature, Assumption \ref{asu:IV}(c) imposes that there are no defiers in our population of participants in order to identify the LATE. We denote the types of compliance behavior by $U_i\in\{\textup{c},\textup{a},\textup{d}\}$.

\subsubsection{Without covariates}
We first consider the 2SLS fit without covariates:
\begin{equation}
\operatorname{tsls}\bigl(y_i \sim 1+D_i \mid 1+Z_i\bigr), 
\label{eq:IV_no_x}
\end{equation}
which corresponds to the moment function
\begin{align}
\psi\bigl(w_i;\alpha_{\textsc{iv}},\beta_{\textsc{iv}} \bigr)
= \tilde Z_{i} 
\left(y_i-\alpha_{\textsc{iv}} -\beta_{\textsc{iv}} D_i  \right)
\quad \text{with} \quad 
\tilde Z_{i} = (1, Z_i)^{\mathrm{T}}.
\label{eq:IV_wc}
\end{align}
Let $\hat{\alpha}_{\textsc{iv}}$ and $\hat{\beta}_{\textsc{iv}}$ denote the intercept and coefficient of $Z_i$ from the 2SLS fit in \eqref{eq:IV_no_x}. Define $\hat{\varepsilon}_i = y_i - \hat{\alpha}_{\textsc{iv}} - \hat{\beta}_{\textsc{iv}} D_i$ as the residual from the same fit.
Define $\hat{V}_{\textsc{hw,iv}}$ as the HW variance estimator for $\hat\beta_{\textsc{iv}}$, obtained as the $(2,2)$ entry of
$\hat V_{\textsc{hw}}$ in \eqref{eq:EHW_ME},
where
\[
\psi\bigl(w_i;\hat\alpha_{\textsc{iv}},\hat\beta_{\textsc{iv}}\bigr)
= \tilde Z_i\,\hat\varepsilon_i.
\]
Importantly, $\hat{V}_{\textsc{hw,iv}}$ comes not from the second-stage OLS but from the $Z$-estimation framework developed in Section \ref{eq:EHW_ME} \citep{AngristPischke2009, Ding2024}.

Let $V_{\textsc{iv}}^{\text{r}}$ denote the asymptotic variance of $\hat{\beta}_{\textsc{iv}}$, which can be derived by applying the $Z$-estimation theorem under random design in Section \ref{sec:Z_random} with the moment function defined in \eqref{eq:IV_wc} and \(\mathbb{E}_\circ = \mathbb{E}_{(y,Z)}\).

\begin{theorem} \label{thm:AsyDist_IV_random}
Under random design and Assumption \ref{asu:IV} with \(\mathbb{E}_\circ = \mathbb{E}_{(y,D, Z)}\) and \(\mathbb{P}_\circ = \mathbb{P}_{(y,D,Z)}\), we have 
$\beta_{\textsc{iv}}^{\textup{r}}  
= \mathbb{E}[Y_i(1) - Y_i(0) \mid  D_i(1) > D_i(0)]$ and
\begin{equation*}
(V_{\textsc{iv}}^{\textup{r}} )^{-1/2} \sqrt{n}(\hat{\beta}_{\textsc{iv}}-\beta_{\textsc{iv}}^{\textup{r}} )
\overset{\textup{d}}{\to} \mathcal{N} (0, 1)
\text{ and }
n \hat{V}_{\textsc{hw,iv}}
= V_{\textsc{iv}}^{\textup{r}} + o(1;\mathbb{P}_{(y,Z,x)}).
\end{equation*}
\end{theorem}
The coefficient of $Z_i$ from the 2SLS fit in \eqref{eq:IV_no_x} identifies the LATE. 
Theorem \ref{thm:AsyDist_IV_random} indicates that under random design, $\hat{V}_{\textsc{hw,iv}}$ is consistent for $V_{\textsc{iv}}^{\textup{r}}$. 
We apply the $Z$-estimation theorem from Section \ref{sec:ME} to prove Theorem \ref{thm:AsyDist_IV_random}.  
Its asymptotic normality result is a special case of \citet[Theorem 4.1]{BugniGao2023} with a single stratum and complete randomization.
\citet{ImbensAngrist1994} derive the classical LATE identification result and
the asymptotic distribution of the IV estimator under i.i.d. sampling. Theorem \ref{thm:AsyDist_IV_random} restates the binary
instrument case in the usual 2SLS regression form within our random-design
framework and records the consistency of the corresponding HW variance
estimator.

\subsubsection{With covariates} \label{sec:add_IV}
Now we consider the 2SLS fit with covariates:
\begin{equation}
\operatorname{tsls}\bigl(y_i \sim 1+D_i+x_i \mid 1+Z_i+x_i\bigr),
\label{eq:IV}
\end{equation}
which corresponds to the moment function
\begin{align}
\psi\bigl(w_i;\alpha_{\textsc{iv,f}},\beta_{\textsc{iv,f}}, {\gamma}_{\textsc{iv,f}}\bigr)
= \tilde Z_{i,\textsc{f}} 
\left(y_i-\alpha_{\textsc{iv,f}} -\beta_{\textsc{iv,f}} D_i - x_i^{\mathrm{T}} \gamma_{\textsc{iv,f}}  \right)
\text{ with } 
\tilde Z_{i,\textsc{f}} = (1, Z_i, x_i)^{\mathrm{T}}.
\label{eq:IV_c}
\end{align}
Let $\hat{\alpha}_{\textsc{iv,f}}$, $\hat{\beta}_{\textsc{iv,f}}$ and $\hat{\gamma}_{\textsc{iv,f}}^{\mathrm{T}}$ denote the intercept and coefficients of $D_i$ and $x_i$ from the 2SLS fit in \eqref{eq:IV}.
Define the residual as $\hat{\varepsilon}_{i,\textsc{f}} = y_i - \hat{\alpha}_{\textsc{iv,f}} - \hat{\beta}_{\textsc{iv,f}} D_i - {x}_i^{\mathrm{T}} \hat{\gamma}_{\textsc{iv,f}} $.
Define $\hat{V}_{\textsc{hw,iv,f}}$ as the HW variance estimator for $\hat\beta_{\textsc{iv}}$, obtained as the $(2,2)$ entry of
$\hat V_{\textsc{hw}}$ in \eqref{eq:EHW_ME},
where
\[
\psi\bigl(w_i;\hat\alpha_{\textsc{iv,f}},\hat\beta_{\textsc{iv,f}}, \hat{\gamma}_{\textsc{iv,f}}\bigr)
= \tilde Z_{i,\textsc{f}} \hat\varepsilon_{i,\textsc{f}}.
\]

\paragraph*{Random Design}
Let $V_{\textsc{iv,f}}^{\text{r}}$ denote the asymptotic variance of $\hat{\beta}_{\textsc{iv,f}}$, which can be derived by applying the estimating-equations theorem under random design in Section \ref{sec:Z_random} with the moment function defined in \eqref{eq:IV_c} and \(\mathbb{E}_\circ = \mathbb{E}_{(y,Z,x)}\).

\begin{theorem} \label{thm:AsyDist_IV_random_x}
Under random design and Assumption \ref{asu:IV} with \(\mathbb{E}_\circ = \mathbb{E}_{(y,D,Z,x)}\) and \(\mathbb{P}_\circ = \mathbb{P}_{(y,D, Z,x)}\), we have 
$\beta_{\textsc{iv}}^{\textup{r}}  
= \mathbb{E}[Y_i(1) - Y_i(0) \mid  D_i(1) > D_i(0)]$
and
\begin{equation*}
(V_{\textsc{iv,f}}^{\textup{r}} )^{-1/2} 
\sqrt{n}(\hat{\beta}_{\textsc{iv,f}}-\beta_{\textsc{iv}}^{\textup{r}} )
\overset{\textup{d}}{\to} \mathcal{N} (0, 1)
\text{ and }
n \hat{V}_{\textsc{hw,iv,f}}
= V_{\textsc{iv,f}}^{\textup{r}} + o(1;\mathbb{P}_{(y,Z,x)}).
\end{equation*}
\end{theorem}

The coefficient of $D_i$ from the 2SLS fit in \eqref{eq:IV} identifies the LATE. 
Theorem \ref{thm:AsyDist_IV_random_x} indicates that under random design, $\hat{V}_{\textsc{hw,iv,f}}$ is consistent for $V_{\textsc{iv,f}}^{\textup{r}}$. 

\paragraph*{Mixed Design}
Define $\pi_{\textup{c}}(x_i) = \mathbb{P}(D_i(1) > D_i(0) \mid x_i)$ as the conditional probability of being a complier given covariates  
and 
$\tau_{\textup{c}}(x_i) = \mathbb{E}(Y_i(1)-Y_i(0)\mid D_i(1) > D_i(0), x_i)$ as the conditional average treatment effect among compliers.
The target parameter under mixed design is
\begin{equation}
\beta_{\textsc{iv}}^{\textup{m}} 
= \frac{ n^{-1} \sum_{i=1}^{n} \tau_{\textup{c}}(x_i) \pi_{\textup{c}}(x_i)  }{ n^{-1} \sum_{i=1}^{n} \pi_{\textup{c}}(x_i) },
\label{eq:beta_IV_m}
\end{equation}
which averages the conditional complier treatment
effects over the fixed covariate values in the sample, weighting each $x_i$
in proportion to its conditional complier probability $\pi_c(x_i)$.
Define the asymptotic bias for $V_{\textsc{iv}}^{\textup{m}}$ as
\begin{equation}
B_{\textsc{iv,f}}^{\textup{m}}
= \frac{ n^{-1} \sum_{i=1}^n 
\pi_{\textup{c}}(x_i)^2
(\tau_{\textup{c}}(x_i) - \beta_{\textsc{iv}}^{\textup{m}} )^2 }{ ( n^{-1} \sum_{i=1}^n \pi_{\textup{c}}(x_i))^2}. 
\label{eq:bias_IV}
\end{equation}
Let $V_{\textsc{iv,f}}^{\text{m}}$ denote the asymptotic variance of $\hat{\beta}_{\textsc{iv,f}}$, which can be derived by applying the $Z$-estimation theorem under mixed design in Section \ref{sec:Z_mixed} with the moment function defined in \eqref{eq:IV_c} and \(\mathbb{E}_\circ = \mathbb{E}_{(y,Z)|x}\).

\begin{theorem} \label{thm:AsyDist_IV_mixed}
Under mixed design and Assumption \ref{asu:IV} with \(\mathbb{E}_\circ = \mathbb{E}_{(y,D,Z)|x}\) and \(\mathbb{P}_\circ = \mathbb{P}_{(y,D,Z)|x}\), we have
\begin{equation*}
(V_{\textsc{iv,f}}^{\textup{m}})^{-1/2} \sqrt{n}(\hat{\beta}_{\textsc{iv,f}}-\beta_{\textsc{iv}}^{\textup{m}} ) \mid X 
\overset{\textup{d}}{\to} \mathcal{N}(0, 1 )
\text{ and }
n \hat{V}_{\textsc{hw,iv,f}} 
= V_{\textsc{iv,f}}^{\textup{m}}
+ B^{\textup{m}}_{\textsc{iv,f}} + o(1;\mathbb{P}_{(y,Z)|x}).
\end{equation*}
\end{theorem}

The coefficient of $D_i$ from the 2SLS fit in \eqref{eq:IV} identifies $\beta_{\textsc{iv}}^{\textup{m}}$ under mixed design. 
Theorem \ref{thm:AsyDist_IV_mixed} indicates that under mixed design, $\hat{V}_{\textsc{hw,iv,f}}$ is conservative for $V_{\textsc{iv,f}}^{\textup{m}}$ with asymptotic bias $B^{\textup{m}}_{\textsc{iv,f}}$ in \eqref{eq:bias_IV}.

One can also consider the interacted 2SLS specification with
complier-centered covariates,
\begin{equation}
\operatorname{tsls}\left(
y_i \sim 1+D_i+\tilde x_i + D_i\tilde x_i 
\mid
1+Z_i+\tilde x_i + Z_i\tilde x_i 
\right),
\label{eq:IV_L}
\end{equation}
where 
\[
\tilde x_i=x_i-\mu_c,
\qquad
\mu_c=\mathbb E[x_i\mid D_i(1)>D_i(0)].
\]
This specification instruments both $D_i$ and its interactions
$D_i\tilde x_i$ by $Z_i$ and $Z_i\tilde x_i$, respectively. 
\citet{DingFellerMiratrix2019} study fully interacted regression and 2SLS fit from
a finite-population, randomization-based perspective. By contrast,
\citet{ZhaoDingLi2026} analyze interacted 2SLS in an i.i.d. superpopulation
framework with conditionally valid instruments.
As in Lin's fully interacted adjustment for completely randomized experiments, centering the covariates creates an additional source of uncertainty under random design. Therefore, the variance estimator obtained by mechanically applying the usual HW formula to the displayed 2SLS moments generally ignores the uncertainty from estimating $\bar x$ and is anti-conservative. A consistent random-design variance estimator can instead be obtained by augmenting the $Z$-estimation to include the moment condition for $\bar x$, analogously to the correction in Theorem \ref{thm:FRA_random}.

Under mixed design, the covariates are conditioned on, so centering does not create additional sampling uncertainty. In this case, the usual HW variance estimator based on the 2SLS moments remains conservative, in the same sense as the mixed-design results above. Since this fully interacted IV specification is less central to our paper, we focus on the additive-covariate IV regression in Theorems \ref{thm:AsyDist_IV_random_x} and \ref{thm:AsyDist_IV_mixed}.

\subsection{Simulation results}
We conduct a Monte Carlo simulation with $n = 1000$ observations per sample and $B = 5000$ simulation replications. 
The treatment is randomized with probability: $\mathbb{P}(Z_i = 1) = 0.5$.
We assume that the population consists of three latent groups: Compliers ($\textup{c}$), Always Takers ($\textup{a}$), and Never Takers (\textup{n}), i.e., $U_i \in \{\textup{c},\textup{a},\textup{n}\}$.
These groups occur with probabilities: $(\mathbb{P}(U_i = \textup{c}), \mathbb{P}(U_i = \textup{a}), \mathbb{P}(U_i = \textup{n})) = (0.7,0.2,0.1)$. 
The covariate $x_i$ is drawn independently of $U_i$, with
$x_i\overset{\mathrm{i.i.d.}}{\sim}\operatorname{Unif}(0,1)$.
The potential outcomes are defined as:
\begin{align*}
Y_{\textup{c},i}(1) 
&= \mu_\textup{c}(1) + \gamma_1 x_i^2 + \varepsilon_{\textup{c},i}(1), \quad \varepsilon_{\textup{c},i}(1) \sim \mathcal{N}(0,1), \\
Y_{\textup{c},i}(0) 
&= \mu_\textup{c}(0) + \gamma_0 x_i + \varepsilon_{\textup{c},i}(0) , \quad \varepsilon_{\textup{c},i}(0)  \sim \mathcal{N}(0,1), \\
Y_{\textup{a},i}(1) 
&= \mu_{\textup{a}}(1) + \gamma_1 x_i^2 + \varepsilon_{\textup{a},i}(1) , \quad \varepsilon_{\textup{a},i}(1)  \sim \mathcal{N}(0,1), \\
Y_{\textup{n},i}(0) 
&= \mu_{\textup{n}}(0) + \gamma_0 x_i^2 + \varepsilon_{\textup{n},i}(0), \quad \varepsilon_{\textup{n},i}(0) \sim \mathcal{N}(0,1),
\end{align*}
where the parameters are set as $(\mu_{\textup{c}}(1), \mu_{\textup{c}}(0), \mu_{\textup{a}}(1), \mu_{\textup{n}}(0)) = (2,1,3,0)$ and $(\gamma_0,\gamma_1)=(1.5,2)$.
The observed outcome is defined as: 
\[
Y_i = Z_i 1\{U_i = \textup{c}\} Y_{\textup{c},i}(1) + (1-Z_i) 1\{U_i = \textup{c}\} Y_{\textup{c},i}(0) + 1\{U_i = \textup{a}\} Y_{\textup{c},i}(1) + 1\{U_i = \textup{n}\} Y_{\textup{n},i}(0).
\]
We present the results for the IV regression without and with covariates in Table \ref{table:IV_simulation}.
As established in Theorem \ref{thm:AsyDist_IV_random}, the estimator from the 2SLS fit without covariates is consistent under the random design, and the HW variance estimator is also consistent in this setting.
Similarly, Theorems \ref{thm:AsyDist_IV_random_x} and \ref{thm:AsyDist_IV_mixed} verifies that the estimator from the 2SLS fit with covariates is consistent under both the random and mixed designs. 
The HW variance estimator remains consistent under the random design but is conservative under the mixed design.

\newcolumntype{P}[1]{>{\centering\arraybackslash}p{#1}}
\newcolumntype{Y}{>{\centering\arraybackslash}X}

\begin{table}[t]
\caption{Simulation results of IV regression}
\label{table:IV_simulation}
\centering
\small

\setlength{\tabcolsep}{3pt}
\renewcommand{\arraystretch}{1.05}

\begin{tabularx}{0.99\textwidth}{
@{}
P{0.18\textwidth}|
P{0.09\textwidth}|
P{0.10\textwidth}|
Y|Y|Y|Y
@{}
} 
\hline	
Specification 
& Design
& estimand
& estimate
& asym.\ SE
& HW SE
& coverage \\
\hline	
no $X$	&	Random	&	2.167	&	2.166	&	0.204	&	0.205	&	0.954	\\
\cline{1-7}	
\multirow{2}{*}{Fisher}	&	Random	&	2.167	&	2.162	&	0.148	&	0.147	&	0.950	\\
\cline{2-7}	
&	Mixed	&	2.337	&	2.333	&	0.126	&	0.138	&	0.969	\\
\bottomrule
\end{tabularx}
\end{table}

\section{Proof of results in Section \ref{sec:ME}}\label{sec:Proof of sec ME}

\begin{lemma}
\label{lem:conditional-ulln}
Let $\{(w_i)\}_{i=1}^n$ be
i.i.d. Suppose that $a(w;\beta)$ is continuous
in $\beta\in\Theta$ almost surely, where $\Theta$
is compact, and $\mathbb{E}_{\circ}[\sup_{b \in\Theta} \|a(w;\beta)\|]< \infty$.
Then
\[
\sup_{b \in\Theta}
\left\|
\frac1n\sum_{i=1}^n
\left[
a(w_i; b)
-
\mathbb{E}_{\circ}\{
a(w_i; b)
\}
\right]
\right\|
=
o(1;\mathbb{P}_{\circ})
\]
with probability measure $\mathbb{P}_{\circ}$.
\end{lemma} 

\begin{proof}
Define $g(w_i;\beta)
=
a(w_i,\beta)
-
\mathbb{E}_\circ\{
a(w_i;\beta)
\}$.
Under the joint probability measure $\mathbb P$, the class $\left\{
g(\cdot; b):
b\in\Theta
\right\}$
is i.i.d. and has the integrable envelope
\[
\sup_{b\in\Theta}
\|a(w_i; b)\|
+
\mathbb{E}_{\circ}\!\left[
\sup_{b\in\Theta}
\|a(w_i;  b)\|
\right].
\]
Indeed, the expectation of this envelope is bounded by $2\mathbb{E}\!\left[
\sup_{\boldsymbol b\in\Theta}
\|a(w_i;  b)\|
\right]
<\infty$.
Therefore, the dominated uniform law of
large numbers \citep[Lemma~2.4]{NeweyMcFadden1994} gives
\[
A_n
:=
\sup_{b\in \Theta}
\left\|
\frac1n\sum_{i=1}^n
g(w_i;b)
\right\|
=o_\mathbb{P}(1).
\]
For every $\varepsilon,\eta>0$,
\[
\begin{aligned}
&\mathbb{P}\left\{
\mathbb{P}_{\circ}(A_n>\varepsilon)
>\eta
\right\}
\leq
\frac1\eta
\mathbb{E}\left[
\mathbb{P}_{\circ}(A_n>\varepsilon)
\right]
=
\frac1\eta\mathbb{P}(A_n>\varepsilon)
\to 0.
\end{aligned}
\]
This proves the result.
\end{proof}

\begin{lemma}\label{lemma:GMM_lemma}
If $\{w_i\}_{i=1}^n$ is i.i.d., $a(w; b)$ is continuous at $b \in \Theta$ with probability one, and there is a neighborhood $\mathscr{N}$ of $\beta$ such that $\mathbb{E}[\sup_{\beta\in\mathscr{N}} \|a(w;\beta)\|]< \infty$,
then for any ${\beta}_n = \beta + o(1;\mathbb{P}_{\circ})$, then $\frac1n\sum_{i=1}^n
a(w_i;{\beta}_n)
=
\frac1n\sum_{i=1}^n
\mathbb{E}_\circ\left[
a(w_i;\beta)
\right]
+
o(1;\mathbb{P}_{\circ})$
with probability measure $\mathbb{P}_{\circ}$.
If $\tilde{\beta}_n
-
\beta_n
= o(1;\mathbb{P}_{\circ})$, then
$\frac1n\sum_{i=1}^n
a(w_i;\tilde{\beta}_n)
=
\frac1n\sum_{i=1}^n
\mathbb{E}_\circ\left[
a(w_i;\beta_n)
\right]
+
o(1;\mathbb{P}_{\circ})$
with probability measure $\mathbb{P}_{\circ}$.
\end{lemma}
\begin{proof}
Under random design with \(\beta^{\diamond} = \beta^{\textup{r}}\), and \(\mathbb{E}_\circ = \mathbb{E}_{(y,x)}\) and \(\mathbb{P}_\circ = \mathbb{P}_{(y,x)}\), the first conclusion follows from Lemma 4.3 of \cite{NeweyMcFadden1994}. 
Under mixed and fixed designs, choose a compact neighborhood
$\mathscr{N}\subset\Theta$ of $\beta$.
By Lemma~\ref{lem:conditional-ulln},
\begin{equation}
\sup_{b\in\mathscr{N}}
\left\|
\frac1n\sum_{i=1}^n
a(w_i,\beta)
-
\frac1n\sum_{i=1}^n
\mathbb{E}_\circ[a(w_i,\beta)]
\right\|
=
o(1;\mathbb{P}_{\circ}).
\label{eq:sup}
\end{equation}
Moreover, continuity in $b \in \Theta$ and the integrable local
envelope imply that $\frac1n\sum_{i=1}^n
\mathbb{E}_\circ[a(w_i,b)]$ is stochastically
equicontinuous at $\beta$. In particular, if
$\beta_n=\beta+o(1;\mathbb{P}_{\circ})$, then
\begin{equation}
\frac1n\sum_{i=1}^n
\mathbb{E}_\circ[a(w_i,\beta_n)]
=
\frac1n\sum_{i=1}^n
\mathbb{E}_\circ[a(w_i,\beta)]
+
o(1;\mathbb{P}_{\circ}).
\label{eq:converge_con}
\end{equation}
Since $\beta_n\in\mathscr{N}$ with probability approaching one, evaluating \eqref{eq:sup} at $\beta_n$ and using
\eqref{eq:converge_con} gives
\[
\frac1n\sum_{i=1}^n
a(w_i;\beta_n)
=
\frac1n\sum_{i=1}^n
\mathbb{E}_\circ[a(w_i; \beta)]
+
o(1;\mathbb{P}_{\circ}).
\]
Finally, if
$\tilde{\beta}_n-\beta_n
=o(1;\mathbb{P}_{\circ})$, then
$\tilde{\beta}_n
=\beta+o(1;\mathbb{P}_{\circ})$.
Applying the first result to
$\tilde{\beta}_n$ and using \eqref{eq:converge_con} yields
\[
\begin{aligned}
\frac1n\sum_{i=1}^n
a(w_i;\tilde{\beta}_n)
&=
\frac1n\sum_{i=1}^n
\mathbb{E}_\circ[a(w_i,\beta)]
+
o(1;\mathbb{P}_{\circ})
=
\frac1n\sum_{i=1}^n
\mathbb{E}_\circ[a(w_i,\beta_n)]
+
o(1;\mathbb{P}_{\circ}),
\end{aligned}
\]
which proves the result under mixed and fixed designs.
\end{proof}

\begin{proof}[Proof of Theorem \ref{thm:ME_random}]
The asymptotic normality is a standard result.
We adapt the proof of Theorems 2.6 and 3.4 in \cite{NeweyMcFadden1994} on the consistency and the asymptotic normality of the GMM estimator to the $Z$-estimation.

By Assumption~\ref{asu:ME}(b) and the dominated uniform law of
large numbers \citep[Lemma~2.4]{NeweyMcFadden1994},
\[
\sup_{b\in\Theta}
\left\|
\bar\psi_n(b)-\mathbb{E}[\psi(w_i,b)]
\right\|
= o_{\mathbb{P}}(1).
\]
By Assumptions~\ref{asu:ME}(a)--(b)
for every $\varepsilon>0$, $c_\varepsilon
:=
\inf_{\substack{b\in\Theta\\
\|b-\beta^\textup{r}\|\geq\varepsilon}}
\|\mathbb{E}[\psi(w_i,b)] \|
>0$.
Since by definition, $\bar\psi_n(\hat\beta)=0$ with probability approaching one, the
uniform convergence above implies $\hat\beta = \beta^\textup{r} + o_{\mathbb{P}}(1)$.
The integral form of Taylor's
theorem gives
\[
0=\bar\psi_n(\hat\beta)
=
\bar\psi_n(\beta^\textup{r})
+ \tilde\Gamma_n (\hat\beta-\beta^\textup{r}),
\]
where
\[
\tilde\Gamma_n
=
\int_0^1
\nabla_b \bar\psi_n\!\left(\beta^\textup{r}+t(\hat\beta-\beta^\textup{r})\right)
\,\textup{d}t .
\]
Multiplying through by $\sqrt n$ and solving yields
\[
\sqrt n(\hat\beta-\beta^\textup{r})
=
-\tilde\Gamma_n^{-1}\sqrt n\,\bar\psi_n(\beta^\textup{r}).
\]
By Assumptions~\ref{asu:ME}(c)--(d), the dominated uniform law of large numbers
\citep[Lemma~2.4]{NeweyMcFadden1994} gives, for a neighborhood
$\mathscr N$ of $\beta^\textup{r}$,
\[
\sup_{b\in\mathscr N}
\left\|
\hat\Gamma_n(b)-\Gamma(b)
\right\|
=o_{\mathbb P}(1).
\]
Since $\hat\beta=\beta^\textup{r}+o_{\mathbb P}(1)$, the line segment between
$\beta^\textup{r}$ and $\hat\beta$ lies in $\mathscr N$ with probability approaching
one. Hence, 
we have $\tilde\Gamma_n
=
\Gamma(\beta^\textup{r})+o_{\mathbb P}(1)$.
The asymptotic normality then follows by the Slutsky's theorem and Assumption~\ref{asu:ME}(e).

It remains to establish consistency of the HW variance estimator $\hat{V}_{\textsc{hw}}$.
By Lemma \ref{lemma:GMM_lemma} with $a(w; \beta) = \psi(w; \beta)\psi(w; \beta)^{\mathrm{T}}$, we have $\hat{\Delta} = \Delta^\textup{r} + o_{\mathbb{P}}(1)$.
By Assumption \ref{asu:ME}(d) and consistency of $\hat{\beta}$, with probability approaching one,
\[
\|\hat{\Gamma}-\Gamma^\textup{r}\| \leqslant\|\hat{\Gamma}-\Gamma(\hat{\beta})\|+\|\Gamma(\hat{\beta})-\Gamma^\textup{r}\| \leqslant \sup _{b \in \mathscr{N}}\left\| \frac{\partial \bar{\psi}(b)}{\partial b^{\mathrm{T}}} - \Gamma(b)\right\|+ \|\Gamma(\hat{\beta})-\Gamma^\textup{r} \| = o_{\mathbb{P}}(1), 
\]
so that $\hat{\Gamma} = \Gamma^\textup{r} + o_{\mathbb{P}}(1)$.
The conclusion then follows from continuity of matrix inversion and multiplication.
\end{proof}

\begin{proof}[Proof of Theorem \ref{thm:ME_fixed}]
Define the conditional population moment function
\[
\Psi^{\mathrm f}_n(b)
=
\frac1n\sum_{i=1}^n
\mathbb{E}\left[
\psi(y_i,x_i;b)
\mid x_i
\right].
\]
We suppress its dependence on $n$ and $X$ below. By
definition, $\Psi^{\mathrm f}_n(\beta^{\mathrm f})
=
0$.
By Assumption \ref{asu:ME}(b) and Lemma \ref{lem:conditional-ulln},
\[
\sup_{b\in\Theta}
\left\|
\bar\psi_n(b)
-
\Psi^{\mathrm f}_n(b)
\right\|
=
o(1;{\mathbb{P}}_{y\mid x}).
\]
By Assumptions \ref{asu:ME}(a)--(b), for every $\varepsilon>0$, $c_{n,\varepsilon}^{\mathrm f}
:=
\inf_{\substack{b\in\Theta\\
\|b-\beta^{\mathrm f}\|
\geq\varepsilon}}
\left\|
\Psi^{\mathrm f}_n(b)
\right\|$
is bounded away from zero with probability approaching one.
Since
$\bar\psi_n(\hat{\beta})= 0$
with probability approaching one, the preceding uniform convergence
implies $\hat{\beta}
=
\beta^{\mathrm f}
+
o(1;{\mathbb{P}}_{y\mid x})$.
Using the integral form of Taylor's theorem,
\[
0=\bar\psi_n(\hat\beta)
=
\bar\psi_n(\beta^{\mathrm f})
+
\tilde\Gamma_n(\hat\beta-\beta^{\mathrm f}),
\]
where
\[
\tilde\Gamma_n
=
\int_0^1
\nabla_b\bar\psi_n\!\left(\beta^{\mathrm f}
+t(\hat\beta-\beta^{\mathrm f})\right)\,\textup{d}t.
\]
Multiplying through by $\sqrt n$ and solving gives
\begin{equation}
\sqrt n(\hat\beta-\beta^{\mathrm f})
=
-\tilde\Gamma_n^{-1}\sqrt n\,\bar\psi_n(\beta^{\mathrm f}).
\label{eq:normal}
\end{equation}
By Assumptions~\ref{asu:ME}(c)--(d), Lemma \ref{lem:conditional-ulln} applied locally to
$\nabla_b\psi$ gives, for a neighborhood $\mathscr N$ of $\beta^{\mathrm f}$,
\[
\sup_{b\in\mathscr N}
\left\|
\nabla_b\bar\psi_n(b)-\Gamma_n^{\mathrm f}(b)
\right\|
=
o(1;\mathbb P_{y\mid x}).
\]
Since $\hat\beta=\beta^{\mathrm f}+o(1;\mathbb P_{y\mid x})$, the line
segment between $\beta^{\mathrm f}$ and $\hat\beta$ lies in $\mathscr N$
with probability approaching one. 
By the local continuity condition in
Assumption~\ref{asu:ME}(c),
\[
\sup_{t\in[0,1]}
\left\|
\Gamma^{\mathrm f}\!\left(\beta^{\mathrm f}
+t(\hat\beta-\beta^{\mathrm f})\right)
-\Gamma^{\mathrm f}(\beta^{\mathrm f})
\right\|
=o(1;\mathbb P_{y\mid x}).
\]
Therefore, $\tilde\Gamma_n
=
\Gamma^{\mathrm f}
+
o(1;\mathbb P_{y\mid x})$.
Let $\mu_i^{\mathrm f}
=
\mathbb{E}\left[
\psi(y_i,x_i;\beta^{\mathrm f})
\mid x_i
\right]$.
Since $\Psi_n^{\mathrm f}(\beta^{\mathrm f})
=
0$,
we have
\begin{equation}
\sqrt n\,
\bar\psi_n(\beta^{\mathrm f})
=
\frac1{\sqrt n}
\sum_{i=1}^n
\left( \psi(y_i,x_i;\beta^{\mathrm f})
- \mu_i^{\mathrm f} \right).
\end{equation}
Conditional on $X$, the variables
$\{\psi(y_i,x_i;\beta^{\mathrm f})
- \mu_i^{\mathrm f}\}_{i=1}^n$ are independent, have
conditional mean zero, and satisfy
\[
\frac1n\sum_{i=1}^n
\mathbb{E}\left[
\left( \psi(y_i,x_i;\beta^{\mathrm f})
- \mu_i^{\mathrm f} \right)
\left( \psi(y_i,x_i;\beta^{\mathrm f})
- \mu_i^{\mathrm f} \right)^{\mathrm{T}}
\mid x_i
\right]
=
\Delta^{\mathrm f}.
\]
Moreover, by Assumption~\ref{asu:ME}(d), for some finite constant $C$,
\[
\begin{aligned}
&\frac1{n^{1+\delta/2}}
\sum_{i=1}^n
\mathbb{E}\left[
\|\psi(y_i,x_i;\beta^{\mathrm f})
- \mu_i^{\mathrm f}\|^{2+\delta}
\mid x_i
\right]
\leq
\frac{C}{n^{\delta/2}}
\frac1n\sum_{i=1}^n
\mathbb{E}\left[
\sup_{b\in \mathscr{N}}
\|\psi(y_i,x_i;b)\|^{2+\delta}
\mid x_i
\right]
=o(1;{\mathbb{P}}_{y\mid x}).
\end{aligned}
\]
Thus, the conditional Lyapunov condition holds, and the
Lindeberg--Feller central limit theorem gives
\begin{equation}
(\Delta^{\mathrm f})^{-1/2}
\sqrt n\,
\bar\psi_n(\beta^{\mathrm f})
\mid X
\overset{\textup{d}}{\to}
\mathcal{N}(0, I).
\label{eq:Lindeberg--Feller}
\end{equation}
By Assumptions~\ref{asu:ME}(c)--(d), Lemma \ref{lem:conditional-ulln} applied locally to
$\nabla_{b}\psi$, and the consistency of
$\hat{\beta}$, $\hat{\Gamma}_n(\bar{\beta})
=
\Gamma^{\mathrm f}
+
o(1;{\mathbb{P}}_{y\mid x})$.
Combining \eqref{eq:normal}--\eqref{eq:Lindeberg--Feller} and applying Slutsky's theorem
gives
\[
(V^{\mathrm f})^{-1/2}
\sqrt n
\left(
\hat{\beta}
-
\beta^{\mathrm f}
\right)
\mid X
\overset{\textup{d}}{\to}
\mathcal{N}(0,I).
\] 
Then we prove the conservativeness of HW variance estimator.
By Lemma \ref{lemma:GMM_lemma} with $a(w;\beta) = \psi(w; \beta)\psi(w; \beta)^{\mathrm{T}}$, $\hat{\Delta} = n^{-1}\sum_{i=1}^n  \mathbb{E}\left[\psi(y_i, x_i;\beta^\textup{f})\psi(y_i, x_i;\beta^\textup{f})^{\mathrm{T}} \mid x_{i} \right] + o(1;{\mathbb{P}}_{y\mid x})$. 
By direct algebra,
\begin{align*}
\Delta^\textup{f} 
= \frac{1}{n}\sum_{i=1}^n \mathbb{E}\left[\psi(y_i, x_i;\beta^\textup{f})\psi(y_i, x_i;\beta^\textup{f})^{\mathrm{T}} \mid x_{i} \right] 
-  \frac{1}{n} \sum_{i=1}^n \mathbb{E}\left[ \psi(y_i,x_i;\beta^\textup{f}) \mid x_{i} \right] \mathbb{E}\left[ \psi(y_i,x_i;\beta^\textup{f}) \mid x_{i} \right]^{\mathrm{T}}.
\end{align*}
Similarly, Lemma \ref{lemma:GMM_lemma} applied to $a(w;\beta) = \nabla_b\psi(w,b)$ gives $\hat{\Gamma} = \Gamma^\textup{f} + o(1;{\mathbb{P}}_{y\mid x})$.
The conclusion then follows from continuity of matrix inversion and multiplication.
\end{proof}

\begin{proof}[Proof of Theorem \ref{thm:ME_mixed}]
The proof follows from a similar argument to Theorem \ref{thm:ME_fixed} but with \(\beta^{\diamond} = \beta^{\textup{m}}\), \(\mathbb{E}_\circ = \mathbb{E}_{(y,x_1)|x_2}\) and \(\mathbb{P}_\circ = \mathbb{P}_{(y,x_1)|x_2}\), so we omit it here.  
\end{proof}

\section{Proof of results in Section \ref{sec:i.i.d.Theory}} \label{sec:proof of i.i.d.Theory}

\begin{proof}[Proof of Theorem \ref{thm:random}]
The asymptotic normality of $\hat{\beta}$ follows from applying Theorem \ref{thm:ME_random} with $Z$-estimation $\psi(w_i;b) = x_i (y_i-x_i^{\mathrm{T}} b)$.
Theorem \ref{thm:ME_random} also ensures the consistency of $\hat{V}_{\textsc{ehw}}$ for $V^{\textup{r}}$.

Under random design, the estimand $\beta^{\textup{r}}$ is the solution to the $Z$-estimation:
\begin{align}
0 = \frac{1}{n} \sum_{i=1}^n \mathbb{E} \left[ x_i (y_i-x_i^{\mathrm{T}} b) \right]
\end{align}
with projection error $\varepsilon_i^{\textup{r}} = y_i-x_i^{\mathrm{T}}\beta^{\textup{r}}$.
The asymptotic normality of $\hat{\beta}$ is a standard result with random regressors, which follows from LLN, CLT and Slutsky's theorem. For the middle part of $\hat{V}_{\textsc{ehw}}$, we have
\begin{align}
& \frac{1}{n}\sum_{i=1}^n \hat{\varepsilon}_i^2 x_ix_i^{\mathrm{T}}
= \frac{1}{n}\sum_{i=1}^n \left(y_i-x_i^{\mathrm{T}}{\beta}^{\mathrm{r}} + x_i^{\mathrm{T}}({\beta}^{\mathrm{r}} - \hat{\beta})\right)^2 x_ix_i^{\mathrm{T}} \notag \\
=& \frac{1}{n}\sum_{i=1}^n (y_i-x_i^{\mathrm{T}}{\beta}^{\mathrm{r}})^2 x_ix_i^{\mathrm{T}}
+ \frac{2}{n}\sum_{i=1}^n (y_i-x_i^{\mathrm{T}}{\beta}^{\mathrm{r}})x_i^{\mathrm{T}}({\beta}^{\mathrm{r}} - \hat{\beta}) x_ix_i^{\mathrm{T}}
+ \frac{1}{n}\sum_{i=1}^n \left(x_i^{\mathrm{T}}({\beta}^{\mathrm{r}} - \hat{\beta})\right)^2 x_ix_i^{\mathrm{T}} \notag \\
=& \mathbb{E}\left[(y_i-x_i^{\mathrm{T}}{\beta}^{\mathrm{r}})^2 x_ix_i^{\mathrm{T}}\right] 
+ o(1;\mathbb{P}_{(y,x)}), \label{eq:consistencyEHW_r}
\end{align}
where the last equality follows from LLN and $\hat{\beta} = {\beta}^{\mathrm{r}} + o(1;\mathbb{P}_{(y,x)})$ follows from \eqref{eq:consistencyEHW_r}, LLN, Slutsky's theorem and Assumption \ref{asu:i.i.d.}.

\end{proof}

\begin{proof}[Proof of Theorem \ref{thm:fixed}]
The asymptotic normality of $\hat{\beta}$ follows from applying Theorem \ref{thm:ME_fixed} with $Z$-estimation $\psi(w_i;b) = x_i (y_i-x_i^{\mathrm{T}} b)$.
Theorem \ref{thm:ME_fixed} also ensures that $\hat{V}_{\textsc{ehw}} = V^{\textup{f}} + B^{\textup{f}} +  o(1;\mathbb{P}_{y\mid x})$.

The asymptotic normality is a standard result with non-stochastic regressors, which follows from LLN, CLT and Slutsky's theorem. 
For the middle part of $\hat{V}_{\textsc{ehw}}$, we have
\begin{align}
& \frac{1}{n}\sum_{i=1}^n (y_i-x_i^{\mathrm{T}}\hat{\beta})^2 x_ix_i^{\mathrm{T}}
= \frac{1}{n}\sum_{i=1}^n \left(y_i-x_i^{\mathrm{T}}{\beta}^{\mathrm{f}} + x_i^{\mathrm{T}}({\beta}^{\mathrm{f}} - \hat{\beta})\right)^2 x_ix_i^{\mathrm{T}} \notag \\
=& \frac{1}{n}\sum_{i=1}^n (y_i-x_i^{\mathrm{T}}{\beta}^{\mathrm{f}})^2 x_ix_i^{\mathrm{T}}
+ \frac{2}{n}\sum_{i=1}^n (y_i-x_i^{\mathrm{T}}{\beta}^{\mathrm{f}})x_i^{\mathrm{T}}({\beta}^{\mathrm{f}} - \hat{\beta}) x_ix_i^{\mathrm{T}}
+ \frac{1}{n}\sum_{i=1}^n (x_i^{\mathrm{T}}({\beta}^{\mathrm{f}} - \hat{\beta}))^2  x_ix_i^{\mathrm{T}} \notag \\
=& \frac{1}{n}\sum_{i=1}^n \sigma^2(x_i) x_ix_i^{\mathrm{T}}
+ \frac{1}{n}\sum_{i=1}^n (\mu(x_i)-x_i^{\mathrm{T}}\beta^{\textup{f}})^2 x_ix_i^{\mathrm{T}}
+ o(1;\mathbb{P}_{y\mid x}), \label{eq:consistencyEHW_f}
\end{align}
where the last equality follows from LLN and $\hat{\beta} - {\beta}^{\mathrm{f}} = o(1;\mathbb{P}_{y\mid x})$. The conservative result that $\hat{V}_{\textsc{ehw}} = V^{\textup{f}} + B^{\textup{f}} +  o(1;\mathbb{P}_{y\mid x})$ follows from \eqref{eq:consistencyEHW_f}, Slutsky's theorem and Assumption \ref{asu:i.i.d.}.

\end{proof}

\begin{proof}[Proof of Theorem \ref{thm:AsyDist_mixed}]
The asymptotic normality of $\hat{\beta}$ follows from applying Theorem \ref{thm:ME_mixed} with $Z$-estimation $\psi(w_i;b) = x_i (y_i-x_i^{\mathrm{T}} b)$.
Theorem \ref{thm:ME_mixed} also ensures that $\hat{V}_{\textsc{ehw}} = V^{\textup{m}} + B^{\textup{m}} +  o(1;\mathbb{P}_{(y,x_1)\mid x_2})$.

Notice that
\[
\sqrt{n}(\hat{\beta} - \beta^{\mathrm{m}})
= \left( \frac{1}{n} \sum_{i=1}^n x_ix_i^{\mathrm{T}} \right)^{-1}  \frac{1}{\sqrt{n}} \sum_{i=1}^n x_i (y_i-x_i^{\mathrm{T}}\beta^{\textup{m}}).
\]
To derive the asymptotic distribution of $\hat{\beta}$ under mixed design,
since $\eta(x_{i2}) = \mathbb{E}[x_i (y_i-x_i^{\mathrm{T}}\beta^{\textup{m}} ) \mid x_{i2}] \ne 0$, we need to center each term:
\begin{align*}
\sqrt{n}(\hat{\beta} - \beta^{\textup{m}})
&= \left( \frac{1}{n} \sum_{i=1}^n x_ix_i^{\mathrm{T}} \right)^{-1} 
\left( \frac{1}{\sqrt{n}} \sum_{i=1}^n \left(x_i (y_i-x_i^{\mathrm{T}}\beta^{\textup{m}}) - \eta(x_{i2}) \right)  + \frac{1}{\sqrt{n}} \sum_{i=1}^n \eta(x_{i2}) \right) \\
&= \left( \frac{1}{n} \sum_{i=1}^n x_ix_i^{\mathrm{T}} \right)^{-1}
\frac{1}{\sqrt{n}} \sum_{i=1}^n \left(x_i (y_i-x_i^{\mathrm{T}}\beta^{\textup{m}}) - \eta(x_{i2})\right),
\end{align*} 
where the second equality follows from the fact that $\sum_{i=1}^n \eta(x_{i2}) = 0$.
Then we derive the asymptotic distribution of $n^{-1/2} \sum_{i=1}^n (x_i (y_i-x_i^{\mathrm{T}}\beta^{\textup{m}}) - \eta(x_{i2}))$.
Conditioning on $x_2$, we have
\begin{align*}
\left( \frac{1}{n}\sum_{i=1}^n \V(x_i \varepsilon_i^{\textup{m}} \mid x_{i2}) \right)^{-1}
\frac{1}{\sqrt{n}} \sum_{i=1}^n \left(x_i (y_i-x_i^{\mathrm{T}}\beta^{\textup{m}}) - \eta(x_{i2})\right) \mid X_2
&\overset{\textup{d}}{\to} \mathcal{N}\left(0, 1 \right),
\end{align*}
where $\varepsilon_i^{\textup{m}} = y_i-x_i^{\mathrm{T}}\beta^{\textup{m}}$.
The asymptotic normality of $\hat{\beta}$ follows from the Slutsky's Theorem:
\[
(V^{\textup{m}})^{-1/2} \sqrt{n}(\hat{\beta} - \beta^{\textup{m}}) \mid X_2
\overset{\textup{d}}{\to}
\mathcal{N}\left(0, I \right)
\]
where 
\begin{equation*}
V^{\mathrm{m}} 
= \left( \frac{1}{n} \sum_{i=1}^n \mathbb{E}(x_ix_i^{\mathrm{T}} \mid x_{i2}) \right)^{-1}
\frac{1}{n} \sum_{i=1}^n \V(x_i \varepsilon_i^{\textup{m}} \mid x_{i2})
\left( \frac{1}{n} \sum_{i=1}^n \mathbb{E}(x_ix_i^{\mathrm{T}} \mid x_{i2}) \right)^{-1}.
\end{equation*}
For the middle part of $\hat{V}_{\textsc{ehw}}$, we have
\begin{align}
& \frac{1}{n}\sum_{i=1}^n (y_i-x_i^{\mathrm{T}}\hat{\beta})^2 x_ix_i^{\mathrm{T}}
~=~ \frac{1}{n}\sum_{i=1}^n \left(y_i-x_i^{\mathrm{T}}{\beta}^{\mathrm{m}} + x_i^{\mathrm{T}}({\beta}^{\mathrm{m}} - \hat{\beta})\right)^2 x_ix_i^{\mathrm{T}} \notag \\
~=~& \frac{1}{n}\sum_{i=1}^n (y_i-x_i^{\mathrm{T}}{\beta}^{\mathrm{m}})^2 x_ix_i^{\mathrm{T}}
+ \frac{2}{n}\sum_{i=1}^n (y_i-x_i^{\mathrm{T}}{\beta}^{\mathrm{m}})x_i^{\mathrm{T}}({\beta}^{\mathrm{m}} - \hat{\beta}) x_ix_i^{\mathrm{T}}
+ \frac{1}{n}\sum_{i=1}^n (x_i^{\mathrm{T}}({\beta}^{\mathrm{m}} - \hat{\beta}) )^2 x_ix_i^{\mathrm{T}} \notag \\
~\overset{(1)}{=}~ & \frac{1}{n}\sum_{i=1}^n \mathbb{E}\left((\varepsilon_i^{\mathrm{m}})^2 x_ix_i^{\mathrm{T}} \mid x_{i2}\right)
+ o(1;\mathbb{P}_{(y,x_1)\mid x_2}) \notag \\
~=~& \frac{1}{n}\sum_{i=1}^n \V(x_i \varepsilon_i^{\textup{m}} \mid x_{i2})
+ \frac{1}{n}\sum_{i=1}^n \eta(x_{i2}) \eta(x_{i2})^{\mathrm{T}}
+ o(1;\mathbb{P}_{(y,x_1)\mid x_2}), \label{eq:consistencyEHW_m}
\end{align}
where the equality (1) follows from LLN and $\hat{\beta} - {\beta}^{\mathrm{m}} = o(1;\mathbb{P}_{(y,x_1)\mid x_2})$. 
The conservative result 
$\hat{V}_{\textsc{ehw}} = V^{\mathrm{m}} 
+ B^{\textup{m}} + o(1;\mathbb{P}_{(y,x_1)\mid x_2})$ follows from \eqref{eq:consistencyEHW_m}, Slutsky's theorem and Assumption \ref{asu:i.i.d.}.
\end{proof}

\section{Proof of results in Section \ref{sec:application}} \label{sec:proof of application}

\subsection{Estimating the ATE under complete randomization}

\subsubsection{Simple difference in means and its regression implementation}
In the two-regressor specification $x_i=(1,Z_i)^\mathrm{T}$, the EHW variance
estimator for the treatment coefficient is
\begin{equation}
\hat{V}_{\textsc{ehw}}
= \frac{1}{n} \left[ \left( \frac{1}{n}\sum_{i=1}^n x_{i} x_{i}^{\mathrm{T}} \right)^{-1}
\left(\frac{1}{n}\sum_{i=1}^n \hat{\varepsilon}_{i}^2 x_{i} x_{i}^{\mathrm{T}} \right)
\left(\frac{1}{n}\sum_{i=1}^n x_{i} x_{i}^{\mathrm{T}} \right)^{-1} \right]_{(2,2)}. 
\label{eq:complete_EHW}
\end{equation}

\begin{proof}[Proof of Theorem \ref{thm:y_Z_r}]
The result is a special case of Theorem \ref{thm:random} with $x_{i} = (1, Z_i)$ where constant $1$ is a degenerated random variable.

\end{proof}

\subsubsection{Fisher's analysis of covariance}
Below we give the detailed form of $\hat{V}_{\textsc{ehw},\textsc{f}}$:
\begin{equation}
\hat{V}_{\textsc{ehw},\textsc{f}}
= \frac{1}{n} \left[
\left( \frac{1}{n}\sum_{i=1}^n x_{i,\textsc{f}}x_{i,\textsc{f}}^{\mathrm{T}} \right)^{-1}
\left( \frac{1}{n}\sum_{i=1}^n \hat{\varepsilon}_{i,\textsc{f}}^2 x_{i,\textsc{f}}x_{i,\textsc{f}}^{\mathrm{T}} \right)
\left( \frac{1}{n}\sum_{i=1}^n x_{i,\textsc{f}}x_{i,\textsc{f}}^{\mathrm{T}} \right)^{-1}
\right]_{(2,2)}. \label{eq:complete_EHW_F}
\end{equation}

\begin{proof}[Proof of Theorem \ref{thm:PRA_r}]
The estimands are $(\alpha_{\textsc{f}}^{\textup{r}}, \beta_{\textsc{f}}^{\textup{r}}, (\gamma_{\textsc{f}}^{\textup{r}})^{\mathrm{T}})^{\mathrm{T}} = \mathbb{E}(x_{i,\textsc{f}} x_{i,\textsc{f}}^{\mathrm{T}} )^{-1}\mathbb{E}(x_{i,\textsc{f}} y_i )$ where $\beta_{\textsc{f}}^{\textup{r}} = \beta^{\textup{r}}$.
Without loss of generality, we assume $\mathbb{E}(x_i)=0$.
Define $\varepsilon_{i,\textsc{f}}^{\textup{r}} = y_i - \beta^{\textup{r}}Z_i - \mathbb{E}(Y_i(0))  - \dot{x}_i^{\mathrm{T}} \gamma_{\textsc{f}}^{\textup{r}}$.
Apply Theorem \ref{thm:random} with $x_{i1} = (1, Z_i)$ and $x_{i2} = \dot{x}_i$:
\begin{align*}
\sqrt{n}(\hat{\beta}_{\textsc{f}} - \beta^{\textup{r}}) 
& \overset{\textup{d}}{\to} 
\mathcal{N}\left(0, \frac{1}{e^2(1-e)^2} \frac{1}{n}\sum_{i=1}^n \V\left((Z_i-e)\varepsilon_{i,\textsc{f}}^{\textup{r}} \right) \right),
\end{align*}
where 
\begin{align*}
(Z_i-e)\varepsilon_{i,\textsc{f}}^{\textup{r}}
=& (Z_i-e)\left(Y_i(1) - \beta^{\textup{r}}Z_i - \mathbb{E}(Y_i(0) \mid x_i)  - \dot{x}_i^{\mathrm{T}} \gamma_{\textsc{f}}^{\textup{r}} \right) \\
=& (1-e)Z_i \varepsilon_{i,\textsc{f}}^{\textup{r}}(1)
- e (1-Z_i)\varepsilon_{i,\textsc{f}}^{\textup{r}}(0).
\end{align*}
Thus, the asymptotic variance of $\hat{\beta}_{\textsc{f}}$ is 
\[
V^{\textup{r}}_{\textsc{f}} 
= \frac{\V(\varepsilon^{\textup{r}}_i(1))}{e} 
+ \frac{\V(\varepsilon^{\textup{r}}_i(0))}{1-e}
\]
with $\varepsilon_i^{\textup{r}}(z)
= Y_i(z) - \mathbb{E}(Y_i(z)) - \dot{x}_i\gamma_\textsc{f}^{\textup{r}}$ for $z\in\{0, 1\}$. 
Theorem \ref{thm:random} ensures that $\hat{V}_{\textsc{ehw,f}} = V_\textsc{f}^{\mathrm{r}} + o(1;\mathbb{P}_{(y,Z, x)})$.

\end{proof}

\begin{proof}[Proof of Theorem \ref{thm:PRA_m}]
Define the estimands as
\[
(\alpha_{\textsc{f}}^{\textup{m}}, \beta_{\textsc{f}}^{\textup{m}}, (\gamma_{\textsc{f}}^{\textup{m}})^{\mathrm{T}})^{\mathrm{T}} 
= \left(\frac{1}{n}\sum_{i=1}^n \mathbb{E}(x_{i,\textsc{f}} x_{i,\textsc{f}}^{\mathrm{T}} \mid x_i) \right)^{-1}
\frac{1}{n} \sum_{i=1}^n \mathbb{E}(x_{i,\textsc{f}} y_i \mid x_i ).
\]
By direct algebra, we have $\alpha_{\textsc{f}}^{\textup{m}} = n^{-1} \sum_{i=1}^n\mathbb{E}(Y_i(0) \mid x_i)$ and $\beta_{\textsc{f}}^{\textup{m}} = \beta^{\textup{m}}$.
Define $\varepsilon_{i,\textsc{f}}^{\textup{m}} = y_i - \beta^{\textup{m}}Z_i - n^{-1} \sum_{i=1}^n\mathbb{E}(Y_i(0) \mid x_i)  - \ddot{x}_i^{\mathrm{T}} \gamma_{\textsc{f}}^{\textup{m}}$.
Apply Theorem \ref{thm:AsyDist_mixed} with $x_{i1} = (1, Z_i)$ and $x_{i2} = \ddot{x}_i$:
\begin{align*}
\sqrt{n}(\hat{\beta}_{\textsc{f}} - \beta^{\textup{m}}) 
& \overset{\textup{d}}{\to} 
\mathcal{N}\left(0, \frac{1}{e^2(1-e)^2} \frac{1}{n}\sum_{i=1}^n \V\left((Z_i-e)\varepsilon_{i,\textsc{f}}^{\textup{m}} \mid {x}_i \right) \right),
\end{align*}
where 
\begin{align*}
(Z_i-e)\varepsilon_{i,\textsc{f}}^{\textup{m}}
=& (Z_i-e)\left(Y_i(1) - \beta^{\textup{m}}Z_i -  \frac{1}{n}\sum_{i=1}^n\mathbb{E}(Y_i(0) \mid x_i)  - \ddot{x}_i^{\mathrm{T}} \gamma_{\textsc{f}}^{\textup{m}} \right) \\
=& (1-e)Z_i\varepsilon_{i,\textsc{f}}^{\textup{m}}(1)
- e (1-Z_i)\varepsilon_{i,\textsc{f}}^{\textup{m}}(0).
\end{align*}
Thus, the asymptotic variance of $\hat{\beta}_{\textsc{f}}$ is 
\begin{align*}
V_{\textsc{f}}^{\textup{m}}
= \frac{1}{n}\sum_{i=1}^n \left( \frac{\mathbb{E}\left(\varepsilon^{\textup{m}}_{i,\textsc{f}}(1)^2 \mid x_i\right)}{e}  
+ \frac{\mathbb{E}\left(\varepsilon^{\textup{m}}_{i,\textsc{f}}(0)^2\mid x_i\right)}{1-e}  
- \mathbb{E}\left(\varepsilon_{i,\textsc{f}}^{\textup{m}}(1) - \varepsilon_{i,\textsc{f}}^{\textup{m}}(0) \mid x_i \right)^2  \right).
\end{align*}
Then we prove the conservativeness of $\hat{V}_{\textsc{ehw},\textsc{f}}$.
By direct algebra, we have
\begin{align*}
\hat{V}_{\textsc{ehw},\textsc{f}}
=&  \left[
\left( \frac{1}{n}\sum_{i=1}^n x_{i,\textsc{f}}x_{i,\textsc{f}}^{\mathrm{T}} \right)^{-1}
\left( \frac{1}{n}\sum_{i=1}^n \hat{\varepsilon}_{i,\textsc{f}}^2 x_{i,\textsc{f}}x_{i,\textsc{f}}^{\mathrm{T}} \right)
\left( \frac{1}{n}\sum_{i=1}^n x_{i,\textsc{f}}x_{i,\textsc{f}}^{\mathrm{T}} \right)^{-1}
\right]_{(2,2)} \\
=& \frac{1}{e^2}\frac{1}{n}\sum_{i=1}^n \hat{\varepsilon}_{i,\textsc{f}}^2Z_i
+ \frac{1}{(1-e)^2}\frac{1}{n}\sum_{i=1}^n \hat{\varepsilon}_{i,\textsc{f}}^2(1-Z_i)
+ o(1;\mathbb{P}_{(y,Z) \mid x}).
\end{align*}
Theorem \ref{thm:AsyDist_mixed} ensures that 
\[
\hat{V}_{\textsc{ehw,f}} = V_\textsc{f}^{\mathrm{m}} 
+ B_\textsc{f}^{\mathrm{m}} + o(1;\mathbb{P}_{(y,Z)\mid x})
\]
where 
\begin{equation*}
B_\textsc{f}^{\textup{m}}
= \left[ \left( \frac{1}{n}\sum_{i=1}^n \mathbb{E}(x_{i,\textsc{f}}x_{i,\textsc{f}}^{\mathrm{T}}\mid x_{i}) \right)^{-1}
\left( \frac{1}{n}\sum_{i=1}^n \mathbb{E}[x_{i,\textsc{f}} \varepsilon_{i,\textsc{f}}^{\textup{m}}\mid x_{i}] \mathbb{E}[x_{i,\textsc{f}} \varepsilon_{i,\textsc{f}}^{\textup{m}}\mid x_{i}]^{\mathrm{T}} \right)
\left( \frac{1}{n}\sum_{i=1}^n \mathbb{E}(x_{i,\textsc{f}}x_{i,\textsc{f}}^{\mathrm{T}}\mid x_{i}) \right)^{-1} \right]_{(2,2)}.
\end{equation*}
We compute each piece respectively. \\
(i) With
\begin{align*}
\frac{1}{n}\sum_{i=1}^n \mathbb{E}(x_{i,\textsc{f}}x_{i,\textsc{f}}^{\mathrm{T}}\mid x_{i})
=& \left(
\begin{array}{ccc}
1 & e & 0  \\
e & e & 0  \\
0 & 0 & \frac{1}{n}\sum_{i=1}^n \ddot{x}_i\ddot{x}_i^{\mathrm{T}}   
\end{array}
\right),
\end{align*}
by direct algebra, we have 
\begin{align*}
\left( \frac{1}{n}\sum_{i=1}^n \mathbb{E}(x_{i,\textsc{f}}x_{i,\textsc{f}}^{\mathrm{T}}\mid x_{i})\right)^{-1}
=& \left(
\begin{array}{ccc}
\frac{e}{e(1-e)} & -\frac{e}{e(1-e)} & 0  \\
-\frac{e}{e(1-e)} & \frac{1}{e(1-e)} & 0  \\
0 & 0 & (\frac{1}{n}\sum_{i=1}^n \ddot{x}_i\ddot{x}_i^{\mathrm{T}} )^{-1}  
\end{array}
\right).
\end{align*}
(ii) By direct algebra,
\begin{align*}
\mathbb{E}[x_{i,\textsc{f}} \varepsilon_{i,\textsc{f}}^{\textup{m}}\mid x_{i}] 
=& 
\left(
\begin{array}{c}
\mathbb{E}[ \varepsilon_{i,\textsc{f}}^{\textup{m}} \mid x_{i}]  \\
\mathbb{E}[ Z_i \varepsilon_{i,\textsc{f}}^{\textup{m}}\mid x_{i}]  \\
x_{i}  \mathbb{E}[ \varepsilon_{i,\textsc{f}}^{\textup{m}} \mid x_{i}] 
\end{array}
\right) 
= \left(
\begin{array}{c}
e \mathbb{E}[ \varepsilon_{i,\textsc{f}}^{\textup{m}}(1) \mid x_{i}] 
+ (1-e) \mathbb{E}[ \varepsilon_{i,\textsc{f}}^{\textup{m}}(0) \mid x_{i}] \\
e \mathbb{E}[ \varepsilon_{i,\textsc{f}}^{\textup{m}}(1) \mid x_{i}]  \\
x_{i}  (e \mathbb{E}[ \varepsilon_{i,\textsc{f}}^{\textup{m}}(1) \mid x_{i}] 
+ (1-e) \mathbb{E}[ \varepsilon_{i,\textsc{f}}^{\textup{m}}(0) \mid x_{i}])
\end{array}
\right)
\end{align*}
and we have
\begin{align*}
\frac{1}{n}\sum_{i=1}^n \mathbb{E}[x_{i,\textsc{f}} \varepsilon_{i,\textsc{f}}^{\textup{m}}\mid x_{i}]  \mathbb{E}[x_{i,\textsc{f}} \varepsilon_{i,\textsc{f}}^{\textup{m}}\mid x_{i}] ^{\mathrm{T}}
=& 
\left(
\begin{array}{ccc}
h_{11} & h_{12} & h_{13}  \\
h_{12}^{\mathrm{T}} & h_{22} & h_{23} \\
h_{13}^{\mathrm{T}} & h_{23}^{\mathrm{T}} & h_{33}  
\end{array}
\right)
\end{align*}
where 
\begin{align*}
h_{11} =& \frac{1}{n}\sum_{i=1}^n (e \mathbb{E}[ \varepsilon_{i,\textsc{f}}^{\textup{m}}(1) \mid x_{i}] 
+ (1-e) \mathbb{E}[ \varepsilon_{i,\textsc{f}}^{\textup{m}}(0) \mid x_{i}])^2, \\
h_{12} =& e \frac{1}{n}\sum_{i=1}^n (e \mathbb{E}[ \varepsilon_{i,\textsc{f}}^{\textup{m}}(1) \mid x_{i}] 
+ (1-e) \mathbb{E}[ \varepsilon_{i,\textsc{f}}^{\textup{m}}(0) \mid x_{i}]) \mathbb{E}[ \varepsilon_{i,\textsc{f}}^{\textup{m}}(1) \mid x_{i}], \\
h_{22} =& e^2 \frac{1}{n}\sum_{i=1}^n \mathbb{E}[ \varepsilon_{i,\textsc{f}}^{\textup{m}}(1) \mid x_{i}]^2.
\end{align*}
We combine these pieces and derive the asymptotic bias as
\begin{align*}
B_{\textsc{f}}^{\textup{m}}
=& \left[ \left( \frac{1}{n}\sum_{i=1}^n \mathbb{E}(x_{i,\textsc{f}}x_{i,\textsc{f}}^{\mathrm{T}}\mid x_{i}) \right)^{-1}
\left( \frac{1}{n}\sum_{i=1}^n \eta(x_{i}) \eta(x_{i})^{\mathrm{T}} \right)
\left( \frac{1}{n}\sum_{i=1}^n \mathbb{E}(x_{i,\textsc{f}}x_{i,\textsc{f}}^{\mathrm{T}}\mid x_{i}) \right)^{-1} \right]_{(2,2)} \\
=& \left[ 
\left(
\begin{array}{ccc}
* & * & 0  \\
\frac{h_{12}-eh_{11}}{e(1-e)} & \frac{h_{22}-eh_{12}}{e(1-e)} & \frac{h_{23}-eh_{13}}{e(1-e)}  \\
* & * & *  
\end{array}
\right)
\left(
\begin{array}{ccc}
\frac{e}{e(1-e)} & -\frac{e}{e(1-e)} & 0  \\
-\frac{e}{e(1-e)} & \frac{1}{e(1-e)} & 0  \\
0 & 0 & *  
\end{array}
\right) \right]_{(2,2)} \\
=& \frac{h_{22}-eh_{12} - e(h_{12}-eh_{11})}{e^2(1-e)^2} \\
=& \frac{1}{n}\sum_{i=1}^n \mathbb{E}\left(\varepsilon_{i,\textsc{f}}^{\textup{m}}(1) - \varepsilon_{i,\textsc{f}}^{\textup{m}}(0) \mid x_{i} \right)^2.
\end{align*}
Thus, we complete the proof.
\end{proof}

\subsubsection{Lin's fully interacted  adjustment with known parameters}
\label{subsec:known-parameters}

Before proving Theorem \ref{thm:FRA_random}, we discuss the case where part of the parameters is known. 
Consider a general type of estimator $\hat{\theta}$ that has as special cases most examples of interest is one that, with probability approaching one, which solves an equation
\begin{equation}
\frac{1}{n}\sum_{i=1}^n g(w_i; \theta, \hat{\gamma})=0, \label{eq:second}
\end{equation}
where $g(w; \theta, {\gamma})$ is a $p_{\theta} \times 1$ vector of functions with the same dimension as $\theta$, and $\hat{\gamma}$ is a first-step estimator.
The estimator can be treated as part of a joint moment estimator if $\hat{\gamma}$ also satisfies a moment condition of the form, with
probability approaching one
\begin{equation}
\frac{1}{n}\sum_{i=1}^n m(w_i; {\gamma})=0 \label{eq:first}
\end{equation}
where $m(w; {\gamma})=0$ is a vector with the same dimension as $\gamma$.
We stack $g(w; \theta, {\gamma})$ and $m(w; {\gamma})$ to form $\tilde{g}(w; \theta,\gamma) = [g(w; \theta, {\gamma})^{\mathrm{T}}, m(w;{\gamma})^{\mathrm{T}}]^{\mathrm{T}}$, then \eqref{eq:second} and \eqref{eq:first} are simply the two components of the joint moment equation:
\begin{equation}
\frac{1}{n}\sum_{i=1}^n \tilde{g}(w_i; \theta, {\gamma})=0. \label{eq:joint}
\end{equation}
Let 
\begin{align}
& G_\theta = \mathbb{E}\left[\frac{\partial}{\partial \theta} g\left(w; \theta_0, \gamma_0\right)\right], 
\quad 
G_\gamma = \mathbb{E}\left[ \frac{\partial}{\partial \gamma} g\left(w; \theta_0, \gamma_0\right)\right], 
\quad 
g(w)=g\left(w; \theta_0, \gamma_0\right), \\
& M=\mathbb{E}\left[ \frac{\partial}{\partial \gamma} m\left(w; \gamma_0\right)\right], 
\quad 
\psi(w)=-M^{-1} m\left(w; \gamma_0\right), \\
& G_{\theta\gamma} = \mathbb{E}\left[\frac{\partial}{\partial (\theta, \gamma)} \tilde{g}\left(w; \theta_0, \gamma_0\right)\right], 
\quad 
\tilde{g}(w)=\tilde{g}\left(w; \theta_0, \gamma_0\right). \label{eq:M_G}
\end{align}

\begin{lemma}\label{lemma:ME_diff}
If we assume $\gamma$ is known and obtain $\hat{\theta}$ by solving $\eqref{eq:second}$, then 
\begin{align*}
(V_{\text{second}})^{-1/2}\sqrt{n}(\hat{\theta}_{\text{second}} - \theta_0) \overset{\textup{d}}{\to} \mathcal{N}(0, I)
\end{align*}
where $V_{\text{second}}
= G_\theta^{-1} \mathbb{E}\left[g(w)g(w)^{\mathrm{T}} \right] G_\theta^{-\mathrm{T}}$.
When $\gamma$ is unknown and we obtain $\hat{\theta}$ by solving $\eqref{eq:joint}$, then 
\begin{align*}
(V_{\text{joint}})^{-1/2}\sqrt{n}(\hat{\theta}_{\text{joint}} - \theta_0) \overset{\textup{d}}{\to} \mathcal{N}(0, I)
\end{align*}
where 
\begin{equation*}
V_{\text{joint}}
= G_\theta^{-1} 
\mathbb{E}\left[(g(w)+G_\gamma \psi(w))(g(w)+G_\gamma \psi(w))^{\mathrm{T}} \right]  
G_\theta^{-\mathrm{T}}.
\end{equation*}

\end{lemma}
\begin{proof}[Proof of Lemma \ref{lemma:ME_diff}]
The asymptotic normality of $\hat{\theta}_{\text{second}}$ follows from applying Theorem \ref{thm:ME_random}. 
By solving the joint moment equations, Theorem \ref{thm:ME_random} ensures that 
\begin{align}
\sqrt{n}
\left\{\left(\begin{array}{c}
\hat{\theta}_{\text{joint}} \\
\hat{\gamma}_{\text{joint}} 
\end{array}\right)
-\left(\begin{array}{c}
\theta_0 \\
\gamma_0
\end{array}\right)\right\} \overset{\textup{d}}{\to} 
\mathcal{N}\left(0, G_{\theta\gamma}^{-1} \mathbb{E}(\tilde{g}(w)\tilde{g}(w)^{\mathrm{T}} )  (G_{\theta\gamma}^{-1})^{\mathrm{T}} \right).
\end{align}
By direct algebra, \\
(i) 
\begin{align*}
G_{\theta\gamma}
=& \mathbb{E}\left[\frac{\partial}{\partial (\theta, \gamma)} \tilde{g}\left(w; \theta_0, \gamma_0\right)\right] 
= \begin{bmatrix}
G_{\theta} & G_{\gamma} \\
0 & M
\end{bmatrix},
\quad 
G_{\theta\gamma}^{-1}
= \mathbb{E}\left[\frac{\partial}{\partial (\theta, \gamma)} \tilde{g}\left(w; \theta_0, \gamma_0\right)\right] 
= \begin{bmatrix}
G_{\theta}^{-1} & -G_{\theta}^{-1} G_{\gamma} M^{-1} \\
0 & M^{-1}
\end{bmatrix}
\end{align*}
(ii)
\begin{align*}
\mathbb{E}(\tilde{g}(w)\tilde{g}(w)^{\mathrm{T}} )
=& 
\begin{bmatrix}
\mathbb{E}(g(w)g(w)^{\mathrm{T}} ) & \mathbb{E}(g(w)m(w,\gamma_0)^{\mathrm{T}} ) \\
\mathbb{E}(g(w)m(w,\gamma_0)^{\mathrm{T}} )^{\mathrm{T}} & \mathbb{E}(m(w,\gamma_0)m(w,\gamma_0)^{\mathrm{T}} )
\end{bmatrix}.
\end{align*}
Therefore,
\begin{align*}
(V_{\text{joint}})^{-1/2}\sqrt{n}(\hat{\theta}_{\text{joint}} - \theta_0) \overset{\textup{d}}{\to} \mathcal{N}(0, I)
\end{align*}
where 
\begin{align*}
& V_{\text{joint}}
= \left[ G_{\theta\gamma}^{-1} \mathbb{E}(\tilde{g}(w)\tilde{g}(w)^{\mathrm{T}} )  (G_{\theta\gamma}^{-1})^{\mathrm{T}}  \right]_{(1:p_\theta, 1:p_\theta)} \\
=& \left[ 
\begin{pmatrix}
G_{\theta}^{-1} & -G_{\theta}^{-1} G_{\gamma} M^{-1} \\
0 & M^{-1}
\end{pmatrix} 
\begin{pmatrix}
\mathbb{E}(g(w)g(w)^{\mathrm{T}} ) & \mathbb{E}(g(w)m(w,\gamma_0)^{\mathrm{T}} ) \\
\mathbb{E}(g(w)m(w,\gamma_0)^{\mathrm{T}} )^{\mathrm{T}} & \mathbb{E}(m(w,\gamma_0)m(w,\gamma_0)^{\mathrm{T}} )
\end{pmatrix} 
\begin{pmatrix}
G_{\theta}^{-\mathrm{T}} & 0 \\
-M^{-1}  G_{\gamma} G_{\theta}^{-1}  & M^{-\mathrm{T}}
\end{pmatrix}  
\right]_{(1:p_\theta, 1:p_\theta)} \\
=& G_{\theta}^{-1} \mathbb{E}(g(w)g(w)^{\mathrm{T}} ) (G_{\theta}^{-1})^{\mathrm{T}}  -G_{\theta}^{-1} G_{\gamma} M^{-1} \mathbb{E}(g(w)m(w,\gamma_0)^{\mathrm{T}} )^{\mathrm{T}} (G_{\theta}^{-1})^{\mathrm{T}} \\
&- G_{\theta}^{-1} \mathbb{E}(g(w)m(w,\gamma_0)^{\mathrm{T}} )  M^{-1}  G_{\gamma} G_{\theta}^{-1}
+ G_{\theta}^{-1} G_{\gamma} M^{-1} \mathbb{E}(m(w,\gamma_0)m(w,\gamma_0)^{\mathrm{T}} )^{\mathrm{T}} M^{-1} G_{\gamma} G_{\theta}^{-1} \\
=& G_\theta^{-1} 
\mathbb{E}\left[(g(w)+G_\gamma \psi(w))(g(w)+G_\gamma \psi(w))^{\mathrm{T}} \right] 
G_\theta^{-1}.
\end{align*}
\end{proof}

If $\mu_x$ is known, then we run the following OLS fit instead:
\begin{equation}
y_i \sim 1 + Z_i + \dot{x}_i + Z_i\cdot \dot{x}_i. \label{eq:Lin_oracle}
\end{equation}
We denote the coefficient of $Z_i$ by $\hat{\beta}_{\textsc{l}}^*$ and $\hat{\varepsilon}_{i,\textsc{l}}^*$ as the residual from the above OLS fit. 
Define $x_{i,\textsc{l}}^* = (1, Z_i, \dot{x}_i^{\mathrm{T}}, Z_i\dot{x}_i^{\mathrm{T}} )^{\mathrm{T}}$.
Let $\mu_1^{\textup{r}} = \mathbb{E}(Y_i(1))$, $\mu_0^{\textup{r}} = \mathbb{E}(Y_i(0))$, $\mu_x = \mathbb{E}(x_i)$, $\gamma_1^{\textup{r}} = \{\operatorname{cov}(x_i)\}^{-1} \operatorname{cov}\{x_i, Y_i(1)\}$ and $\gamma_0^{\textup{r}} = \{\operatorname{cov}(x_i)\}^{-1} \operatorname{cov}\{x_i, Y_i(0)\}$.
\begin{lemma}\label{lemma:Lin_known_x}
Recall $\beta^{\textup{r}} = \mathbb{E}(Y_i(1) - Y_i(0) )$.
Under random design and Assumption \ref{asu:i.i.d.}, we have
\[
(V^{* \textup{r}}_{\textsc{l}})^{-1/2}\sqrt{n}(\hat{\beta}_{\textsc{l}}^* - \beta^{\textup{r}})  \overset{\textup{d}}{\to} \mathcal{N}(0, I)
\text{ and }
\hat{V}_{\textsc{ehw},\textsc{l}}^*
= V_{\textsc{l}}^{* \textup{r}}
+ o(1;\mathbb{P}_{(y,Z,x)}),
\]
where 
\begin{equation*}
V^{* \textup{r}}_{\textsc{l}}
= \frac{\mathbb{E}\left[ \left\{Y_i(1)-\left(x_i-\mu_x\right)^{\mathrm{T}} \gamma_1^{\textup{r}}-\mu_1^{\textup{r}}\right\}^2\right]}{e}
+\frac{\mathbb{E}\left[ \left\{Y_i(0)-\left(x_i-\mu_x\right)^{\mathrm{T}} \gamma_0^{\textup{r}}-\mu_0^{\textup{r}}\right\}^2\right]}{1-e}
\end{equation*}
and
\begin{equation}
\hat{V}_{\textsc{ehw},\textsc{l}}^*
= \left[
\left( \frac{1}{n}\sum_{i=1}^n x_{i,\textsc{l}}^* x_{i,\textsc{l}}^{* \top} \right)^{-1}
\left( \frac{1}{n}\sum_{i=1}^n \hat{\varepsilon}_{i,\textsc{l}}^{*2} x_{i,\textsc{l}}^*x_{i,\textsc{l}}^{* \top} \right)
\left( \frac{1}{n}\sum_{i=1}^n x_{i,\textsc{l}}^* x_{i,\textsc{l}}^{* \top} \right)^{-1}
\right]_{(2,2)}. \label{eq:EHW_Lin_oracle}
\end{equation}

\end{lemma}
\begin{proof}[Proof of Lemma \ref{lemma:Lin_known_x}]
Let $\hat{\gamma}_1^*$ and $\hat{\gamma}_0^*$ denote the coefficient vectors of $x_i-\mu_x$ from the OLS fits of
\begin{align}
y_i \sim 1+\left(x_i-\mu_x\right) & \text { over }\left\{i: Z_i=1\right\}  \label{eq:Lin_1_oracle}, \\
y_i \sim 1+\left(x_i-\mu_x\right) & \text { over }\left\{i: Z_i=0\right\}, \label{eq:Lin_0_oracle}
\end{align}
respectively. Let $\hat{Y}_{\textsc{l}}^*(1)$ and $\hat{Y}_{\textsc{l}}^*(0)$ denote the intercepts from \eqref{eq:Lin_1_oracle} and \eqref{eq:Lin_0_oracle}, respectively.
That $\hat{\beta}_{\textsc{l}}^* = \hat{Y}_{\textsc{l}}^*(1)-\hat{Y}_{\textsc{l}}^*(0)$ follows from properties of least squares. To verify the explicit form of $\hat{Y}_{\textsc{l}}(1)$ and $\hat{Y}_{\textsc{l}}(0)$, observe that the residual from \eqref{eq:Lin_1_oracle} equals $Y_i(1)-\left(x_i-\mu_{x}\right)^{\mathrm{T}} \hat{\gamma}_1^*-\hat{Y}_{\textsc{l}}^*(1)$ for units with $Z_i=1$. 
The first-order condition ensures
$$
\sum_{i=1}^n Z_i \cdot\left(Y_i(1)-\left(x_i-\mu_x\right)^{\mathrm{T}} \hat{\gamma}_1^*-\hat{Y}_{\textsc{l}}^*(1)\right)=0
$$
This gives the expression of $\hat{Y}_{\textsc{l}}^*(1)$:
\begin{align*}
\hat{Y}_{\textsc{l}}^*(1)
=& \frac{\sum_{i=1}^n Z_i \cdot\left(Y_i(1)-\left(x_i-\mu_x\right)^{\mathrm{T}} \hat{\gamma}_1^*\right)}{\sum_{i=1}^n Z_i}.
\end{align*}
The expression of $\hat{Y}_{\textsc{l}}^*(0)$ follows by symmetry:
\begin{equation*}
\hat{Y}_{\textsc{l}}^*(0)
= \frac{\sum_{i=1}^n \left(1-Z_i\right) \cdot\left(Y_i(0)-\left(x_i-\mu_x \right)^{\mathrm{T}} \hat{\gamma}_0^*\right)}{\sum_{i=1}^n \left(1-Z_i\right)}.
\end{equation*}
The first-order conditions of \eqref{eq:Lin_1_oracle} and \eqref{eq:Lin_0_oracle} ensure that $\left(\mu_1, \gamma_1\right)=(\hat{Y}_{\textsc{l}}^*(1), \hat{\gamma}_1^*)$ and $\left(\mu_0, \gamma_0\right)=(\hat{Y}_{\textsc{l}}^*(0), \hat{\gamma}_0^*)$ solve
$$
0=\sum_{i=1}^n Z_i \cdot\left(Y_i(1)-\left(x_i-\mu_x\right)^{\mathrm{T}} \gamma_1-\mu_1\right)\left(\begin{array}{c}
1 \\
x_i-\mu_x
\end{array}\right).
$$
This ensures that $\theta = \left(\mu_1, \gamma_1, \mu_0, \gamma_0\right)= \left(\hat{Y}_{\textsc{l}}^*(1), \hat{\gamma}_1^*, \hat{Y}_{\textsc{l}}^*(0), \hat{\gamma}_0^*\right)$ jointly solves
$$
0=n^{-1} \sum_{i=1}^n \eta\left(Y_i(1), Y_i(0), x_i, Z_i; \theta\right)
=n^{-1} \sum_{i=1}^n \left(\begin{array}{c}
\eta_1\left(Y_i(1), x_i, Z_i; \mu_1, \gamma_1\right) \\
\eta_0\left(Y_i(0), x_i, Z_i; \mu_0, \gamma_0\right) 
\end{array}\right)
$$
and
\begin{align}
\eta^*=
\left(\begin{array}{c}
\eta_1 \\
\eta_0 
\end{array}\right) 
\text { with } 
\quad \begin{aligned}
\eta_1 
& =Z_i \cdot\left\{y_i-\left(x_i-\mu_x\right)^{\mathrm{T}} \gamma_1-\mu_1\right\}
\left(\begin{array}{c}
1 \\
x_i-\mu_x
\end{array}\right), \\
\eta_0 
& =(1-Z_i)  \cdot\left\{y_i-\left(x_i-\mu_x\right)^{\mathrm{T}} \gamma_0-\mu_0\right\}
\left(\begin{array}{c}
1 \\
x_i-\mu_x
\end{array}\right).
\end{aligned} 
\label{eq:Lin_ME_oracle}
\end{align}
Direct algebra ensures that $\theta=\theta^{\textup{r}}=\left(\mu_1^{\textup{r}}, \gamma_1^{\textup{r}}, \mu_0^{\textup{r}}, \gamma_0^{\textup{r}}\right)$ solves $\mathbb{E}\{\eta^*(\theta)\}=0$. 
Theorem \ref{thm:ME_random} ensures that 
\begin{align}
\sqrt{n}\left\{\left(\begin{array}{c}
\hat{Y}_{\textsc{l}}^*(1) \\
\hat{\gamma}_1^* \\
\hat{Y}_{\textsc{l}}^*(0) \\
\hat{\gamma}_0^* 
\end{array}\right)-\left(\begin{array}{c}
\mu_1^{\textup{r}} \\
\gamma_1^{\textup{r}} \\
\mu_0^{\textup{r}} \\
\gamma_0^{\textup{r}} 
\end{array}\right)\right\} \overset{\textup{d}}{\to} 
\mathcal{N}\left(0,( A_{\textsc{l}}^{*\textup{r}})^{-1}  D_{\textsc{l}}^{*\textup{r}} \left(A_{\textsc{l}}^{*\textup{r}} \right)^{-\mathrm{T}}\right), \label{eq:Lin_normal_oracle}
\end{align}
where $A_{\textsc{l}}^{*\textup{r}}$ and $D_{\textsc{l}}^{*\textup{r}} $ are the values of $-\mathbb{E}\left(\frac{\partial}{\partial\left(\mu_1, \gamma_1, \mu_0, \gamma_0\right)} \eta^*\right)$ and $\mathbb{E}\left(\eta^* \eta^{* \top}\right)$ evaluated at $\theta=\theta^{\textup{r}}$. 
We compute below $A_{\textsc{l}}^{*\textup{r}}$ and $D_{\textsc{l}}^{*\textup{r}} $, respectively.

\noindent \textbf{Computing $D_{\textsc{L}}^{*\textup{r}} $:}
Let $\left(\eta^{\textup{r}}, \eta_1^{\textup{r}}, \eta_0^{\textup{r}} \right)$ denote the value of $\left(\eta, \eta_1, \eta_0 \right)$ evaluated at $\theta=\theta^{\textup{r}}$. 
From \eqref{eq:Lin_ME_oracle}, we have
\begin{align}
\eta^{\textup{r}}
=\left(\begin{array}{c}
\eta_1^{\textup{r}} \\
\eta_0^{\textup{r}} 
\end{array}\right) \text { with } \quad 
\begin{aligned}
\eta_1^{\textup{r}} 
& =Z_i \cdot\left\{Y_i(1)-\left(x_i-\mu_x\right)^{\mathrm{T}} \gamma_1^{\textup{r}}-\mu_1^{\textup{r}}\right\}\left(\begin{array}{c}
1 \\
x_i-\mu_x
\end{array}\right), \\
\eta_0^{\textup{r}} 
& =(1-Z_i) \cdot\left\{Y_i(0)-\left(x_i-\mu_x\right)^{\mathrm{T}} \gamma_0^{\textup{r}}-\mu_0^{\textup{r}}\right\}\left(\begin{array}{c}
1 \\
x_i-\mu_x
\end{array}\right).
\end{aligned} \label{eq:Lin_eta_oracle}
\end{align}
The $D_{\textsc{l}}^{*\textup{r}}$ matrix equals
\begin{align}
D_{\textsc{l}}^{*\textup{r}} 
= \left.\mathbb{E}\left(\eta^* (\eta^{*})^\mathrm{T}\right)\right|_{\theta=\theta^{\textup{r}}}
= \mathbb{E}\left(\eta^{*\textup{r}} (\eta^{*{\textup{r}} })^\mathrm{T}\right)
= \left(\begin{array}{cc|cc}
d_{11} & d_{12} & 0 & 0  \\
d_{12}^{\mathrm{T}} & d_{22} & 0 & 0 \\
\hline 0 & 0 & d_{33} & d_{34}   \\
0 & 0 & d_{34}^{\mathrm{T}} & d_{44} 
\end{array}\right), 
\label{eq:Lin_B_oracle}
\end{align}
where
\begin{align}
d_{11} 
& = e \cdot \mathbb{E}\left[ \left\{Y_i(1)-\left(x_i-\mu_x\right)^{\mathrm{T}} \gamma_1^{\textup{r}}-\mu_1^{\textup{r}}\right\}^2\right], \label{eq:d_11} \\
d_{33} 
& = (1-e) \cdot \mathbb{E}\left[ \left\{Y_i(0)-\left(x_i-\mu_x\right)^{\mathrm{T}} \gamma_0^{\textup{r}}-\mu_0^{\textup{r}}\right\}^2\right], \label{eq:d_33}
\end{align} 
by symmetry. 

\noindent \textbf{Compute $A_{\textsc{L}}^{*\textup{r}}$:}
With a slight abuse of notation, let $\frac{\partial}{\partial \mu_1} \eta_1^{\textup{r}}$ denote the value of $\frac{\partial}{\partial \mu_1} \eta_1$ evaluated at $\theta^{\textup{r}}$. Similarly define other partial derivatives. 
From \eqref{eq:Lin_ME_oracle}, we have
$$
\left.\frac{\partial}{\partial\left(\mu_1, \gamma_1^{\mathrm{T}}, \mu_0, \gamma_0^{\mathrm{T}} \right)} \eta^* \right|_{\theta=\theta^{\textup{r}}}
=\left(\begin{array}{cc}
\frac{\partial}{\partial\left(\mu_1, \gamma_1^{\mathrm{T}}\right)} \eta_1^{\textup{r}} & 0  \\
0 & \frac{\partial}{\partial\left(\mu_0, \gamma_0^{\mathrm{T}}\right)} \eta_0^{\textup{r}} 
\end{array}\right),
$$
where
\begin{equation}
\frac{\partial}{\partial\left(\mu_1, \gamma_1^{\mathrm{T}}\right)} \eta_1^{\textup{r}} 
=-Z_i \cdot\left(\begin{array}{c}
1 \\
x_i-\mu_x
\end{array}\right)\left(1,\left(x_i-\mu_x\right)^{\mathrm{T}}\right).  \label{eq:eta_mu_oracle}
\end{equation}
Accordingly, we have
\begin{align*}
A_{\textsc{l}}^{*\textup{r}}=-\left.\mathbb{E}\left(\frac{\partial}{\partial\left(\mu_1, \gamma_1^{\mathrm{T}}, \mu_0, \gamma_0^{\mathrm{T}} \right)} \eta^r\right)\right|_{\theta=\theta^r}
=& -\mathbb{E}\left(\left.\frac{\partial}{\partial\left(\mu_1, \gamma_1^{\mathrm{T}}, \mu_0, \gamma_0^{\mathrm{T}} \right)} \eta^*\right|_{\theta=\theta^*}\right) \\
=& \left(\begin{array}{cc|cc}
g_{11} & g_{12} & 0 & 0   \\
g_{21} & g_{22} & 0 & 0    \\
\hline 0 & 0 & g_{33} & g_{34} \\
0 & 0 & g_{43} & g_{44} 
\end{array}\right),
\end{align*}
where
\begin{align}
\left(\begin{array}{ll}
g_{11} & g_{12} \\
g_{21} & g_{22}
\end{array}\right)
=& -\mathbb{E}\left(\frac{\partial}{\partial (\mu_1, \gamma_1^{\mathrm{T}})} \eta_1^{\textup{r}}\right) 
= \mathbb{E}\left\{Z_i  \cdot\left(\begin{array}{c}
1 \\
x_i-\mu_x
\end{array}\right)\left(1,\left(x_i-\mu_x\right)^{\mathrm{T}}\right)\right\} \notag \\
=& 
e \cdot \mathbb{E}\left\{\left(\begin{array}{c}
1 \\
x_i-\mu_x
\end{array}\right)\left(1,\left(x_i-\mu_x\right)^{\mathrm{T}}\right)\right\}
=e \cdot\left(\begin{array}{cc}
1 & 0 \\
0 & \operatorname{cov}\left(x_i\right)
\end{array}\right), \label{eq:g11_oracle}
\end{align}
and
\begin{equation}
\left(\begin{array}{ll}
g_{33} & g_{34} \\
g_{43} & g_{44}
\end{array}\right)
= (1-e) \left(\begin{array}{cc}
1 & 0 \\
0 & \operatorname{cov}\left(x_i\right)
\end{array}\right). \label{eq:g33_oracle}
\end{equation}
This ensures
$$
A_{\textsc{l}}^{*\textup{r}}=\left(\begin{array}{cccc}
g_{11} & 0 & 0 & 0  \\
0 & g_{22} & 0 & 0  \\
0 & 0 & g_{33} & 0  \\
0 & 0 & 0 & g_{44} 
\end{array}\right),
$$
where $g_{11}=e $ and $ g_{33}=1-e$. 
By the formula of block matrix inverse, we have
$$
(A_{\textsc{l}}^{*\textup{r}} )^{-1}
=\left(\begin{array}{cccc}
g_{11}^{-1} & 0 & 0 & 0   \\
0 & g_{22}^{-1} & 0 & 0   \\
0 & 0 & g_{33}^{-1} & 0   \\
0 & 0 & 0 & g_{44}^{-1} 
\end{array}\right)
$$
with
\begin{align}
\left(\begin{array}{llll}
1 & 0 & 0 & 0   \\
0 & 0 & 1 & 0  
\end{array}\right) A_{\textsc{l}}^{*\textup{r}-1}  
& =
\left(\begin{array}{cccc}
g_{11}^{-1} & 0 & 0 & 0 \\
0 & 0 & g_{33}^{-1} & 0
\end{array}\right). 
\label{eq:Lin_A_oracle}
\end{align}
\paragraph*{Compute $A_{\textsc{L}}^{*\textup{r}-1} D_{\textsc{L}}^{*\textup{r}} \left(A_{\textsc{L}}^{*\textup{r}-\mathrm{T}} \right)$:}
Direct algebra ensures
\begin{align}
& \left(\begin{array}{llll}
1 & 0 & 0 & 0  \\
0 & 0 & 1 & 0  
\end{array}\right)
\left(\begin{array}{cccc}
d_{11} & d_{12} & 0 & 0    \\
d_{12}^{\mathrm{T}} & d_{22} & 0  & 0  \\
0 & 0 & d_{33} & d_{34}  \\
0 & 0 & d_{34}^{\mathrm{T}} & d_{44}   
\end{array}\right)\left(\begin{array}{cc}
1 & 0 \\
0 & 0 \\
0 & 1 \\
0 & 0 
\end{array}\right) 
= \left(\begin{array}{cc}
d_{11} & 0 \\
0 & d_{33}
\end{array}\right). \label{eq:Lin_meat_oracle}
\end{align}
Equations \eqref{eq:Lin_normal_oracle}, \eqref{eq:Lin_B_oracle}, \eqref{eq:Lin_A_oracle}, and \eqref{eq:Lin_meat_oracle} together ensure
\begin{align*}
V_{\textsc{l}}^*
=& \operatorname{cov}\left\{\left(\begin{array}{c}
\hat{Y}_{\textsc{l}}^*(1) \\
\hat{Y}_{\textsc{l}}^*(0)
\end{array}\right)\right\}
= \operatorname{cov}
\left\{\left(\begin{array}{llll}
1 & 0 & 0 & 0  \\
0 & 0 & 1 & 0 
\end{array}\right)
\left(\begin{array}{c}
\hat{Y}_{\textsc{l}}^*(1) \\
\hat{\gamma}_1 \\
\hat{Y}_{\textsc{l}}^*(0) \\
\hat{\gamma}_0 
\end{array}\right)\right\} \\
=& \left(\begin{array}{llll}
1 & 0 & 0 & 0   \\
0 & 0 & 1 & 0 
\end{array}\right) 
(A_{\textsc{l}}^{*\textup{r}})^{-1}  D_{\textsc{l}}^{*\textup{r}} \left(A_{\textsc{l}}^{*\textup{r}} \right)^{-\mathrm{T}}
\left(\begin{array}{ll}
1 & 0 \\
0 & 0 \\
0 & 1 \\
0 & 0 
\end{array}\right) \\
=& \left(\begin{array}{cccc}
g_{11}^{-1} & 0 & 0 & 0 \\
0 & 0 & g_{33}^{-1} & 0
\end{array}\right)
\left(\begin{array}{cc}
  d_{11} & 0 \\
  0 & d_{33}
  \end{array}\right)
\left(\begin{array}{cc}
g_{11}^{-1} & 0  \\
0 & 0 \\
0 & g_{33}^{-1} \\
0 & 0
\end{array}\right) \\
=& \left(\begin{array}{cc}
\frac{d_{11}}{e^2} & 0 \\
0 & \frac{d_{33}}{(1-e)^2}
\end{array}\right).
\end{align*}
This ensures
\begin{align*} 
& V_{\textsc{l}}^{*\textup{r}}
=\operatorname{var}\left(\hat{\beta}_{\textsc{l}}\right) 
=\operatorname{var}\left\{\hat{Y}_{\textsc{l}}^*(1)-\hat{Y}_{\textsc{l}}^*(0)\right\}
=\operatorname{var}\left\{(1,-1)\left(\begin{array}{c}
\hat{Y}_{\textsc{l}}^*(1) \\
\hat{Y}_{\textsc{l}}^*(0)
\end{array}\right)\right\} 
=(1,-1) V_{\textsc{l}}^*
\left(\begin{array}{c}
1 \\
-1
\end{array}\right) \\
& =\frac{\mathbb{E}\left[ \left\{Y_i(1)-\left(x_i-\mu_x\right)^{\mathrm{T}} \gamma_1^{\textup{r}}-\mu_1^{\textup{r}}\right\}^2\right]}{e}
+\frac{\mathbb{E}\left[ \left\{Y_i(0)-\left(x_i-\mu_x\right)^{\mathrm{T}} \gamma_0^{\textup{r}}-\mu_0^{\textup{r}}\right\}^2\right]}{1-e}.
\end{align*}
Theorem \ref{thm:ME_random} ensures that $\hat{V}_{\textsc{ehw},\textsc{l}}^* = V_{\textsc{l}}^{*\textup{r}} + o(1;\mathbb{P}_{(y,Z,x)})$.
\end{proof}

The EHW variance estimator in \eqref{eq:EHW_Lin_oracle} has an equivalent form.
Define $\tilde{x}_{i,\textsc{l}}^* = [Z_i, Z_i\dot{x}_i, 1-Z_i, (1-Z_i)\dot{x}_i]^{\mathrm{T}}$. 
Applying the EHW covariance estimator in \eqref{eq:EHW_ME} with moment equations in \eqref{eq:Lin_ME_oracle}, we have 
\begin{align}
\hat{V}_{\textsc{ehw}}^*
=& \left[
\left( \frac{1}{n}\sum_{i=1}^n 
\tilde{x}_{i,\textsc{l}}^* \tilde{x}_{i,\textsc{l}}^{*\mathrm{T}} 
\right)^{-1}
\left( \frac{1}{n}\sum_{i=1}^n \hat{\varepsilon}_{i,\textsc{l}}^{*2} 
\tilde{x}_{i,\textsc{l}}^* \tilde{x}_{i,\textsc{l}}^{*\mathrm{T}}
\right)
\left( \frac{1}{n}\sum_{i=1}^n 
\tilde{x}_{i,\textsc{l}}^* \tilde{x}_{i,\textsc{l}}^{*\mathrm{T}} 
\right)^{-1}
\right]_{(1,1)+(3,3)-2(1,3)}. 
\label{eq:EHW_oracle}
\end{align} 
Define
\begin{align}
R=\left(\begin{array}{cccc}
1 & 0 & 1 & 0 \\
1 & 0 & 0 & 0 \\
0 & I_{p_x} & 0 & I_{p_x} \\
0 & I_{p_x} & 0 & 0
\end{array}\right), 
\quad R^{-1}=\left(\begin{array}{cccc}
0 & 1 & 0 & 0 \\
0 & 0 & 0 & I_{p_x} \\
1 & -1 & 0 & 0 \\
0 & 0 & I_{p_x} & -I_{p_x}
\end{array}\right), 
\quad 
x_{i,\textsc{l}}^* = R \tilde{x}_{i,\textsc{l}}^*. 
\label{eq:R}
\end{align}
Define
$$
\begin{aligned}
\eqref{eq:EHW_Lin_oracle}
& =\left[ \left( \frac{1}{n}\sum_{i=1}^n x_{i,\textsc{l}}^* x_{i,\textsc{l}}^{* \top} \right)^{-1}
\left( \frac{1}{n}\sum_{i=1}^n \hat{\varepsilon}_{i,\textsc{l}}^{*2} x_{i,\textsc{l}}^*x_{i,\textsc{l}}^{* \top} \right)
\left( \frac{1}{n}\sum_{i=1}^n x_{i,\textsc{l}}^* x_{i,\textsc{l}}^{* \top} \right)^{-1} \right]_{(2,2)} \\
& =\left[ R^{-1}
\left( \frac{1}{n}\sum_{i=1}^n 
\tilde{x}_{i,\textsc{l}}^* \tilde{x}_{i,\textsc{l}}^{*\mathrm{T}} 
\right)^{-1}
\left( \frac{1}{n}\sum_{i=1}^n \hat{\varepsilon}_{i,\textsc{l}}^{*2} 
\tilde{x}_{i,\textsc{l}}^* \tilde{x}_{i,\textsc{l}}^{*\mathrm{T}}
\right)
\left( \frac{1}{n}\sum_{i=1}^n 
\tilde{x}_{i,\textsc{l}}^* \tilde{x}_{i,\textsc{l}}^{*\mathrm{T}} 
\right)^{-1} R^{-1}
\right]_{(2,2)}
= \eqref{eq:EHW_oracle}. 
\end{aligned}
$$
Theorem \ref{thm:ME_random} ensures that $\hat{V}_{\textsc{ehw}}^* = V_{\textsc{l}}^{*\textup{r}} + o(1;\mathbb{P}_{(y,Z,x)})$.

\subsubsection{Lin's fully interacted  adjustment with unknown parameters}
\label{subsec:unknown-parameters}
Now we consider the case where $\mu_x$ is unknown. 
Let $\hat{\gamma}_1$ and $\hat{\gamma}_0$ denote the coefficient vectors of $x_i-\bar{x}$ from the OLS fits of
\begin{align}
y_i \sim 1 + \left(x_i-\bar{x}\right) & \text { over }\left\{i: Z_i=1\right\}  \label{eq:Lin_1} \\
y_i \sim 1 + \left(x_i-\bar{x}\right) & \text { over }\left\{i: Z_i=0\right\}, \label{eq:Lin_0}
\end{align}
respectively. Let $\hat{Y}_{\textsc{l}}(1)$ and $\hat{Y}_{\textsc{l}}(0)$ denote the intercepts from \eqref{eq:Lin_1} and \eqref{eq:Lin_0}, respectively.

Lemma \ref{lemma:Lin_beta} below is a standard algebraic characterization of Lin's
regression-adjusted estimator; see, for example, Proposition 1 of \citet{ZhaoDing2021}.
\begin{lemma}\label{lemma:Lin_beta}
$\hat{\beta}_{\textsc{l}} = \hat{Y}_{\textsc{l}}(1)-\hat{Y}_{\textsc{l}}(0)$ with
\begin{align*}
\hat{Y}_{\textsc{l}}(1)
=& \frac{\sum_{i=1}^n Z_i \cdot\left(Y_i(1)-\left(x_i-\bar{x}\right)^{\mathrm{T}} \hat{\gamma}_1\right)}{\sum_{i=1}^n Z_i}, \\
\hat{Y}_{\textsc{l}}(0)
=& \frac{\sum_{i=1}^n \left(1-Z_i\right) \cdot\left(Y_i(0)-\left(x_i-\bar{x}\right)^{\mathrm{T}} \hat{\gamma}_0\right)}{\sum_{i=1}^n \left(1-Z_i\right)} .
\end{align*}
\end{lemma}
\begin{proof}[Proof of Lemma \ref{lemma:Lin_beta}]
That $\hat{\beta}_{\textsc{l}} = \hat{Y}_{\textsc{l}}(1)-\hat{Y}_{\textsc{l}}(0)$ follows from properties of least squares. To verify the explicit form of $\hat{Y}_{\textsc{l}}(1)$ and $\hat{Y}_{\textsc{l}}(0)$, observe that the residual from \eqref{eq:Lin_1} equals $Y_i(1)-\left(x_i-\bar{x}\right)^{\mathrm{T}} \hat{\gamma}_1-\hat{Y}_{\textsc{l}}(1)$ for units with $Z_i=1$. The first-order condition ensures
$$
\sum_{i=1}^n Z_i \cdot\left(Y_i(1)-\left(x_i-\bar{x}\right)^{\mathrm{T}} \hat{\gamma}_1-\hat{Y}_{\textsc{l}}(1)\right)=0
$$
This verifies the expression of $\hat{Y}_{\textsc{l}}(1)$. The expression of $\hat{Y}_{\textsc{l}}(0)$ follows by symmetry.

\end{proof}

Below we give the detailed form of $\hat{V}_{\textsc{ehw},\textsc{l}}$:
\begin{align}
\hat{V}_{\textsc{ehw},\textsc{l}}
~=~& \frac{1}{n} \left[
\left( \frac{1}{n}\sum_{i=1}^n x_{i,\textsc{l}}x_{i,\textsc{l}}^{\mathrm{T}} \right)^{-1}
\left( \frac{1}{n}\sum_{i=1}^n \hat{\varepsilon}_{i,\textsc{l}}^2 x_{i,\textsc{l}}x_{i,\textsc{l}}^{\mathrm{T}} \right)
\left( \frac{1}{n}\sum_{i=1}^n x_{i,\textsc{l}}x_{i,\textsc{l}}^{\mathrm{T}} \right)^{-1}
\right]_{(2,2)}. 
\label{eq:EHW_L} 
\end{align}

\begin{proof}[Proof of Theorem \ref{thm:FRA_random}]
\label{proof:thm:FRA_random}
The asymptotic normality of $\hat{\beta}_{\textsc{l}}$ is proved by \citet[Theorem 5.1]{NegiWooldridge2021}. Here we prove it by applying the theory of $Z$-estimation with augmented moment equations.
\paragraph*{Asymptotic normality}
Recall from Lemma \ref{lemma:Lin_beta} that $(\hat{Y}_{\textsc{l}}(1), \hat{\gamma}_1)$ and $(\hat{Y}_{\textsc{l}}(0), \hat{\gamma}_0)$ are the intercepts and coefficient vectors of $\left(x_i-\bar{x}\right)$ from the OLS fits of \eqref{eq:Lin_1} and \eqref{eq:Lin_0}, respectively, with $\hat{\beta}_{\textsc{l}} =\hat{Y}_{\textsc{l}}(1)-\hat{Y}_{\textsc{l}}(0)$. The first-order conditions of \eqref{eq:Lin_1} and \eqref{eq:Lin_0} ensure that $\left(\mu_1, \gamma_1\right)=(\hat{Y}_{\textsc{l}}(1), \hat{\gamma}_1)$ and $\left(\mu_0, \gamma_0\right)=(\hat{Y}_{\textsc{l}}(0), \hat{\gamma}_0)$ solve
$$
0=\sum_{i=1}^n Z_i \cdot\left(Y_i(1)-\left(x_i-\bar{x}\right)^{\mathrm{T}} \gamma_1-\mu_1\right)\left(\begin{array}{c}
1 \\
x_i-\bar{x}
\end{array}\right).
$$
This ensures that $\left(\mu_1, \gamma_1, \mu_0, \gamma_0, \mu_x\right)= \left(\hat{Y}_{\textsc{l}}(1), \hat{\gamma}_1, \hat{Y}_{\textsc{l}}(0), \hat{\gamma}_0, \bar{x}\right)$ jointly solves
$$
0=n^{-1} \sum_{i=1}^n \eta\left(Y_i(1), Y_i(0), x_i, Z_i; \theta\right)
=n^{-1} \sum_{i=1}^n \left(\begin{array}{c}
\eta_1\left(Y_i(1), x_i, Z_i; \mu_1, \gamma_1, \mu_x\right) \\
\eta_0\left(Y_i(0), x_i, Z_i; \mu_0, \gamma_0, \mu_x\right) \\
\eta_x\left(x_i ; \mu_x\right) 
\end{array}\right)
$$
where $\theta=\left(\mu_1, \gamma_1, \mu_0, \gamma_0,\mu_x\right)$ and
\begin{align}
\eta=
\left(\begin{array}{c}
\eta^* \\
\eta_x 
\end{array}\right) 
\text { with } 
\quad \begin{aligned}
\eta_x 
&= x_i-\mu_x
\end{aligned} \label{eq:Lin_ME}
\end{align}
and $\eta^*$ is defined in \eqref{eq:Lin_ME_oracle}.
Direct algebra ensures that $\theta=\theta^{\textup{r}}=\left(\mu_1^{\textup{r}}, \gamma_1^{\textup{r}}, \mu_0^{\textup{r}}, \gamma_0^{\textup{r}},\mu_x\right)$ solves $\mathbb{E}\{\eta(\theta)\}=0$. 
Lemma \ref{lemma:ME_diff} ensures that 
\begin{align}
\sqrt{n}\left\{\left(\begin{array}{c}
\hat{Y}_{\textsc{l}}(1) \\
\hat{\gamma}_1 \\
\hat{Y}_{\textsc{l}}(0) \\
\hat{\gamma}_0 
\end{array}\right)-\left(\begin{array}{c}
\mu_1^{\textup{r}} \\
\gamma_1^{\textup{r}} \\
\mu_0^{\textup{r}} \\
\gamma_0^{\textup{r}} 
\end{array}\right)\right\} \overset{\textup{d}}{\to} 
\mathcal{N}\left(0, A_{\textsc{l}}^{\textup{r}-1}  D_{\textsc{l}}^{\textup{r}} \left(A_{\textsc{l}}^{\textup{r}-\mathrm{T}} \right) \right), 
\label{eq:Lin_normal}
\end{align}
where
\begin{align*}
A_{\textsc{l}}^{\textup{r}}
=& - \mathbb{E}\left(\frac{\partial}{\partial\left(\mu_1, \gamma_1, \mu_0, \gamma_0 \right)} \eta^*  \right)
= A_{\textsc{l}}^{*\textup{r}} \\
D_{\textsc{l}}^{\textup{r}}
=& \mathbb{E}\left( \left(\eta^* - \mathbb{E}\left[\frac{\partial \eta^*}{\partial\mu_x }\right] \mathbb{E}\left[ \frac{\partial}{\partial \mu_x} \eta_x \right]^{-1} \eta_x \right)\left( \eta^* - \mathbb{E}\left[\frac{\partial \eta^*}{\partial\mu_x }\right] \mathbb{E}\left[ \frac{\partial}{\partial \mu_x} \eta_x \right]^{-1} \eta_x \right)^{\mathrm{T}} \right) 
\end{align*}
evaluated at $\theta=\theta^{\textup{r}}$. We compute below $D_{\textsc{l}}^{\textup{r}} $.

\noindent \textbf{Computing $D_{\textsc{L}}^{\textup{r}} $:}
With a slight abuse of notation, let $\frac{\partial}{\partial \mu_1} \eta_1^{\textup{r}}$ denote the value of $\frac{\partial}{\partial \mu_1} \eta_1$ evaluated at $\theta^{\textup{r}}$. 
Observe that
\begin{align}
\frac{\partial}{\partial \mu_x^{\mathrm{T}}} \eta_1^{\textup{r}} 
=& \left\{Y_i(1)-\left(x_i-\mu_x\right)^{\mathrm{T}} \gamma_1^{\textup{r}} - \mu_1^{\textup{r}} \right\}\left(\begin{array}{c}
0 \\
-I_J
\end{array}\right)+\left(\begin{array}{c}
1 \\
x_i-\mu_x
\end{array}\right) (\gamma_1^{\textup{r}})^{\mathrm{T}} \notag \\
=& Z_i 
\left(\begin{array}{c}
(\gamma_1^{\textup{r}})^{\mathrm{T}} \\
-\left\{Y_i(1)-\left( x_i-\mu_x \right)^{\mathrm{T}} \gamma_1^{\textup{r}}-\mu_1^{\textup{r}} \right\} I_J+\left( x_i-\mu_x \right) (\gamma_1^{\textup{r}})^{\mathrm{T}}
\end{array}\right) \\
\frac{\partial}{\partial\left(\mu_1^{\mathrm{T}}, \gamma_1^{\mathrm{T}}\right)} \eta_1^{\textup{r}} 
=& -Z_i \cdot\left(\begin{array}{c}
1 \\
x_i-\mu_x
\end{array}\right)\left(1,\left(x_i-\mu_x \right)^{\mathrm{T}}\right), \\
\frac{\partial}{\partial \mu_x^{\mathrm{T}}} \eta_x^{\textup{r}} 
& =-I_J .
\end{align}
Therefore,
\begin{align*}
& \eta^{\textup{r}} - \mathbb{E}\left[\frac{\partial \eta^{\textup{r}}}{\partial\mu_x }\right] \mathbb{E}\left[ \frac{\partial}{\partial \mu_x} \eta_x \right]^{-1} \eta_x 
= \eta^{\textup{r}} + 
\left(\begin{array}{c}
e (\gamma_1^{\textup{r}})^{\mathrm{T}} \dot{x}_i \\
0 \\
(1-e) (\gamma_0^{\textup{r}})^{\mathrm{T}} \dot{x}_i \\
0
\end{array}\right).
\end{align*}
The $D_{\textsc{l}}^{\textup{r}} $ matrix equals
\begin{align*}
D_{\textsc{l}}^{\textup{r}} 
=& \left(\begin{array}{cc|cc}
d_{11} & d_{12} & 0 & 0   \\
d_{12}^{\mathrm{T}} & d_{22} & 0 & 0  \\
\hline 0 & 0 & d_{33} & d_{34}   \\
0 & 0 & d_{34}^{\mathrm{T}} & d_{44}   
\end{array}\right)
+ \left(\begin{array}{cccc}
e^2 (\gamma_1^{\textup{r}})^{\mathrm{T}} \operatorname{cov}\left(x_i\right) \gamma_1^{\textup{r}} & 0 & e(1-e) (\gamma_1^{\textup{r}})^{\mathrm{T}} \operatorname{cov}\left(x_i\right) \gamma_0^{\textup{r}} & 0   \\
0 & 0 & 0 & 0  \\
e(1-e) (\gamma_0^{\textup{r}})^{\mathrm{T}} \operatorname{cov}\left(x_i\right) \gamma_1^{\textup{r}} & 0 & (1-e)^2 (\gamma_1^{\textup{r}})^{\mathrm{T}} \operatorname{cov}\left(x_i\right) \gamma_1^{\textup{r}} & 0  \\
0 & 0 & 0 & 0
\end{array}\right) 
\end{align*}
where (i) $d_{11}$ is defined in \eqref{eq:d_11} and $d_{22}$ is defined in \eqref{eq:d_33},
(ii)
\begin{align*}
\mathbb{E}\left[\left\{Y_i(1)-\mu_1^{\textup{r}}\right\}\left(x_i-\mu_x\right)^{\mathrm{T}}\right]-(\gamma_1^{{\textup{r}}} )^\mathrm{T} \mathbb{E}\left\{\left(x_i-\mu_x \right)\left(x_i-\mu_x\right)^{\mathrm{T}}\right\}=0.
\end{align*}

\noindent \textbf{Compute $A_{\textsc{L}}^{\textup{r}-1}  D_{\textsc{L}}^{\textup{r}} \left(A_{\textsc{L}}^{\textup{r}-\mathrm{T}} \right)$:}
Equations \eqref{eq:Lin_B_oracle}, \eqref{eq:Lin_normal}, \eqref{eq:Lin_A_oracle}, and \eqref{eq:Lin_meat_oracle} together ensure
\begin{align*}
V_{\textsc{l}}
=& \operatorname{cov}\left\{\left(\begin{array}{c}
\hat{Y}_{\textsc{l}}(1) \\
\hat{Y}_{\textsc{l}}(0)
\end{array}\right)\right\}
= \operatorname{cov}\left\{\left(\begin{array}{llll}
1 & 0 & 0 & 0   \\
0 & 0 & 1 & 0 
\end{array}\right)\left(\begin{array}{c}
\hat{Y}_{\textsc{l}}(1) \\
\hat{\gamma}_1 \\
\hat{Y}_{\textsc{l}}(0) \\
\hat{\gamma}_0 
\end{array}\right)\right\} \\
=& \left(\begin{array}{llll}
1 & 0 & 0 & 0   \\
0 & 0 & 1 & 0  
\end{array}\right) 
A_{\textsc{l}}^{\textup{r}-1}  D_{\textsc{l}}^{\textup{r}} \left(A_{\textsc{l}}^{\textup{r}-\mathrm{T}} \right) \left(\begin{array}{ll}
1 & 0 \\
0 & 0 \\
0 & 1 \\
0 & 0 
\end{array}\right) \\
=&  \left(\begin{array}{cccc}
g_{11}^{-1} & 0 & 0 & 0 \\
0 & 0 & g_{33}^{-1} & 0
\end{array}\right)
D_{\textsc{l}}^{\textup{r}}
\left(\begin{array}{cc}
g_{11}^{-1} & 0 \\
0 & 0 \\
0 & g_{33}^{-1} \\
0 & 0 
\end{array}\right)  \\
=& \left(\begin{array}{cc}
\frac{d_{11}}{e^2} & 0 \\
0 & \frac{d_{33}}{(1-e)^2}
\end{array}\right) + 
\left(\gamma_1^{\textup{r}}, \gamma_0^{\textup{r}}\right)^\mathrm{T}
\operatorname{cov}\left(x_i\right)
\left(\gamma_1^{\textup{r}}, \gamma_0^{\textup{r}}\right).
\end{align*}
This ensures
\begin{align}
& V_{\textsc{l}}^{\textup{r}}
=\operatorname{var}\left(\hat{\beta}_{\textsc{l}}\right) 
=\operatorname{var}\left\{\hat{Y}_{\textsc{l}}(1)-\hat{Y}_{\textsc{l}}(0)\right\}
=\operatorname{var}\left\{(1,-1)\left(\begin{array}{c}
\hat{Y}_{\textsc{l}}(1) \\
\hat{Y}_{\textsc{l}}(0)
\end{array}\right)\right\} 
=(1,-1) V_{\textsc{l}}\left(\begin{array}{c}
1 \\
-1
\end{array}\right) \notag \\
& =\frac{\mathbb{E}\left[ \left\{Y_i(1)-\left(x_i-\mu_x\right)^{\mathrm{T}} \gamma_1^{\textup{r}}-\mu_1^{\textup{r}}\right\}^2\right]}{e}
+ \frac{\mathbb{E}\left[ \left\{Y_i(0)-\left(x_i-\mu_x\right)^{\mathrm{T}} \gamma_0^{\textup{r}}-\mu_0^{\textup{r}}\right\}^2\right]}{1-e}
+ \left(\gamma_1^{\textup{r}}-\gamma_0^{\textup{r}}\right)^{\mathrm{T}} \operatorname{cov}\left(x_i\right)\left(\gamma_1^{\textup{r}}-\gamma_0^{\textup{r}}\right) \notag \\
&= V_{\textsc{l}}^{*\textup{r}} + \left(\gamma_1^{\textup{r}}-\gamma_0^{\textup{r}}\right)^{\mathrm{T}} \operatorname{cov}\left(x_i\right)\left(\gamma_1^{\textup{r}}-\gamma_0^{\textup{r}}\right). \label{eq:correction}
\end{align}

\paragraph*{Consistency:}
The proof that $\hat{V}_{\textsc{ehw},\textsc{l},\text{adj}} = V_{\textsc{l}}^{\textup{r}} + o(1;\mathbb{P}_{(y,Z,x)})$ is similar to that of Theorem \ref{thm:cluster_adj_r}, so we omit it here.

\paragraph*{Anti-conservativeness}
The anti-conservativeness of $\hat{V}_{\textsc{ehw},\textsc{l}}$ follows from Theorem \ref{thm:FRA_random}, i.i.d. of $x_i$ and consistency of $\hat{\gamma}_0$ and $\hat{\gamma}_1$.

\end{proof}

\begin{proof}[Proof of Theorem \ref{thm:FRA_m}]
Recall Lemma \ref{lemma:Lin_beta}, $\theta = \left(\mu_1, \gamma_1, \mu_0, \gamma_0\right)= \left(\hat{Y}_{\textsc{l}}(1), \hat{\gamma}_1, \hat{Y}_{\textsc{l}}(0), \hat{\gamma}_0\right)$ jointly solves
$$
0=n^{-1} \sum_{i=1}^n \eta\left(Y_i(1), Y_i(0), x_i, Z_i; \theta\right)
=n^{-1} \sum_{i=1}^n \left(\begin{array}{c}
\eta_{1}(Y_i(1), x_i, Z_i; \mu_1, \gamma_1 ) \\
\eta_{0}(Y_i(0), x_i, Z_i; \mu_0, \gamma_0 ) 
\end{array}\right)
$$
and
\begin{align}
\eta=
\left(\begin{array}{c}
\eta_1 \\
\eta_0 
\end{array}\right) 
\text { with } 
\quad \begin{aligned}
\eta_1 
& =Z_i \cdot\left\{y_i-\left(x_i-\bar{x}\right)^{\mathrm{T}} \gamma_1-\mu_1\right\}
\left(\begin{array}{c}
1 \\
x_i-\bar{x}
\end{array}\right), \\
\eta_0 
& =(1-Z_i)  \cdot\left\{y_i-\left(x_i-\bar{x}\right)^{\mathrm{T}} \gamma_0-\mu_0\right\}
\left(
\begin{array}{c}
1 \\
x_i-\bar{x}
\end{array}
\right).
\end{aligned} 
\label{eq:Lin_ME_mixed}
\end{align}
Direct algebra ensures that $\theta=\theta^\textup{m}=\left(\mu_1^\textup{m}, \gamma_1^\textup{m}, \mu_0^\textup{m}, \gamma_0^\textup{m} \right)$ solves $n^{-1}\sum_{i=1}^n \mathbb{E}\{\eta(\theta)\mid x_i\}=0$. Theorem \ref{thm:ME_mixed} ensures that 
\begin{align}
\sqrt{n}\left\{\left(\begin{array}{c}
\hat{Y}_{\textsc{l}}(1) \\
\hat{\gamma}_1 \\
\hat{Y}_{\textsc{l}}(0) \\
\hat{\gamma}_0 
\end{array}\right)-\left(\begin{array}{c}
\mu_1^\textup{m} \\
\gamma_1^\textup{m} \\
\mu_0^\textup{m} \\
\gamma_0^\textup{m} 
\end{array}\right)\right\} 
\overset{\textup{d}}{\to} 
\mathcal{N}\left(0, (A_{\textsc{l}}^{\textup{m}-1})  D_{\textsc{l}}^{\textup{m}} \left(A_{\textsc{l}}^{\textup{m}-\mathrm{T}} \right) \right), 
\label{eq:Lin_normal_mixed}
\end{align}
where $A_{\textsc{l}}^{\textup{m}}$ and $D_{\textsc{l}}^{\textup{m}} $ are the values of $-n^{-1}\sum_{i=1}^n \mathbb{E}\left(\frac{\partial}{\partial\left(\mu_1, \gamma_1, \mu_0, \gamma_0 \right)} \eta\mid x_i\right)$ and $n^{-1}\sum_{i=1}^n \V(\eta(\theta) \mid x_i)$ evaluated at $\theta=\theta^{\textup{m}}$. We compute below $A_{\textsc{l}}^{\textup{m}}$ and $D_{\textsc{l}}^{\textup{m}} $, respectively.

\paragraph*{Computing $D_{\textsc{L}}^{\textup{m}} $:}
Let $\left(\eta^{\textup{m}}, \eta_{1}^{\textup{m}}, \eta_{0}^{\textup{m}} \right)$ denote the value of $\left(\eta, \eta_1, \eta_0\right)$ evaluated at $\theta=\theta^{\textup{m}}$. From \eqref{eq:Lin_ME_mixed}, we have
\begin{align}
\eta^{\textup{m}}=
\left(\begin{array}{c}
\eta_1^{\textup{m}} \\
\eta_0^{\textup{m}} 
\end{array}\right) 
\text { with } 
\quad \begin{aligned}
\eta_1^{\textup{m}} 
& =Z_i \cdot\left\{Y_i(1)-\left(x_i-\bar{x}\right)^{\mathrm{T}} \gamma_1^{\textup{m}}-\mu_1^{\textup{m}}\right\}
\left(\begin{array}{c}
1 \\
x_i-\bar{x}
\end{array}\right), \\
\eta_0^{\textup{m}} 
& =(1-Z_i)  \cdot\left\{Y_i(0)-\left(x_i-\bar{x}\right)^{\mathrm{T}} \gamma_0^{\textup{m}}-\mu_0^{\textup{m}}\right\}
\left(\begin{array}{c}
1 \\
x_i-\bar{x}
\end{array}\right).
\end{aligned} \label{eq:Lin_eta_mixed}
\end{align}
Define $\varepsilon_{i}^{\textup{m}}(z) = Y_i(z)-\left(x_i-\bar{x}\right)^{\mathrm{T}} \gamma_z^{\textup{m}}-\mu_z^{\textup{m}}$.
The $D_{\textsc{l}}^{\textup{m}} $ matrix equals
\begin{equation}
D_{\textsc{l}}^{\textup{m}} 
=\frac{1}{n}\sum_{i=1}^n \V\left(\eta^{\textup{m}} \mid x_i \right)
=\left(\begin{array}{cc|cc}
d_{11} & d_{12} & d_{13} & d_{14}   \\
d_{12}^{\mathrm{T}} & d_{22} & d_{23} & d_{24}  \\
\hline 
d_{13}^{\mathrm{T}} & d_{23}^{\mathrm{T}} & d_{33} & d_{34} \\
d_{14}^{\mathrm{T}} & d_{24}^{\mathrm{T}} & d_{34}^{\mathrm{T}} & d_{44}  
\end{array}\right)
=\left(\begin{array}{cc}
D_{11} & D_{12}  \\
D_{12}^{\mathrm{T}} & D_{22}  
\end{array}\right), \label{eq:Lin_B_mixed}
\end{equation}
where
\begin{align}
d_{11} 
&= \frac{1}{n}\sum_{i=1}^n \left(e \mathbb{E}((\varepsilon_i^{\textup{m}}(1))^2 \mid x_i) - e^2 \mathbb{E}(\varepsilon_i^{\textup{m}}(1) \mid x_i)^2 \right), \label{eq:d_11_mixed} \\
d_{33} 
&= \frac{1}{n}\sum_{i=1}^n \left( (1-e) \mathbb{E}((\varepsilon_i^{\textup{m}}(0))^2 \mid x_i) - (1-e)^2 \mathbb{E}(\varepsilon_i^{\textup{m}}(0) \mid x_i)^2 \right), \label{eq:d_33_mixed}  \\
d_{13} 
&= - e(1-e) \frac{1}{n}\sum_{i=1}^n \mathbb{E}(\varepsilon_i^{\textup{m}}(1) \mid x_i) \mathbb{E}(\varepsilon_i^{\textup{m}}(0) \mid x_i). 
\label{eq:d_13_mixed}
\end{align}

\noindent \textbf{Compute $A_{\textsc{L}}^{\textup{m}}$:}
With a slight abuse of notation, let $\frac{\partial}{\partial \mu_1} \eta_1^*$ denote the value of $\frac{\partial}{\partial \mu_1} \eta_1$ evaluated at $\theta^*$. Similarly define other partial derivatives. 
From \eqref{eq:Lin_ME_mixed}, we have
$$
\left.\frac{\partial}{\partial\left(\mu_1, \gamma_1^{\mathrm{T}}, \mu_0, \gamma_0^{\mathrm{T}} \right)} \eta\right|_{\theta=\theta^{\textup{m}}}
=\left(\begin{array}{ccc}
\frac{\partial}{\partial\left(\mu_1, \gamma_1^{\mathrm{T}}\right)} \eta_1^{\textup{m}} & 0  \\
0 & \frac{\partial}{\partial\left(\mu_0, \gamma_0^{\mathrm{T}}\right)} \eta_0^{\textup{m}}
\end{array}\right),
$$
where
$$
\begin{aligned}
\frac{\partial}{\partial\left(\mu_1, \gamma_1^{\mathrm{T}}\right)} \eta_1^{\textup{m}} 
& =-Z_i \cdot\left(\begin{array}{c}
1 \\
x_i-\bar{x}
\end{array}\right)
\left(1,\left(x_i-\bar{x} \right)^{\mathrm{T}}\right).
\end{aligned}
$$
Accordingly, we have
\begin{align}
A_{\textsc{l}}^{\textup{m}} 
= - \frac{1}{n}\sum_{i=1}^n \left.\mathbb{E}\left(\frac{\partial}{\partial\left(\mu_1, \gamma_1^{\mathrm{T}}, \mu_0, \gamma_0^{\mathrm{T}}  \right)} \eta \mid x_i\right)\right|_{\theta=\theta^{\textup{m}}}
= \left(\begin{array}{cc|cc}
g_{11} & g_{12} & 0 & 0 \\
g_{21} & g_{22} & 0 & 0 \\
\hline 0 & 0 & g_{33} & g_{34}  \\
0 & 0 & g_{43} & g_{44} 
\end{array}\right), 
\label{eq:A_mixed_L}
\end{align}
where
\begin{align*}
\left(\begin{array}{ll}
g_{11} & g_{12} \\
g_{21} & g_{22}
\end{array}\right)
=& -\frac{1}{n}\sum_{i=1}^n \mathbb{E}\left(\frac{\partial}{\partial\left(\mu_1, \gamma_1^{\mathrm{T}}\right)} \eta_1^{\textup{m}}\mid x_i \right) 
= \frac{1}{n}\sum_{i=1}^n \mathbb{E}\left\{Z_i  \cdot\left(\begin{array}{c}
1 \\
x_i-\bar{x}
\end{array}\right)\left(1,\left(x_i-\bar{x}\right)^{\mathrm{T}}\right) \mid x_i \right\} \\
=& e \cdot\left(\begin{array}{cc}
1 & 0 \\
0 & \frac{1}{n} \sum_{i=1}^n (x_i-\bar{x})\left(x_i-\bar{x}\right)^{\mathrm{T}}
\end{array}\right), 
\end{align*}
and by symmetry,
$$
\left(\begin{array}{ll}
g_{33} & g_{34} \\
g_{43} & g_{44}
\end{array}\right)
= (1-e) \left(\begin{array}{cc}
1 & 0 \\
0 & \frac{1}{n} \sum_{i=1}^n (x_i-\bar{x})\left(x_i-\bar{x}\right)^{\mathrm{T}}
\end{array}\right).
$$
By the formula of block matrix inverse, we have
$$
A_{\textsc{l}}^{\textup{m}-1} 
=\left(\begin{array}{cccc}
g_{11}^{-1} & 0 & 0 & 0  \\
0 & g_{22}^{-1} & 0 & 0   \\
0 & 0 & g_{33}^{-1} & 0   \\
0 & 0 & 0 & g_{44}^{-1} 
\end{array}\right)
$$
with
\begin{align}
\left(\begin{array}{llll}
1 & 0 & 0 & 0  \\
0 & 0 & 1 & 0 
\end{array}\right) A_{\textsc{l}}^{\textup{m}-1}  
=\left(\begin{array}{cc}
g_{11}^{-1} & 0 \\
0 & g_{33}^{-1}
\end{array}\right)\left(\begin{array}{cccc}
1 & 0 & 0 & 0  \\
0 & 0 & 1 & 0  
\end{array}\right). 
\label{eq:Lin_A_mixed}
\end{align}

\noindent\textbf{Compute $A_{\textsc{L}}^{\textup{m}-1}  D_{\textsc{L}}^{\textup{m}} \left(A_{\textsc{L}}^{\textup{m}-\mathrm{T}} \right)$:}
Direct algebra ensures
\begin{align}
& \left(\begin{array}{llll}
1 & 0 & 0 & 0  \\
0 & 0 & 1 & 0 
\end{array}\right)
\left(\begin{array}{cc|cc}
d_{11} & d_{12} & d_{13} & d_{14}   \\
d_{12}^{\mathrm{T}} & d_{22} & d_{23} & d_{24}  \\
\hline 
d_{13}^{\mathrm{T}} & d_{23}^{\mathrm{T}} & d_{33} & d_{34} \\
d_{14}^{\mathrm{T}} & d_{24}^{\mathrm{T}} & d_{34}^{\mathrm{T}} & d_{44}  
\end{array}\right)
\left(\begin{array}{cc}
1 & 0 \\
0 & 0 \\
0 & 1 \\
0 & 0 
\end{array}\right) 
= \left(\begin{array}{cc}
d_{11} & d_{13} \\
d_{13}^{\mathrm{T}} & d_{33}
\end{array}\right). 
\label{eq:Lin_meat_mixed}
\end{align}
Equations \eqref{eq:Lin_normal_mixed}, \eqref{eq:Lin_B_mixed}, \eqref{eq:Lin_A_mixed}, and \eqref{eq:Lin_meat_mixed} together ensure
\begin{align*}
& V_{\textsc{l}}
= \operatorname{cov}\left\{\left(\begin{array}{c}
\hat{Y}_{\textsc{l}}(1) \\
\hat{Y}_{\textsc{l}}(0)
\end{array}\right) \mid x \right\}
= \operatorname{cov}\left\{\left(\begin{array}{llll}
1 & 0 & 0 & 0   \\
0 & 0 & 1 & 0 
\end{array}\right)\left(\begin{array}{c}
\hat{Y}_{\textsc{l}}(1) \\
\hat{\gamma}_1 \\
\hat{Y}_{\textsc{l}}(0) \\
\hat{\gamma}_0 
\end{array}\right) \mid x \right\} \\
=& \left(\begin{array}{llll}
1 & 0 & 0 & 0  \\
0 & 0 & 1 & 0 
\end{array}\right) (A_{\textsc{l}}^{\textup{m}-1})  D_{\textsc{l}}^{\textup{m}} \left(A_{\textsc{l}}^{\textup{m}-\mathrm{T}} \right) \left(\begin{array}{ll}
1 & 0 \\
0 & 0 \\
0 & 1 \\
0 & 0 
\end{array}\right) \\
=& \left(\begin{array}{cc}g_{11}^{-1} & 0 \\ 0 & g_{33}^{-1}\end{array}\right) 
\left(\begin{array}{cccc}
1 & 0 & 0 & 0  \\
0 & 0 & 1 & 0  
\end{array}\right)
\left(\begin{array}{cccc}
d_{11} & d_{12} & d_{13} & d_{14}   \\
d_{12}^{\mathrm{T}} & d_{22} & d_{23} & d_{24}  \\
d_{13}^{\mathrm{T}} & d_{23}^{\mathrm{T}} & d_{33} & d_{34} \\
d_{14}^{\mathrm{T}} & d_{24}^{\mathrm{T}} & d_{34}^{\mathrm{T}} & d_{44}  
\end{array}\right)
\left(\begin{array}{ccc}
1 & 0 \\
0 & 0 \\
0 & 1 \\
0 & 0 
\end{array}\right) 
\left(\begin{array}{cc}
g_{11}^{-1} & 0 \\
0 & g_{33}^{-1}
\end{array}\right) \\
=& \left(\begin{array}{cc}
g_{11}^{-1} & 0 \\
0 & g_{33}^{-1}
\end{array}\right) 
\left(\begin{array}{cc}
d_{11} & d_{13} \\
d_{13}^{\mathrm{T}} & d_{33}
\end{array}\right)
\left(\begin{array}{cc}
g_{11}^{-1} & 0 \\
0 & g_{33}^{-1}
\end{array}\right) \\
=& \left(\begin{array}{cc}
\frac{d_{11}}{e^2} & \frac{d_{13}}{e(1-e)} \\
\frac{d_{13}^{\mathrm{T}}}{e(1-e)} & \frac{d_{33}}{(1-e)^2}
\end{array}\right).
\end{align*}
This ensures
\begin{align*}
& V_{\textsc{l}}^{\textup{m}}
=\operatorname{var}\left(\hat{\beta}_{\textsc{l}} \mid x \right) 
=\operatorname{var}\left\{\hat{Y}_{\textsc{l}}(1)-\hat{Y}_{\textsc{l}}(0) \mid x\right\}
=\operatorname{var}\left\{(1,-1)\left(\begin{array}{c}
\hat{Y}_{\textsc{l}}(1) \\
\hat{Y}_{\textsc{l}}(0)
\end{array}\right) \mid x \right\} 
=(1,-1) V_{\textsc{l}}\left(\begin{array}{c}
1 \\
-1
\end{array}\right) \\
=& \frac{d_{11}}{e^2} - 2\frac{d_{13}}{e(1-e)} + \frac{d_{33}}{(1-e)^2} \\
=& \frac{\frac{1}{n}\sum_{i=1}^n \left(e \mathbb{E}((\varepsilon_{i,\textsc{l}}^{\textup{m}}(1))^2 \mid x_i) - e^2 \mathbb{E}(\varepsilon_{i,\textsc{l}}^{\textup{m}}(1) \mid x_i)^2 \right)}{e^2} + \frac{\frac{1}{n}\sum_{i=1}^n \left((1-e) \mathbb{E}((\varepsilon_{i,\textsc{l}}^{\textup{m}}(0))^2 \mid x_i) - (1-e)^2 \mathbb{E}(\varepsilon_{i,\textsc{l}}^{\textup{m}}(0) \mid x_i)^2 \right)}{(1-e)^2} \\
&- 2 \frac{e(1-e) \frac{1}{n}\sum_{i=1}^n \mathbb{E}(\varepsilon_{i,\textsc{l}}^{\textup{m}}(1) \mid x_i) \mathbb{E}(\varepsilon_{i,\textsc{l}}^{\textup{m}}(0) \mid x_i)}{e(1-e)} \\
=& \frac{\frac{1}{n}\sum_{i=1}^n \mathbb{E}((\varepsilon_{i,\textsc{l}}^{\textup{m}}(1))^2 \mid x_i) }{e} 
+ \frac{\frac{1}{n}\sum_{i=1}^n \mathbb{E}((\varepsilon_{i,\textsc{l}}^{\textup{m}}(0))^2 \mid x_i) }{1-e}
- \frac{1}{n}\sum_{i=1}^n \mathbb{E}(\varepsilon_{i,\textsc{l}}^{\textup{m}}(1) - \varepsilon_{i,\textsc{l}}^{\textup{m}}(0) \mid x_i)^2.
\end{align*}

\paragraph*{Conservativeness:}
Define $\tilde{x}_{i,\textsc{l}} = (Z_i, Z_i\ddot{x}_i, 1-Z_i, (1-Z_i)\ddot{x}_i)^{\mathrm{T}}$. 
Applying the EHW variance estimator in \eqref{eq:EHW_ME} with moment equations in \eqref{eq:Lin_ME_mixed}, we have 
\begin{align}
\tilde{V}_{\textsc{ehw},\textsc{l}}
=& \left[
\left( \frac{1}{n}\sum_{i=1}^n 
\tilde{x}_{i,\textsc{l}} \tilde{x}_{i,\textsc{l}}^{\mathrm{T}} 
\right)^{-1}
\left( \frac{1}{n}\sum_{i=1}^n \hat{\varepsilon}_{i,\textsc{l}}^{2} 
\tilde{x}_{i,\textsc{l}} \tilde{x}_{i,\textsc{l}}^{\mathrm{T}}
\right)
\left( \frac{1}{n}\sum_{i=1}^n 
\tilde{x}_{i,\textsc{l}} \tilde{x}_{i,\textsc{l}}^{\mathrm{T}} 
\right)^{-1}
\right]_{(1,1)+(3,3)-2(1,3)}. 
\label{eq:EHW_lin_equi}
\end{align} 
Recall $R$ and $R^{-1}$ in \eqref{eq:R} and we have 
$x_{i,\textsc{l}} = R \tilde{x}_{i,\textsc{l}}$.
The EHW variance estimator of $\hat{\beta}_{\textsc{l}}$ has the following equivalent forms:
\begin{align*}
\eqref{eq:EHW_L}
& =\left[ R^{-1}
\left( \frac{1}{n}\sum_{i=1}^n 
\tilde{x}_{i,\textsc{l}} \tilde{x}_{i,\textsc{l}}^{\mathrm{T}} 
\right)^{-1}
\left( \frac{1}{n}\sum_{i=1}^n \hat{\varepsilon}_{i,\textsc{l}}^{2} 
\tilde{x}_{i,\textsc{l}} \tilde{x}_{i,\textsc{l}}^{\mathrm{T}}
\right)
\left( \frac{1}{n}\sum_{i=1}^n 
\tilde{x}_{i,\textsc{l}} \tilde{x}_{i,\textsc{l}}^{\mathrm{T}} 
\right)^{-1} R^{-1}
\right]_{(2,2)}
= \eqref{eq:EHW_lin_equi}. 
\end{align*}
Theorem \ref{thm:ME_mixed} ensures that 
\[
\tilde{V}_{\textsc{ehw},\textsc{l}}
= V_{\textsc{l}}^\textup{m}
+ B_{\textsc{l}}^\textup{m}
+ o(1;\mathbb{P}_{(y,x_1)\mid x_2})  
\]
with the bias term
\begin{equation*}
B_{\textsc{l}}^\textup{m}
= \left[(A_{\textsc{l}}^\textup{m})^{-1} \left( \frac{1}{n} \sum_{i=1}^n \mathbb{E}\left[ \eta(\theta^\textup{m}) \mid x_{i} \right] \mathbb{E}\left[ \eta(\theta^\textup{m}) \mid x_{i} \right]^{\mathrm{T}}  \right) (A_{\textsc{l}}^\textup{m})^{-\mathrm{T}} \right]_{(1,1)+(3,3)-2(1,3)}.
\end{equation*}
We compute the ``middle'' matrix:
\begin{align*}
& \frac{1}{n} \sum_{i=1}^n \mathbb{E}\left[ \eta(\theta^\textup{m}) \mid x_{i} \right] \mathbb{E}\left[ \eta(\theta^\textup{m}) \mid x_{i} \right]^{\mathrm{T}}
= \left(
\begin{array}{cccc}
h_{11} & h_{12} & h_{13} & h_{14} \\
h_{12}^{\mathrm{T}} & h_{22} & h_{23} & h_{24} \\
h_{13}^{\mathrm{T}} & h_{23}^{\mathrm{T}} & h_{33} & h_{34} \\
h_{14}^{\mathrm{T}} & h_{24}^{\mathrm{T}} & h_{34}^{\mathrm{T}} & h_{44} \\
\end{array}
\right)
\end{align*}
where 
\begin{align*}
h_{11}
~=~ & e^2 \frac{1}{n} \sum_{i=1}^n \mathbb{E}(\varepsilon_{i,\textsc{l}}^{\textup{m}}(1) \mid x_{i})^2 \\
h_{33}
~=~ & (1-e)^2 \frac{1}{n} \sum_{i=1}^n \mathbb{E}(\varepsilon_{i,\textsc{l}}^{\textup{m}}(0) \mid x_{i})^2 \\
h_{13}
~=~ & e(1-e) \frac{1}{n} \sum_{i=1}^n \mathbb{E}(\varepsilon_{i,\textsc{l}}^{\textup{m}}(1) \mid x_{i}) \mathbb{E}(\varepsilon_{i,\textsc{l}}^{\textup{m}}(0) \mid x_{i}).
\end{align*}
Therefore,
\begin{align*}
&  \left[(A_{\textsc{l}}^\textup{m})^{-1} \left( \frac{1}{n} \sum_{i=1}^n \mathbb{E}\left[ \eta(\theta^\textup{m}) \mid x_{i} \right] \mathbb{E}\left[ \eta(\theta^\textup{m}) \mid x_{i} \right]^{\mathrm{T}}  \right) (A_{\textsc{l}}^{\textup{m}})^{-{\mathrm{T}}} \right]_{(1,1)+(3,3)-2(1,3)} \\
=& (1,-1) \left(\begin{array}{llll}
1 & 0 & 0 & 0  \\
0 & 0 & 1 & 0 
\end{array}\right) A_{\textsc{l}}^{\textup{m}-1}  \left( \frac{1}{n} \sum_{i=1}^n \mathbb{E}\left[ \eta(\theta^\textup{m}) \mid x_{i} \right] \mathbb{E}\left[ \eta(\theta^\textup{m}) \mid x_{i} \right]^{\mathrm{T}}  \right) \left(A_{\textsc{l}}^{\textup{m}-{\mathrm{T}}} \right)\left(\begin{array}{ll}
1 & 0 \\
0 & 0 \\
0 & 1 \\
0 & 0 
\end{array}\right)
\left(\begin{array}{c}
1 \\
-1
\end{array}\right) \\
=&  (1,-1) \left(\begin{array}{cc}
g_{11}^{-1} & 0 \\
0 & g_{33}^{-1}
\end{array}\right) 
\left(\begin{array}{cc}
h_{11} & h_{13} \\
h_{13}^{\mathrm{T}} & h_{33}
\end{array}\right)
\left(\begin{array}{cc}
g_{11}^{-1} & 0 \\
0 & g_{33}^{-1}
\end{array}\right)
\left(\begin{array}{c}
1 \\
-1
\end{array}\right) \\
=& \frac{h_{11}}{e^2} + \frac{h_{33}}{(1-e)^2} - 2\frac{h_{13}}{e(1-e)}  \\
=& \frac{1}{n} \sum_{i=1}^n \mathbb{E}(\varepsilon_{i,\textsc{l}}^{\textup{m}}(1) \mid x_{i})^2 
+ \frac{1}{n} \sum_{i=1}^n \mathbb{E}(\varepsilon_{i,\textsc{l}}^{\textup{m}}(0) \mid x_{i})^2
- 2 \frac{1}{n} \sum_{i=1}^n \mathbb{E}(\varepsilon_{i,\textsc{l}}^{\textup{m}}(1) \mid x_{i}) \mathbb{E}(\varepsilon_{i,\textsc{l}}^{\textup{m}}(0) \mid x_{i}) \\
=& \frac{1}{n} \sum_{i=1}^n \mathbb{E}(\varepsilon_{i,\textsc{l}}^{\textup{m}}(1) - \varepsilon_{i,\textsc{l}}^{\textup{m}}(0) \mid x_{i})^2.
\end{align*}
Then we verify the form of the bias term.

\end{proof}

\begin{proof}[Proof of Proposition~\ref{prop:fisher-lin-mixed}]
Define
\[
\mu_{iz}
=
\mathbb{E}\{y_i(z)\mid x_i\}
-
\frac{1}{n}\sum_{j=1}^n \mathbb{E}\{y_j(z)\mid x_j\},
\qquad z\in\{0,1\}.
\]
For arbitrary vectors $b_1$ and $b_0$, let $r_{iz}(b_z)=\mu_{iz}-\ddot{x}_i^{\mathrm{T}}b_z$.
The component of the asymptotic variance that depends on
$(b_1,b_0)$ is
\begin{align*}
&\frac{1}{n}\sum_{i=1}^n
\left\{
\frac{r_{i1}(b_1)^2}{e}
+
\frac{r_{i0}(b_0)^2}{1-e}
-
\bigl(r_{i1}(b_1)-r_{i0}(b_0)\bigr)^2
\right\} 
=
\frac{1}{e(1-e)}
\frac{1}{n}\sum_{i=1}^n
\left\{
(1-e)r_{i1}(b_1)+e r_{i0}(b_0)
\right\}^2.
\end{align*}
For Lin's fully interacted adjustment,
$b_z=\gamma_z^{\mathrm{m}}$, so the coefficient used to linearly
approximate
\[
(1-e)\mu_{i1}+e\mu_{i0}
\]
is $(1-e)\gamma_1^{\mathrm{m}}+e\gamma_0^{\mathrm{m}}$.
For Fisher's additive adjustment, the common slope is
\[
\gamma_{\textsc{f}}^{\mathrm{m}}
=
(1-e)\gamma_0^{\mathrm{m}}
+
e\gamma_1^{\mathrm{m}}.
\]
Because
$(1-e)\gamma_1^{\mathrm{m}}+e\gamma_0^{\mathrm{m}}$
is the least-squares projection coefficient of
$(1-e)\mu_{i1}+e\mu_{i0}$ on $\ddot{x}_i$, the projection theorem gives
\begin{align*}
V_{\textsc{f}}^{\mathrm{m}}-V_{\textsc{l}}^{\mathrm{m}}
&=
\frac{1}{e(1-e)}
\left[
\gamma_{\textsc{f}}^{\mathrm{m}}
-
\left\{
(1-e)\gamma_1^{\mathrm{m}}
+
e\gamma_0^{\mathrm{m}}
\right\}
\right]^{\mathrm{T}}
\Sigma_x 
\times
\left[
\gamma_{\textsc{f}}^{\mathrm{m}}
-
\left\{
(1-e)\gamma_1^{\mathrm{m}}
+
e\gamma_0^{\mathrm{m}}
\right\}
\right].
\end{align*}
Finally,
\[
\gamma_{\textsc{f}}^{\mathrm{m}}
-
\left\{
(1-e)\gamma_1^{\mathrm{m}}
+
e\gamma_0^{\mathrm{m}}
\right\}
=
(2e-1)
\left(\gamma_1^{\mathrm{m}}-\gamma_0^{\mathrm{m}}\right),
\]
which proves the result.
\end{proof}

\subsection{IV Regression} \label{sec:proof_iv}

\subsubsection{Without covariates} \label{sec:proof_iv_noX}
Below we give the closed-form of the HW variance estimator for $\hat{\beta}_{\textsc{iv}}$ without covariates:
\begin{align}
\hat{V}_{\textsc{hw,iv}}
=& \frac{1}{n} \left[
\left( \frac{1}{n}\sum_{i=1}^n \tilde{Z}_i \tilde{D}_i^{\mathrm{T}} \right)^{-1}
\left( \frac{1}{n}\sum_{i=1}^n \hat{\varepsilon}_i^2 \tilde{Z}_i \tilde{Z}_i^{\mathrm{T}} \right)
\left( \frac{1}{n}\sum_{i=1}^n \tilde{D}_i \tilde{Z}_i^{\mathrm{T}} \right)^{-1}
\right]_{(2,2)}
\label{eq:EHW_IV}
\end{align}
where $\tilde{D}_i = (1, D_i)^{\mathrm{T}}$.

\begin{proof}[Proof of Theorem \ref{thm:AsyDist_IV_random}]
To simplify notations, define $\mu(d,z)=\mathbb{E}({Y}_i(d)\mid D_i(z)=d)$, $\sigma^{2}(d,z)=\V({Y}_i(d)\mid D_i(z)=d)$, and $\pi_{D(z)} = \mathbb{P}(D_i(z)=1 )$ for $z\in\{0,1\}$.
The proof is the same as the proof of Theorem \ref{thm:AsyDist_IV_random_x}. So we omit it.

\end{proof}

\subsubsection{With covariates} \label{sec:proof_iv_withX}
Below we give the closed-form of the HW variance estimator for $\hat{\beta}_{\textsc{iv}}$ with covariates:
\begin{equation}
\hat{V}_{\textsc{hw,iv,f}}
= \frac{1}{n} \left[
\left( \frac{1}{n}\sum_{i=1}^n \tilde{Z}_{i,\textsc{f}} \tilde{D}_{i,\textsc{f}}^{\mathrm{T}} \right)^{-1}
\left( \frac{1}{n}\sum_{i=1}^n \hat{\varepsilon}_{i,\textsc{f}}^2 \tilde{Z}_{i,\textsc{f}} \tilde{Z}_{i,\textsc{f}}^{\mathrm{T}} \right)
\left( \frac{1}{n}\sum_{i=1}^n \tilde{D}_{i,\textsc{f}} \tilde{Z}_{i,\textsc{f}}^{\mathrm{T}} \right)^{-1}
\right]_{(2,2)}
\label{eq:EHW_IV_x}
\end{equation}
where $\tilde{D}_{i,\textsc{f}} = (1, D_i, {x}_i^{\mathrm{T}})^{\mathrm{T}}$.

\begin{proof}[Proof of Theorem \ref{thm:AsyDist_IV_random_x}]

We prove the results by applying the theory of $Z$-estimation.

\paragraph*{Asymptotic normality.}
Recall that $(\hat{\alpha}_{\textsc{iv,f}}, \hat{\beta}_{\textsc{iv,f}}, \hat{\gamma}^{\mathrm{T}}_{\textsc{iv,f}})$ are the intercept and coefficients of $D_i$ and $x_i$ from the 2SLS fits of \eqref{eq:IV}. The first-order conditions of \eqref{eq:IV} ensure that $\theta=(\alpha, \beta, \gamma)=(\hat{\alpha}_{\textsc{iv,f}}, \hat{\beta}_{\textsc{iv,f}}, \hat{\gamma}^{\mathrm{T}}_{\textsc{iv,f}})$ jointly solves
$$
0=n^{-1} \sum_{i=1}^n \eta(y_i, Z_i, D_i, x_i; \theta)
=n^{-1} \sum_{i=1}^n \left(\begin{array}{c}
\eta_{\alpha}(y_i, Z_i, D_i, x_i; \theta) \\
\eta_{\beta}(y_i, Z_i, D_i, x_i; \theta) \\
\eta_{\gamma}(y_i, Z_i, D_i, x_i; \theta) 
\end{array}\right)
$$
and
\begin{align}
\eta=
\left(\begin{array}{c}
\eta_{\alpha} \\
\eta_{\beta} \\
\eta_{\gamma}
\end{array}\right) 
\text { with } 
\quad \begin{aligned}
\eta_{\alpha}
& = y_i- \alpha - \beta D_i - \dot{x}_i^{\mathrm{T}} \gamma, \\
\eta_{\beta} 
& = (y_i- \alpha - \beta D_i - \dot{x}_i^{\mathrm{T}} \gamma) Z_i, \\
\eta_{\gamma}
&= (y_i- \alpha - \beta D_i - \dot{x}_i^{\mathrm{T}} \gamma) \dot{x}_i. 
\end{aligned} \label{eq:IV_ME}
\end{align}
Direct algebra ensures that $\theta=\theta^{\textup{r}}=\left(\alpha_{\textsc{iv,f}}^{\textup{r}}, \beta_{\textsc{iv,f}}^{\textup{r}}, \gamma_{\textsc{iv,f}}^{\textup{r}}\right)$ solves $\mathbb{E}(\eta(\theta))=0$:
\begin{align*}
\alpha_{\textsc{iv,f}}^{\textup{r}} 
=& \frac{\pi_{D(1)} (\pi_{D(0)}\mu(1,0) + (1-\pi_{D(0)})\mu(0,0)) }{\pi_{D(1)} - \pi_{D(0)}}  
- \frac{\pi_{D(0)} (\pi_{D(1)}\mu(1,1) + (1-\pi_{D(1)})\mu(0,1))}{\pi_{D(1)} - \pi_{D(0)}}, \\
\beta_{\textsc{iv,f}}^{\textup{r}} 
=& \mathbb{E}[Y_i(1) - Y_i(0) \mid  D_i(1) > D_i(0)]
= \beta_{\textsc{iv}}^{\textup{r}} , \\
\gamma_{\textsc{iv,f}}^{\textup{r}}
=& \mathbb{E}(\dot{x}_i\dot{x}_i^{\mathrm{T}})^{-1}\mathbb{E}(x_i(y_i- \alpha_{\textsc{iv,f}}^{\textup{r}} - \beta_{\textsc{iv}}^{\textup{r}} D_i)). 
\end{align*}
Define $\tilde{Y}^{\textup{r}}_i(d) = Y_i(d) - x_i^{\mathrm{T}}\gamma_{\textsc{iv,f}}^{\textup{r}}$, $\tilde{\mu}(d,z)=\mathbb{E}(\tilde{Y}_i(d)\mid D_i(z)=d)$ and $\tilde{\sigma}^{2}(d,z)=\V(\tilde{Y}_i(d)\mid D_i(z)=d)$.
Theorem \ref{thm:ME_random} ensures
\begin{align}
\sqrt{n}\left\{\left(\begin{array}{c}
\hat{\alpha}_{\textsc{iv,f}} \\
\hat{\beta}_{\textsc{iv,f}} \\
\hat{\gamma}_{\textsc{iv,f}} 
\end{array}\right)-\left(\begin{array}{c}
\alpha_{\textsc{iv,f}}^{\textup{r}} \\
\beta_{\textsc{iv}}^{\textup{r}} \\
\gamma_{\textsc{iv,f}}^{\textup{r}} 
\end{array}\right)\right\} \overset{\textup{d}}{\to} 
\mathcal{N}\left(0, (A_{\textsc{f}}^{{\textup{r}}-1}) D_{\textsc{f}}^{\textup{r}}\left(A_{\textsc{f}}^{{\textup{r}}-\mathrm{T}}\right)\right), \label{eq:IV_normal_random}
\end{align}
where $A_{\textsc{f}}^{\textup{r}}$ and $D_{\textsc{f}}^{\textup{r}}$ are the values of $-\mathbb{E}\left(\frac{\partial}{\partial\left(\alpha, \beta, \gamma \right)} \eta\right)$ and $\mathbb{E}\left(\eta \eta^{\mathrm{T}}\right)$ evaluated at $\theta=\theta^{\textup{r}}$. 
We compute below $A_{\textsc{f}}^{\textup{r}}$ and $D_{\textsc{f}}^{\textup{r}}$, respectively.

\noindent\textbf{Compute $D_{\textsc{F}}^{\textup{r}}$:}
Let $(\eta^{\textup{r}}, \eta_{\alpha}^{\textup{r}}, \eta_{\beta}^{\textup{r}}, \eta_{\gamma}^{\textup{r}})$ denote the value of $(\eta, \eta_{\alpha}, \eta_{\beta}, \eta_{\gamma})$ evaluated at $\theta=\theta^{\textup{r}}$. 
From \eqref{eq:IV_ME}, we have
\begin{align}
\eta^{\textup{r}}
=\left(\begin{array}{c}
\eta_{\alpha}^{\textup{r}} \\
\eta_{\beta}^{\textup{r}} \\
\eta_{\gamma}^{\textup{r}} 
\end{array}\right) \text { with } \quad 
\begin{aligned}
\eta_{\alpha}^{\textup{r}}
& = y_i- \alpha_{\textsc{iv,f}}^{\textup{r}} - \beta_{\textsc{iv}}^{\textup{r}} D_i - \dot{x}_i^{\mathrm{T}} \gamma_{\textsc{iv,f}}^{\textup{r}}, \\
\eta_{\beta}^{\textup{r}}
& = (y_i- \alpha_{\textsc{iv,f}}^{\textup{r}} - \beta_{\textsc{iv}}^{\textup{r}} D_i - \dot{x}_i^{\mathrm{T}} \gamma_{\textsc{iv,f}}^{\textup{r}}) Z_i, \\
\eta_{\gamma}^{\textup{r}}
&= (y_i- \alpha_{\textsc{iv,f}}^{\textup{r}} - \beta_{\textsc{iv}}^{\textup{r}} D_i - \dot{x}_i^{\mathrm{T}} \gamma_{\textsc{iv,f}}^{\textup{r}}) \dot{x}_i. 
\end{aligned} \label{eq:IV_eta}
\end{align}
The $D_{\textsc{f}}^{\textup{r}}$ matrix equals
\begin{equation}
D_{\textsc{f}}^{\textup{r}} 
=\left.\mathbb{E}\left(\eta \eta^{\mathrm{T}}\right)\right|_{\theta=\theta^{\textup{r}}}=\mathbb{E}\left(\eta^{\textup{r}} (\eta^{{\textup{r}} })^\mathrm{T} \right)
=\left(\begin{array}{ccc}
d_{11} & d_{12} & d_{13}  \\
d_{12}^{\mathrm{T}} & d_{22} & d_{23}  \\
d_{13}^{\mathrm{T}} & d_{23}^{\mathrm{T}} & d_{33} 
\end{array}\right). \label{eq:IV_B}
\end{equation}
Observe that
\begin{align*}
d_{11}
=& \mathbb{E}\left(\eta_{\alpha}^{\textup{r}} (\eta_{\alpha}^{{\textup{r}}})^\mathrm{T}\right) 
= \mathbb{E}\left[ (y_i- \alpha_{\textsc{iv,f}}^{\textup{r}} - \beta_{\textsc{iv}}^{\textup{r}} D_i - \dot{x}_i^{\mathrm{T}} \gamma_{\textsc{iv,f}}^{\textup{r}})^2 \right] \\
=& \left(
\begin{array}{l}
\mathbb{E}\left[ Z_iD_i(y_i- \alpha_{\textsc{iv,f}}^{\textup{r}} - \beta_{\textsc{iv}}^{\textup{r}} D_i - \dot{x}_i^{\mathrm{T}} \gamma_{\textsc{iv,f}}^{\textup{r}})^2 \right] \\
+ \mathbb{E}\left[ Z_i(1-D_i)(y_i- \alpha_{\textsc{iv,f}}^{\textup{r}} - \beta_{\textsc{iv}}^{\textup{r}} D_i - \dot{x}_i^{\mathrm{T}} \gamma_{\textsc{iv,f}}^{\textup{r}})^2 \right] \\
+ \mathbb{E}\left[ (1-Z_i)D_i(y_i- \alpha_{\textsc{iv,f}}^{\textup{r}} - \beta_{\textsc{iv}}^{\textup{r}} D_i - \dot{x}_i^{\mathrm{T}} \gamma_{\textsc{iv,f}}^{\textup{r}})^2 \right] \\
+ \mathbb{E}\left[ (1-Z_i)(1-D_i)(y_i- \alpha_{\textsc{iv,f}}^{\textup{r}} - \beta_{\textsc{iv}}^{\textup{r}} D_i - \dot{x}_i^{\mathrm{T}} \gamma_{\textsc{iv,f}}^{\textup{r}})^2 \right] 
\end{array}
\right).
\end{align*}
We only derive the result of $\mathbb{E}\left[ Z_iD_i(y_i- \alpha_{\textsc{iv,f}}^{\textup{r}} - \beta_{\textsc{iv}}^{\textup{r}} D_i - \dot{x}_i^{\mathrm{T}} \gamma_{\textsc{iv,f}}^{\textup{r}})^2 \right]$, as the proof of the remaining three terms follows from analogous arguments. By direct algebra, 
\begin{align*}
&  \mathbb{E}\left[ Z_iD_i(y_i- \alpha_{\textsc{iv,f}}^{\textup{r}} - \beta_{\textsc{iv}}^{\textup{r}} D_i - \dot{x}_i^{\mathrm{T}} \gamma_{\textsc{iv,f}}^{\textup{r}})^2 \right]  \\
=& \mathbb{E}\left[ Z_iD_i(1) \left(Y_i(1) - \dot{x}_i^{\mathrm{T}} \gamma_{\textsc{iv,f}}^{\textup{r}} - \mu(1,1) + \mu(1,1) - \alpha_{\textsc{iv,f}}^{\textup{r}} - \beta_{\textsc{iv}}^{\textup{r}} D_i \right)^2 \right] \\
=& e \pi_{D(1)} 
\left(\sigma ^{2}(1,1)
+(1-\pi_{D(1)})^2(\mu(1,1)-\mu(0,1)-\beta_{\textsc{iv}}^{\textup{r}})^2
\right).
\end{align*}
Therefore,
\begin{align}
d_{11}
=& \mathbb{E}\left( \eta_{\alpha}^{\textup{r}} (\eta_{\alpha}^{{\textup{r}}})^\mathrm{T} \right) 
=\mathbb{E}\left[ (y_i- \alpha_{\textsc{iv,f}}^{\textup{r}} - \beta_{\textsc{iv}}^{\textup{r}} D_i - \dot{x}_i^{\mathrm{T}} \gamma_{\textsc{iv,f}}^{\textup{r}} )^2 \right] \notag \\
=& \left(
\begin{array}{l}
e\pi _{D(1)}\times  
\left(\sigma ^{2}(1,1)
+(1-\pi_{D(1)})^2(\mu(1,1)-\mu(0,1)-\beta_{\textsc{iv}}^{\textup{r}})^2
\right) \\ 
+ e(1-\pi _{D(1)})\times 
\left(\sigma ^{2}(0,1
)+(\pi_{D(1)})^2(\mu(1,1)-\mu(0,1)-\beta_{\textsc{iv}}^{\textup{r}})^2
\right) \\ 
+ (1-e)\pi _{D(0)}\times  
\left( \sigma ^{2}(1,0)
+(1-\pi_{D(0)})^2(\mu(1,0)-\mu(0,0)-\beta_{\textsc{iv}}^{\textup{r}})^2
\right) \\ 
+ (1-e)(1-\pi _{D(0)})\times 
\left( \sigma^{2}(0,0)
+(\pi_{D(0)})^2(\mu(1,0)-\mu(0,0)-\beta_{\textsc{iv}}^{\textup{r}})^2
\right)
\end{array}
\right), \label{eq:IV_D11} \\
d_{22}
=& \mathbb{E}\left( \eta_{\beta}^{\textup{r}} (\eta_{\beta}^{{\textup{r}}})^\mathrm{T} \right) 
=\mathbb{E}\left[ (y_i- \alpha_{\textsc{iv,f}}^{\textup{r}} - \beta_{\textsc{iv}}^{\textup{r}} D_i - \dot{x}_i^{\mathrm{T}} \gamma_{\textsc{iv,f}}^{\textup{r}} )^2 Z_i \right] \notag \\
=& \left(
\begin{array}{l}
e\pi _{D(1)}\times  
\left(\sigma ^{2}(1,1)
+(1-\pi_{D(1)})^2(\mu(1,1)-\mu(0,1)-\beta_{\textsc{iv}}^{\textup{r}})^2
\right) \\ 
+ e(1-\pi_{D(1)})\times 
\left(\sigma ^{2}(0,1
)+(\pi_{D(1)})^2(\mu(1,1)-\mu(0,1)-\beta_{\textsc{iv}}^{\textup{r}})^2
\right) 
\end{array} 
\right),
\label{eq:IV_D22}
\end{align}
and 
\begin{align*}
d_{12}
=\mathbb{E}\left(\eta_{\alpha}^{\textup{r}} (\eta_{\beta}^{{\textup{r}}})^\mathrm{T}\right) 
=\mathbb{E}\left[ (y_i- \alpha_{\textsc{iv,f}}^{\textup{r}} - \beta_{\textsc{iv}}^{\textup{r}} D_i - \dot{x}_i^{\mathrm{T}} \gamma_{\textsc{iv,f}}^{\textup{r}} )^2 Z_i \right] 
= d_{22}.
\end{align*}
The terms $d_{13}$, $d_{23}$ and $d_{33}$ are irrelevant here. So we omit them.

\noindent\textbf{Compute $A_{\textsc{F}}^{\textup{r}}$:}
With a slight abuse of notation, let $\frac{\partial}{\partial \alpha} \eta_{\alpha}^{\textup{r}}$ denote the value of $\frac{\partial}{\partial \alpha} \eta_\alpha$ evaluated at $\theta^{\textup{r}}$. Similarly define other partial derivatives.
From \eqref{eq:IV_ME}, we have
$$
\left.\frac{\partial}{\partial\left( \alpha, \beta, \gamma \right)} \eta\right|_{\theta=\theta^{\textup{r}}}
=\left(\begin{array}{ccc}
\frac{\partial}{\partial \alpha} \eta_\alpha^{\textup{r}} & \frac{\partial}{\partial \beta} \eta_\alpha^{\textup{r}} & \frac{\partial}{\partial \gamma^{\mathrm{T}}} \eta_\alpha^{\textup{r}}  \\
\frac{\partial}{\partial \alpha} \eta_\beta^{\textup{r}} & \frac{\partial}{\partial \beta} \eta_\beta^{\textup{r}} & \frac{\partial}{\partial \gamma^{\mathrm{T}}} \eta_\beta^{\textup{r}}  \\
\frac{\partial}{\partial \alpha} \eta_\gamma^{\textup{r}} & \frac{\partial}{\partial \beta} \eta_\gamma^{\textup{r}} & \frac{\partial}{\partial \gamma^{\mathrm{T}}} \eta_\gamma^{\textup{r}}  \\ 
\end{array}\right),
$$
where
\begin{align*}
\frac{\partial}{\partial \alpha} \eta_\alpha^{\textup{r}}
=& -1,  \quad
\frac{\partial}{\partial \beta} \eta_\alpha^{\textup{r}}
= -D_i,  \quad
\frac{\partial}{\partial \gamma} \eta_\alpha^{\textup{r}}
= - \dot{x}_i^{\mathrm{T}}, \\
\frac{\partial}{\partial \alpha} \eta_\beta^{\textup{r}}
=& -Z_i,  \quad
\frac{\partial}{\partial \beta} \eta_\beta^{\textup{r}}
= -Z_iD_i,  \quad
\frac{\partial}{\partial \gamma} \eta_\beta^{\textup{r}}
= -Z_i \dot{x}_i^{\mathrm{T}}, \\
\frac{\partial}{\partial \alpha} \eta_\gamma^{\textup{r}}
=& - \dot{x}_i,  \quad
\frac{\partial}{\partial \beta} \eta_\gamma^{\textup{r}}
= - \dot{x}_iD_i,  \quad
\frac{\partial}{\partial \gamma} \eta_\gamma^{\textup{r}}
= - \dot{x}_i \dot{x}_i^{\mathrm{T}}.
\end{align*}
Accordingly, we have
\begin{align}
A_{\textsc{f}}^{\textup{r}}
=-\left.\mathbb{E}\left(\frac{\partial}{\partial\left( \alpha, \beta, \gamma \right)} \eta\right)\right|_{\theta=\theta^r}
=& -\mathbb{E}\left(\left.\frac{\partial}{\partial\left( \alpha, \beta, \gamma \right)} \eta\right|_{\theta=\theta^r}\right) 
= \left(\begin{array}{ccc}
1 & \mathbb{E}(D_i) & 0 \\
e & \mathbb{E}(Z_iD_i) & 0  \\
0 & \mathbb{E}(D_i \dot{x}_i) & \mathbb{E}(\dot{x}_i \dot{x}_i^{\mathrm{T}}) 
\end{array}\right). 
\label{eq:IV_A}
\end{align}
By the formula of matrix inverse, we have
$$
A_{\textsc{f}}^{\textup{r}-1} 
= \left(\begin{array}{ccc}
\frac{\mathbb{E}(Z_iD_i)}{\mathbb{E}(Z_iD_i)-e\mathbb{E}(D_i)} & \frac{ -\mathbb{E}(D_i)}{\mathbb{E}(Z_iD_i)-e\mathbb{E}(D_i)} & 0 \\
\frac{-e}{\mathbb{E}(Z_iD_i)-e\mathbb{E}(D_i)} & \frac{1}{\mathbb{E}(Z_iD_i)-e\mathbb{E}(D_i)} & 0  \\
* & *  & \mathbb{E}(\dot{x}_i \dot{x}_i^{\mathrm{T}})^{-1} 
\end{array}\right)
$$
where * denotes the terms that are irrelevant here.
By direct algebra,
\begin{align*}
\mathbb{E}(Z_iD_i)-e\mathbb{E}(D_i)
=& \mathbb{E}(Z_iD_i(1)) - e(Z_iD_i(1)+(1-Z_i)D_i(0)) 
= e(1-e) (\pi_{D(1)} - \pi_{D(0)}).
\end{align*}
\noindent\textbf{Compute $(A_{\textsc{F}}^{\textup{r}-1}) D_{\textsc{F}}^{\textup{r}} \left(A_{\textsc{F}}^{\textup{r} -\mathrm{T}} \right)$:}
Direct algebra ensures
\begin{align}
V_{\textsc{iv,f}}^{\textup{r}}  
&= \left[ (A_{\textsc{f}}^{\textup{r}-1})  D_{\textsc{f}}^{\textup{r}} \left(A_{\textsc{f}}^{\textup{r}-\mathrm{T}} \right) \right]_{(2,2)} 
= \frac{(1-e)D_{22} - e(D_{22}-eD_{11})}{e^2(1-e)^2 (\pi_{D(1)} - \pi_{D(0)})^2} \notag \\
&=
\frac{1}{\mathbb{P}(U_i = \textup{c})^2}
\left[
\begin{array}{c}
\frac{
\tilde\sigma^2(1,1)\pi_D(1)
+
\tilde\sigma^2(0,1)\{1-\pi_D(1)\}
}{e} 
+
\frac{
\tilde\sigma^2(1,0)\pi_D(0)
+
\tilde\sigma^2(0,0)\{1-\pi_D(0)\}
}{1-e} \\
+
\frac{
\{\tilde\mu(1,1)-\tilde\mu(0,1)-\beta_{\textsc{iv}}^{\textup{r}} \}^2
\pi_D(1)\{1-\pi_D(1)\}
}{e} 
+
\frac{
\{\tilde\mu(1,0)-\tilde\mu(0,0)-\beta_{\textsc{iv}}^{\textup{r}}\}^2
\pi_D(0)\{1-\pi_D(0)\}
}{1-e}
\end{array} 
\right].
\label{eq:Lin_meat}
\end{align}

\paragraph*{Consistency of $\hat{V}_{\textsc{HW,F}}$:}
By \eqref{eq:EHW_ME}, the HW variance estimator for $\hat{\beta}_{\textsc{iv,f}}$ is
\begin{align}
\hat{V}_{\textsc{hw,iv,f}} 
=& \left[
\left( \frac{1}{n}\sum_{i=1}^n \tilde{D}_i \tilde{Z}_i^{\mathrm{T}} \right)^{-1}
\left( \frac{1}{n}\sum_{i=1}^n \hat{\varepsilon}_{i,\textsc{f}}^2 \tilde{Z}_i \tilde{Z}_i^{\mathrm{T}} \right)
\left( \frac{1}{n}\sum_{i=1}^n \tilde{Z}_i \tilde{D}_i^{\mathrm{T}} \right)^{-1}
\right]_{(2,2)} \notag \\
=& \frac{ (1-\frac{1}{n}\sum_{i=1}^n Z_i)^2 \frac{1}{n}\sum_{i=1}^n Z_i\hat{\varepsilon}_{i,\textsc{f}}^2 + (\frac{1}{n}\sum_{i=1}^n Z_i)^2 \frac{1}{n}\sum_{i=1}^n (1-Z_i) \hat{\varepsilon}_{i,\textsc{f}}^2 }{ \left( \frac{1}{n}\sum_{i=1}^n Z_iD_i -  \frac{1}{n}\sum_{i=1}^n D_i \frac{1}{n}\sum_{i=1}^n D_i \right)^2 }. 
\label{eq:EHW_IV}
\end{align}
Theorem \ref{thm:ME_random} ensures the consistency of $\hat{V}_{\textsc{hw,iv,f}}$: 
\[
\hat{V}_{\textsc{hw,iv,f}} 
=  V_{\textsc{iv,f}}^\textup{r}
+ o(1;\mathbb{P}_{(y,Z,D)}).
\]
\end{proof}

\begin{proof}[Proof of Theorem \ref{thm:AsyDist_IV_mixed}]
We prove the results by applying the theory of $Z$-estimation.

\paragraph*{Asymptotic normality:}
Direct algebra ensures that $\theta=\theta^{\textup{m}}=\left(\alpha_{\textsc{iv}}^{\textup{m}}, \beta_{\textsc{iv}}^{\textup{m}}, \gamma_{\textsc{iv}}^{\textup{m}}\right)$ solves
\[
\frac{1}{n}\sum_{i=1}^n \mathbb{E}(\eta(y_i, Z_i, D_i, x_i; \theta) \mid x_i)=0
\] 
with $\eta(\theta)$ defined in \eqref{eq:IV_ME}:
\begin{align*}
\alpha_{\textsc{iv,f}}^{\textup{m}} 
=& \frac{\frac{1}{n}\sum_{i=1}^n\pi_{D_i(1)} \left[ \frac{1}{n}\sum_{i=1}^n \pi_{D_i(0)} \mu_i(1,0,x) + \frac{1}{n}\sum_{i=1}^n (1-\pi_{D_i(0)}) \mu_i(0,0,x) \right] }{ \frac{1}{n}\sum_{i=1}^n (\pi_{D_i(1)}-\pi_{D_i(0)})}  \\
&- \frac{ \frac{1}{n}\sum_{i=1}^n\pi_{D_i(0)} \left[\frac{1}{n}\sum_{i=1}^n \pi_{D_i(1)} \mu_i(1,1,x) + \frac{1}{n}\sum_{i=1}^n (1-\pi_{D_i(1)}) \mu_i(0,1,x) \right] }{ \frac{1}{n}\sum_{i=1}^n (\pi_{D_i(1)}-\pi_{D_i(0)})} , \\
\beta_{\textsc{iv,f}}^{\textup{m}} 
=& \frac{ n^{-1} \sum_{i=1}^{n} \tau_{\textup{c}}(x_i) \pi_{\textup{c}}(x_i)  }{ n^{-1} \sum_{i=1}^{n} \pi_{\textup{c}}(x_i) }
= \beta_{\textsc{iv}}^{\textup{m}} , \\
\gamma_{\textsc{iv,f}}^{\textup{m}}
=& \left(\frac{1}{n}\sum_{i=1}^n \ddot{x}_i \ddot{x}_i^{\mathrm{T}} \right)^{-1}
\frac{1}{n}\sum_{i=1}^n \mathbb{E}(\ddot{x}_i(y_i- \alpha_{\textsc{iv,f}}^{\textup{m}} - \beta_{\textsc{iv}}^{\textup{m}} D_i) \mid x_i).
\end{align*}
Define $\tilde{Y}_i(d)^{\textup{m}} = Y_i(d) - \ddot{x}_i^{\mathrm{T}}\gamma_{\textsc{iv,f}}^{\textup{m}}$,
$\mu_i(d,z,x) = \mathbb{E}[\tilde{Y}_i(d)^{\textup{m}} | D_i(z)=d, x_i=x]$ and $\pi_{D_i(z)} = \mathbb{P}(D_i(z)=1 \mid x_i)$ for $d \in \{0,1\}$ and $z\in\{0, 1\}$.
Theorem \ref{thm:ME_mixed} ensures
\begin{align}
\sqrt{n}\left\{\left(\begin{array}{c}
\hat{\alpha}_{\textsc{iv,f}} \\
\hat{\beta}_{\textsc{iv,f}} \\
\hat{\gamma}_{\textsc{iv,f}} 
\end{array}\right)-\left(\begin{array}{c}
\alpha_{\textsc{iv,f}}^{\textup{m}} \\
\beta_{\textsc{iv,f}}^{\textup{m}} \\
\gamma_{\textsc{iv,f}}^{\textup{m}} 
\end{array}\right)\right\} \overset{\textup{d}}{\to} 
\mathcal{N}\left(0, (A^{{\textup{m}}-1} )D^{\textup{m}}\left(A^{{\textup{m}}-\mathrm{T}}\right)\right), \label{eq:IV_normal_mixed}
\end{align}
where
\begin{align*}
A_{\textsc{f}}^\textup{m} 
&= - \frac{1}{n} \sum_{i=1}^n \mathbb{E}\left(  \frac{\partial \eta(y_i, Z_i, D_i, x_i; \theta^\textup{m})}{\partial (\alpha, \beta, \gamma)}  \mid x_{i} \right)
\text{ and }
D_{\textsc{f}}^\textup{m} 
= \frac{1}{n} \sum_{i=1}^n \V\left[ \eta(y_i, Z_i, D_i, x_i; \theta^\textup{m}) \mid x_{i} \right].
\end{align*}
We compute below $A^{\textup{m}}$ and $D^{\textup{m}}$, respectively.

\noindent\textbf{Compute $D^{\textup{m}}$:}
Let $(\eta^{\textup{m}}, \eta_{\alpha}^{\textup{m}}, \eta_{\beta}^{\textup{m}}, \eta_{\gamma}^{\textup{m}})$ denote the value of $(\eta, \eta_{\alpha}, \eta_{\beta}, \eta_{\gamma})$ evaluated at $\theta=\theta^{\textup{m}}$. From \eqref{eq:IV_ME}, we have
\begin{align}
\eta^{\textup{m}}
=\left(\begin{array}{c}
\eta_{\alpha}^{\textup{m}} \\
\eta_{\beta}^{\textup{m}} \\
\eta_{\gamma}^{\textup{m}} 
\end{array}\right) \text { with } \quad 
\begin{aligned}
\eta_{\alpha}^{\textup{m}}
& = y_i- \alpha_{\textsc{iv,f}}^{\textup{m}} - \beta_{\textsc{iv}}^{\textup{m}} D_i - \gamma_{\textsc{iv,f}}^{\textup{m}} \ddot{x}_i, \\
\eta_{\beta}^{\textup{m}}
& = (y_i- \alpha_{\textsc{iv,f}}^{\textup{m}} - \beta_{\textsc{iv}}^{\textup{m}} D_i - \gamma_{\textsc{iv,f}}^{\textup{m}} \ddot{x}_i) Z_i, \\
\eta_{\gamma}^{\textup{m}}
&= (y_i- \alpha_{\textsc{iv,f}}^{\textup{m}} - \beta_{\textsc{iv}}^{\textup{m}} D_i - \gamma_{\textsc{iv,f}}^{\textup{m}} \ddot{x}_i) \ddot{x}_i. 
\end{aligned} \label{eq:IV_eta_mixed}
\end{align}
The $D_{\textsc{f}}^{\textup{m}}$ matrix equals
\begin{equation}
D_{\textsc{f}}^{\textup{m}} 
= \frac{1}{n} \sum_{i=1}^n \V
\left( \eta(y_i, Z_i, D_i, x_i; \theta^\textup{m}) \mid x_{i} \right)
=\left(\begin{array}{ccc}
d_{11} & d_{12} & d_{13}  \\
d_{12}^{\mathrm{T}} & d_{22} & d_{23}  \\
d_{13}^{\mathrm{T}} & d_{23}^{\mathrm{T}} & d_{33} 
\end{array}\right). \label{eq:IV_B_mixed}
\end{equation}
Therefore,
\begin{align}
d_{11}
=& \frac{1}{n} \sum_{i=1}^n \V
\left( \eta_{\beta}^{\textup{m}} \mid x_{i} \right) 
= \frac{1}{n} \sum_{i=1}^n 
\left( \mathbb{E}( (\eta_{\alpha}^{\textup{m}})^2 \mid x_{i} )  - \mathbb{E}(\eta_{\alpha}^{\textup{m}}\mid x_i)^2 \right), \label{eq:IV_D11_m} \\
d_{22}
=& \frac{1}{n} \sum_{i=1}^n \V
\left( \eta_{\beta}^{\textup{m}} \mid x_{i} \right) 
= \frac{1}{n} \sum_{i=1}^n 
\left( \mathbb{E}( (\eta_{\beta}^{\textup{m}})^2 \mid x_{i} )  - \mathbb{E}(\eta_{\beta}^{\textup{m}}\mid x_i)^2 \right),  \label{eq:IV_D22_m} \\
d_{12}
=& \frac{1}{n} \sum_{i=1}^n \mathbb{C}
\left( \eta_{\alpha}^{\textup{m}}, \eta_{\beta}^{\textup{m}} \mid x_{i} \right) 
= \frac{1}{n} \sum_{i=1}^n (\mathbb{E}
\left( \eta_{\alpha}^{\textup{m}} \eta_{\beta}^{\textup{m}} \mid x_{i} \right) 
- \mathbb{E}
\left( \eta_{\alpha}^{\textup{m}} \mid x_{i} \right) 
\mathbb{E}
\left( \eta_{\beta}^{\textup{m}} \mid x_{i} \right)) \notag  \\
=& \frac{1}{n} \sum_{i=1}^n (\mathbb{E}
\left( (\eta_{\beta}^{\textup{m}})^2 \mid x_{i} \right) 
- \mathbb{E}
\left( \eta_{\alpha}^{\textup{m}} \mid x_{i} \right) 
\mathbb{E}
\left( \eta_{\beta}^{\textup{m}}  \mid x_{i} \right)).
\end{align}
\noindent\textbf{Compute $A_{\textsc{F}}^{\textup{m}}$:}
With a slight abuse of notation, let $\frac{\partial}{\partial \alpha} \eta_{\alpha}^{\textup{m}}$ denote the value of $\frac{\partial}{\partial \alpha} \eta_\alpha$ evaluated at $\theta^{\textup{m}}$. Similarly define other partial derivatives.
From \eqref{eq:IV_ME}, we have
\begin{align}
A_{\textsc{f}}^{\textup{m}}
=& - \frac{1}{n} \sum_{i=1}^n \mathbb{E}\left(  \frac{\partial}{\partial (\alpha, \beta, \gamma)} \eta(y_i, Z_i, D_i, x_i; \theta^\textup{m}) \mid x_{i} \right) \notag \\
=& \left(\begin{array}{ccc}
1 & \frac{1}{n}\sum_{i=1}^n \mathbb{E}(D_i\mid x_i) & 0 \\
e & \frac{1}{n}\sum_{i=1}^n \mathbb{E}(Z_i D_i\mid x_i) & 0  \\
0 & \frac{1}{n}\sum_{i=1}^n \mathbb{E}(D_i\mid x_i)x_i & \frac{1}{n} \sum_{i=1}^n x_i x_i^{\mathrm{T}} 
\end{array}\right). \label{eq:IV_A}
\end{align}
By the formula of matrix inverse, we have
$$
(A_{\textsc{f}}^{\textup{m}})^{-1} 
= \left(\begin{array}{ccc}
\frac{\frac{1}{n}\sum_{i=1}^n\mathbb{E}(Z_iD_i\mid x_i)}{\frac{1}{n}\sum_{i=1}^n\mathbb{E}(Z_iD_i\mid x_i)-e\frac{1}{n}\sum_{i=1}^n\mathbb{E}(D_i\mid x_i)} & \frac{ -\frac{1}{n}\sum_{i=1}^n\mathbb{E}(D_i\mid x_i)}{\frac{1}{n}\sum_{i=1}^n\mathbb{E}(Z_iD_i\mid x_i)-e\frac{1}{n}\sum_{i=1}^n\mathbb{E}(D_i\mid x_i)} & 0 \\
\frac{-e}{\frac{1}{n}\sum_{i=1}^n\mathbb{E}(Z_iD_i\mid x_i)-e\frac{1}{n}\sum_{i=1}^n\mathbb{E}(D_i\mid x_i)} & \frac{1}{\frac{1}{n}\sum_{i=1}^n\mathbb{E}(Z_iD_i\mid x_i)-e\frac{1}{n}\sum_{i=1}^n\mathbb{E}(D_i\mid x_i)} & 0  \\
* & *  & (\frac{1}{n}\sum_{i=1}^n \ddot{x}_i \ddot{x}_i^{\mathrm{T}})^{-1}
\end{array}\right)
$$
where * denotes the terms that are irrelevant here.
By direct algebra,
\begin{align*}
\frac{1}{n}\sum_{i=1}^n\mathbb{E}(Z_iD_i\mid x_i)-e\frac{1}{n}\sum_{i=1}^n\mathbb{E}(D_i\mid x_i)
= e(1-e) \frac{1}{n}\sum_{i=1}^n (\pi_{D_i(1)} - \pi_{D_i(0)} ).
\end{align*}
\noindent\textbf{Compute $(A_{\textsc{F}}^{\textup{m}} )^{-1} D_{\textsc{F}}^{\textup{m}} \left(A_{\textsc{F}}^{\textup{m}-\mathrm{T}} \right)$:}
Direct algebra ensures
\begin{align}
& V_{\textsc{iv,f}}^{\textup{m}}  
= \left[ (A_{\textsc{f}}^{\textup{m}})^{-1}  D_{\textsc{f}}^{\textup{m}} \left(A_{\textsc{f}}^{\textup{m}-\mathrm{T}} \right) \right]_{(2,2)} 
= \frac{ d_{22} - e d_{12} - e (d_{12} - e d_{11}) }{e^2(1-e)^2 (\frac{1}{n}\sum_{i=1}^n (\pi_{D_i(1)} - \pi_{D_i(0)} ))^2 } \notag \\
=& \frac{ (1-2e) \frac{1}{n} \sum_{i=1}^n \mathbb{E}( (\eta_{\beta}^{\textup{m}})^2 \mid x_{i} ) + e^2 \frac{1}{n} \sum_{i=1}^n \mathbb{E}( (\eta_{\alpha}^{\textup{m}})^2 \mid x_{i} ) - \frac{1}{n} \sum_{i=1}^n (e  \mathbb{E}( (\eta_{\alpha}^{\textup{m}}) \mid x_{i} ) - \mathbb{E}( (\eta_{\beta}^{\textup{m}}) \mid x_{i} ))^2 }{e^2(1-e)^2 (\frac{1}{n}\sum_{i=1}^n (\pi_{D_i(1)} - \pi_{D_i(0)} ))^2 }. \notag
\end{align}
Take the subgroup $(Z_i=1,D_i=1)$ for example:
\begin{align*}
\frac{1}{n} \sum_{i=1}^n \mathbb{E}
\left( Z_iD_i (\eta_{\alpha}^{\textup{m}})^2 \mid x_{i} \right)
=& \frac{1}{n} \sum_{i=1}^n \mathbb{E}
\left( Z_iD_i \left(Y_i(1) - \ddot{x}_i^\mathrm{T} \gamma_{\textsc{iv}}^{\textup{m}} - \mu_i(1,1,x) + \mu_i(1,1,x) - \alpha_{\textsc{iv,f}}^{\textup{m}} - \beta_{\textsc{iv}}^{\textup{m}} \right)^2 \mid x_{i} \right) \\
=& \frac{1}{n} \sum_{i=1}^n e \pi_{D_i(1)}
(\sigma_i^2(1,1,x) + (\mu_i(1,1,x) - \alpha_{\textsc{iv,f}}^{\textup{m}} - \beta_{\textsc{iv}}^{\textup{m}})^2 ).
\end{align*}
So we have
\begin{align*}
\frac{1}{n} \sum_{i=1}^n \mathbb{E}
\left( (\eta_{\alpha}^{\textup{m}})^2 \mid x_{i} \right)
=& \frac{1}{n} \sum_{i=1}^n 
\left(
\begin{array}{l}
e \pi_{D_i(1)}
(\sigma_i^2(1,1,x) + (\mu_i(1,1,x) - \alpha_{\textsc{iv,f}}^{\textup{m}} - \beta_{\textsc{iv}}^{\textup{m}})^2 ) \\ 
+ e (1-\pi_{D_i(1)})
(\sigma_i^2(0,1,x) + (\mu_i(0,1,x) - \alpha_{\textsc{iv,f}}^{\textup{m}} )^2 ) \\ 
+ (1-e) \pi_{D_i(0)}
(\sigma_i^2(1,0,x) + (\mu_i(1,0,x) - \alpha_{\textsc{iv,f}}^{\textup{m}} - \beta_{\textsc{iv}}^{\textup{m}})^2 ) \\ 
+ (1-e) (1-\pi_{D_i(0)})
(\sigma_i^2(0,0,x) + (\mu_i(0,0,x) - \alpha_{\textsc{iv,f}}^{\textup{m}} )^2 )
\end{array}
\right) \\
\frac{1}{n} \sum_{i=1}^n \mathbb{E}
\left( (\eta_{\beta}^{\textup{m}})^2 \mid x_{i} \right)
=& \frac{1}{n} \sum_{i=1}^n 
\left(
\begin{array}{l}
e \pi_{D_i(1)}
(\sigma_i^2(1,1,x) + (\mu_i(1,1,x) - \alpha_{\textsc{iv,f}}^{\textup{m}} - \beta_{\textsc{iv}}^{\textup{m}})^2 ) \\ 
+ e (1-\pi_{D_i(1)})
(\sigma_i^2(0,1,x) + (\mu_i(0,1,x) - \alpha_{\textsc{iv,f}}^{\textup{m}} )^2 ) 
\end{array}
\right).
\end{align*}
Again, take the subgroup $(Z_i=1,D_i=1)$ for example:
\begin{align*}
\mathbb{E}
\left( Z_iD_i \eta_{\alpha}^{\textup{m}} \mid x_{i} \right)
=& 
\left( Z_iD_i \left(Y_i(1) - x_i^\mathrm{T} \gamma_{\textsc{iv,f}}^{\textup{m}} - \alpha_{\textsc{iv,f}}^{\textup{m}} - \beta_{\textsc{iv}}^{\textup{m}} \right) \mid x_{i} \right) 
= e \pi_{D_i(1)}
(\mu_i(1,1,x) - \alpha_{\textsc{iv,f}}^{\textup{m}} - \beta_{\textsc{iv}}^{\textup{m}}).
\end{align*}
So we have
\begin{align*}
e \mathbb{E}
\left( \eta_{\alpha}^{\textup{m}} \mid x_{i} \right)
- \mathbb{E}
\left( \eta_{\beta}^{\textup{m}} \mid x_{i} \right) 
=&  
\left(
\begin{array}{l}
+ e (1-e) \pi_{D_i(0)} (\mu_i(1,0,x) - \alpha_{\textsc{iv,f}}^{\textup{m}} - \beta_{\textsc{iv}}^{\textup{m}}) \\ 
+ e (1-e) (1-\pi_{D_i(0)}) (\mu_i(0,0,x) - \alpha_{\textsc{iv,f}}^{\textup{m}}) \\
- e (1-e) \pi_{D_i(1)} (\mu_i(1,1,x) - \alpha_{\textsc{iv,f}}^{\textup{m}} - \beta_{\textsc{iv}}^{\textup{m}}) \\ 
- e (1-e) (1-\pi_{D_i(1)}) (\mu_i(0,1,x) - \alpha_{\textsc{iv,f}}^{\textup{m}}) 
\end{array}
\right) \\
=& e (1-e) 
(\pi_{D_i(1)} - \pi_{D_i(0)})(\beta_{\textsc{iv}}^{\textup{m}} - \tau_{\textup{c}}(x_i)).
\end{align*}
Also, by direct algebra, 
\begin{align*}
& \mu_i(1,1,x) - \alpha_{\textsc{iv,f}}^{\textup{m}} - \beta_{\textsc{iv}}^{\textup{m}} \\
=& (1-\pi_{D_i(1)})(\mu_i(1,1,x)-\mu_i(0,1,x)-\beta_{\textsc{iv}}^{\textup{m}} ) 
+ \pi_{D_i(1)}\mu_i(1,1,x)
+ (1-\pi_{D_i(1)}) \mu_i(0,1,x)
-\pi_{D_i(1)} \beta_{\textsc{iv}}^{\textup{m}} \\
&- \frac{1}{n}\sum_{i=1}^n (\pi_{D_i(1)}\mu_i(1,1,x)
+ (1-\pi_{D_i(1)}) \mu_i(0,1,x) - \pi_{D_i(1)} \beta_{\textsc{iv}}^{\textup{m}}), \\
& \mu_i(0,1,x) - {\alpha}_{\textsc{iv,f}}^{\textup{m}}  \\
=& - \pi_{D_i(1)}(\mu_i(1,1,x)-\mu_i(0,1,x)-\beta_{\textsc{iv}}^{\textup{m}} )
+ (\pi_{D_i(1)}(\mu_i(1,1,x) 
+ (1-\pi_{D_i(1)})\mu_i(0,1,x)
- \pi_{D_i(1)} \beta_{\textsc{iv}}^{\textup{m}} )) \\
&- \frac{1}{n}\sum_{i=1}^n (\pi_{D_i(1)}(\mu_i(1,1,x) 
+ (1-\pi_{D_i(1)})\mu_i(0,1,x) - \pi_{D_i(1)} \beta_{\textsc{iv}}^{\textup{m}} )),  \\
& \mu_i(1,0,x) - {\alpha}_{\textsc{iv,f}}^{\textup{m}} - {\beta}_{\textsc{iv}}^{\textup{m}}  \\
=& (1-\pi_{D_i(0)}) (\mu_i(1,0,x) - \mu_i(0,0,x) - \beta_{\textsc{iv}}^{\textup{m}}) 
+ \pi_{D_i(1)}\mu_i(1,1,x) + (1-\pi_{D_i(1)})\mu_i(0,1,x)
-\pi_{D_i(1)} \beta_{\textsc{iv}}^{\textup{m}} \\
&- \frac{1}{n}\sum_{i=1}^n (\pi_{D_i(1)} \mu_i(1,1,x) + (1-\pi_{D_i(1)}) \mu_i(0,1,x) - \pi_{D_i(1)} \beta_{\textsc{iv}}^{\textup{m}}   ) 
- (\pi_{D_i(1)} - \pi_{D_i(0)}) 
(\tau_{\textup{c}}(x) - \beta_{\textsc{iv}}^{\textup{m}} ), \\
& \mu_i(0,0,x) - {\alpha}_{\textsc{iv,f}}^{\textup{m}}  \\
=& - \pi_{D_i(0)}(\mu_i(1,0,x)-\mu_i(0,0,x)-\beta_{\textsc{iv}}^{\textup{m}} ) 
+ \pi_{D_i(1)}\mu_i(1,1,x) + (1-\pi_{D_i(1)})\mu_i(0,1,x) - \pi_{D_i(1)} \beta_{\textsc{iv}}^{\textup{m}} \\
&- \frac{1}{n}\sum_{i=1}^n 
( \pi_{D_i(1)}\mu_i(1,1,x) + (1-\pi_{D_i(1)})\mu_i(0,1,x) - \pi_{D_i(1)} \beta_{\textsc{iv}}^{\textup{m}} ) 
-  (\pi_{D_i(1)} - \pi_{D_i(0)}) 
(\tau_{\textup{c}}(x) - \beta_{\textsc{iv}}^{\textup{m}}). 
\end{align*}
Therefore,
\begin{align}
& V_{\textsc{iv,f}}^{\textup{m}}  
= \frac{ (1-2e) \frac{1}{n} \sum_{i=1}^n \mathbb{E}( (\eta_{\beta}^{\textup{m}})^2 \mid x_{i} ) + e^2 \frac{1}{n} \sum_{i=1}^n \mathbb{E}( (\eta_{\alpha}^{\textup{m}})^2 \mid x_{i} ) - \frac{1}{n} \sum_{i=1}^n (e  \mathbb{E}( (\eta_{\alpha}^{\textup{m}}) \mid x_{i} ) - \mathbb{E}( (\eta_{\beta}^{\textup{m}}) \mid x_{i} ))^2 }{e^2(1-e)^2 (\frac{1}{n}\sum_{i=1}^n (\pi_{D_i(1)} - \pi_{D_i(0)} ))^2 } \notag \\
=& \frac{1}{(\frac{1}{n}\sum_{i=1}^n \pi_{\textup{c}}(x_i))^2} 
\frac{1}{n} \sum_{i=1}^n 
\left[
\begin{array}{l}
\frac{\pi_{D_i(1)}\sigma_i^2(1,1,x) + (1-\pi_{D_i(1)})\sigma_i^2(0,1,x)}{e}
+ \frac{\pi_{D_i(0)}\sigma_i^2(1,0,x) + (1-\pi_{D_i(0)})\sigma_i^2(0,0,x)}{1-e} \\
+ \frac{(\mu_i(1,1,x)-\mu_i(0,1,x)-\beta_{\textsc{iv}}^{\textup{m}} )^2 \pi_{D_i(1)} (1-\pi_{D_i(1)}) }{e} 
+ \frac{(\mu_i(1,0,x)-\mu_i(0,0,x)-\beta_{\textsc{iv}}^{\textup{m}} )^2 \pi_{D_i(0)} (1-\pi_{D_i(0)}) }{1-e} \\
+ \frac{1}{e} \left[
\begin{array}{l}
\pi_{D_i(1)}\mu_i(1,1,x)
+ (1-\pi_{D_i(1)}) \mu_i(0,1,x)
-\pi_{D_i(1)} \beta_{\textsc{iv}}^{\textup{m}} \\
- \frac{1}{n}\sum_{i=1}^n (\pi_{D_i(1)}\mu_i(1,1,x)
+ (1-\pi_{D_i(1)}) \mu_i(0,1,x) - \pi_{D_i(1)} \beta_{\textsc{iv}}^{\textup{m}})
\end{array}
\right]^2 \\
+ \frac{1}{1-e} \left[
\begin{array}{l}
\pi_{D_i(1)}\mu_i(1,1,x) + (1-\pi_{D_i(1)})\mu_i(0,1,x) - \pi_{D_i(1)} \beta_{\textsc{iv}}^{\textup{m}} \\
- \frac{1}{n}\sum_{i=1}^n 
( \pi_{D_i(1)}\mu_i(1,1,x) + (1-\pi_{D_i(1)})\mu_i(0,1,x) - \pi_{D_i(1)} \beta_{\textsc{iv}}^{\textup{m}} ) \\
-  (\pi_{D_i(1)} - \pi_{D_i(0)}) 
( \tau_{\textup{c}}(x_i) - \beta_{\textsc{iv}}^{\textup{m}} )
\end{array}
\right]^2 \\
+ 2e \pi_{D_i(1)} (1-\pi_{D_i(1)})(\mu_i(1,1,x)-\mu_i(0,1,x)-\beta_{\textsc{iv}}^{\textup{m}} )   \\
\times \left[
\begin{array}{l}
\pi_{D_i(1)}\mu_i(1,1,x)
+ (1-\pi_{D_i(1)}) \mu_i(0,1,x)
-\pi_{D_i(1)} \beta_{\textsc{iv}}^{\textup{m}} \\
- \frac{1}{n}\sum_{i=1}^n (\pi_{D_i(1)}\mu_i(1,1,x)
+ (1-\pi_{D_i(1)}) \mu_i(0,1,x) - \pi_{D_i(1)} \beta_{\textsc{iv}}^{\textup{m}})
\end{array}
\right] / e^2 \\
- 2e \pi_{D_i(1)} (1-\pi_{D_i(1)})(\mu_i(1,1,x)-\mu_i(0,1,x)-\beta_{\textsc{iv}}^{\textup{m}} )   \\
\times \left[
\begin{array}{l}
\pi_{D_i(1)}\mu_i(1,1,x)
+ (1-\pi_{D_i(1)}) \mu_i(0,1,x)
-\pi_{D_i(1)} \beta_{\textsc{iv}}^{\textup{m}} \\
- \frac{1}{n}\sum_{i=1}^n (\pi_{D_i(1)}\mu_i(1,1,x)
+ (1-\pi_{D_i(1)}) \mu_i(0,1,x) - \pi_{D_i(1)} \beta_{\textsc{iv}}^{\textup{m}})
\end{array}
\right] / e^2 \\
+ 2 (1-e) \pi_{D_i(0)} (1-\pi_{D_i(0)})(\mu_i(1,0,x)-\mu_i(0,0,x)-\beta_{\textsc{iv}}^{\textup{m}} )   \\
\times \left[
\begin{array}{l}
\pi_{D_i(1)}\mu_i(1,1,x)
+ (1-\pi_{D_i(1)}) \mu_i(0,1,x)
-\pi_{D_i(1)} \beta^{\textup{m}} \\
- \frac{1}{n}\sum_{i=1}^n (\pi_{D_i(1)}\mu_i(1,1,x)
+ (1-\pi_{D_i(1)}) \mu_i(0,1,x) - \pi_{D_i(1)} \beta_{\textsc{iv}}^{\textup{m}}) \\
-  (\pi_{D_i(1)} - \pi_{D_i(0)}) 
(\tau_{\textup{c}}(x_i) - \beta_{\textsc{iv}}^{\textup{m}})
\end{array}
\right] / (1-e)^2 \\
- 2 (1-e) \pi_{D_i(0)} (1-\pi_{D_i(0)})(\mu_i(1,0,x)-\mu_i(0,0,x)-\beta_{\textsc{iv}}^{\textup{m}} )   \\
\times \left[
\begin{array}{l}
\pi_{D_i(1)}\mu_i(1,1,x)
+ (1-\pi_{D_i(1)}) \mu_i(0,1,x)
-\pi_{D_i(1)} \beta_{\textsc{iv}}^{\textup{m}} \\
- \frac{1}{n}\sum_{i=1}^n (\pi_{D_i(1)}\mu_i(1,1,x)
+ (1-\pi_{D_i(1)}) \mu_i(0,1,x) - \pi_{D_i(1)} \beta_{\textsc{iv}}^{\textup{m}}) \\
-  (\pi_{D_i(1)} - \pi_{D_i(0)}) 
(\tau_{\textup{c}}(x_i) - \beta_{\textsc{iv}}^{\textup{m}})
\end{array}
\right] / (1-e)^2 \\
+ (\pi_{D_i(1)} - \pi_{D_i(0)})^2 ( \tau_{\textup{c}}(x_i) - \beta_{\textsc{iv}}^{\textup{m}} )^2
\end{array}
\right] \notag \\
=& \frac{1 }{(\frac{1}{n}\sum_{i=1}^n \pi_{\textup{c}}(x_i))^2} 
\frac{1}{n} \sum_{i=1}^n
\left[
\begin{array}{l}
\frac{\pi_{D_i(1)}\sigma_i^2(1,1,x) + (1-\pi_{D_i(1)})\sigma_i^2(0,1,x)}{e}
+ \frac{\pi_{D_i(0)}\sigma_i^2(1,0,x) + (1-\pi_{D_i(0)})\sigma_i^2(0,0,x)}{1-e} \\
+ \frac{(\mu_i(1,1,x)-\mu_i(0,1,x)-\beta_{\textsc{iv}}^{\textup{m}} )^2 \pi_{D_i(1)} (1-\pi_{D_i(1)}) }{e} 
+ \frac{(\mu_i(1,0,x)-\mu_i(0,0,x)-\beta_{\textsc{iv}}^{\textup{m}} )^2 \pi_{D_i(0)} (1-\pi_{D_i(0)}) }{1-e} \\
+ \frac{1}{e} \left[
\begin{array}{l}
\pi_{D_i(1)}\mu_i(1,1,x)
+ (1-\pi_{D_i(1)}) \mu_i(0,1,x)
-\pi_{D_i(1)} \beta_{\textsc{iv}}^{\textup{m}} \\
- \frac{1}{n}\sum_{i=1}^n (\pi_{D_i(1)}\mu_i(1,1,x)
+ (1-\pi_{D_i(1)}) \mu_i(0,1,x) - \pi_{D_i(1)} \beta_{\textsc{iv}}^{\textup{m}})
\end{array}
\right]^2 \\
+ \frac{1}{1-e} \left[
\begin{array}{l}
\pi_{D_i(1)}\mu_i(1,1,x) + (1-\pi_{D_i(1)})\mu_i(0,1,x) - \pi_{D_i(1)} \beta_{\textsc{iv}}^{\textup{m}} \\
- \frac{1}{n}\sum_{i=1}^n 
( \pi_{D_i(1)}\mu_i(1,1,x) + (1-\pi_{D_i(1)})\mu_i(0,1,x) - \pi_{D_i(1)} \beta_{\textsc{iv}}^{\textup{m}} ) \\
-  (\pi_{D_i(1)} - \pi_{D_i(0)}) 
(\tau_{\textup{c}}(x_i) - \beta_{\textsc{iv}}^{\textup{m}})
\end{array}
\right]^2 \\
+ (\pi_{D_i(1)} - \pi_{D_i(0)})^2 ( \tau_{\textup{c}}(x_i) - \beta_{\textsc{iv}}^{\textup{m}} )^2
\end{array}
\right]. 
\label{eq:EHW_IV_mixed}
\end{align}

\paragraph*{Conservativeness:} 
Theorem \ref{thm:ME_mixed} ensures
\[
\hat{V}_{\textsc{hw,iv,f}} 
=  V_{\textsc{iv,f}}^\textup{m}
+ B_{\textsc{iv,f}}^\textup{m}
+ o(1;\mathbb{P}_{(y,Z,D)\mid x})
\]
where $\hat{V}_{\textsc{hw,iv,f}} $ is defined in \eqref{eq:EHW_IV} and 
\begin{equation*}
B_{\textsc{iv,f}}^\textup{m}
= \left[ (A_{\textsc{f}}^{{\textup{m}}-1} ) \left( \frac{1}{n} \sum_{i=1}^n \mathbb{E}\left[ \eta(\theta^\textup{m}) \mid x_{i} \right] \mathbb{E}\left[ \eta(\theta^\textup{m}) \mid x_{i} \right]^{\mathrm{T}}  \right) \left(A_{\textsc{f}}^{{\textup{m}}-{\mathrm{T}}}\right) \right]_{(2,2)}.
\end{equation*}
We derive the ``middle'' part below:
\begin{align*}
\frac{1}{n} \sum_{i=1}^n \mathbb{E}\left[ \eta(\theta^\textup{m}) \mid x_{i} \right] \mathbb{E}\left[ \eta(\theta^\textup{m}) \mid x_{i} \right]^{\mathrm{T}}
=& \frac{1}{n} \sum_{i=1}^n 
\left(
\begin{array}{ccc}
\mathbb{E}(\eta_{\alpha}^{\textup{m}} \mid x_{i})^2 
& \mathbb{E}(\eta_{\alpha}^{\textup{m}} \mid x_{i})
\mathbb{E}(\eta_{\beta}^{\textup{m}} \mid x_{i}) 
& * \\
\mathbb{E}(\eta_{\alpha}^{\textup{m}} \mid x_{i}) \mathbb{E}(\eta_{\beta}^{\textup{m}} \mid x_{i}) 
& \mathbb{E}(\eta_{\beta}^{\textup{m}} \mid x_{i})^2 
& *  \\
* & * & *
\end{array}
\right)
\end{align*}
where * denotes the terms that are irrelevant. 
Therefore, by direct algebra, the asymptotic bias is 
\begin{align*}
B_{\textsc{iv,f}}^\textup{m}
=& \left[ (A_{\textsc{f}}^{{\textup{m}}-1} ) \left( \frac{1}{n} \sum_{i=1}^n \mathbb{E}\left[ \eta(\theta^\textup{m}) \mid x_{i} \right] \mathbb{E}\left[ \eta(\theta^\textup{m}) \mid x_{i} \right]^{\mathrm{T}}  \right) \left(A_{\textsc{f}}^{{\textup{m}}-\mathrm{T}}\right) \right]_{(2,2)} \notag \\
=& \frac{ \frac{1}{n}\sum_{i=1}^n (e \mathbb{E}(\eta_{\alpha}) - \mathbb{E}(\eta_{\beta}))^2  }{e^2(1-e)^2 (\frac{1}{n}\sum_{i=1}^n (\pi_{D_i(1)} - \pi_{D_i(0)} ))^2 } \notag \\
=& \frac{\frac{1}{n}\sum_{i=1}^n
(\pi_{D_i(1)} - \pi_{D_i(0)})^2( \tau_{\textup{c}}(x_i) - \beta_{\textsc{iv}}^{\textup{m}})^2}{ (\frac{1}{n}\sum_{i=1}^n (\pi_{D_i(1)} - \pi_{D_i(0)} ))^2 }. 
\label{eq:IV_mixed_bias}
\end{align*}
\end{proof}

\section{Proof of results in Section \ref{sec:cluster}}
\label{sec:Proof of cluster}
\subsection{M-estimation with clustered data}
\begin{proof}[Proof of Theorem \ref{thm:ME_random_cluster}]
The proof follows by applying Theorem \ref{thm:ME_random} with cluster average $\frac{M}{N}\psi(y_{ij}, x_{ij};\beta^\textup{r})$. So we omit it.

\end{proof}

\begin{proof}[proof of Theorem \ref{thm:ME_fixed_cluster}]
The proof follows by applying Theorem \ref{thm:ME_fixed} with cluster average $\frac{M}{N}\psi(y_{ij}, x_{ij};\beta^\textup{f})$. So we omit it.

\end{proof}

\begin{proof}[Proof of Theorem \ref{thm:ME_mixed_cluster}]
The proof follows by applying Theorem \ref{thm:ME_mixed} with cluster average $\frac{M}{N}\psi(y_{ij}, x_{ij};\beta^\textup{m})$. So we omit it.

\end{proof}

\subsection{Without covariates}
\begin{proof}[Proof of Theorem \ref{thm:cluster_random}]

The result follows from applying Theorem \ref{thm:ME_random_cluster} with $Z$-estimation $\psi(w_{ij}; a, b) = (1, Z_{ij})^{\mathrm{T}} (y_{ij} - a -b Z_{ij})$.
The proof is analogous to that of Theorem \ref{thm:cluster_add_r}, so we omit it.

\end{proof}

\subsection{Additive regression} \label{proof of Cluster2}
\begin{proof}[Proof of Theorem \ref{thm:cluster_add_r}]
Let $\hat{\gamma}_{\textsc{f}}$ denote the coefficient of ${x}_{ij}$ from the OLS fit in \eqref{eq:cluster_add}. 
\[
\hat{\beta}_{\textsc{f}}
= \frac{1}{n_{\mathcal{T}}} \sum_{ij} Z_i \left(y_{ij} - {x}_{ij}^{\mathrm{T}}\hat{\gamma}_{\textsc{f}} \right) 
- \frac{1}{n_{\mathcal{C}}} \sum_{ij} (1-Z_i) \left(y_{ij} - {x}_{ij}^{\mathrm{T}}\hat{\gamma}_{\textsc{f}} \right).
\]
Without loss of generality, we assume $\mu_x=0$.
Denote $(\hat{\alpha}_{\textsc{f}}, \hat{\beta}_{\textsc{f}}, \hat{\gamma}_{\textsc{f}})$ as the intercepts and coefficient vectors of $Z_i$ and $x_{ij}$ from the OLS fits of \eqref{eq:cluster_add}. The first-order conditions of \eqref{eq:cluster_add} ensure that $\theta=\left(\alpha, \beta, \gamma\right)=(\hat{\alpha}_{\textsc{f}}, \hat{\beta}_{\textsc{f}},\hat{\gamma}_{\textsc{f}})$ jointly solve
$$
0= N^{-1} \sum_{ij} \eta\left( y_{ij}, x_{ij}, Z_i; \theta\right)
= N^{-1} \sum_{ij} 
\eta \left(y_{ij}, x_{ij}, Z_i; \theta\right)
$$
where 
\begin{align}
\eta= \left(y_{ij} - \alpha - \beta Z_i - \gamma^{\mathrm{T}} x_{ij} \right)
\left(\begin{array}{c}
1 \\
Z_i \\
x_{ij}
\end{array}\right). 
\label{eq:Lin_ME_cluster}
\end{align}
Direct algebra ensures that $\theta=\theta^{\textup{r}}=\left(\alpha_{\textsc{f}}^{\textup{r}}, \beta_{\textsc{f}}^{\textup{r}}, \gamma_{\textsc{f}}^{\textup{r}}\right)$ solves $E\{\eta(\theta)\}=0$. Theorem \ref{thm:ME_random_cluster} ensures that 
\begin{align}
\sqrt{M}\left\{\left(\begin{array}{c}
\hat{\alpha}_{\textsc{f}} \\
\hat{\beta}_{\textsc{f}}  \\
\hat{\gamma}_{\textsc{f}}  
\end{array}\right)-\left(\begin{array}{c}
\alpha_{\textsc{f}}^{\textup{r}} \\
\beta_{\textsc{f}}^{\textup{r}} \\
\gamma_{\textsc{f}}^{\textup{r}} 
\end{array}\right)\right\} \overset{\textup{d}}{\to} 
\mathcal{N}\left(0, A_{\textsc{f}}^{\textup{r}-1}  D_{\textsc{f}}^{\textup{r}} \left(A_{\textsc{f}}^{\textup{r}} \right)^{-\mathrm{T}}\right),
\label{eq:Lin_normal_cluster_F}
\end{align}
where 
\begin{align*}
A_{\textsc{f}}^{\textup{r}}
=& - \frac{1}{M} \sum_{i=1}^M \mathbb{E}\left[ \sum_{j=1}^{n_i} \frac{\partial}{\partial (\alpha,\beta,\gamma)} \eta \right], \\
D_{\textsc{f}}^{\textup{r}}
=& \frac{1}{M} \sum_{i=1}^M \mathbb{E}\left[ \left( \sum_{j=1}^{n_i}\eta(w_{ij};\theta^\textup{r}) \right) \left( \sum_{j=1}^{n_i}\eta(w_{ij};\theta^\textup{r})\right)^{\mathrm{T}} \right]
\end{align*}
evaluated at $\theta=\theta^{\textup{r}}$. 
By direct algebra, $\beta_{\textsc{f}}^{\textup{r}} = \beta^{\textup{r}}$.
We compute below $A_{\textsc{f}}^{\textup{r}}$ and $D_{\textsc{f}}^{\textup{r}} $, respectively.

\paragraph*{Compute $A_{\textsc{F}}^{\textup{r}}$:}
By direct algebra,
\begin{align*}
A_{\textsc{f}}^{\textup{r}}
=& \frac{1}{M} \sum_{i=1}^M  \sum_{j=1}^{n_i} \mathbb{E}
\begin{pmatrix}
1 & Z_i & x_{ij}^{\mathrm{T}} \\
Z_i & Z_i & Z_ix_{ij}^{\mathrm{T}} \\
x_{ij} & Z_ix_{ij} & x_{ij}x_{ij}^{\mathrm{T}} 
\end{pmatrix} 
= \frac{N}{M} \begin{pmatrix}
1 & e & 0 \\
e & e & 0 \\
0 & 0 & \mathbb{E}(x_{ij}x_{ij}^{\mathrm{T}})
\end{pmatrix}
\end{align*}
and its inverse is 
\begin{align*}
(A_{\textsc{f}}^{\textup{r}})^{-1}
= \frac{M}{N} \begin{pmatrix}
\frac{e}{e(1-e)} & -\frac{e}{e(1-e)} & 0 \\
-\frac{e}{e(1-e)} & \frac{1}{e(1-e)} & 0 \\
0 & 0 & (\mathbb{E}(x_{ij}x_{ij}^{\mathrm{T}}))^{-1} 
\end{pmatrix}.
\end{align*}
\paragraph*{Compute $D_{\textsc{F}}^{\textup{r}}$:} 
Define $\varepsilon_{ij,\textsc{f}}^{\textup{r}} = y_{ij} - \alpha_{\textsc{f}}^{\textup{r}} - \beta^{\textup{r}} Z_i - x_{ij}^{\mathrm{T}} \gamma_{\textsc{f}}^{{\textup{r}}}$.
By direct algebra, 
\begin{align*}
& D_{\textsc{f}}^{\textup{r}}
= \frac{1}{M} \sum_{i=1}^M \mathbb{E}\left[ \left( \sum_{j=1}^{n_i}\eta( \theta^\textup{r}) \right) \left( \sum_{j=1}^{n_i}\eta( \theta^\textup{r})\right)^{\mathrm{T}} \right] 
= \begin{pmatrix}
d_{11} & d_{12} & d_{13} \\
d_{12}^{\mathrm{T}} & d_{22} & d_{23} \\
d_{13}^{\mathrm{T}} & d_{23}^{\mathrm{T}} & d_{33}
\end{pmatrix} 
\end{align*}
where 
\begin{align*}
d_{11}
=& \frac{1}{M} \sum_{i=1}^M \mathbb{E} \left[ \left( \sum_{j=1}^{n_i} \varepsilon_{ij,\textsc{f}}^\textup{r} \right)^2 \right],
\quad 
d_{12}
= \frac{1}{M} \sum_{i=1}^M \mathbb{E} \left[ \left( \sum_{j=1}^{n_i} \varepsilon_{ij,\textsc{f}}^\textup{r} \right)\left( \sum_{j=1}^{n_i} Z_i \varepsilon_{ij,\textsc{f}}^\textup{r} \right) \right], \\
d_{13}
=& \frac{1}{M} \sum_{i=1}^M \mathbb{E} \left[ \left( \sum_{j=1}^{n_i} \varepsilon_{ij,\textsc{f}}^\textup{r} \right)\left( \sum_{j=1}^{n_i} \varepsilon_{ij,\textsc{f}}^\textup{r} x_{ij}^{\mathrm{T}} \right) \right],
\quad 
d_{22}
= \frac{1}{M} \sum_{i=1}^M \mathbb{E} \left[ \left( \sum_{j=1}^{n_i} Z_i \varepsilon_{ij,\textsc{f}}^\textup{r}  \right)^2 \right] = d_{12}, \\
d_{23}
=& \frac{1}{M} \sum_{i=1}^M \mathbb{E} \left[ \left( \sum_{j=1}^{n_i} Z_i \varepsilon_{ij,\textsc{f}}^\textup{r}  \right) \left( \sum_{j=1}^{n_i} \varepsilon_{ij,\textsc{f}}^\textup{r} x_{ij}^{\mathrm{T}} \right) \right], 
\quad
d_{33}
= \frac{1}{M} \sum_{i=1}^M \mathbb{E} \left[ \left( \sum_{j=1}^{n_i} \varepsilon_{ij,\textsc{f}}^\textup{r} x_{ij} \right) \left( \sum_{j=1}^{n_i} \varepsilon_{ij,\textsc{f}}^\textup{r} x_{ij} \right)^{\mathrm{T}} \right].
\end{align*}

\paragraph*{Compute $A_{\textsc{F}}^{\textup{r}-1}  D_{\textsc{F}}^{\textup{r}} \left(A_{\textsc{F}}^{\textup{r}-\mathrm{T}} \right)$:}
That \eqref{eq:Lin_normal_cluster_F} implies that 
\begin{align}
\sqrt{M}(\hat{\beta}_{\textsc{f}} -  \beta^{\textup{r}} ) 
\overset{\textup{d}}{\to} 
\mathcal{N}\left(0, \left[A_{\textsc{f}}^{\textup{r}-1} D_{\textsc{f}}^{\textup{r}} \left(A_{\textsc{f}}^{\textup{r}-\mathrm{T}} \right)\right]_{(2,2)}\right).
\end{align}
By direct algebra,
\begin{align*}
& \left[A_{\textsc{f}}^{\textup{r}-1}  D_{\textsc{f}}^{\textup{r}} \left(A_{\textsc{f}}^{\textup{r}-\mathrm{T}} \right) \right]_{(2,2)} \\
=& \frac{M^2}{N^2} \left[ \begin{pmatrix}
\frac{e}{e(1-e)} & -\frac{e}{e(1-e)} & 0 \\
-\frac{e}{e(1-e)} & \frac{1}{e(1-e)} & 0 \\
0 & 0 & *
\end{pmatrix}
\begin{pmatrix}
d_{11} & d_{12} & d_{13} \\
d_{12} & d_{22} & d_{23} \\
d_{13}^{\mathrm{T}} & d_{23}^{\mathrm{T}} & d_{33}
\end{pmatrix} 
\begin{pmatrix}
\frac{e}{e(1-e)} & -\frac{e}{e(1-e)} & 0 \\
-\frac{e}{e(1-e)} & \frac{1}{e(1-e)} & 0 \\
0 & 0 & *
\end{pmatrix}
\right]_{(2,2)} \\
=& \frac{M^2}{N^2} \left[ \begin{pmatrix}
* & * & * \\
\frac{d_{12}^{\mathrm{T}} - e d_{11}}{e(1-e)} & \frac{d_{22} - e d_{12}}{e(1-e)} & \frac{d_{23} - e d_{13}}{e(1-e)} \\
* & * & *
\end{pmatrix} 
\begin{pmatrix}
\frac{e}{e(1-e)} & -\frac{e}{e(1-e)} & 0 \\
-\frac{e}{e(1-e)} & \frac{1}{e(1-e)} & 0 \\
0 & 0 & *
\end{pmatrix}
\right]_{(2,2)} \\
=&  \frac{M^2}{N^2} \frac{d_{22} - 2 e d_{12}  + e^2 d_{11}}{e^2(1-e)^2},
\end{align*}
where * denotes the terms that are not relevant here and 
\begin{align}
d_{11}
=& \frac{1}{M} \sum_{i=1}^M \mathbb{E} \left[ \left( \sum_{j=1}^{n_i} \varepsilon_{ij,\textsc{f}}^\textup{r} \right)^2 \right]
= e \frac{1}{M} \sum_{i=1}^M \mathbb{E} \left[ \left( \sum_{j=1}^{n_i} \varepsilon_{ij,\textsc{f}}^\textup{r}(1) \right)^2 \right]
+ (1-e) \frac{1}{M} \sum_{i=1}^M \mathbb{E} \left[ \left( \sum_{j=1}^{n_i} \varepsilon_{ij,\textsc{f}}^\textup{r}(0) \right)^2 \right], \\
d_{22}
=& \frac{1}{M} \sum_{i=1}^M \mathbb{E} \left[ Z_i \left( \sum_{j=1}^{n_i} \varepsilon_{ij,\textsc{f}}^\textup{r} \right)^2 \right]
= e \frac{1}{M} \sum_{i=1}^M \mathbb{E} \left[ \left( \sum_{j=1}^{n_i} \varepsilon_{ij,\textsc{f}}^\textup{r}(1) \right)^2 \right].
\end{align}
Therefore,
\begin{align*}
V^\textup{r}_{\textsc{f}}
= \left[A_{\textsc{f}}^{\textup{r}-1}  D_{\textsc{f}}^{\textup{r}} \left(A_{\textsc{f}}^{\textup{r}-1} \right)^{\mathrm{T}}\right]_{(2,2)}
=& \frac{1}{M} \sum_{i=1}^M \left( \frac{\V({\varepsilon}^{\textup{r}}_{i\cdot,\textsc{f}}(1))}{e} + \frac{\V({\varepsilon}^{\textup{r}}_{i\cdot,\textsc{f}}(0))}{1-e} \right).
\end{align*}
Theorem \ref{thm:ME_random_cluster} ensures that 
\[
\left[  \left( \frac{1}{M} \sum_{ij} \frac{\partial}{\partial \theta^{\mathrm{T}}} \eta( \hat{\theta}) \right)^{-1}
\left( \frac{1}{M} \sum_{i=1}^{M} \eta_{i}(\hat{\theta})^{\mathrm{T}} \eta_{i}(\hat{\theta}) \right)
\left( \frac{1}{M} \sum_{ij} \frac{\partial}{\partial \theta^{\mathrm{T}}} \eta( \hat{\theta}) \right)^{-1} \right]_{(2,2)} 
= V^\textup{r} + o(1;\mathbb{P}_{(y,x)}).
\]
By direct algebra,
\begin{align*}
- \frac{1}{M} \sum_{ij} \frac{\partial}{\partial \theta^{\mathrm{T}}} \eta(\hat{\theta})
=& \frac{1}{M}  \begin{pmatrix}
N & \sum_{ij} Z_i &  \sum_{ij}(x_{ij} - \bar{x})^{\mathrm{T}} \\
\sum_{ij} Z_i & \sum_{ij} Z_i & \sum_{ij} Z_i (x_{ij} - \bar{x})^{\mathrm{T}} \\
\sum_{ij}(x_{ij} - \bar{x}) &  \sum_{ij}Z_i(x_{ij} - \bar{x}) &  \sum_{ij} (x_{ij} - \bar{x}) (x_{ij} - \bar{x})^{\mathrm{T}}
\end{pmatrix} 
= \frac{1}{M} (X_{\textsc{f}}^{\mathrm{T}} X_{\textsc{f}})
\end{align*}
and 
\begin{align*}
\frac{1}{M} \sum_{i=1}^{M} \eta_{i}(\hat{\theta})^{\mathrm{T}} \eta_{i}(\hat{\theta})
=& \frac{1}{M} \sum_{i=1}^{M} X_{i,\textsc{f}}^{\mathrm{T}} \hat{U}_{i,\textsc{f}} X_{i,\textsc{f}}.
\end{align*}
Thus, we complete the proof by showing 
\begin{equation*}
\hat{V}_{\textsc{lz,f}}
=
\left[  \left( \frac{1}{M} \sum_{ij} \frac{\partial}{\partial \theta^{\mathrm{T}}} \eta(\hat{\theta}) \right)^{-1}
\left( \frac{1}{M} \sum_{i=1}^{M} \eta_{i}(\hat{\theta})^{\mathrm{T}} \eta_{i}(\hat{\theta}) \right)
\left( \frac{1}{M} \sum_{ij} \frac{\partial}{\partial \theta^{\mathrm{T}}} \eta(\hat{\theta}) \right)^{-1} \right]_{(2,2)}
\end{equation*}
where $\hat{V}_{\textsc{lz,f}}$ is defined in \eqref{eq:V_LZ_F}.

\end{proof}

\begin{proof}[Proof of Theorem \ref{thm:cluster_add_m}]
Without loss of generality, we center $x_{ij}$ around $\bar{x}$. 
Recall $(\hat{\alpha}_{\textsc{f}}, \hat{\beta}_{\textsc{f}}, \hat{\gamma}_{\textsc{f}})$ are the intercepts and coefficient vectors of $Z_i$ and $\left(x_{ij}-\bar{x}\right)$ from the OLS fits of \eqref{eq:cluster_add}. The first-order conditions of \eqref{eq:cluster_add} ensure that $\theta=\left(\alpha, \beta, \gamma\right)=(\hat{\alpha}_{\textsc{f}}, \hat{\beta}_{\textsc{f}},\hat{\gamma}_{\textsc{f}})$ jointly solve
$$
0= N^{-1} \sum_{ij} \eta\left( y_{ij}, x_{ij}, Z_i; \theta\right)
= N^{-1} \sum_{ij} 
\eta \left(y_{ij}, x_{ij}, Z_i; \theta\right)
$$
where 
\begin{align}
\eta= \left(y_{ij} - \alpha - \beta Z_i - \left(x_{ij}-\bar{x}\right)^{\mathrm{T}} \gamma  \right)
\left(\begin{array}{c}
1 \\
Z_i \\
x_{ij}-\bar{x}
\end{array}\right). 
\end{align}
Direct algebra ensures that $\theta=\theta^{\textup{m}}=\left(\alpha_{\textsc{f}}^{\textup{m}}, \beta_{\textsc{f}}^{\textup{m}}, \gamma_{\textsc{f}}^{\textup{m}}\right)$ solves $N^{-1} \sum_{ij} E\{\eta(\theta^{\textup{m}})\mid x_{ij}\}=0$. Theorem \ref{thm:ME_random_cluster} ensures that 
\begin{align}
\sqrt{M}\left\{\left(\begin{array}{c}
\hat{\alpha}_{\textsc{f}} \\
\hat{\beta}_{\textsc{f}}  \\
\hat{\gamma}_{\textsc{f}}  
\end{array}\right)-\left(\begin{array}{c}
\alpha_{\textsc{f}}^{\textup{m}} \\
\beta_{\textsc{f}}^{\textup{m}} \\
\gamma_{\textsc{f}}^{\textup{m}} 
\end{array}\right)\right\} \overset{\textup{d}}{\to} 
\mathcal{N}\left(0, A_{\textsc{f}}^{\textup{m}-1}  D_{\textsc{f}}^{\textup{m}} \left(A_{\textsc{f}}^{\textup{m}-1} \right)^{\mathrm{T}}\right), 
\label{eq:Lin_normal_cluster_mixed}
\end{align}
where 
\begin{align*}
A_{\textsc{f}}^{\textup{m}}
=& - \frac{1}{M} \sum_{i=1}^M \mathbb{E}\left[  \sum_{j=1}^{n_i} \frac{\partial}{\partial (\alpha,\beta,\gamma)} \eta \mid x_i \right], \\
D_{\textsc{f}}^{\textup{m}}
=& \frac{1}{M} \sum_{i=1}^M \mathbb{E}\left[ \left( \sum_{j=1}^{n_i}\eta(w_{ij};\theta^\textup{m}) \right) \left( \sum_{j=1}^{n_i}\eta(w_{ij};\theta^\textup{m})\right)^{\mathrm{T}} \mid X \right]
\end{align*}
evaluated at $\theta=\theta^{\textup{m}}$. Direct algebra ensures that $\beta_{\textsc{f}}^{\textup{m}} = \beta^{\textup{m}}$. We compute below $A_{\textsc{f}}^{\textup{m}}$ and $D_{\textsc{f}}^{\textup{m}}$, respectively.

\paragraph*{Compute $A_{\textsc{F}}^{\textup{m}}$:}
By direct algebra, we have
\begin{align*}
A_{\textsc{f}}^{\textup{m}}
=& \frac{1}{M} \sum_{i=1}^M \sum_{j=1}^{n_i} \mathbb{E}
\left[
\begin{pmatrix}
1 & Z_i & (x_{ij}-\bar{x})^{\mathrm{T}} \\
Z_i & Z_i & Z_i(x_{ij}-\bar{x})^{\mathrm{T}} \\
(x_{ij}-\bar{x}) & Z_i(x_{ij}-\bar{x}) & (x_{ij}-\bar{x})(x_{ij}-\bar{x})^{\mathrm{T}} 
\end{pmatrix} \mid X 
\right]\\
=& \frac{N}{M}  
\begin{pmatrix}
1 & e & 0 \\
e & e & 0 \\
0 & 0 & \frac{1}{N}\sum_{ij}(x_{ij}-\bar{x})(x_{ij}-\bar{x})^{\mathrm{T}}  
\end{pmatrix}
\end{align*}
and its inverse is 
\begin{align*}
(A_{\textsc{f}}^{\textup{m}})^{-1}
=& \frac{M}{N}  
\begin{pmatrix}
\frac{e}{e(1-e)} & -\frac{e}{e(1-e)} & 0 \\
-\frac{e}{e(1-e)} & \frac{1}{e(1-e)} & 0 \\
0 & 0 & (\frac{1}{N}\sum_{ij}(x_{ij}-\bar{x})(x_{ij}-\bar{x})^{\mathrm{T}})^{-1} 
\end{pmatrix}.
\end{align*}
\paragraph*{Compute $D_{\textsc{F}}^{\textup{m}}$:} 
Define $\varepsilon_{ij,\textsc{f}}^{\textup{m}} = y_{ij} - \alpha_{\textsc{f}}^{\textup{m}} - \beta^{\textup{m}} Z_i - \left(x_{ij}-\bar{x}\right)^{\mathrm{T}} \gamma_{\textsc{f}}^{{\textup{m}} } $.
By direct algebra, we have
\begin{align*}
& D_{\textsc{f}}^{\textup{m}}
= \frac{1}{M} \sum_{i=1}^M \V \left[ \left( \sum_{j=1}^{n_i}\eta( \theta^\textup{m}) \right) \mid x_i \right] 
= \begin{pmatrix}
d_{11} & d_{12} & d_{13} \\
d_{12}^{\mathrm{T}} & d_{22} & d_{23} \\
d_{13}^{\mathrm{T}} & d_{23}^{\mathrm{T}} & d_{33}
\end{pmatrix} 
\end{align*}
where 
\begin{align*}
d_{11}
=& \frac{1}{M} \sum_{i=1}^M \V \left[   \sum_{j=1}^{n_i} \varepsilon_{ij,\textsc{f}}^\textup{m}  \mid x_i \right],
\quad 
d_{12}
= \frac{1}{M} \sum_{i=1}^M \mathbb{C} \left( \sum_{j=1}^{n_i} \varepsilon_{ij,\textsc{f}}^\textup{m}, \sum_{j=1}^{n_i} Z_i \varepsilon_{ij,\textsc{f}}^\textup{m} \mid x_i \right), \\
d_{13}
=& \frac{1}{M} \sum_{i=1}^M \mathbb{C} \left( \sum_{j=1}^{n_i} \varepsilon_{ij,\textsc{f}}^\textup{m}, \sum_{j=1}^{n_i} \varepsilon_{ij,\textsc{f}}^\textup{m} (x_{ij}-\bar{x})^{\mathrm{T}}  \mid x_i \right),
\quad 
d_{22}
= \frac{1}{M} \sum_{i=1}^M \V \left[ \sum_{j=1}^{n_i} Z_i \varepsilon_{ij,\textsc{f}}^\textup{m} \mid x_i \right] , \\
d_{23}
=& \frac{1}{M} \sum_{i=1}^M \mathbb{C} \left( \sum_{j=1}^{n_i} Z_i \varepsilon_{ij,\textsc{f}}^\textup{m},  \sum_{j=1}^{n_i} \varepsilon_{ij,\textsc{f}}^\textup{m} (x_{ij}-\bar{x})^{\mathrm{T}} \right), 
\quad
d_{33}
= \frac{1}{M} \sum_{i=1}^M \V \left[  \sum_{j=1}^{n_i} \varepsilon_{ij,\textsc{f}}^\textup{m} (x_{ij}-\bar{x}) \mid x_i \right].
\end{align*}

\paragraph*{Compute $A_{\textsc{F}}^{\textup{m}-1} D_{\textsc{F}}^{\textup{m}} \left(A_{\textsc{F}}^{\textup{m}-\mathrm{T}} \right)$:}
That \eqref{eq:Lin_normal_cluster_mixed} implies that 
\begin{align}
\sqrt{n}(\hat{\beta}_{\textsc{f}} - \beta^{\textup{m}} ) 
\overset{\textup{d}}{\to} 
\mathcal{N}\left(0, \left[A_{\textsc{f}}^{\textup{m}-1} D_{\textsc{f}}^{\textup{m}} \left(A_{\textsc{f}}^{\textup{m}-\mathrm{T}} \right) \right]_{(2,2)}\right).
\end{align}
By direct algebra,
\begin{align*}
& \left[A_{\textsc{f}}^{\textup{m}-1} D_{\textsc{f}}^{\textup{m}} \left(A_{\textsc{f}}^{\textup{m}-1} \right)^{\mathrm{T}}\right]_{(2,2)} \\
=& \frac{M^2}{N^2} \left[ \begin{pmatrix}
\frac{e}{e(1-e)} & -\frac{e}{e(1-e)} & 0 \\
-\frac{e}{e(1-e)} & \frac{1}{e(1-e)} & 0 \\
0 & 0 & *
\end{pmatrix}
\begin{pmatrix}
d_{11} & d_{12} & d_{13} \\
d_{12}^{\mathrm{T}} & d_{22} & d_{23} \\
d_{13}^{\mathrm{T}} & d_{23}^{\mathrm{T}} & d_{33}
\end{pmatrix} 
\begin{pmatrix}
\frac{e}{e(1-e)} & -\frac{e}{e(1-e)} & 0 \\
-\frac{e}{e(1-e)} & \frac{1}{e(1-e)} & 0 \\
0 & 0 & *
\end{pmatrix}
\right]_{(2,2)} \\
=& \frac{M^2}{N^2} \left[ \begin{pmatrix}
* & * & * \\
\frac{d_{12} - e d_{11}}{e(1-e)} & \frac{d_{22} - e d_{12}}{e(1-e)} & \frac{d_{23} - e d_{13}}{e(1-e)} \\
* & * & *
\end{pmatrix} 
\begin{pmatrix}
\frac{e}{e(1-e)} & -\frac{e}{e(1-e)} & 0 \\
-\frac{e}{e(1-e)} & \frac{1}{e(1-e)} & 0 \\
0 & 0 & *
\end{pmatrix}
\right]_{(2,2)} \\
=& \frac{M^2}{N^2} \cdot \frac{d_{22} - 2 e d_{12}  + e^2 d_{11}}{e^2(1-e)^2},
\end{align*}
where
\begin{align*}
\frac{M^2}{N^2} d_{22}
=&  \frac{M}{N^2} \sum_{i=1}^M \V \left[ Z_i \sum_{j=1}^{n_i} \varepsilon_{ij,\textsc{f}}^\textup{m} \mid x_i \right] \\
=& \frac{M}{N^2} \sum_{i=1}^M \mathbb{E} \left[ Z_i \left( \sum_{j=1}^{n_i} \varepsilon_{ij,\textsc{f}}^\textup{m}(1) \right)^2 \mid x_i \right] 
- \frac{M}{N^2} \sum_{i=1}^M \left( \mathbb{E} \left[ Z_i  \sum_{j=1}^{n_i} \varepsilon_{ij,\textsc{f}}^\textup{m}(1) \mid x_i \right] \right)^2 \\
=& e \cdot \frac{1}{M} \sum_{i=1}^M \mathbb{E} \left[ \left( {\varepsilon}^{\textup{m}}_{i\cdot,\textsc{f}}(1) \right)^2 \mid x_i \right] 
- e^2 \cdot \frac{1}{M} \sum_{i=1}^M \mathbb{E} \left[ {\varepsilon}^{\textup{m}}_{i\cdot,\textsc{f}}(1) \mid x_i \right] ^2, \\
\frac{M^2}{N^2} d_{11}
=&  \frac{M}{N^2} \sum_{i=1}^M V \left[   \sum_{j=1}^{n_i} \varepsilon_{ij,\textsc{f}}^\textup{m}  \mid x_i \right]
= \frac{M}{N^2} \sum_{i=1}^M \mathbb{E} \left[ \left( \sum_{j=1}^{n_i} \varepsilon_{ij,\textsc{f}}^\textup{m} \right)^2 \mid x_i \right] 
- \frac{M}{N^2} \sum_{i=1}^M \left( \mathbb{E} \left[ \sum_{j=1}^{n_i} \varepsilon_{ij,\textsc{f}}^\textup{m} \mid x_i \right] \right)^2 \\
=& e \cdot \frac{1}{M} \sum_{i=1}^M \mathbb{E} \left[ \left( {\varepsilon}^{\textup{m}}_{i\cdot,\textsc{f}}(1) \right)^2  \mid x_i \right] 
+ (1-e) \cdot \frac{1}{M} \sum_{i=1}^M \mathbb{E} \left[ \left( {\varepsilon}^{\textup{m}}_{i\cdot,\textsc{f}}(0) \right)^2 \mid x_i \right] \\
-&  \frac{e^2 }{M} \sum_{i=1}^M \mathbb{E} \left[ {\varepsilon}^{\textup{m}}_{i\cdot,\textsc{f}}(1) \mid x_i \right] ^2
- \frac{(1-e)^2 }{M} \sum_{i=1}^M \mathbb{E} \left[ {\varepsilon}^{\textup{m}}_{i\cdot,\textsc{f}}(0) \mid x_i \right] ^2
- \frac{2 e(1-e)}{M} \sum_{i=1}^M \mathbb{E} \left[ {\varepsilon}^{\textup{m}}_{i\cdot,\textsc{f}}(1) \mid x_i \right] 
\mathbb{E} \left[ {\varepsilon}^{\textup{m}}_{i\cdot,\textsc{f}}(0) \mid x_i \right], \\
\frac{M^2}{N^2} d_{12}
=&  \frac{M}{N^2} \sum_{i=1}^M \mathbb{C} \left( \sum_{j=1}^{n_i} \varepsilon_{ij}^\textup{m}, \sum_{j=1}^{n_i} Z_i \varepsilon_{ij,\textsc{f}}^\textup{m} \mid x_i \right) \\
=& \frac{M}{N^2} \sum_{i=1}^M \mathbb{E} \left[ Z_i \left( \sum_{j=1}^{n_i} \varepsilon_{ij,\textsc{f}}^\textup{m} \right)^2 \mid x_i \right] 
- \frac{M}{N^2} \sum_{i=1}^M \mathbb{E} \left[ \sum_{j=1}^{n_i} \varepsilon_{ij,\textsc{f}}^\textup{m} \mid x_i \right]
\mathbb{E} \left[ Z_i  \sum_{j=1}^{n_i} \varepsilon_{ij,\textsc{f}}^\textup{m} \mid x_i \right] \\
=& \frac{e}{M} \sum_{i=1}^M \mathbb{E} \left[ \left( {\varepsilon}^{\textup{m}}_{i\cdot,\textsc{f}}(1) \right)^2 \mid x_i \right] 
- \frac{e(1-e) }{M} \sum_{i=1}^M \mathbb{E} \left[ {\varepsilon}^{\textup{m}}_{i\cdot,\textsc{f}}(1) \right] 
\mathbb{E} \left[ {\varepsilon}^{\textup{m}}_{i\cdot,\textsc{f}}(0) \mid x_i \right] 
- \frac{e^2 }{M} \sum_{i=1}^M \mathbb{E} \left[ {\varepsilon}^{\textup{m}}_{i\cdot,\textsc{f}}(1) \mid x_i \right]^2.
\end{align*}
Therefore,
\begin{align*}
V^\textup{m}_{\textsc{f}}
=& \left[A_{\textsc{f}}^{\textup{m}-1}  B_{\textsc{f}}^{\textup{m}} \left(A_{\textsc{f}}^{\textup{m}-1} \right)^{\mathrm{T}}\right]_{(2,2)} \\
=& \frac{1}{M} \sum_{i=1}^M \left( 
\frac{ \mathbb{E}\left( {\varepsilon}^{\textup{m}}_{i\cdot,\textsc{f}}(1)^2 \mid x_i \right)}{e}
+ \frac{\mathbb{E}\left( {\varepsilon}^{\textup{m}}_{i\cdot,\textsc{f}}(0)^2 \mid x_i \right)}{1-e}  - \mathbb{E}\left( {\varepsilon}^{\textup{m}}_{i\cdot,\textsc{f}}(1) - {\varepsilon}^{\textup{m}}_{i\cdot,\textsc{f}}(0) \mid x_i \right)^2 \right).
\end{align*}
Theorem \ref{thm:ME_mixed_cluster} ensures that $\hat{V}_{\textsc{lz,f}}
=  V_{\textsc{f}}^\textup{m}
+ B_{\textsc{f}}^\textup{m}
+ o(1;\mathbb{P}_{(y,Z)\mid x})$ where
\begin{equation}
B_{\textsc{f}}^\textup{m}
= \left[ (A_{\textsc{f}}^\textup{m})^{-1} \left( \frac{1}{M} \sum_{i=1}^M \mathbb{E}\left( \sum_{j=1}^{n_i} \eta(\theta^\textup{m}) \mid x_i \right) \mathbb{E}\left( \sum_{j=1}^{n_i} \eta(\theta^\textup{m}) \mid x_i \right)^{\mathrm{T}}  \right) (A_{\textsc{f}}^{\textup{m} -1})^{\mathrm{T}} \right]_{(2,2)}
\end{equation}
where
\begin{align*}
\mathbb{E}\left( \sum_{j=1}^{n_i} \eta(\theta^\textup{m}) \mid x_i \right) 
=& 
\begin{pmatrix}
\mathbb{E}( \sum_{j=1}^{n_i} \varepsilon_{ij,\textsc{f}}^\textup{m} \mid x_i ) \\
\mathbb{E}(\sum_{j=1}^{n_i}  Z_i \varepsilon_{ij,\textsc{f}}^\textup{m} \mid x_i ) \\
\mathbb{E}(\sum_{j=1}^{n_i} \varepsilon_{ij,\textsc{f}}^\textup{m} (x_{ij} - \bar{x}) \mid x_i ) 
\end{pmatrix}.
\end{align*} 
Therefore, 
\begin{align*}
& \frac{1}{M} \sum_{i=1}^M \mathbb{E}\left(  \sum_{j=1}^{n_i} \eta(\theta^\textup{m}) \mid x_i \right) \mathbb{E}\left( \sum_{j=1}^{n_i} \eta(\theta^\textup{m}) \mid x_i \right)^{\mathrm{T}} 
= \begin{pmatrix}
h_{11} & h_{12} & h_{13} \\
h_{12}^{\mathrm{T}} & h_{22} & h_{23} \\
h_{13}^{\mathrm{T}} & h_{23}^{\mathrm{T}} & h_{33}
\end{pmatrix} 
\end{align*}
where 
\begin{align*}
& \frac{M^2}{N^2} h_{11}
= \frac{M}{N^2} \sum_{i=1}^M \mathbb{E}\left(\sum_{j=1}^{n_i} \varepsilon_{ij,\textsc{f}}^\textup{m} \mid x_i \right)^2 \\
=&  \frac{e^2}{M} \sum_{i=1}^M \mathbb{E}\left( {\varepsilon}^{\textup{m}}_{i\cdot,\textsc{f}}(1) \mid x_i \right)^2
+  \frac{(1-e)^2 }{M} \sum_{i=1}^M \mathbb{E}\left( {\varepsilon}^{\textup{m}}_{i\cdot,\textsc{f}}(0) \mid x_i \right)^2
+  2 \frac{ e(1-e)}{M} \sum_{i=1}^M \mathbb{E}\left(  {\varepsilon}^{\textup{m}}_{i\cdot,\textsc{f}}(0) \mid x_i \right)\mathbb{E}\left(  {\varepsilon}^{\textup{m}}_{i\cdot,\textsc{f}}(1) \mid x_i \right),   \\
& \frac{M^2}{N^2} h_{12} 
= \frac{M}{N^2} \sum_{i=1}^M \mathbb{E}\left(\sum_{j=1}^{n_i} \varepsilon_{ij,\textsc{f}}^\textup{m} \mid x_i \right) \mathbb{E}\left(Z_i \sum_{j=1}^{n_i} \varepsilon_{ij,\textsc{f}}^\textup{m} \mid x_i \right) \\
=& e^2  \cdot \frac{1}{M} \sum_{i=1}^M \mathbb{E}\left( {\varepsilon}^{\textup{m}}_{i\cdot,\textsc{f}}(1) \mid x_i \right)^2
+ e(1-e) \cdot \frac{1}{M} \sum_{i=1}^M \mathbb{E}\left( {\varepsilon}^{\textup{m}}_{i\cdot,\textsc{f}}(0) \mid x_i \right)\mathbb{E}\left( {\varepsilon}^{\textup{m}}_{i\cdot,\textsc{f}}(1) \mid x_i \right), \\
& \frac{M^2}{N^2} h_{22}
= \frac{M}{N^2} \sum_{i=1}^M  \mathbb{E}\left(Z_i \sum_{j=1}^{n_i} \varepsilon_{ij,\textsc{f}}^\textup{m} \mid x_i \right)^2
= e^2  \cdot \frac{1}{M} \sum_{i=1}^M \mathbb{E}\left( {\varepsilon}^{\textup{m}}_{i\cdot,\textsc{f}}(1) \mid x_i \right)^2. 
\end{align*}
Therefore, we verify the asymptotic bias: 
\begin{align*}
& B_{\textsc{f}}^\textup{m}
= \frac{h_{22} - 2 e h_{12}  + e^2 h_{11}}{e^2(1-e)^2}
= \frac{1}{M} \sum_{i=1}^M \mathbb{E}\left( {\varepsilon}^{\textup{m}}_{i\cdot,\textsc{f}}(1) - {\varepsilon}^{\textup{m}}_{i\cdot,\textsc{f}}(0) \mid x_i \right)^2.
\end{align*}

\end{proof}

\subsection{Fully-interacted regression} \label{proof of Cluster3}

\begin{lemma}\label{lemma:aij}
Under Assumption \ref{asu:cluster}, if $(a_{ij})_{1\le i \le M, 1\le j\le n_i}$ satisfies $n^{-1}\sum_{ij} \mathbb{E}(a_{ij}^2) = O(1)$ under probability measure $\mathbb{P}$ and $Z_i \indep a_{ij}$, then 
\[
\frac{1}{N} \sum_{ij } Z_i a_{ij} -  \frac{1}{N}\sum_{ij} e \mathbb{E}(a_{ij}) 
= O_{\mathbb{P}}(\Omega^{1/2}).
\]
If further $\Omega = o(1)$, then 
\[
\frac{1}{n_{\mathcal{T}}} \sum_{ij } Z_i a_{ij} - \frac{1}{N} \sum_{ij} \mathbb{E}(a_{ij}) = O_{\mathbb{P}}(\Omega^{1/2}).
\]
\end{lemma}

\begin{proof}[Proof of Lemma \ref{lemma:aij}]
First, $\mathbb{E}(n^{-1} \sum_{ij } Z_i a_{ij}) 
=  N^{-1} e \sum_{ij} \mathbb{E}(a_{ij})$. Then we verify the variance:
\begin{align*}
& \mathrm{var}\left( \frac{1}{N} \sum_{ij } Z_i a_{ij} \right) 
= \frac{(eM)^2}{N^2} 
\mathrm{var}\left( \frac{1}{eM} \sum_{i=1}^M Z_i \sum_{j=1}^{n_i} a_{ij} \right) 
\le \frac{(eM)^2}{N^2} \frac{\sum_{i=1}^M \mathbb{E}[(\sum_{j=1}^{n_i} a_{ij})^2] }{eM(M-1)} \\
&\le \frac{eM}{(M-1) N^2} \sum_{i=1}^M n_i \sum_{j=1}^{n_i} \mathbb{E} \left[ a_{ij}^2 \right] 
\le \frac{eM}{M-1} \Omega \left( \frac{1}{N} \sum_{ij} \mathbb{E} \left[ a_{ij}^2 \right]  \right)
= O_{\mathbb{P}}(1) \Omega O_{\mathbb{P}}(1) 
= O_{\mathbb{P}}(\Omega).
\end{align*}
When $a_{ij} = 1$ for all $(i,j)$, then it implies that $n_{\mathcal{T}}/N - e = O_{\mathbb{P}}(\Omega^{1/2})$. 
If we further assume $\Omega = o_{\mathbb{P}}(1)$, then 
\begin{align*}
\frac{1}{n_{\mathcal{T}}} \sum_{ij} Z_i a_{ij}
&= \frac{1}{N} \sum_{ij} Z_i a_{ij} \frac{N}{n_{\mathcal{T}}}
= \left( \frac{1}{N} \sum_{ij} \mathbb{E}(a_{ij}) e + O_{\mathbb{P}}(\Omega^{1/2}) \right)
\left( \frac{1}{e} + O_{\mathbb{P}}(\Omega^{1/2}) \right) \\
&= \frac{1}{N} \sum_{ij} \mathbb{E}(a_{ij}) + O_{\mathbb{P}}(\Omega^{1/2}) + O_{\mathbb{P}}(\Omega) 
= \frac{1}{N} \sum_{ij} \mathbb{E}(a_{ij}) + O_{\mathbb{P}}(\Omega^{1/2}).
\end{align*}
\end{proof}

\begin{lemma}\label{lemma:inverse}
If $\Delta=O_{\mathbb{P}}(\mu)$ with $\mu=o(1)$, and $\Lambda$ converges in probability to a finite and invertible matrix, then $(\Lambda+\Delta)^{-1}-\Lambda^{-1}=O_{\mathbb{P}}(\mu)$.
\end{lemma}
Lemma \ref{lemma:inverse} is Lemma A5 in \cite{SuDing2021}, which is useful for deriving the probability limit of the inverse of a matrix.

Let $\hat{\gamma}_1$ and $\hat{\gamma}_0$ denote the coefficient vectors of $x_{ij}-\bar{x}$ from the OLS fits of
\begin{align}
y_{ij} \sim 1+\left(x_{ij}-\bar{x}\right) & \text { over }\left\{(i,j): Z_{ij}=1\right\}  \label{eq:Lin_1_cluster} \\
y_{ij} \sim 1+\left(x_{ij}-\bar{x}\right) & \text { over }\left\{(i,j): Z_{ij}=0\right\} , \label{eq:Lin_0_cluster}
\end{align}
respectively. Let $\hat{Y}_{\textsc{l}}(1)$ and $\hat{Y}_{\textsc{l}}(0)$ denote the intercepts from \eqref{eq:Lin_1_cluster} and \eqref{eq:Lin_0_cluster}, respectively.

\begin{lemma}\label{lemma:Lin_beta_cluster}
$\hat{\beta}_{\textsc{l}} = \hat{Y}_{\textsc{l}}(1)-\hat{Y}_{\textsc{l}}(0)$ with
\begin{align*}
\hat{Y}_{\textsc{l}}(1)
=& \frac{\sum_{ij} Z_i \cdot\left(Y_{ij}(1)-\left(x_{ij}-\bar{x}\right)^{\mathrm{T}} \hat{\gamma}_1\right)}{\sum_{ij} Z_{ij}}, \\
\hat{Y}_{\textsc{l}}(0)
=& \frac{\sum_{ij} \left(1-Z_i\right) \cdot\left(Y_{ij}(0)-\left(x_{ij}-\bar{x}\right)^{\mathrm{T}} \hat{\gamma}_0\right)}{\sum_{ij} \left(1-Z_{ij}\right)} .
\end{align*}
\end{lemma}
\begin{proof}[Proof of Lemma \ref{lemma:Lin_beta_cluster}]
That $\hat{\beta}_{\textsc{l}} = \hat{Y}_{\textsc{l}}(1)-\hat{Y}_{\textsc{l}}(0)$ follows from properties of least squares. To verify the explicit form of $\hat{Y}_{\textsc{l}}(1)$ and $\hat{Y}_{\textsc{l}}(0)$, observe that the residual from \eqref{eq:Lin_1_cluster} equals $Y_{ij}(1)-\left(x_{ij}-\bar{x}\right)^{\mathrm{T}} \hat{\gamma}_1-\hat{Y}_{\textsc{l}}(1)$ for units with $Z_{ij}=1$. The first-order condition ensures
$$
\sum_{ij} Z_{ij} \cdot\left(Y_{ij}(1)-\left(x_{ij}-\bar{x}\right)^{\mathrm{T}} \hat{\gamma}_1-\hat{Y}_{\textsc{l}}(1)\right)=0
$$
This verifies the expression of $\hat{Y}_{\textsc{l}}(1)$. The expression of $\hat{Y}_{\textsc{l}}(0)$ follows by symmetry.

\end{proof}

\begin{lemma} \label{lemma:QIz}
If $\Omega = o(1)$, then $\hat{\gamma}_1 - \gamma^v_z = O_{\mathbb{P}}(\Omega^{1/2})$ for $z \in \{0, 1\}$ and $v \in \{\textup{r}, \textup{m}\}$. 
\end{lemma}

\begin{proof}
See proof in Section \ref{sec:proof_lemma}.
\end{proof}

\begin{lemma} \label{lemma:cluster_random}
Let $r_{ij}^\textup{r}(z) = \varepsilon^{\textup{m}}_{ij}(z) - \ddot{x}_{i j}^{\mathrm{T}} \gamma^{\textup{m}}_z$. 
If $\Omega = o(M^{-2/3})$, then 
\begin{align}
& \frac{M}{N^{2}} \sum_{i=1}^{M} Z_{i}\left(\sum_{j=1}^{n_{i}} \hat{\varepsilon}_{i j,\textsc{l}}\right)^{2}
- e \frac{1}{M} \sum_{i=1}^{M} \mathbb{E}\left( {\varepsilon}^{\textup{m}}_{i\cdot,\textsc{l}}(1)^2 \mid x_i \right)
= o_{\mathbb{P}}(1),
\label{Equ: A1_random} \\
& \frac{M}{N^{2}} \sum_{i=1}^{M} Z_{i} \sum_{j=1}^{n_{i}} \hat{\varepsilon}_{i j,\textsc{l}} \sum_{j=1}^{n_{i}} \hat{\varepsilon}_{i j,\textsc{l}} \ddot{x}_{i j}
=O_{\mathbb{P}}(M \Omega), 
\label{Equ: A2_random} \\
& \frac{M}{N^{2}} \sum_{i=1}^{M} Z_{i} \sum_{j=1}^{n_{i}} \hat{\varepsilon}_{i j,\textsc{l}} \ddot{x}_{i j} \sum_{j=1}^{n_{i}} \hat{\varepsilon}_{i j,\textsc{l}} \ddot{x}_{i j}^{\mathrm{T}}
=O_{\mathbb{P}}(M \Omega). 
\label{Equ: A3_random} 
\end{align}
\end{lemma}

\begin{proof}
See proof in Section \ref{sec:proof_lemma}.
\end{proof}

\begin{proof}[Proof of Theorem \ref{thm:cluster_adj_r}]
Recall from Lemma \ref{lemma:Lin_beta_cluster} that $(\hat{Y}_{\textsc{l}}(1), \hat{\gamma}_1)$ and $(\hat{Y}_{\textsc{l}}(0), \hat{\gamma}_0)$ are the intercepts and coefficient vectors of $\left(x_i-\bar{x}\right)$ from the OLS fits of \eqref{eq:Lin_1_cluster} and \eqref{eq:Lin_0_cluster}, respectively, with $\hat{\beta}_{\textsc{l}} =\hat{Y}_{\textsc{l}}(1)-\hat{Y}_{\textsc{l}}(0)$. The first-order conditions of \eqref{eq:Lin_1_cluster} ensure that $\left(\mu_1, \gamma_1\right)=(\hat{Y}_{\textsc{l}}(1), \hat{\gamma}_1)$ solve
$$
0=\sum_{ij} Z_i \cdot\left(Y_{ij}(1)-\left(x_{ij}-\bar{x}\right)^{\mathrm{T}} \gamma_1-\mu_1\right)\left(\begin{array}{c}
1 \\
x_{ij}-\bar{x}
\end{array}\right).
$$
This ensures that $\theta=\left(\mu_1, \gamma_1, \mu_0, \gamma_0, \mu_x\right)= \left(\hat{Y}_{\textsc{l}}(1), \hat{\gamma}_1, \hat{Y}_{\textsc{l}}(0), \hat{\gamma}_0, \bar{x}\right)$ jointly solves
$$
0= N^{-1} \sum_{ij} \eta\left(Y_{ij}(1), Y_{ij}(0), x_{ij}, Z_i; \theta\right)
= N^{-1} \sum_{ij} \left(\begin{array}{c}
\eta_1\left(Y_{ij}(1), x_{ij}, Z_i; \mu_1, \gamma_1, \mu_x\right) \\
\eta_0\left(Y_{ij}(0), x_{ij}, Z_i; \mu_0, \gamma_0, \mu_x\right) \\
\eta_x\left(x_{ij} ; \mu_x\right) 
\end{array}\right)
$$
where 
\begin{align}
\eta=
\left(\begin{array}{c}
\eta_1 \\
\eta_0 \\
\eta_x 
\end{array}\right) 
\text { with } 
\quad \begin{aligned}
\eta_1 
& =Z_i \cdot\left\{y_{ij}-\left(x_{ij}-\mu_x\right)^{\mathrm{T}} \gamma_1-\mu_1\right\}
\left(\begin{array}{c}
1 \\
x_{ij}-\mu_x
\end{array}\right), \\
\eta_0 
& =(1-Z_i)  \cdot\left\{y_{ij}-\left(x_{ij}-\mu_x\right)^{\mathrm{T}} \gamma_0-\mu_0\right\}
\left(\begin{array}{c}
1 \\
x_{ij}-\mu_x
\end{array}\right), \\
\eta_x 
&= x_{ij}-\mu_x.
\end{aligned} \label{eq:Lin_ME_cluster}
\end{align}
Direct algebra ensures that $\theta=\theta^{\textup{r}}=\left(\mu_1^{\textup{r}}, \gamma_1^{\textup{r}}, \mu_0^{\textup{r}}, \gamma_0^{\textup{r}},\mu_x\right)$ solves $\mathbb E\{\eta(\theta)\}=0$.

Theorem \ref{thm:ME_random_cluster} ensures that 
\begin{align}
\sqrt{M}\left\{\left(\begin{array}{c}
\hat{Y}_{\textsc{l}}(1) \\
\hat{\gamma}_1 \\
\hat{Y}_{\textsc{l}}(0) \\
\hat{\gamma}_0 \\
\bar{x}
\end{array}\right)-\left(\begin{array}{c}
\mu_1^{\textup{r}} \\
\gamma_1^{\textup{r}} \\
\mu_0^{\textup{r}} \\
\gamma_0^{\textup{r}} \\
\mu_x
\end{array}\right)\right\} \overset{\textup{d}}{\to} 
\mathcal{N}\left(0, A_{\textsc{l}}^{\textup{r}-1}  D_{\textsc{l}}^{\textup{r}} \left(A_{\textsc{l}}^{\textup{r}-\mathrm{T}} \right)\right), \label{eq:Lin_normal_cluster_L}
\end{align}
where
\begin{align*}
A_{\textsc{l}}^{\textup{r}}
=& - \frac{1}{M} \sum_{i=1}^M \mathbb{E}\left[  \sum_{j=1}^{n_i} \frac{\partial}{\partial (\mu_1,\gamma_1,\mu_0,\gamma_0, \mu_x)} \eta \right] \\
D_{\textsc{l}}^{\textup{r}}
=& \frac{1}{M} \sum_{i=1}^M \mathbb{E}\left[ \left(  \sum_{j=1}^{n_i}\eta(\theta^\textup{r}) \right) \left( \sum_{j=1}^{n_i}\eta(\theta^\textup{r})\right)^{\mathrm{T}} \right]
\end{align*}
evaluated at $\theta=\theta^{\textup{r}}$. 
We compute below $A_{\textsc{l}}^{\textup{r}}$ and $D_{\textsc{l}}^{\textup{r}} $, respectively.

\noindent\textbf{Compute $A_{\textsc{L}}^{\textup{r}} $:}
With a slight abuse of notation, let $\frac{\partial}{\partial \mu_1} \eta_1^{\textup{r}}$ denote the value of $\frac{\partial}{\partial \mu_1} \eta_1$ evaluated at $\theta^{\textup{r}}$. Similarly define other partial derivatives. Observe that
\begin{align*}
\frac{\partial}{\partial \mu_x^{\mathrm{T}}} \eta_1^{\textup{r}} 
=& Z_i \left(\begin{array}{c}
\gamma_1^{\mathrm{T}} \\
-\left\{Y_{ij}(1)-\left(x_{ij}-\mu_x\right)^{\mathrm{T}} \gamma_1-\mu_1\right\} I_J+\left(x_{ij}-\mu_x\right) \gamma_1^{\mathrm{T}}
\end{array}\right). 
\end{align*}
From \eqref{eq:Lin_ME_cluster}, we have
\[
\left.\frac{\partial}{\partial\left(\mu_1, \gamma_1, \mu_0, \gamma_0, \mu_x\right)} \eta\right|_{\theta=\theta^{\textup{r}}}
=\left(\begin{array}{ccc}
\frac{\partial}{\partial\left(\mu_1^{\mathrm{T}}, \gamma_1^{\mathrm{T}}\right)} \eta_1^{\textup{r}} & 0 & \frac{\partial}{\partial \mu_x^{\mathrm{T}}} \eta_1^{\textup{r}} \\
0 & \frac{\partial}{\partial\left(\mu_0^{\mathrm{T}}, \gamma_0^{\mathrm{T}}\right)} \eta_0^{\textup{r}} & \frac{\partial}{\partial \mu_x^{\mathrm{T}}} \eta_0^{\textup{r}} \\
0 & 0 & \frac{\partial}{\partial \mu_x^{\mathrm{T}}} \eta_x^{\textup{r}} 
\end{array}\right)
\]
where
\begin{align*}
\frac{\partial}{\partial\left(\mu_1^{\mathrm{T}}, \gamma_1^{\mathrm{T}}\right)} \eta_1^{\textup{r}}
& =-Z_i \cdot\left(\begin{array}{c}
1 \\
x_{ij}-\mu_x
\end{array}\right)\left(1,\left(x_{ij}-\mu_x\right)^{\mathrm{T}}\right), \\
\frac{\partial}{\partial \mu_x^{\mathrm{T}}} \eta_1^{\textup{r}}
& = Z_i \cdot\left(\begin{array}{c}
\gamma_1^{{\textup{r}} \mathrm{~T}} \\
-\left\{Y_{ij}(1)-\left(x_{ij}-\mu_x^*\right)^{\mathrm{T}} \gamma_1^{\textup{r}}-\mu_1^{\textup{r}}\right\} I_J+\left(x_{ij}-\mu_x\right) (\gamma_1^{{\textup{r}} })^{\mathrm{T}}
\end{array}\right), \\
\frac{\partial}{\partial \mu_x^{\mathrm{T}}} \eta_x^{\textup{r}}
& =-I_J . 
\end{align*}
Accordingly, we have
\begin{align*}
A_{\textsc{L}}^{\textup{r}} 
= - \frac{1}{M} \sum_{i=1}^M \mathbb{E}\left[ \frac{M}{N} \sum_{j=1}^{n_i} \frac{\partial}{\partial (\mu_1,\gamma_1,\mu_0,\gamma_0, \mu_x)} \eta \right] 
=\left(\begin{array}{cc|cc|c}
g_{11} & g_{12} & 0 & 0 & g_{15}   \\
g_{21} & g_{22} & 0 & 0 & g_{25}   \\
\hline 0 & 0 & g_{33} & g_{34} & g_{35}  \\
0 & 0 & g_{43} & g_{44} & g_{45}  \\
\hline 0 & 0 & 0 & 0 & g_{55} 
\end{array}\right),
\end{align*}
where
\begin{align*}
& \left(\begin{array}{ll}
g_{11} & g_{12} \\
g_{21} & g_{22}
\end{array}\right)
= - \frac{1}{M} \sum_{i=1}^M \mathbb{E}\left[ \frac{M}{N} \sum_{j=1}^{n_i} \frac{\partial}{\partial (\mu_1,\gamma_1)} \eta_1^{\textup{r}} \right] \\
& = \frac{1}{N} \sum_{i=1}^M \sum_{j=1}^{n_i} 
\mathbb{E}\left\{Z_i \left(\begin{array}{c}
1 \\
x_{ij}-\mu_x
\end{array}\right)
\left(1,\left(x_{ij}-\mu_x \right)^{\mathrm{T}}\right)\right\}
=e \cdot\left(\begin{array}{cc}
1 & 0 \\
0 & \operatorname{cov}\left(x_{ij}\right)
\end{array}\right) \text {, } \\
& \left(\begin{array}{l}
g_{15} \\
g_{25}
\end{array}\right)
=- \frac{1}{N} \sum_{i=1}^M \sum_{j=1}^{n_i} \mathbb{E}\left(\frac{\partial}{\partial \mu_x^{\mathrm{T}}} \eta_1^{\textup{r}} \right) \\
& =- \frac{1}{N} \sum_{i=1}^M \sum_{j=1}^{n_i} 
\mathbb{E}\left\{Z_i \left(\begin{array}{c}
(\gamma_1^{{\textup{r}} })^{\mathrm{T}} \\
-\left\{Y_{ij}(1)-\left(x_{ij}-\mu_x\right)^{\mathrm{T}} \gamma_1^{\textup{r}}-\mu_1^{\textup{r}}\right\} I_J + \left(x_{ij}-\mu_x\right) (\gamma_1^{{\textup{r}} })^{\mathrm{T}}
\end{array}\right)\right\} 
=-e \cdot\left(\begin{array}{c}
(\gamma_1^{{\textup{r}} })^{\mathrm{T}} \\
0_{J \times J}
\end{array}\right) \text {, } \\
& \left(\begin{array}{ll}
g_{33} & g_{34} \\
g_{43} & g_{44}
\end{array}\right)
=(1-e) \cdot\left(\begin{array}{cc}
1 & 0 \\
0 & \operatorname{cov}\left(x_{ij}\right)
\end{array}\right), 
\quad
\left(\begin{array}{l}
g_{35} \\
g_{45}
\end{array}\right)=
-(1-e) \cdot\left(\begin{array}{c}
(\gamma_0^{{\textup{r}} })^{\mathrm{T}} \\
0_{J \times J}
\end{array}\right), 
\end{align*}
and
$$
g_{55} = - \frac{1}{N} \sum_{i=1}^M \sum_{j=1}^{n_i} \mathbb{E}\left(\frac{\partial}{\partial \mu_x^{\mathrm{T}}} \eta_x^{\textup{r}}\right)
= I_J.
$$
This ensures
$$
A_{\textsc{l}}^{\textup{r}}
=\left(\begin{array}{cccc|c}
g_{11} & 0 & 0 & 0 & g_{15}   \\
0 & g_{22} & 0 & 0 & 0   \\
0 & 0 & g_{33} & 0 & g_{35}   \\
0 & 0 & 0 & g_{44} & 0   \\
\hline 0 & 0 & 0 & 0 & I 
\end{array}\right)=\left(\begin{array}{c|c}
G_{11} & G_{12} \\
\hline 0 & G_{22}
\end{array}\right),
$$
where $g_{11}=e, g_{93}=1-e, g_{15}=-e \gamma_1^\textup{r}$, and $g_{35}=-(1-e) (\gamma_0^{\textup{r}})^\mathrm{T}$. By the formula of block matrix inverse, we have
$$
A_{\textsc{l}}^{\textup{r} -1}
=\left(\begin{array}{cc}
G_{11}^{-1} & -G_{11}^{-1} G_{12} G_{22}^{-1} \\
0 & G_{22}^{-1}
\end{array}\right)
= \left(\begin{array}{cccc|cc}
g_{11}^{-1} & 0 & 0 & 0 & -g_{11}^{-1} g_{15}  \\
0 & g_{22}^{-1} & 0 & 0 & 0   \\
0 & 0 & g_{33}^{-1} & 0 & -g_{33}^{-1} g_{35}  \\
0 & 0 & 0 & g_{44}^{-1} & 0  \\
\hline 0 & 0 & 0 & 0 & I 
\end{array}\right)
$$
with
\begin{align*}
\left(\begin{array}{lllll}
1 & 0 & 0 & 0 & 0  \\
0 & 0 & 1 & 0 & 0  
\end{array}\right) A_{\textsc{l}}^{\textup{r} -1} 
& =\left(\begin{array}{ccccc}
g_{11}^{-1} & 0 & 0 & 0 & -g_{11}^{-1} g_{15}  \\
0 & 0 & g_{33}^{-1} & 0 & -g_{33}^{-1} g_{35}  
\end{array}\right) \\
& = \left(\begin{array}{cc}
g_{11}^{-1} & 0 \\
0 & g_{33}^{-1}
\end{array}\right)\left(\begin{array}{ccccc}
1 & 0 & 0 & 0 & -g_{15}  \\
0 & 0 & 1 & 0 & -g_{35} 
\end{array}\right). 
\end{align*}

\noindent\textbf{Computing $D_{\textsc{L}}^{\textup{r}} $:}
Let $\left(\eta^{\textup{r}}, \eta_1^{\textup{r}}, \eta_0^{\textup{r}} \right)$ denote the value of $\left(\eta, \eta_1, \eta_0 \right)$ evaluated at $\theta=\theta^{\textup{r}}$. From \eqref{eq:Lin_ME_cluster}, we have
\begin{align}
\eta^{\textup{r}}
=\left(\begin{array}{c}
\eta_1^{\textup{r}} \\
\eta_0^{\textup{r}} \\
\eta_x^{\textup{r}}
\end{array}\right) 
\text { with } \quad 
\begin{aligned}
\eta_1^{\textup{r}} 
& =Z_i \cdot\left\{Y_{ij}(1)-\left(x_{ij}-\mu_x\right)^{\mathrm{T}} \gamma_1^{\textup{r}}-\mu_1^{\textup{r}}\right\}\left(\begin{array}{c}
1 \\
x_{ij}-\mu_x
\end{array}\right), \\
\eta_0^{\textup{r}} 
& =(1-Z_i) \cdot\left\{Y_{ij}(0)-\left(x_i-\mu_x\right)^{\mathrm{T}} \gamma_0^{\textup{r}}-\mu_0^{\textup{r}}\right\}\left(\begin{array}{c}
1 \\
x_{ij}-\mu_x
\end{array}\right), \\
\eta_x^{\textup{r}}
&= x_{ij}-\mu_x.
\end{aligned} \label{eq:Lin_eta_oracle}
\end{align}
The $D_{\textsc{l}}^{\textup{r}} $ matrix equals
$$
D_{\textsc{l}}^{\textup{r}} 
=\left(\begin{array}{cc|cc|c}
d_{11} & d_{12} & 0 & 0 & d_{15} \\
d_{12}^{\mathrm{T}} & d_{22} & 0 & 0 & d_{25} \\
\hline 0 & 0 & d_{33} & d_{34} & d_{35}   \\
0 & 0 & d_{34}^{\mathrm{T}} & d_{44}  & d_{45} \\
\hline
d_{15}^{\mathrm{T}} & d_{25}^{\mathrm{T}} & d_{35}^{\mathrm{T}} & d_{45}^{\mathrm{T}} & d_{55}
\end{array}\right),
$$
where 
\begin{align*}
d_{11} 
& = e \frac{1}{M} \sum_{i=1}^M  \V \left( {\varepsilon}^{\textup{r}}_{i\cdot,\textsc{l}}(1)\right), \quad
d_{33} 
= (1-e) \frac{1}{M} \sum_{i=1}^M  \V \left( {\varepsilon}^{\textup{r}}_{i\cdot,\textsc{l}}(0)\right), \\
d_{55} 
&= \frac{1}{M}\sum_{i=1}^M \mathbb{E}\left( \left(\frac{M}{N} \sum_{j=1}^{n_i}(x_{ij} - \mu_x) \right)\left(\frac{M}{N} \sum_{j=1}^{n_i}(x_{ij} - \mu_x) \right)^{\mathrm{T}}  \right)
=\frac{1}{M}\sum_{i=1}^M \V \left( {x}_{i\cdot}\right) \\
d_{15}
&= \frac{1}{M}\sum_{i=1}^M \mathbb{E}\left[Z_i  \frac{M}{N} \sum_{j=1}^{n_i}\left\{Y_{ij}(1)-\left(x_{ij}-\mu_x\right)^{\mathrm{T}} \gamma_1^{\textup{r}}-\mu_1^{\textup{r}}\right\} \frac{M}{N} \sum_{j=1}^{n_i} \left(x_{ij}-\mu_x^{\textup{r}}\right)^{\mathrm{T}}\right] \\
&= e \cdot\left(\mathbb{E}\left[\left\{Y_{ij}(1)-\mu_1^{\textup{r}} \right\}\left(x_i-\mu_x\right)^{\mathrm{T}}\right]-\gamma_1^{{\textup{r}} \mathrm{~T}} \mathbb{E}\left\{\left(x_{ij}-\mu_x\right)\left(x_{ij}-\mu_x\right)^{\mathrm{T}}\right\}\right)
=0, \\
d_{35}
&= \frac{1}{M}\sum_{i=1}^M \mathbb{E}\left[(1-Z_i)  \frac{M}{N} \sum_{j=1}^{n_i}\left\{Y_{ij}(0)-\left(x_{ij}-\mu_x\right)^{\mathrm{T}} \gamma_0^{\textup{r}}-\mu_0^{\textup{r}}\right\} \frac{M}{N} \sum_{j=1}^{n_i} \left(x_{ij}-\mu_x \right)^{\mathrm{T}}\right] \\
&= (1-e) \cdot\left(\mathbb{E}\left[\left\{Y_{ij}(0)-\mu_0^{\textup{r}}\right\}\left(x_{ij}-\mu_x \right)^{\mathrm{T}}\right]-(\gamma_0^{{\textup{r}}})^{\mathrm{T}} \mathbb{E}\left\{\left(x_{ij}-\mu_x\right)\left(x_{ij}-\mu_x \right)^{\mathrm{T}}\right\}\right)=0,
\end{align*}
We omit other terms that are irrelevant here.

\noindent\textbf{Compute $A_{\textsc{L}}^{\textup{r}-1} D_{\textsc{L}}^{\textup{r}} \left(A_{\textsc{L}}^{\textup{r}-\mathrm{T}} \right)$}
By direct algebra,
\begin{align*}
& \left(\begin{array}{cccc|c}
1 & 0 & 0 & 0 & -g_{15}  \\
0 & 0 & 1 & 0 & -g_{35}  
\end{array}\right)\left(\begin{array}{cccc|c}
d_{11} & d_{12} & 0 & 0 & 0   \\
d_{12}^{\mathrm{T}} & d_{22} & 0 & 0 & d_{25} \\
0 & 0 & d_{33} & d_{34} & 0   \\
0 & 0 & d_{34}^{\mathrm{T}} & d_{44} & d_{45}   \\
\hline 0 & d_{25}^{\mathrm{T}} & 0 & d_{45}^{\mathrm{T}} & d_{55} 
\end{array}\right)\left(\begin{array}{cc}
1 & 0 \\
0 & 0 \\
0 & 1 \\
0 & 0 \\
\hline-g_{15}^{\mathrm{T}} & -g_{35}^{\mathrm{T}} 
\end{array}\right) \\
=& \left(\begin{array}{ll}
d_{11} & 0   \\
0 &  d_{33}
\end{array}\right)
+ \left(\begin{array}{c}
g_{15} \\
g_{35}
\end{array}\right)
\left(\begin{array}{c}
d_{55}
\end{array}\right)
\left(\begin{array}{cc}
g_{15} & g_{35}
\end{array}\right).
\end{align*}
Then
\begin{align*}
& V_{\textsc{l}}
= \operatorname{cov}\left\{\left(\begin{array}{c}
\hat{Y}_{\textsc{l}}(1) \\
\hat{Y}_{\textsc{l}}(0)
\end{array}\right)\right\}
= \operatorname{cov}\left\{\left(\begin{array}{lllll}
1 & 0 & 0 & 0 & 0 \\
0 & 0 & 1 & 0 & 0
\end{array}\right)
\left(\begin{array}{c}
\hat{Y}_{\textsc{l}}(1) \\
\hat{\gamma}_1 \\
\hat{Y}_{\textsc{l}}(0) \\
\hat{\gamma}_0 \\
\bar{x}
\end{array}\right)\right\} \\
=& \left(\begin{array}{lllll}
1 & 0 & 0 & 0  & 0 \\
0 & 0 & 1 & 0  & 0
\end{array}\right) 
A_{\textsc{l}}^{\textup{r}-1}  D_{\textsc{l}}^{\textup{r}} \left(A_{\textsc{l}}^{\textup{r}-\mathrm{T}} \right) \left(\begin{array}{ll}
1 & 0 \\
0 & 0 \\
0 & 1 \\
0 & 0 \\
0 & 0
\end{array}\right) \\
=& \left(\begin{array}{cc}
g_{11}^{-1} & 0 \\
0 & g_{33}^{-1}
\end{array}\right)
\left(\begin{array}{ccccc}
1 & 0 & 0 & 0 & -g_{15}  \\
0 & 0 & 1 & 0 & -g_{35} 
\end{array}\right)
D_{\textsc{l}}^{\textup{r}}
\left(\begin{array}{cc}
1 & 0 \\
0 & 0 \\
0 & 1 \\
0 & 0 \\
-g_{15} & -g_{35}
\end{array}\right)  
\left(\begin{array}{cc}
g_{11}^{-1} & 0 \\
0 & g_{33}^{-1}
\end{array}\right)  \\
=& \left(\begin{array}{cc}
g_{11}^{-1} & 0 \\
0 & g_{33}^{-1}
\end{array}\right)
\left(\begin{array}{ll}
d_{11} & 0   \\
0 &  d_{33}
\end{array}\right)
\left(\begin{array}{cc}
g_{11}^{-1} & 0 \\
0 & g_{33}^{-1}
\end{array}\right) \\
&+ \left(\begin{array}{cc}
g_{11}^{-1} & 0 \\
0 & g_{33}^{-1}
\end{array}\right)
\left(\begin{array}{c}
g_{15} \\
g_{35}
\end{array}\right)
\left(\begin{array}{c}
d_{55}
\end{array}\right)
\left(\begin{array}{cc}
g_{15} & g_{35}
\end{array}\right)
\left(\begin{array}{cc}
g_{11}^{-1} & 0 \\
0 & g_{33}^{-1}
\end{array}\right)\\
=& \left(\begin{array}{cc}
  \frac{d_{11}}{g_{11}^2} & 0 \\
  0 & \frac{d_{33}}{g_{33}^2}
  \end{array}\right)
+ \left(\begin{array}{c}
\frac{g_{15}}{g_{11}} \\
\frac{g_{35}}{g_{33}}
\end{array}\right)
\left(\begin{array}{c}
d_{55}
\end{array}\right)
\left(\begin{array}{cc}
\frac{g_{15}}{g_{11}} &  \frac{g_{35}}{g_{33}}
\end{array}\right) \\
=& \left(\begin{array}{cc}
\frac{d_{11}}{g_{11}^2} & 0 \\
0 & \frac{d_{33}}{g_{33}^2}
\end{array}\right)
+ \left(\begin{array}{c}
\gamma_1^{\textup{r}} \\
\gamma_0^{\textup{r}}
\end{array}\right)
\left(\begin{array}{c}
d_{55}
\end{array}\right)
\left(\begin{array}{cc}
\gamma_1^{\textup{r}} &  \gamma_0^{\textup{r}}
\end{array}\right).   
\end{align*}
This ensures
\begin{align*}
V_{\textsc{l}}^{\textup{r}}
& = \operatorname{var}\left\{(1,-1)\left(\begin{array}{c}
\hat{Y}_{\textsc{l}}(1) \\
\hat{Y}_{\textsc{l}}(0)
\end{array}\right)\right\} 
=(1,-1) V_{\textsc{l}}\left(\begin{array}{c}
1 \\
-1
\end{array}\right)  \\
& = \frac{ \frac{1}{M} \sum_{i=1}^M  \V \left( {\varepsilon}^{\textup{r}}_{i\cdot,\textsc{l}}(1)\right) }{e}
+ \frac{ \frac{1}{M} \sum_{i=1}^M  \V \left( {\varepsilon}^{\textup{r}}_{i\cdot,\textsc{l}}(0)\right) }{1-e}
+ \left(\gamma_1^{\textup{r}}-\gamma_0^{\textup{r}}\right)^{\mathrm{T}} \frac{1}{M}\sum_{i=1}^M \V \left( {x}_{i\cdot}\right)
\left(\gamma_1^{\textup{r}}-\gamma_0^{\textup{r}}\right).
\end{align*} 
\paragraph*{Consistency:}
Now we show that 
\[
\hat{V}_{\textsc{lz,l}}
= \frac{ \frac{1}{M} \sum_{i=1}^M  \V \left( {\varepsilon}^{\textup{r}}_{i\cdot,\textsc{l}}(1)\right) }{e}
+ \frac{ \frac{1}{M} \sum_{i=1}^M  \V \left( {\varepsilon}^{\textup{r}}_{i\cdot,\textsc{l}}(0)\right) }{1-e} + o_{\mathbb{P}}(1).
\]
We follow the proof of Theorem 2 in \cite{SuDing2021}.
Define $\tilde{X}_{i,\textsc{l}}$ as an $n_{i} \times\left(2+2 p_{x}\right)$ matrix with row $j$ equaling $\tilde{x}_{i j,\textsc{l}}^{\mathrm{T}}= (Z_{i j}, 1-Z_{i j}, Z_{i j} \ddot{x}_{i j}^{\mathrm{T}}, (1-Z_{i j}) \ddot{x}_{i j}^{\mathrm{T}})$, stacked as an $N \times\left(2+2 p_{x}\right)$ matrix $\tilde{X}_{\textsc{l}}$. 
\citet[Lemma A8]{SuDing2021} shows the following equivalent form:
\[
\hat{V}_{\textsc{lz,l}} 
= \left[\left(\tilde{X}_{\textsc{l}}^{\mathrm{T}} \tilde{X}_{\textsc{l}}\right)^{-1}\left(\sum_{i=1}^{M} \tilde{X}_{i,\textsc{l}}^{\mathrm{T}} \hat{U}_{i,\textsc{l}} \tilde{X}_{i,\textsc{l}}\right)\left(\tilde{X}_{\textsc{l}}^{\mathrm{T}} \tilde{X}_{\textsc{l}}\right)^{-1}\right]_{(1,1)+(2,2)-2(1,2)}.
\]
Define
\begin{align*}
& G_{\textsc{l}}
= \tilde{X}_{\textsc{l}}^{\mathrm{T}}\tilde{X}_{\textsc{l}} / N
= \frac{1}{N} 
\left(
\begin{array}{cc:cc}
n_{\mathcal{T}} & 0 & \sum_{i j }Z_i{\ddot{x}_{i j}^{\mathrm{T}}} & 0 \\
0  & n_{\mathcal{C}} & 0 & \sum_{ij} (1-Z_i) \ddot{x}_{i j}^{\mathrm{T}}  \\
\hdashline 
\sum_{i j }Z_i \ddot{x}_{i j} & 0 & \sum_{ij} Z_i \ddot{x}_{i j} \ddot{x}_{i j}^{\mathrm{T}}  & 0 \\
0 & \sum_{ij} (1-Z_i) \ddot{x}_{i j} & 0 & \sum_{ij} (1-Z_i) \ddot{x}_{i j} \ddot{x}_{i j}^{\mathrm{T}} 
\end{array}
\right).  
\end{align*}
Define 
\[
\Lambda_{\textsc{l}}^{\textup{r}}
= \left(
\begin{array}{cc:cc}
e & 0 & 0  & 0 \\
0 & 1-e  & 0 & 0  \\
\hdashline 
0 & 0 &  e \sum_{i j}\mathbb{E}[ x_{i j} x_{i j}^{\mathrm{T}} ]/ N & 0 \\
0 & 0  & 0 & (1-e) \sum_{i j}\mathbb{E}[ x_{i j} x_{i j}^{\mathrm{T}} ]/ N 
\end{array}
\right) 
\]
as the expection of $G_{\textsc{l}}$.
Define $H_{\textsc{l}} = M / N^2 \sum_{i=1}^M H_{i,\textsc{l}}$ with
\begin{align*}
& H_{i,\textsc{l}} 
= \tilde{X}_{i,\textsc{l}}^{\mathrm{T}}\hat{U}_{i,\textsc{l}}\tilde{X}_{i,\textsc{l}} \\
= & \left(
\begin{array}{cccc}
Z_{i}( \sum_{j} \hat{\varepsilon}_{ij,\textsc{l}})^{2} & 0 & Z_{i} \sum_{j} \hat{\varepsilon}_{i j,\textsc{l}} \sum_{j} \hat{\varepsilon}_{ij,\textsc{l}} \ddot{x}_{i j}^{\mathrm{T}} & 0 \\
0 & \left(1-Z_{i}\right)( \sum_{j} \hat{\varepsilon}_{i j,\textsc{l}})^{2} & 0 & \left(1-Z_{i}\right) \sum_{j} \hat{\varepsilon}_{i j,\textsc{l}} \sum_{j} \hat{\varepsilon}_{i j,\textsc{l}} \ddot{x}_{i j}^{\mathrm{T}} \\
* & 0 & Z_{i} \sum_{j} \hat{\varepsilon}_{i j,\textsc{l}} \ddot{x}_{i j} \sum_{j} \hat{\varepsilon}_{i j,\textsc{l}} \ddot{x}_{i j}^{\mathrm{T}} & 0 \\
0 & * & 0 & \left(1-Z_{i}\right) \sum_{j} \hat{\varepsilon}_{i j,\textsc{l}} \ddot{x}_{i j} \sum_{j} \hat{\varepsilon}_{i j,\textsc{l}} \ddot{x}_{i j}^{\mathrm{T}}
\end{array}
\right)
\end{align*}
where the * elements can be determined by symmetry.

Follow the proof of \citet[Theorem 2]{SuDing2021}, it is sufficient to show that
\begin{equation}
\left[G_{\textsc{l}}^{-1} H_{\textsc{l}} G_{\textsc{l}}^{-1}\right]_{(1-2,1-2)}
- \left[(\Lambda_{\textsc{l}}^{\textup{r}})^{-1} H_{\textsc{l}} (\Lambda_{\textsc{l}}^{\textup{r}})^{-1} \right]_{(1-2,1-2)}
= o_{\mathbb{P}}(1).
\label{eq:GHG}
\end{equation}
By Lemma \ref{lemma:cluster_random}, 
\begin{align*}
\left[(\Lambda_{\textsc{l}}^{\textup{r}})^{-1} H_{\textsc{l}} (\Lambda_{\textsc{l}}^{\textup{r}})^{-1} \right]_{(1,1) + (2,2) - 2(1,2)}
&= \frac{M \sum_{i=1}^{M} Z_{i}\left(\sum_{j=1}^{n_{i}} \hat{\varepsilon}_{i j,\textsc{l}}\right)^{2}}{N^{2} e^{2}}+\frac{M \sum_{i=1}^{M}\left(1-Z_{i}\right)\left(\sum_{j=1}^{n_{i}} \hat{\varepsilon}_{i j,\textsc{l}}\right)^{2}}{N^{2}(1-e)^{2}} \\
&= \frac{1}{M} \sum_{i=1}^M \left( \frac{\V( {\varepsilon}^{\textup{r}}_{i\cdot,\textsc{l}}(1) )}{e} 
+ \frac{\V( {\varepsilon}^{\textup{r}}_{i\cdot,\textsc{l}}(0) )}{1-e} \right)
+ o_{\mathbb{P}}(1).
\end{align*}
We complete the proof by showing \eqref{eq:GHG}.
Define $\Delta=G_{\textsc{l}}-\Lambda_{\textsc{l}}^{\textup{r}}$, and $\Psi=G_{\textsc{l}}^{-1}-(\Lambda_{\textsc{l}}^{\textup{r}})^{-1}$.
\begin{align*}
G_{\textsc{l}}^{-1} H_{\textsc{l}} G_{\textsc{l}}^{-1} - (\Lambda_{\textsc{l}}^{\textup{r}})^{-1} H_{\textsc{l}} (\Lambda_{\textsc{l}}^{\textup{r}})^{-1}
& = \left((\Lambda_{\textsc{l}}^{\textup{r}})^{-1}+\Psi\right) H_{\textsc{l}}\left((\Lambda_{\textsc{l}}^{\textup{r}})^{-1}+\Psi\right)-(\Lambda_{\textsc{l}}^{\textup{r}})^{-1} H_{\textsc{l}} (\Lambda_{\textsc{l}}^{\textup{r}})^{-1} \\
& = \Psi H_{\textsc{l}} (\Lambda_{\textsc{l}}^{\textup{r}})^{-1}+(\Lambda_{\textsc{l}}^{\textup{r}})^{-1} H_{\textsc{l}} \Psi+\Psi H_{\textsc{l}} \Psi,  
\end{align*}
with where $\Psi_{11}, \Psi_{12}, \Psi_{21}$, and $\Psi_{22}$ are $2 \times 2,2 \times 2 p_x, 2 p_x \times 2$, and $2 p_x \times 2 p_x$ submatrices of $\Psi$, respectively, and the $*$ elements do not matter in the proof. Therefore, by Lemma \ref{lemma:cluster_random},
\begin{align*}
{\left[\Psi H_{\textsc{l}} (\Lambda_{\textsc{l}}^{\textup{r}})^{-1}\right]_{(1-2,1-2)} } 
& =\Psi_{11} H_{\textsc{l},11} (\Lambda_{\textsc{l}}^{\textup{r}})_{11}^{-1}+\Psi_{12} H_{\textsc{l},21} (\Lambda_{\textsc{l}}^{\textup{r}})_{11}^{-1} \\
& =O_{\mathbb{P}}\left(\Omega^{1 / 2}\right) O_{\mathbb{P}}(M \Omega) O_{\mathbb{P}}(1)+O_{\mathbb{P}}\left(\Omega^{1 / 2}\right) O_{\mathbb{P}}(M \Omega) O_{\mathbb{P}}(1)=o_{\mathbb{P}}(1), \\
\Psi H_{\textsc{l}} \Psi 
& =O_{\mathbb{P}}\left(\Omega^{1 / 2}\right) O_{\mathbb{P}}(M \Omega) O_{\mathbb{P}}\left(\Omega^{1 / 2}\right)=O_{\mathbb{P}}\left(M \Omega^2\right)=o_{\mathbb{P}}(1), 
\end{align*}
which imply \eqref{eq:GHG}.
Therefore, we complete the proof.

\paragraph*{Anti-conservativeness}
The anti-conservativeness of $\hat{V}_{\textsc{lz},\textsc{l}}$ follows immediately.

\end{proof}

\begin{proof}[Proof of Theorem \ref{thm:cluster_adj_m}]
Recall Lemma \ref{lemma:Lin_beta}, $\theta=\left(\mu_1, \gamma_1, \mu_0, \gamma_0\right)= \left(\hat{Y}_{\textsc{l}}(1), \hat{\gamma}_1, \hat{Y}_{\textsc{l}}(0), \hat{\gamma}_0\right)$ jointly solves
$$
0= N^{-1} \sum_{ij} \eta\left(Y_{ij}(1), Y_{ij}(0), x_{ij}, Z_i; \theta\right)
= N^{-1} \sum_{ij}  \left(\begin{array}{c}
\eta_{1}\left(Y_{ij}(1), x_{ij}, Z_i; \mu_1, \gamma_1 \right) \\
\eta_{0}\left(Y_{ij}(0), x_{ij}, Z_i; \mu_0, \gamma_0 \right) 
\end{array}\right)
$$
where 
\begin{align}
\eta=
\left(\begin{array}{c}
\eta_1 \\
\eta_0 
\end{array}\right) 
\text { with } 
\quad \begin{aligned}
\eta_1 
& =Z_i \cdot\left\{y_{ij}-\left(x_{ij}-\bar{x}\right)^{\mathrm{T}} \gamma_1-\mu_1\right\}
\left(\begin{array}{c}
1 \\
x_{ij}-\bar{x}
\end{array}\right), \\
\eta_0 
& =(1-Z_i)  \cdot\left\{y_{ij}-\left(x_{ij}-\bar{x}\right)^{\mathrm{T}} \gamma_0-\mu_0\right\}
\left(\begin{array}{c}
1 \\
x_{ij}-\bar{x}
\end{array}\right).
\end{aligned} \label{eq:Lin_ME_mixed_cluster}
\end{align}
Direct algebra ensures that $\theta=\theta^\textup{m}=\left(\mu_1^\textup{m}, \gamma_1^\textup{m}, \mu_0^\textup{m}, \gamma_0^\textup{m} \right)$ solves $N^{-1} \sum_{ij} E\{\eta(\theta)\mid x\}=0$. 
Theorem \ref{thm:ME_mixed_cluster} ensures that 
\begin{align}
\sqrt{M}\left\{\left(\begin{array}{c}
\hat{Y}_{\textsc{l}}(1) \\
\hat{\gamma}_1 \\
\hat{Y}_{\textsc{l}}(0) \\
\hat{\gamma}_0 
\end{array}\right)-\left(\begin{array}{c}
\mu_1^\textup{m} \\
\gamma_1^\textup{m} \\
\mu_0^\textup{m} \\
\gamma_0^\textup{m} 
\end{array}\right)\right\} 
\overset{\textup{d}}{\to} 
\mathcal{N}\left(0, A_{\textsc{l}}^{\textup{m}-1}  D_{\textsc{l}}^{\textup{m}} \left(A_{\textsc{l}}^{\textup{m}-\mathrm{T}} \right) \right), \label{eq:Lin_normal_mixed_cluster}
\end{align}
where
\begin{align*}
A_{\textsc{l}}^{\textup{m}}
&= \frac{1}{M} \sum_{i=1}^M \mathbb{E}\left[ \sum_{j=1}^{n_i} \frac{\partial}{\partial \beta^\mathrm{T}} \psi(y_{ij}, Z_i, x_{ij};\theta^\textup{m} ) \mid x_{i} \right] \\
D_{\textsc{l}}^\textup{m} 
&= \frac{1}{M} \sum_{i=1}^M \V\left[ \sum_{j=1}^{n_i} \psi(y_{ij}, Z_i, x_{ij};\theta^\textup{m}) \mid x_{i} \right].
\end{align*}
We compute below $A_{\textsc{l}}^{\textup{m}}$ and $D_{\textsc{l}}^{\textup{m}} $, respectively.

\paragraph*{Computing $D_{\textsc{L}}^{\textup{m}} $:}
The $D_{\textsc{l}}^{\textup{m}} $ matrix equals
\begin{align}
D_{\textsc{l}}^{\textup{m}} 
= \frac{1}{M} \sum_{i=1}^M \V\left[ \sum_{j=1}^{n_i} \eta(y_{ij},x_{ij};\theta^\textup{m}) \mid x_{i} \right]
= \left(\begin{array}{cc|cc}
d_{11} & d_{12} & d_{13} & d_{14}   \\
d_{12}^{\mathrm{T}} & d_{22} & d_{23} & d_{24}  \\
\hline 
d_{13}^{\mathrm{T}} & d_{23}^{\mathrm{T}} & d_{33} & d_{34} \\
d_{14}^{\mathrm{T}} & d_{24}^{\mathrm{T}} & d_{34}^{\mathrm{T}} & d_{44}  
\end{array}\right),
\label{eq:Lin_B_mixed_cluster} 
\end{align}
where
\begin{align*}
d_{11}
=& \frac{1}{M} \sum_{i=1}^M \V\left[ \sum_{j=1}^{n_i} \eta_1(y_{ij},x_{ij};\theta^\textup{m}) \mid x_{i} \right]
= \frac{N^2}{M^2} \frac{1}{M} \sum_{i=1}^M \left( e \mathbb{E}(\tilde{\varepsilon}^{\textup{m}}_{i\cdot,\textsc{l}}(1)^2 \mid x) - e^2 \mathbb{E}(\tilde{\varepsilon}^{\textup{m}}_{i\cdot,\textsc{l}}(1))^2 \mid x_i \right), \\
d_{22}
=& \frac{1}{M} \sum_{i=1}^M \V\left[  \sum_{j=1}^{n_i} \eta_0(y_{ij},x_{ij};\theta^\textup{m}) \mid x_{i} \right]
= \frac{N^2}{M^2} \frac{1}{M} \sum_{i=1}^M \left( (1-e) \mathbb{E}(\tilde{\varepsilon}^{\textup{m}}_{i\cdot,\textsc{l}}(0)^2 \mid x_i) - (1-e)^2 \mathbb{E}(\tilde{\varepsilon}^{\textup{m}}_{i\cdot,\textsc{l}}(0))^2 \mid x_i \right), \\
d_{13} 
=& - \frac{N^2}{M^2} e(1-e) \frac{1}{M}\sum_{i=1}^M \mathbb{E}(\tilde{\varepsilon}^{\textup{m}}_{i\cdot,\textsc{l}}(1) \mid x_i ) \mathbb{E}(\tilde{\varepsilon}^{\textup{m}}_{i\cdot,\textsc{l}}(0) \mid x_i).
\end{align*}

\paragraph*{Compute $A_{\textsc{L}}^{\textup{m}}$:}
With a slight abuse of notation, let $\frac{\partial}{\partial \mu_1} \eta_1^{\textup{m}}$ denote the value of $\frac{\partial}{\partial \mu_1} \eta_1$ evaluated at $\theta^{\textup{m}}$. Similarly define other partial derivatives. 
From \eqref{eq:Lin_ME_mixed_cluster}, we have
$$
\left.\frac{\partial}{\partial\left(\mu_1, \gamma_1, \mu_0, \gamma_0 \right)} \eta\right|_{\theta=\theta^{\textup{m}}}
=\left(\begin{array}{ccc}
\frac{\partial}{\partial\left(\mu_1, \gamma_1^{\mathrm{T}}\right)} \eta_1^{\textup{m}} & 0  \\
0 & \frac{\partial}{\partial\left(\mu_0, \gamma_0^{\mathrm{T}}\right)} \eta_0^{\textup{m}}
\end{array}\right),
$$
where
$$
\begin{aligned}
\frac{\partial}{\partial\left(\mu_1, \gamma_1^{\mathrm{T}}\right)} \eta_1^{\textup{m}} 
& = - Z _i \cdot\left(\begin{array}{c}
1 \\
x_{ij} - \bar{x}
\end{array}\right)
\left(1,\left(x_{ij}-\bar{x} \right)^{\mathrm{T}}\right).
\end{aligned}
$$
Accordingly, we have
\begin{align*}
A_{\textsc{l}}^{\textup{m}} 
= - \frac{1}{M} \sum_{i=1}^M \left.\mathbb{E}\left( \sum_{j=1}^{n_i} \frac{\partial}{\partial\left(\mu_1, \gamma_1, \mu_0, \gamma_0 \right)} \eta \mid x_i \right)\right|_{\theta=\theta^{\textup{m}}}
= \left(\begin{array}{cc|cc}
g_{11} & g_{12} & 0 & 0 \\
g_{21} & g_{22} & 0 & 0 \\
\hline 0 & 0 & g_{33} & g_{34}  \\
0 & 0 & g_{43} & g_{44} 
\end{array}\right),
\end{align*}
where
\begin{align*}
& \left(\begin{array}{ll}
g_{11} & g_{12} \\
g_{21} & g_{22}
\end{array}\right)
= - \frac{1}{M} \sum_{i=1}^M  \mathbb{E}\left(  \sum_{j=1}^{n_i} \frac{\partial}{\partial\left(\mu_1^{\mathrm{T}}, \gamma_1^{\mathrm{T}}\right)} \eta_1^{\textup{m}}\mid x_i \right) \\
&= \frac{1}{M} \sum_{i=1}^M \mathbb{E}\left\{  \sum_{j=1}^{n_i} Z_i  \left(\begin{array}{c}
1 \\
x_{ij}-\bar{x}
\end{array}\right)\left(1,\left(x_{ij}-\bar{x}\right)^{\mathrm{T}}\right) \mid x_i \right\} 
= \frac{N}{M} \cdot e 
\left(\begin{array}{cc}
1 & 0 \\
0 & \frac{1}{N} \sum_{i=1}^M \sum_{j=1}^{n_i} (x_{ij}-\bar{x})(x_{ij}-\bar{x})^{\mathrm{T}}
\end{array}\right), 
\end{align*}
and by symmetry,
\[
\left(\begin{array}{ll}
g_{33} & g_{34} \\
g_{43} & g_{44}
\end{array}\right)
= \frac{N}{M} \cdot (1-e) \left(\begin{array}{cc}
1 & 0 \\
0 & \frac{1}{N} \sum_{i=1}^M \sum_{j=1}^{n_i} (x_{ij}-\bar{x})(x_{ij}-\bar{x})^{\mathrm{T}}
\end{array}\right).
\]
By the formula of block matrix inverse, we have
$$
A_{\textsc{l}}^{\textup{m}-1} 
=\left(\begin{array}{cccc}
g_{11}^{-1} & 0 & 0 & 0  \\
0 & g_{22}^{-1} & 0 & 0   \\
0 & 0 & g_{33}^{-1} & 0   \\
0 & 0 & 0 & g_{44}^{-1} 
\end{array}\right)
$$
with
\begin{align}
\left(\begin{array}{llll}
1 & 0 & 0 & 0  \\
0 & 0 & 1 & 0 
\end{array}\right) A_{\textsc{l}}^{\textup{m}-1}  
& =\left(\begin{array}{cccc}
g_{11}^{-1} & 0 & 0 & 0  \\
0 & 0 & g_{33}^{-1} & 0 
\end{array}\right) 
=\left(\begin{array}{cc}
g_{11}^{-1} & 0 \\
0 & g_{33}^{-1}
\end{array}\right)\left(\begin{array}{cccc}
1 & 0 & 0 & 0  \\
0 & 0 & 1 & 0  
\end{array}\right). \label{eq:Lin_A_mixed_cluster}
\end{align}

\paragraph*{Compute $A_{\textsc{L}}^{\textup{m}-1}  D_{\textsc{L}}^{\textup{m}} \left(A_{\textsc{L}}^{\textup{m}-\mathrm{T}} \right)$:}
Equations \eqref{eq:Lin_meat_mixed}, \eqref{eq:Lin_normal_mixed_cluster}, \eqref{eq:Lin_B_mixed_cluster}, and \eqref{eq:Lin_A_mixed_cluster} together ensure
\begin{align*}
& V_{\textsc{l}}
= \operatorname{cov}\left\{\left(\begin{array}{c}
\hat{Y}_{\textsc{l}}(1) \\
\hat{Y}_{\textsc{l}}(0)
\end{array}\right) \mid x \right\}
= \operatorname{cov}\left\{\left(\begin{array}{llll}
1 & 0 & 0 & 0   \\
0 & 0 & 1 & 0 
\end{array}\right)\left(\begin{array}{c}
\hat{Y}_{\textsc{l}}(1) \\
\hat{\gamma}_1 \\
\hat{Y}_{\textsc{l}}(0) \\
\hat{\gamma}_0 
\end{array}\right) \mid x \right\} \\
=& \left(\begin{array}{llll}
1 & 0 & 0 & 0  \\
0 & 0 & 1 & 0 
\end{array}\right) A_{\textsc{l}}^{\textup{m}-1}  D_{\textsc{l}}^{\textup{m}} \left(A_{\textsc{l}}^{\textup{m}-\mathrm{T}} \right) 
\left(\begin{array}{ll}
1 & 0 \\
0 & 0 \\
0 & 1 \\
0 & 0 
\end{array}\right) \\
=& \left(\begin{array}{cc}
g_{11}^{-1} & 0 \\
0 & g_{33}^{-1}
\end{array}\right) 
\left(\begin{array}{cc}
d_{11} & d_{13} \\
d_{13}^{\mathrm{T}} & d_{33}
\end{array}\right)
\left(\begin{array}{cc}
g_{11}^{-1} & 0 \\
0 & g_{33}^{-1}
\end{array}\right) \\
=& \frac{M^2}{N^2} \left(\begin{array}{cc}
\frac{d_{11}}{e^2 } & \frac{d_{13}}{e(1-e) } \\
\frac{d_{13}^{\mathrm{T}}}{e(1-e) } & \frac{d_{22}}{(1-e)^2 }
\end{array}\right).
\end{align*}
This ensures
\begin{align*}
& V_{\textsc{l}}^{\textup{m}}
= \operatorname{var}\left(\hat{\beta}_{\textsc{l}} \mid x \right) 
= \operatorname{var}\left\{\hat{Y}_{\textsc{l}}(1)-\hat{Y}_{\textsc{l}}(0) \mid x\right\}
= \operatorname{var}\left\{(1,-1)\left(\begin{array}{c}
\hat{Y}_{\textsc{l}}(1) \\
\hat{Y}_{\textsc{l}}(0)
\end{array}\right) \mid x \right\} 
= (1,-1) V_{\textsc{l}}\left(\begin{array}{c}
1 \\
-1
\end{array}\right) \\
=& \frac{M^2}{N^2} \cdot \left( \frac{d_{11}}{e^2} - 2\frac{d_{13}}{e(1-e)} + \frac{d_{33}}{(1-e)^2} \right) \\
=& \frac{ \frac{1}{M} \sum_{i=1}^M \left( e \mathbb{E}( {\varepsilon}^{\textup{m}}_{i\cdot,\textsc{l}}(1)^2 \mid x) - e^2 \mathbb{E}( {\varepsilon}^{\textup{m}}_{i\cdot,\textsc{l}}(1))^2 \mid x_i \right) }{e^2} + \frac{\frac{1}{M} \sum_{i=1}^M \left( e^2 \mathbb{E}( {\varepsilon}^{\textup{m}}_{i\cdot,\textsc{l}}(0)^2 \mid x_i ) - (1-e)^2 \mathbb{E}({\varepsilon}^{\textup{m}}_{i\cdot,\textsc{l}}(0))^2 \mid x_i \right)}{(1-e)^2} \\
& + 2 \frac{e(1-e) \frac{1}{M}\sum_{i=1}^M \mathbb{E}( {\varepsilon}^{\textup{m}}_{i\cdot,\textsc{l}}(1) \mid x_i ) \mathbb{E}( {\varepsilon}^{\textup{m}}_{i\cdot,\textsc{l}}(0) \mid x_i)}{e(1-e)} \\
=& \frac{1}{M} \sum_{i=1}^M \left( 
\frac{ \mathbb{E}\left( {\varepsilon}^{\textup{m}}_{i\cdot,\textsc{l}}(1)^2 \mid x_i \right)}{e}
+ \frac{\mathbb{E}\left( {\varepsilon}^{\textup{m}}_{i\cdot,\textsc{l}}(0)^2 \mid x_i \right)}{1-e}  - \mathbb{E}\left(  {\varepsilon}^{\textup{m}}_{i\cdot,\textsc{l}}(1) -  {\varepsilon}^{\textup{m}}_{i\cdot,\textsc{l}}(0) \mid x_i \right)^2 \right).
\end{align*}

\paragraph*{Conservativeness:}
Theorem \ref{thm:ME_mixed_cluster} ensures that 
\[
\hat{V}_{\textsc{lz,l}} 
= V_{\textsc{l}}^\textup{m}
+ B_{\textsc{l}}^\textup{m}
+ o(1;\mathbb{P}_{(y,Z)\mid x})  
\]
with 
\begin{equation*}
B_{\textsc{l}}^\textup{m}
= \left[ \left( A_{\textsc{l}}^{\textup{m}} \right)^{-1} \left( \frac{1}{M} \sum_{i=1}^M \mathbb{E}\left( \sum_{j=1}^{n_i} \eta(\theta^\textup{m}) \mid x_i \right) \mathbb{E}\left( \sum_{j=1}^{n_i} \eta(\theta^\textup{m}) \mid x_i \right)^{\mathrm{T}}  \right) (A_{\textsc{l}}^{\textup{m}})^{-{\mathrm{T}}} \right]_{(1,1)+(3,3)-2(1,3)}.
\end{equation*}
We compute the ``middle'' matrix below:
\begin{align*}
& \frac{1}{M} \sum_{i=1}^M \mathbb{E}\left(  \sum_{j=1}^{n_i} \eta(\theta^\textup{m}) \mid x_i \right) \mathbb{E}\left( \sum_{j=1}^{n_i} \eta(\theta^\textup{m}) \mid x_i \right)^{\mathrm{T}}  
= \left(
\begin{array}{cccc}
h_{11} & h_{12} & h_{13} & h_{14} \\
h_{12}^{\mathrm{T}} & h_{22} & h_{23} & h_{24} \\
h_{13}^{\mathrm{T}} & h_{23}^{\mathrm{T}} & h_{33} & h_{34} \\
h_{14}^{\mathrm{T}} & h_{24}^{\mathrm{T}} & h_{34}^{\mathrm{T}} & h_{44} \\
\end{array}
\right)
\end{align*}
where 
\begin{align*}
h_{11}
=& \frac{N^2}{M^2} e^2 \frac{1}{M} \sum_{i=1}^M \mathbb{E}( {\varepsilon}^{\textup{m}}_{i\cdot,\textsc{l}}(1) \mid x)^2 \\
h_{33}
=& \frac{N^2}{M^2} (1-e)^2 \frac{1}{M} \sum_{i=1}^M \mathbb{E}( {\varepsilon}^{\textup{m}}_{i\cdot,\textsc{l}}(0) \mid x)^2 \\
h_{13}
=& \frac{N^2}{M^2} e(1-e) \frac{1}{M} \sum_{i=1}^M \mathbb{E}( {\varepsilon}^{\textup{m}}_{i\cdot,\textsc{l}}(1) \mid x_i) \mathbb{E}({\varepsilon}^{\textup{m}}_{i\cdot,\textsc{l}}(0) \mid x).
\end{align*}
Therefore,
\begin{align*}
& B_{\textsc{l}}^\textup{m} 
= (1,-1) \left(\begin{array}{llll}
1 & 0 & 0 & 0  \\
0 & 0 & 1 & 0 
\end{array}\right) A_{\textsc{l}}^{\textup{m}-1}  
\left(
\begin{array}{cccc}
h_{11} & h_{12} & h_{13} & h_{14} \\
h_{12}^{\mathrm{T}} & h_{22} & h_{23} & h_{24} \\
h_{13}^{\mathrm{T}} & h_{23}^{\mathrm{T}} & h_{33} & h_{34} \\
h_{14}^{\mathrm{T}} & h_{24}^{\mathrm{T}} & h_{34}^{\mathrm{T}} & h_{44} \\
\end{array}
\right)
\left(A_{\textsc{l}}^{\textup{m}} \right)^{-\mathrm{T}}\left(\begin{array}{ll}
1 & 0 \\
0 & 0 \\
0 & 1 \\
0 & 0 
\end{array}\right)
\left(\begin{array}{c}
1 \\
-1
\end{array}\right) \\
=&  (1,-1) \left(\begin{array}{cc}
g_{11}^{-1} & 0 \\
0 & g_{33}^{-1}
\end{array}\right) 
\left(\begin{array}{cc}
h_{11} & h_{13} \\
h_{13}^{\mathrm{T}} & h_{33}
\end{array}\right)
\left(\begin{array}{cc}
g_{11}^{-1} & 0 \\
0 & g_{33}^{-1}
\end{array}\right)
\left(\begin{array}{c}
1 \\
-1
\end{array}\right) \\
=& \frac{M^2}{N^2} \left( \frac{h_{11}}{e^2} - 2\frac{h_{13}}{e(1-e)} + \frac{h_{33}}{(1-e)^2} \right) \\
=& \frac{1}{M} \sum_{i=1}^M \mathbb{E}( {\varepsilon}^{\textup{m}}_{i\cdot,\textsc{l}}(1) \mid x_i)^2 
+ \frac{1}{M} \sum_{i=1}^M \mathbb{E}( {\varepsilon}^{\textup{m}}_{i\cdot,\textsc{l}}(0) \mid x_i)^2
- 2 \frac{1}{M} \sum_{i=1}^M \mathbb{E}( {\varepsilon}^{\textup{m}}_{i\cdot,\textsc{l}}(1) \mid x_i) \mathbb{E}( {\varepsilon}^{\textup{m}}_{i\cdot,\textsc{l}}(0) \mid x_i) \\
=& \frac{1}{M} \sum_{i=1}^M \mathbb{E}( {\varepsilon}^{\textup{m}}_{i\cdot,\textsc{l}}(1) - {\varepsilon}^{\textup{m}}_{i\cdot,\textsc{l}}(0) \mid x )^2.
\end{align*}
Thus, we complete the proof.

\end{proof}

\subsection{Proof of useful lemmas} \label{sec:proof_lemma}

\begin{proof}[Proof of Lemma \ref{lemma:QIz}]
We only prove the result for $z=1$, and omit the proof for $z=0$. 
For the treated subsample, the OLS coefficient on the residualized covariate
is given by
\begin{align*}
\hat{\gamma}_1
&= \left( \sum_{ij} Z_i (\ddot{x}_{ij} - \bar{x}_{\mathcal{T}}) (\ddot{x}_{ij} - \bar{x}_{\mathcal{T}})^{\mathrm{T}} \right)^{-1}   
\left( \sum_{ij} Z_i (\ddot{x}_{ij} - \bar{x}_{\mathcal{T}})(y_{ij} - \bar{Y}_{\mathcal{T}}) \right).
\end{align*}
Under random design, given $\Omega = o(1)$, Lemma \ref{lemma:aij} ensures that 
\begin{align*}
& \frac{1}{N} \sum_{ij } Z_i \ddot{x}_{ij}y_{ij}
= \frac{e}{N} \sum_{ij} \mathbb{E}(\ddot{x}_{ij}Y_{ij}(1)) + O_{\mathbb{P}}(\Omega^{1/2}),  \\
& \bar{x}_{\mathcal{T}}
= \frac{1}{n_{\mathcal{T}}} \sum_{ij} Z_i \ddot{x}_{ij}
= \frac{1}{N}\sum_{i=1}^M n_i\mathbb{E}(\ddot{x}_{ij}) + O_{\mathbb{P}}(\Omega^{1/2}) , \\
& \bar{Y}_{\mathcal{T}}
= \frac{1}{n_{\mathcal{T}}} \sum_{ij} Z_i y_{ij}
= \mathbb{E}(Y_{ij}(1)) + O_{\mathbb{P}}(\Omega^{1/2}).
\end{align*}
Therefore, the numerator and denominator of $\hat{\gamma}_1$ satisfy 
\begin{align*}
\frac{1}{N} \sum_{ij} Z_i (\ddot{x}_{ij} - \bar{x}_{\mathcal{T}})(y_{ij} - \bar{Y}_{\mathcal{T}})
&= \frac{1}{N} \sum_{ij} Z_i \ddot{x}_{ij} y_{ij}
- \frac{n_{\mathcal{T}}}{N} \bar{x}_{\mathcal{T}} \bar{Y}_{\mathcal{T}}  
= \frac{e}{N} \sum_{ij} \mathbb{E}(\ddot{x}_{ij}Y_{ij}(1)) + O_{\mathbb{P}}(\Omega^{1/2}), \\
\frac{1}{N} \sum_{ij} Z_i (\ddot{x}_{ij} - \bar{x}_{\mathcal{T}}) (\ddot{x}_{ij} - \bar{x}_{\mathcal{T}})^{\mathrm{T}} 
&= \frac{1}{N} \sum_{ij} Z_i \ddot{x}_{ij} \ddot{x}_{ij}^{\mathrm{T}} 
- \frac{n_{\mathcal{T}}}{N}\bar{x}_{\mathcal{T}}\bar{x}_{\mathcal{T}}^{\mathrm{T}} 
= \frac{e}{N} \sum_{ij} \mathbb{E}(\ddot{x}_{ij}\ddot{x}_{ij}^{\mathrm{T}})
+ O_{\mathbb{P}}(\Omega^{1/2}).
\end{align*}
By Lemma \ref{lemma:inverse}, we can show that $\hat{\gamma}_1 - \gamma_1^{\textup{r}} = O_{\mathbb{P}}(\Omega^{1/2})$
with
\[
\gamma_1^{\textup{r}} 
= \left(\sum_{ij} \mathbb{E}(\ddot{x}_{ij}\ddot{x}_{ij}^{\mathrm{T}}) \right)^{-1}
\left(\sum_{ij} \mathbb{E}(\ddot{x}_{ij}Y_{ij}(1))\right)
= \left(\sum_{ij} \mathbb{E}(\ddot{x}_{ij}\ddot{x}_{ij}^{\mathrm{T}}) \right)^{-1}
\left(\sum_{ij} \mathbb{E}(\ddot{x}_{ij}\varepsilon_{ij}(1))\right).
\]
Under mixed design, given $\Omega = o(1)$, Lemma \ref{lemma:aij} ensures that 
\begin{align*}
& \frac{1}{N} \sum_{ij } Z_i \ddot{x}_{ij}y_{ij}
= \frac{e}{N} \sum_{ij} \ddot{x}_{ij} \mathbb{E}[Y_{ij}(1) \mid x_{i}] + O_{\mathbb{P}}(\Omega^{1/2}),  \\
& \bar{x}_{\mathcal{T}}
= \frac{1}{n_{\mathcal{T}}} \sum_{ij} Z_i \ddot{x}_{ij}
= O_{\mathbb{P}}(\Omega^{1/2}), \\
& \bar{Y}_{\mathcal{T}}
= \frac{1}{n_{\mathcal{T}}} \sum_{ij} Z_i y_{ij}
= \frac{1}{N}\sum_{ij} \mathbb{E}[Y_{ij}(1) \mid x_{i}] + O_{\mathbb{P}}(\Omega^{1/2}).
\end{align*}
Therefore, under mixed design, the numerator and denominator of $\hat{\gamma}_1$ satisfy 
\begin{align*}
\frac{1}{N} \sum_{ij } Z_i (\ddot{x}_{ij} - \bar{x}_{\mathcal{T}})(y_{ij} - \bar{Y}_{\mathcal{T}})
&= \frac{1}{N} \sum_{ij } Z_i \ddot{x}_{ij} y_{ij}
- \frac{n_{\mathcal{T}}}{N} \bar{x}_{\mathcal{T}} \bar{Y}_{\mathcal{T}}  
= \frac{e}{N} \sum_{ij} \ddot{x}_{ij} \mathbb{E}[Y_{ij}(1) \mid x_{i}] + O_{\mathbb{P}}(\Omega^{1/2}), \\
\frac{1}{N} \sum_{ij } Z_i (\ddot{x}_{ij} - \bar{x}_{\mathcal{T}}) (\ddot{x}_{ij} - \bar{x}_{\mathcal{T}})^{\mathrm{T}} 
&= \frac{1}{N} \sum_{ij } Z_i \ddot{x}_{ij} \ddot{x}_{ij}^{\mathrm{T}} 
- \frac{n_{\mathcal{T}}}{N}\bar{x}_{\mathcal{T}}\bar{x}_{\mathcal{T}}^{\mathrm{T}} 
= \frac{e}{N} \sum_{ij} \ddot{x}_{ij}\ddot{x}_{ij}^{\mathrm{T}}
+ O_{\mathbb{P}}(\Omega^{1/2}). 
\end{align*}
Again, by Lemma \ref{lemma:inverse}, we can show that $\hat{\gamma}_1 - \gamma^{\textup{m}}_1 = O_{\mathbb{P}}(\Omega^{1/2})$
with
\[
\gamma^{\textup{m}}_1 
= \left(\sum_{ij} \ddot{x}_{ij}\ddot{x}_{ij}^{\mathrm{T}} \right)^{-1}
\left(\sum_{ij} \ddot{x}_{ij} \mathbb{E}[Y_{ij}(1) \mid x_{i}]\right).
\]
\end{proof}

\begin{proof}[Proof of Lemma \ref{lemma:cluster_random}]
Without loss of generality, we assume that $x_{ij}$ is one-dimensional.
To prove  (\ref{Equ: A1_random}), we follow the proof of Lemma A10 in \cite{SuDing2021}. Define $\bar{Y}_{\mathcal{T}}
= n_{\mathcal{T}}^{-1} \sum_{ij} Z_i y_{ij}$. 
By the definition of $\hat{\varepsilon}_{i j,\textsc{l}}$,
\begin{align*}
& \frac{M}{N^{2}} \sum_{i=1}^{M} Z_{i}\left(\sum_{j=1}^{n_{i}} \hat{\varepsilon}_{i j,\textsc{l}} \right)^{2} 
= \frac{M}{N^{2}} \sum_{i=1}^{M} Z_{i}\left(\sum_{j=1}^{n_{i}} 
\left( 
Y_{ij}(1) - \bar{Y}_{\mathcal{T}} - (x_{ij} - \bar{x}_{\mathcal{T}}) \hat{\gamma}_1 
\right) \right)^{2} \\
=& \frac{M}{N^{2}} \sum_{i=1}^{M} Z_{i}\left(\sum_{j=1}^{n_{i}}
\left[
r^{\textup{r}}_{i j}(1)
+ \mathbb{E}(Y_{ij}(1))
+ (x_{ij} - \mathbb{E}(x_{ij})) \gamma^{\textup{r}}_1
- \bar{Y}_{\mathcal{T}} - (\ddot{x}_{ij} - \bar{x}_{\mathcal{T}}) \hat{\gamma}_1 
\right]\right)^{2} \\
=& \frac{M}{N^{2}} \sum_{i=1}^{M} Z_{i}
\left(
\sum_{j=1}^{n_{i}}\left[
r^{\textup{r}}_{i j}(1)
+ (x_{i j} - \mathbb{E}(x_{ij} )) (\gamma^{\textup{r}}_1-\hat{\gamma}_1)
- \left( \bar{\varepsilon}_{\mathcal{T}} - (\bar{x}_{\mathcal{T}} - \mathbb{E}(x_{ij})) \hat{\gamma}_1 \right) 
\right]
\right)^{2} \\
=& T_{1}+T_{2}+T_{3}+T_{4}-T_{5}-T_{6}
\end{align*}
where 
\begin{align*}
T_{1}
&=\frac{M}{N^{2}} \sum_{i=1}^{M} Z_{i}\left(\sum_{j=1}^{n_{i}} r^{\textup{r}}_{i j}(1)\right)^{2}, 
\quad 
T_{2}
= \frac{M}{N^{2}} \sum_{i=1}^{M} Z_{i}\left(\sum_{j=1}^{n_{i}} 
(x_{i j} - \mathbb{E}(x_{ij} )) \right)^{2}
\left(\gamma^{\textup{r}}_1-\hat{\gamma}_1 \right)^{2}, \\
T_{3}
&=\frac{M}{N^{2}} \sum_{i=1}^{M} Z_{i} n_{i}^{2}
\left( \bar{\varepsilon}_{\mathcal{T}} - (\bar{x}_{\mathcal{T}} - \mathbb{E}(x_{ij})) \hat{\gamma}_1 \right)^{2}, \\
T_{4}
&= \frac{2 M}{N^{2}} \sum_{i=1}^{M} Z_{i} \sum_{j=1}^{n_{i}} r_{i j}^{\textup{r}}(1) \sum_{j=1}^{n_{i}} (x_{i j} - \mathbb{E}(x_{ij} ))
\left(\gamma^{\textup{r}}_1-\hat{\gamma}_1 \right), \\
T_{5}
&=\frac{2 M}{N^{2}} \sum_{i=1}^{M} Z_{i} \sum_{j=1}^{n_{i}} r^{\textup{r}}_{i j}(1) n_{i}
\left( \bar{\varepsilon}_{\mathcal{T}} - (\bar{x}_{\mathcal{T}} - \mathbb{E}(x_{ij})) \hat{\gamma}_1 \right), \\
T_{6}
&= \frac{2 M}{N^{2}} \sum_{i=1}^{M} Z_{i} \sum_{j=1}^{n_{i}} 
(x_{i j} - \mathbb{E}(x_{ij} ))
\left( \gamma^{\textup{r}}_1-\hat{\gamma}_1 \right)
n_{i}
\left( \bar{\varepsilon}_{\mathcal{T}} - (\bar{x}_{\mathcal{T}} - \mathbb{E}(x_{ij})) \hat{\gamma}_1 \right).
\end{align*}
Similarly, we claim that except for $T_1$, all other terms $T_2-T_6$ are of order $o_{\mathbb{P}}(1)$. 
We show that $T_4 = o_{\mathbb{P}}(1)$ and omit the proofs for other terms. It is bounded from the above by
\begin{align*}
|T_4|
\le & \frac{2 M}{N^{2}} \sum_{i=1}^{M} 
\left| \sum_{j=1}^{n_{i}} r_{i j}^{\textup{r}}(1) \sum_{j=1}^{n_{i}} (x_{i j} - \mathbb{E}(x_{ij} )) \right|
\left| \gamma^{\textup{r}}_1-\hat{\gamma}_1 \right| \\
\le & \frac{M}{N^{2}} \sum_{i=1}^{M} 
\left[ \left(\sum_{j=1}^{n_{i}} r_{i j}^{\textup{r}}(1) \right)^2 
+ \left(\sum_{j=1}^{n_{i}} (x_{i j} - \mathbb{E}(x_{ij} )) \right)^2 \right]
\left| \gamma^{\textup{r}}_1-\hat{\gamma}_1 \right|.
\end{align*}
By Lemma \ref{lemma:QIz}, $\gamma^{\textup{r}}_1-\hat{\gamma}_1 = O_{\mathbb{P}}\left(\Omega^{1 / 2}\right)$
and by assumption \ref{asu:cluster},
\begin{align*}
\mathbb{E}\left[ \frac{M}{N^{2}} \sum_{i=1}^{M} 
\left(\sum_{j=1}^{n_{i}} r_{i j}^{\textup{r}}(1) \right)^2 \right]
& \leq_{\mathrm{HI}} \frac{M}{N^{2}} \sum_{i=1}^{M} n_i
\sum_{j=1}^{n_{i}} \mathbb{E}\left[ (r_{i j}^{\textup{r}}(1))^2 \right]
= O(M\Omega), \\
\V\left[ \frac{M}{N^{2}} \sum_{i=1}^{M} 
\left(\sum_{j=1}^{n_{i}} r_{i j}^{\textup{r}}(1) \right)^2 \right]
& \leq \frac{M^2}{N^{4}} \sum_{i=1}^{M} 
\mathbb{E} \left[ \left(\sum_{j=1}^{n_{i}} r_{i j}^{\textup{r}}(1) \right)^4 \right]
\le_{\mathrm{HI}} \frac{M^2}{N^{4}} \sum_{i=1}^{M} n_i^3 \sum_{j=1}^{n_{i}}
\mathbb{E} [ r_{i j}^{\textup{r}}(1)^4 ]
\leq O\left(M^{2} \Omega^{3}\right) 
= o(1), 
\end{align*}
and
\begin{align*}
\mathbb{E}\left[ \frac{M}{N^{2}} \sum_{i=1}^{M} 
\left(\sum_{j=1}^{n_{i}} (x_{i j} - \mathbb{E}(x_{ij} )) \right)^2 \right]
& \leq_{\mathrm{HI}} \frac{M}{N^{2}} \sum_{i=1}^{M} n_i
\sum_{j=1}^{n_{i}}  \mathbb{E} \left[(x_{i j} - \mathbb{E}(x_{ij} ))^2\right]
= O(M\Omega), \\
\V\left[ \frac{M}{N^{2}} \sum_{i=1}^{M} 
\left(\sum_{j=1}^{n_{i}} (x_{i j} - \mathbb{E}(x_{ij} )) \right)^2 \right]
& \leq \frac{M^2}{N^{4}} \sum_{i=1}^{M} 
\mathbb{E} \left[ \left(\sum_{j=1}^{n_{i}} (x_{i j} - \mathbb{E}(x_{ij} )) \right)^4 \right] \\
& \le_{\mathrm{HI}} \frac{M^2}{N^{4}} \sum_{i=1}^{M} n_i^3 \sum_{j=1}^{n_{i}}
\mathbb{E} ((x_{i j} - \mathbb{E}(x_{ij} )) ^4)
\leq O\left(M^{2} \Omega^{3}\right) 
= o(1),
\end{align*}
we have $|T_4|\le O_{\mathbb{P}}(M\Omega) O_{\mathbb{P}} (\Omega^{1 / 2}) = o_{\mathbb{P}}(1)$ using the assumption $\Omega=o(M^{-2 / 3})$.

To finish the proof of (\ref{Equ: A1_random}), we only need to verify that $T_{1}$ differs from its mean by a term of order $o_{\mathbb{P}}(1)$, which follows from Chebyshev's inequality and the variance calculation:
\[
\operatorname{\mathrm{var}}\left(T_{1}\right)
\leq \frac{M^{2}}{N^{4}} \sum_{i=1}^{M}\mathbb{E}\left[\left(\sum_{j=1}^{n_{i}} r^{\textup{r}}_{i j}(1)\right)^{4} \right]
\leq_{\mathrm{HI}} \frac{M^{2}}{N^{4}} \sum_{i=1}^{M} n_{i}^{3} \sum_{j=1}^{n_{i}} \mathbb{E}\left[ r^{\textup{r}}_{i j}(1)^{4} \right]
\leq O\left(M^{2} \Omega^{3}\right) 
= o(1).
\]
Second, we prove (\ref{Equ: A2_random}). 
Define $\tilde{\varepsilon}_{ij}^{\textup{r}}(1) = \varepsilon_{ij}^{\textup{r}}(1) - (x_{ij}-\mathbb{E}(x_{ij} )) \hat{\gamma}_1 - (\bar{\varepsilon}_{\mathcal{T}} - (\bar{x}_{\mathcal{T}} - \mathbb{E}(x_{ij} )) \hat{\gamma}_1)$. 
By definition of $\hat{\varepsilon}_{i j,\textsc{l}}$,
\begin{align*}
&\frac{M}{N^{2}} \sum_{i=1}^{M} Z_{i} \sum_{j=1}^{n_{i}} \hat{\varepsilon}_{i j,\textsc{l}} \sum_{j=1}^{n_{i}} \hat{\varepsilon}_{i j,\textsc{l}} x_{i j} \\
=& \frac{M}{N^{2}} \sum_{i=1}^{M} Z_{i} \sum_{j=1}^{n_{i}}
\left[
Y_{ij}(1) - \bar{Y}_{\mathcal{T}} - (x_{ij} - \bar{x}_{\mathcal{T}}) \hat{\gamma}_1 
\right] 
\sum_{j=1}^{n_{i}} x_{i j}
\left[
Y_{ij}(1) - \bar{Y}_{\mathcal{T}} - (x_{ij} - \bar{x}_{\mathcal{T}}) \hat{\gamma}_1
\right]  \\
=& \frac{M}{N^{2}} \sum_{i=1}^{M} Z_{i} \sum_{j=1}^{n_{i}}
\tilde{\varepsilon}_{ij}^{\textup{r}}(1) 
\sum_{j=1}^{n_{i}} x_{i j}
\tilde{\varepsilon}_{ij}^{\textup{r}}(1) \\
=& T_{7}-T_{8}-T_{9}-T_{10}+T_{11}+T_{12}-T_{13}+T_{14}+T_{15}    
\end{align*}
where
\begin{align*}
T_{7}
&= \frac{M}{N^{2}} \sum_{i=1}^{M} Z_{i} \sum_{j=1}^{n_{i}} \varepsilon^{\textup{r}}_{i j}(1) \sum_{j=1}^{n_{i}} \varepsilon^{\textup{r}}_{i j}(1) (x_{i j} - \mathbb{E}(x_{ij} )), \\ 
T_{8}
&= \frac{M}{N^{2}} \sum_{i=1}^{M} Z_{i} \sum_{j=1}^{n_{i}} \varepsilon^{\textup{r}}_{i j}(1) \sum_{j=1}^{n_{i}} (x_{i j} - \mathbb{E}(x_{ij} ))^{2} \hat{\gamma}_1,\\
T_{9}
&= \frac{M}{N^{2}} \sum_{i=1}^{M} Z_{i} \sum_{j=1}^{n_{i}} \varepsilon^{\textup{r}}_{i j}(1) \sum_{j=1}^{n_{i}} (x_{i j} - \mathbb{E}(x_{ij} ))
\left(\bar{\varepsilon}_{\mathcal{T}}-(\bar{x}_{\mathcal{T}} - \mathbb{E}(x_{ij} )) \hat{\gamma}_1 \right), \\
T_{10} 
&= \frac{M}{N^{2}} \sum_{i=1}^{M} Z_{i} \sum_{j=1}^{n_{i}} x_{i j} \sum_{j=1}^{n_{i}} \varepsilon^{\textup{r}}_{i j}(1) (x_{i j} - \mathbb{E}(x_{ij} )) \hat{\gamma}_1,\\
T_{11}
&= \frac{M}{N^{2}} \sum_{i=1}^{M} Z_{i} \sum_{j=1}^{n_{i}} (x_{i j} - \mathbb{E}(x_{ij} )) \sum_{j=1}^{n_{i}} (x_{i j} - \mathbb{E}(x_{ij} ))^{2} \hat{\gamma}_1^{2}, \\
T_{12}
&= \frac{M}{N^{2}} \sum_{i=1}^{M} Z_{i}\left(\sum_{j=1}^{n_{i}} (x_{i j} - \mathbb{E}(x_{ij} ))\right)^{2} \hat{\gamma}_1 \left(\bar{\varepsilon}_{\mathcal{T}}-(\bar{x}_{\mathcal{T}} - \mathbb{E}(x_{ij} )) \hat{\gamma}_1 \right),\\
T_{13}
&= \frac{M}{N^{2}} \sum_{i=1}^{M} Z_{i} n_{i} \sum_{j=1}^{n_{i}} \varepsilon^{\textup{r}}_{i j}(1) (x_{i j} - \mathbb{E}(x_{ij} ))\left(\bar{\varepsilon}_{\mathcal{T}}-(\bar{x}_{\mathcal{T}} - \mathbb{E}(x_{ij} )) \hat{\gamma}_1\right), \\ 
T_{14} 
&=\frac{M}{N^{2}} \sum_{i=1}^{M} Z_{i} n_{i} \sum_{j=1}^{n_{i}} (x_{i j} - \mathbb{E}(x_{ij} ))^{2}\left(\bar{\varepsilon}_{\mathcal{T}}-(\bar{x}_{\mathcal{T}} - \mathbb{E}(x_{ij} )) \hat{\gamma}_1 \right) \hat{\gamma}_1,\\
T_{15}
&= \frac{M}{N^{2}} \sum_{i=1}^{M} Z_{i} n_{i} \sum_{j=1}^{n_{i}} (x_{i j} - \mathbb{E}(x_{ij} ))\left(\bar{\varepsilon}_{\mathcal{T}}-(\bar{x}_{\mathcal{T}} - \mathbb{E}(x_{ij} )) \hat{\gamma}_1 \right)^{2}.    
\end{align*}
We claim that all terms $T_{7}-T_{15}$ are of order $O_{\mathbb{P}}(M \Omega)$. We show that $T_{7}=O_{\mathbb{P}}(M \Omega)$ and omit the proofs for other terms.
It is bounded from the above by
\begin{align*}
\left|T_7\right| 
& \leq \frac{M}{N^2} \sum_{i=1}^M\left|\sum_{j=1}^{n_i} \varepsilon_{i j}^{\textup{r}}(1) \sum_{j=1}^{n_i} \varepsilon_{i j}^{\textup{r}}(1) (x_{i j} - \mathbb{E}(x_{ij} )) \right| \\ 
& \leq \frac{M}{N^2} \sum_{i=1}^M\left[\left(\sum_{j=1}^{n_i} \varepsilon_{i j}^{\textup{r}}(1)\right)^2+\left(\sum_{j=1}^{n_i} \varepsilon_{i j}^{\textup{r}}(1) (x_{i j} - \mathbb{E}(x_{ij} ))\right)^2\right].
\end{align*}
By Assumption \ref{asu:cluster},
\begin{align*}
\mathbb{E}\left[ \frac{M}{N^{2}} \sum_{i=1}^{M} 
\left(\sum_{j=1}^{n_{i}} \varepsilon_{i j}^{\textup{r}}(1) \right)^2 \right]
& \leq_{\mathrm{HI}} \frac{M}{N^{2}} \sum_{i=1}^{M} n_i
\sum_{j=1}^{n_{i}} \mathbb{E}\left[ (\varepsilon_{i j}^{\textup{r}}(1))^2 \right]
= O(M\Omega), \\
V\left[ \frac{M}{N^{2}} \sum_{i=1}^{M} 
\left(\sum_{j=1}^{n_{i}} \varepsilon_{i j}^{\textup{r}}(1) \right)^2 \right]
& \leq \frac{M^2}{N^{4}} \sum_{i=1}^{M} 
\mathbb{E} \left[ \left(\sum_{j=1}^{n_{i}} \varepsilon_{i j}^{\textup{r}}(1) \right)^4 \right] \\
& \le_{\mathrm{HI}} \frac{M^2}{N^{4}} \sum_{i=1}^{M} n_i^3 \sum_{j=1}^{n_{i}}
\mathbb{E} [ r_{i j}^{\textup{r}}(1)^4 ]
\leq O\left(M^{2} \Omega^{3}\right) 
= o(1),
\end{align*}
and 
\begin{align*}
& \mathbb{E}\left[ \frac{M}{N^{2}} \sum_{i=1}^{M} 
\left(\sum_{j=1}^{n_{i}} \varepsilon_{i j}^{\textup{r}}(1)(x_{i j} - \mathbb{E}(x_{ij} )) \right)^2 \right]
\leq_{\mathrm{HI}} \frac{M}{N^{2}} \sum_{i=1}^{M} n_i
\sum_{j=1}^{n_{i}} \mathbb{E}\left[ (\varepsilon_{i j}^{\textup{r}}(1)(x_{i j} - \mathbb{E}(x_{ij} )))^2 \right]
= O(M\Omega), \\
& V\left[ \frac{M}{N^{2}} \sum_{i=1}^{M} 
\left(\sum_{j=1}^{n_{i}} \varepsilon_{i j}^{\textup{r}}(1) (x_{i j} - \mathbb{E}(x_{ij} ))\right)^2 \right]
\leq \frac{M^2}{N^{4}} \sum_{i=1}^{M} 
\mathbb{E} \left[ \left(\sum_{j=1}^{n_{i}} \varepsilon_{i j}^{\textup{r}}(1) (x_{i j} - \mathbb{E}(x_{ij} ))\right)^4 \right] \\
& \le_{\mathrm{HI}} \frac{M^2}{N^{4}} \sum_{i=1}^{M} n_i^3 \sum_{j=1}^{n_{i}}
\mathbb{E} [ (\varepsilon_{i j}^{\textup{r}}(1)(x_{i j} - \mathbb{E}(x_{ij} )))^4 ]
\leq O\left(M^{2} \Omega^{3}\right) 
= o(1). 
\end{align*}
The proof of \eqref{Equ: A3_random} is similar to \eqref{Equ: A2_random}, and we omit it. Therefore, $H = O_{\mathbb{P}}(M \Omega)$ follows from \eqref{Equ: A1_random}–\eqref{Equ: A3_random}.

\end{proof}

\end{document}